\documentclass [leqno,11pt]{amsart}    
\usepackage{amssymb,amsmath,latexsym,amsfonts,amsbsy,amsthm,mathtools,graphicx,color,subeqnarray,cases,bm}      
\usepackage{float} 
\usepackage{hyperref} 
\usepackage{scalerel,stackengine}    
\stackMath  
\newcommand\reallywidehat[1]{%
\savestack{\tmpbox}{\stretchto{%
  \scaleto{%
    \scalerel*[\widthof{\ensuremath{#1}}]{\kern-.6pt\bigwedge\kern-.6pt}%
    {\rule[-\textheight/2]{1ex}{\textheight}}
  }{\textheight}%
}{0.5ex}}%
\stackon[1pt]{#1}{\tmpbox}%
}

\definecolor{myback}{RGB}{204,232,207}

\usepackage{cancel}
\usepackage[dvipsnames]{xcolor}
\numberwithin{equation}{section}

\allowdisplaybreaks

\numberwithin{equation}{section}

\allowdisplaybreaks

\let\al=\alpha
\let\b=\beta

\let\la=\lambda

\let\s=\sigma
\let\f=\frac

\let\om=\omega

\let\na=\nabla

\let\pa=\partial
\let\ep=\varepsilon
\let\ka=\kappa

\def\rmA{\mathrm{A}}

\def\tu{\widetilde{u}}

\def\eqdef{\buildrel\hbox{\footnotesize def}\over =}

\def\bbT{\mathbb{T}}

\newcommand{\beq}{\begin{equation}}
\newcommand{\eeq}{\end{equation}}
\newcommand{\ben}{\begin{eqnarray}}
\newcommand{\een}{\end{eqnarray}}
\newcommand{\beno}{\begin{eqnarray*}}
\newcommand{\eeno}{\end{eqnarray*}}

\newtheorem{theorem}{Theorem}[section]

\newtheorem{lemma}[theorem]{Lemma}
\newtheorem{proposition}[theorem]{Proposition}
\newtheorem{corol}[theorem]{Corollary}
\newtheorem{remark}[theorem]{Remark}

\begin{document}

\title{Asymptotic stability of two-dimensional Couette flow in a viscous fluid}

\author{Hui Li}
\address{Department of Mathematics, New York University Abu Dhabi, Saadiyat Island, P.O. Box 129188, Abu Dhabi, United Arab Emirates.}
\email{lihuiahu@126.com, lihui@nyu.edu}

\author{Nader Masmoudi}
\address{NYUAD Research Institute, New York University Abu Dhabi, Saadiyat Island, Abu Dhabi, P.O. Box 129188, United Arab Emirates\\
Courant Institute of Mathematical Sciences, New York University, 251 Mercer Street New York, NY 10012 USA}
\email{masmoudi@cims.nyu.edu}

\author{Weiren Zhao}
\address{Department of Mathematics, New York University Abu Dhabi, Saadiyat Island, P.O. Box 129188, Abu Dhabi, United Arab Emirates.}
\email{zjzjzwr@126.com, wz19@nyu.edu}

\begin{abstract}
In this paper, we study the nonlinear asymptotic stability of Couette flow for the two-dimensional Navier-Stokes equation with small viscosity $\nu>0$ in $\mathbb{T}\times\mathbb{R}$. It's generally known the nonlinear asymptotic stability of the Couette flow depends closely on the size and regularity of the initial perturbation, which yields the stability threshold problem. This work studies the relationship between the size and the regularity of the initial perturbation that makes the nonlinear asymptotic stability holds. More precisely, we proved that if the initial perturbation is in some Gevrey-$\f1s$ class with size $\ep\nu^{\b}$ where $s\geq \f{1-3\beta}{2-3\beta}$ and $\b\in [0,\f13]$, then the nonlinear asymptotic stability holds. 
\end{abstract}
\maketitle

\tableofcontents
\section{Introduction} 
\subsection{Two dimensional Navier-Stokes equations}
We consider the two-dimensional incompressible Navier-Stokes equations in $\Omega=\bbT\times\mathbb{R}$:
\begin{equation}\label{eq-nl-NS}
  \left\{
    \begin{array}{ll}
      \pa_tU+U\cdot\nabla U+\nabla P-\nu\Delta U=0,\\
      \nabla\cdot U=0,
    \end{array}
  \right.
\end{equation}
where $\nu>0$ denotes the viscosity. We denote by $U=(U^{(1)},U^{(2)})$ the velocity and $P$ the pressure. Let $W=-\pa_yU^{(1)}+\pa_xU^{(2)}$ be the vorticity, which satisfies
\begin{align*}
  \pa_t W+U\cdot\nabla W-\nu\Delta W=0.
\end{align*}
The system \eqref{eq-nl-NS} has a steady solution with $U=(y,0)$ and $W=-1$, which is the so-called Couette flow. 

In the paper, we focus on the asymptotic stability of Couette flow $(y,0)$. It is natural to introduce the perturbation. Let $u=(u^{(1)},u^{(2)})=U-(y, 0)$ and $\om=W-(-1)$, then $\omega$ satisfies  
\begin{equation}\label{eq-nl-om-chan}
  \left\{
    \begin{array}{ll}
      \partial_{t} \omega+y \partial_{x} \omega-\nu \Delta \omega=-u \cdot \nabla \omega , \\
      u=-\nabla^{\perp}(-\Delta)^{-1} \omega=\big(\pa_y(-\Delta)^{-1} \omega,-\pa_x(-\Delta)^{-1} \omega\big),\\
      \omega|_{t=0}=\omega_{i n}.
    \end{array}
  \right.
\end{equation}

We also introduce the linearized equation:
\begin{equation}\label{eq:LNS2}
  \left\{
    \begin{array}{ll}
      \partial_{t} \omega+y \partial_{x} \omega-\nu \Delta \omega=0,\\
      \psi=\Delta^{-1}\omega,\\
      \omega |_{t=0}=\omega_{in},
    \end{array}
  \right.
\end{equation}
where $\psi$ is the stream function.

Kelvin \cite{Kelvin1887} studied the linearized system \eqref{eq:LNS2}. If we denote by $\hat{\om}(t,k,\eta)$ the Fourier transform of $\om(t,x,y)$, then the solution of \eqref{eq:LNS2} can be write as
\begin{equation}\label{eq: Lin-sol}
\begin{split}
&\hat{\om}(t,k,\eta)=\hat{\om}_{in}(k,\eta+kt)\exp\left(-\nu\int_0^t|k|^2+|\eta-ks+kt|^2ds\right),\\
&\hat{\psi}(t,k,\eta)=\frac{-\hat{\om}_{in}(k,\eta+kt)}{k^2+\eta^2}\exp\left(-\nu\int_0^t|k|^2+|\eta-ks+kt|^2ds\right),
\end{split}
\end{equation}
which gives that
\begin{equation}\label{eq: ED_and_ID}
  \begin{aligned}    
&\|\pa_yP_{\neq}\psi\|_{L^2}+\langle t\rangle \|\pa_xP_{\neq}\psi\|_{L^2}\leq C\langle t\rangle^{-1}e^{-c\nu t^3}\|P_{\neq}\om_{in}\|_{H^{2}},\\
&\|P_{\neq}\om\|_{L^2}\leq C\|P_{\neq}\om_{in}\|_{L^2}e^{-c\nu t^3},
  \end{aligned}
\end{equation}
where here we denote by $P_{\neq}f=f(x,y)-\frac{1}{2\pi}\int_{\mathbb{T}}f(x,y)dx$ the projection to nonzero mode of $f$. The first inequality in \eqref{eq: ED_and_ID} is the inviscid damping and the second one is the enhanced dissipation. 

\subsection{Historical comments}
Understanding the stability and instability of laminar shear flows at a high Reynolds number has been a classical question in applied fluid mechanics since the early experiments of Reynolds. It was suggested by Lord Kelvin \cite{Kelvin1887} that indeed the flow may be stable, but the stability threshold is decreasing as $\nu\to 0$, resulting in transition at a finite Reynolds number in any real system. In \cite{BGM2017}, Bedrossian, Germain, and Masmoudi formulated the following  stability threshold problem:

{\it Given a norm $\|\cdot\|_X$, find a $\beta=\beta(X)$ so that
\begin{align*}
  &\|\omega_{in}\|_{X}\leq \nu^{\beta} \Rightarrow \text{stability},\\
&\|\omega_{in}\|_{X}\gg \nu^{\beta} \Rightarrow \text{instability}.
\end{align*}
}
In \cite{MasmoudiZhao2020cpde}, Masmoudi and Zhao reformulated this in terms of nonlinear enhanced dissipation and inviscid damping which yield asymptotic stability:

1. {\it Given a norm $\|\cdot\|_{X}$ $(X\subset L^2)$, determine a $\beta=\beta(X)$ so that if the initial vorticity satisfies $\|\om_{in}\|_{X}\ll \nu^{\beta}$, then for $t>0$
\begin{equation}\label{eq: enha-invis}
  \|\omega_{\neq}\|_{L^2_{x,y}}\leq C\|P_{\neq}\omega_{in}\|_{X}e^{-c\nu^{\frac{1}{3}}t}\quad \text{and}\quad
\|u_{\neq}\|_{L^2_{t,x,y}}\leq C\|P_{\neq}\omega_{in}\|_{X},
\end{equation}
hold for the Navier-Stokes equation \eqref{eq-nl-om-chan}.}

2. {\it Given $\beta$, is there an optimal function space $X\subset L^2$ so that if the initial vorticity satisfies $\|\omega_{in}\|_{X}\ll \nu^{\beta}$, then \eqref{eq: enha-invis} holds for the Navier-Stokes equation \eqref{eq-nl-om-chan}?}

Many works in mathematics and physics have been devoted to estimating $\beta$. In \cite{BVW2018}, the authors proved that if $X=H^{s}$ with $s>1$ and $\beta\geq \f12$, then the system is stable. In \cite{MasmoudiZhao2020cpde}, the authors proved that the asymptotic stability holds if $X=H_x^{log}L^2_y$ and $\beta\geq \f12$. This result was recently proved to be optimal in \cite{LMZ2022}. In \cite{BMV2016}, the authors proved that if  $\beta=0$, then the asymptotic holds if the initial perturbation is in some Gevrey-$m$ class with $1\leq m< 2$. In \cite{MasmoudiZhao2019}, the authors considered the asymptotic stability in higher Sobolev spaces $H^s$ with $s\geq 40$, and they proved the asymptotic stability holds for $\beta\geq \f13$. All the results show that the $\beta$ is expected to depend non-trivially on the norm $X$. {\bf In this paper, we give the first result about the relationship between the $\beta$ and $X$. }The main result is as follows:

\begin{theorem}\label{Thm: main}
Let $0<\lambda_1<\lambda_0$, $\beta\in [0,\frac{1}{3}]$ and $s(\beta)=\frac{1-3\beta}{2-3\beta}\in [0,\frac{1}{2}]$. 
For any integer $\s\geq 11$, there exist $0<\ep_0,\nu_0<1$, such that for all $0<\nu\leq \nu_0$ and $0<\ep\leq \ep_0$, if $\om_{in}$ satisfies
\begin{align*}
  \left\|u_{in}\right\|_{L^2}^2+\|\om_{in}\|_{\mathcal{G}^{s(\beta),\lambda_0,\s}}^2=\sum_{k}\int_{\mathbb{R}}e^{\lambda_0|k,\eta|^{s(\beta)}}\langle k,\eta\rangle^{2\s}|\widehat{\om}_{in}(k,\eta)|^2d\eta \leq \ep^2\nu^{2\beta}
\end{align*}
then the solution $\om(t)$ of \eqref{eq-nl-om-chan} with initial data $\om_{in}$ satisfies the following properties: \\
1. Global stability in Gevrey class,  
\begin{align}\label{eq:GSinG}
  \left\|\om\left(t,x+ty+\Phi(t,y),y\right)\right\|_{\mathcal{G}^{s(\beta),\lambda_1,\s}}\leq C\ep\nu^{\beta}, 
\end{align}
where $\Phi(t,y)$ is given explicitly by 
\begin{align*}
\Phi(t,y)=\int_0^te^{\nu(t-\tau)\pa_y^2}\left(\frac{1}{2\pi}\int_{\mathbb{T}}u^{(1)}(\tau,x,y)dx\right)d\tau.
\end{align*}
2. Inviscid damping, 
\begin{equation}
\left\|P_{\neq}u^{(1)}\right\|_{L^2}+\langle t\rangle\left\|u^{(2)}\right\|_{L^2}\leq C\ep\nu^{\beta}e^{-c\nu^{\frac{1}{3}} t}\langle t\rangle^{-1}.
\end{equation}
3. Enhanced dissipation, 
\begin{equation}\label{eq: enha-thm}
\left\|P_{\neq}\om(t)\right\|_{L^2}\leq C\ep\nu^{\beta}e^{-c\nu^{\frac{1}{3}} t}.
\end{equation}
The constants $c, C$ are independent of $\nu$ and $\ep$. 
\end{theorem}
Before discussing the stability mechanism, let us make some comments about the theorem and list some interesting problems here. 
\begin{remark}
The constants $\la_1,\la_0,\ep_0, \nu_0$ and $c,C$ are uniform in $\beta$. 
\end{remark}

\begin{remark}
The theorem also holds for the Gevrey-$\f1s$ class with $s\geq s(\b)=\f{1-3\b}{2-3\b}$. In the case $s>s(\b)$, One can even shorten the proof by using a simpler time-dependent Fourier multiplier. 
\end{remark}

\begin{remark}
If $\beta=0$, the size of the initial perturbation is independent of the viscosity $\nu$. We obtain the nonlinear asymptotic stability in the Gevrey-$2$ class, which slightly improves the results of \cite{BMV2016}. 

If $\beta=\f13$, the regularity of the initial perturbation is $H^{\s}$ with $\s\geq 11$. The results contain the results of \cite{MasmoudiZhao2019}. It is an interesting problem to study the smallest $\s$ for $\beta=\f13$ that makes the nonlinear asymptotic stability hold. 
\end{remark}

\begin{remark}
It is interesting to study the optimality of the relationship, i.e. 
\beno
\|\om_{in}\|_{\mathcal{G}^{s(\beta),\la_0,\s}}\gg \nu^{\beta-} \ \Rightarrow \ instability. 
\eeno
\end{remark}
\begin{remark}
In a recent paper \cite{LMZ2022}, we proved that if $X=L^2$, $\beta=\f12$ is optimal. The relationship between regularity and size of the initial perturbation that makes the nonlinear asymptotic stability holds is open for $\beta\in (\f12,\f13)$. 
\end{remark}

\subsection{Inviscid damping and enhanced dissipation}
In this section, let us discuss the stability mechanism: inviscid damping and enhanced dissipation.  

In \cite{Orr1907}, Orr observed an important phenomenon that the velocity will tend to 0 as $t\to \infty$. This phenomenon is called inviscid damping, which is the analog in hydrodynamics of Landau damping found by Landau \cite{Landau1946}, which predicted the rapid decay of the electric field of the linearized Vlasov equation around a homogeneous equilibrium. Mouhot and Villani \cite{MouhotVillani2011} made a breakthrough and proved nonlinear Landau damping for the perturbation in Gevrey class(see also \cite{BMM2016}). In this case, the mechanism leading to the damping is the vorticity mixing driven by shear flow or Orr mechanism \cite{Orr1907}. 
Lin and Zeng \cite{LinZeng2011} constructed the traveling wave solution near Couette flow which implies that the nonlinear inviscid damping is not true for perturbations of the Couette flow in $H^{s}$ ($s<\f32$). We also refer to \cite{LWZZ2020} for a recent result about the traveling wave solutions near general shear flows. 
Bedrossian and Masmoudi \cite{BM2015} proved nonlinear inviscid damping around the Couette flow in Gevrey-$m$  class ($1\leq m<2$). 
Deng and Masmoudi \cite{DM2018}  proved some instability for initial perturbations in Gevrey-$m$ class ($m>2$). 
We refer to \cite{IonescuJia2020cmp,IonescuJia2021} and references therein for other related interesting results.
Due to the presence of the nonlocal term for general shear flows, the inviscid damping for general shear flows is a challenging problem even at a linear level. In the case of the finite channel, Case \cite{Case1960} gave a
formal proof of $t^{-1}$ decay for the velocity. 
Lin and Zeng \cite{LinZeng2011} present the optimal linear decay estimates of the velocity
for the data in Sobolev space. Rosencrans and Sattinger \cite{RosSat1966} gave $t^{-1}$ decay of the stream function with a continuous spectrum projection for analytic monotone shear flow.
Stepin \cite{Stepin1995} proved $t^{-\nu}(\nu<\mu_0)$ decay of the stream function for the monotone shear flow $u(y)\in C^{2+\mu_0}(\mu_0>\f12)$ without inflection point. Zillinger \cite{Zillinger2017} proved the same decay estimates as those of Couette flow given by Lin and Zeng \cite{LinZeng2011} for a class of monotone shear flow
in Sobolev spaces. However, his result imposed a strong assumption that $L\|u''\|_{W^{3,\infty}}$ is small, where $L$ is the wave length with respect to $x$. Wei, Zhang, and Zhao \cite{WeiZhangZhao2018} removed the smallness assumption in \cite{Zillinger2017} and proved the linear inviscid damping for general monotone shear flows in Sobolev space. We also refer to \cite{Jia2020siam, Jia2020arma} for a simplified proof and the linear inviscid damping in Gevrey class.  Recently, Ionescu-Jia \cite{IJ2020}, and Masmoudi-Zhao \cite{MasmoudiZhao2020} proved that the nonlinear inviscid damping holds for general linear stable monotone shear flows. 
For non-monotone flows such as the Poiseuille flow and the Kolmogorov flow, another dynamic phenomenon should be taken into consideration, which is the so-called vorticity depletion phenomena, predicted by Bouchet and Morita \cite{BM2010} and later proved by Wei, Zhang and Zhao \cite{WeiZhangZhao2019, WeiZhangZhao2020}. See also \cite{BCV2017, LMZZ2021, IonescuJia2021} for similar depletion phenomena in various systems.
Such damping caused by mixing is a common phenomenon in fluid dynamics, see \cite{LMZZ2021,RenZhao2017, RenWeiZhang-siam-2022,ZhaiZhangZhao2021} for the plasma fluid, see \cite{bedrossian2021nonlinear, YangLin2018, MSZ2020} for the stratified fluid, see \cite{AntonelliDolceMarcati2020, AntonelliDolceMarcati2021, ZengZhangZi2021} for the compressible fluid, and see \cite{WangZhangZhu2020, WeiZhangZhu2020cmp} for the geophysical fluid.

The second estimate in \eqref{eq: ED_and_ID} is the enhanced dissipation, which is also referred to as the `shear-diffusion mechanism'. This decay rate is much faster than the diffusive decay of $e^{-\nu t}$. The mechanism leading to the enhanced dissipation is also due to vorticity mixing. Generally speaking, the sheared velocity sends information to a higher frequency, and the diffusion term eliminates the information in the higher frequency. The linear enhanced dissipation for the transport diffusion equation with different sheared velocities is well studied. The degeneracy rate of sheared velocity is related to the enhanced dissipation rate \cite{DRM2021,BC2017,Coti2020,Wei2021}. Also, the stronger diffusion gives stronger enhanced dissipation \cite{He2021}. Moreover, the nonlocal term re-enhances the enhanced dissipation \cite{li2021metastability}. The nonlinear enhanced dissipation is harder even for the Couette flow in $\mathbb{T}\times\mathbb{R}$. The main reason is that the Orr mechanism interacts poorly with the nonlinear term for the nonlinear system, creating a weakly nonlinear effect referred to as an echo. There are two necessary ways to control (compete against) the echo cascades. 
One is to assume enough smallness of the initial perturbations such that the rapid growth of the enstrophy may not happen before the enhanced dissipative time-scale $\nu^{-\frac{1}{3}}$, see \cite{MasmoudiZhao2020cpde, MasmoudiZhao2019}. 
The other is to assume enough regularity (Gevrey class) of the initial perturbations such that one can pay enough regularity to control the growth caused by the echo cascade, see \cite{BMV2016}. This paper considers the case when the initial perturbation is either less regular or not small enough. 
We summarize the stability threshold results for Couette flow in different domains in the following tables: 
\begin{table}[H]
\centering
\caption{2D Couette flow} 
\medskip
 \begin{tabular}{|c|c|c|c|c|}
\hline
Space & $\beta$ &Boundary&Reference&Optimality \\
\hline
$H_x^{log}L_y^2$ & $\geq\frac{1}{2}$ &No&\cite{BVW2018,MasmoudiZhao2020cpde}&Yes \cite{LMZ2022} \\
\hline
$H^{\s}$ & $\geq \frac{1}{3}$ &No&\cite{MasmoudiZhao2019}&Open\\
\hline
Gevrey-${2_-}$ & $0$ &No&\cite{BMV2016}&\cite{DM2018} \\
\hline
Gevrey-${\f{1}{s(\b)}}$ & $\beta$ &No&this paper &Open\\
\hline
$H^{1}$ & $\geq \frac{1}{2}$ &Non-slip&\cite{CLWZ2020}&Open \\
\hline
\end{tabular} 
\end{table} 
\begin{table}[H]
\centering
\caption{3D Couette flow} 
\medskip
 \begin{tabular}{|c|c|c|c|}
\hline
Space & $\beta$ &Boundary&Reference \\
\hline
$H^{1}$ & $\geq 1$ &No&\cite{BGM2015,BGM2017,BGM2020,WeiZhang2020}\\
\hline
$H^{1}$ & $\geq 1$ &Non-slip&\cite{ChenWeiZhang2020}\\
\hline
\end{tabular}  
\end{table} 
We also refer to \cite{DengWuZhang2021, MasmoudiZhaiZhao2022, ZhaiZhao2022} and \cite{Liss2020cmp} for the nonlinear enhanced dissipation of Couette flow in stratified fluid and plasma fluid. 
Let us also mention some other recent progress on the stability problem of different types of shear flows in different domains for incompressible fluid:
\begin{itemize}
\item 2D Kolmogorov flow: \cite{LinXu2019,IMM2019,WeiZhangZhao2020}
\item 2D Poiseuille flow: \cite{CotiElgindiWidmayer2020,ding2020enhanced}
\item 2D shear flow:
\cite{GNRS2020, Jiahao2022}
\item 3D Kolmogorov flow: \cite{LiWeiZhang2020}
\item 3D pipe Poiseuille flow: \cite{chen2019linear}
\end{itemize}

\subsection{Notation and conventions}\label{sec-notation}
A convention we generally use is to denote the discrete $x$ (or $z$) frequencies as subscripts. By convention, we always use Greek letters such as $\eta$ and $\xi$ to denote frequencies in the $y$ or $v$ direction and lowercase Latin characters commonly used as indices such as $j,k,l,m$ to denote frequencies in the $x$ or $z$ direction (which are discrete). Another convention we use is to denote $M, N$ as dyadic integers $M, N\in \mathbb D$  where
\begin{align*}
  \mathbb D=\left\{\frac{1}{2},1,2,\dots,2^j,\dots\right\}=2^{\mathbb N}\cup \frac{1}{2}.
\end{align*}
This will be useful when defining Littlewood–Paley projections and paraproduct decompositions. We will use the same $\rmA$ for $\rmA(t,\na) f=(\rmA(\eta) \hat{f}(\eta))^{\vee}$ or $\rmA \hat{f}=\rmA(\eta) \hat{f}(\eta)$, where $\rmA(t,\na)$ is a Fourier multiplier.

We use the notation $f \lesssim g$ when there exists a constant $C>0$ independent of the parameters of interest such that $f \leq C g$ (we define $g \gtrsim f$ analogously). Similarly, we use the notation $f \approx g$ when there exists $C>0$ such that $C^{-1} g \leq f \leq C g$. 

We use the notation $c$ to denote a constant such that $0\le c<1$ which may be different from line to line.

We will denote the $l^{1}$ vector norm $|k, \eta|=|k|+|\eta|$, which by convention is the norm taken in our work. Similarly, given a scalar or vector in $\mathbb{R}^{n}$ we denote
$$
\langle v\rangle=\left(1+|v|^{2}\right)^{\frac{1}{2}} .
$$

We denote
\begin{align*}
   \Delta_{L;z,v}=\pa_z^2+(\pa_v-t\pa_z)^2,\quad \nabla_{L;z,v}=\left(\pa_z,\pa_v-t\pa_z\right).
\end{align*}
When no confusion can arise, we shall write $\Delta_{L}$ and $\nabla_{L}$ instead of $\Delta_{L;z,v}$ and $\nabla_{L;z,v}$.

Given a function $f(t,x,y)$, we denote its Fourier transform in $(x,y)$ by
\begin{align*}
   \hat f(k,\xi)=\frac{1}{4\pi^2}\int_{\mathbb T}\int_{\mathbb R} f(x,y)e^{-ikx}e^{- i\xi y} dy dx.
\end{align*}
We denote the projection to the $k$th mode of $f(x,y)$ by
\begin{align*}
  P_kf(x,y)=f_k(x,y)=\frac{1}{2\pi}\left(\int_{\mathbb{T}}f(x',y)e^{-ikx'}dx'\right)e^{ikx},
\end{align*}
and denote the projection to the non-zero mode by
\begin{align*}
  P_{\neq}f(x,y)=f_{\neq}(x,y)=\sum_{k\in\mathbb Z/0}f_k(x,y).
\end{align*}
and the projection to zero mode by
\begin{align*}
  \left(P_{0}f\right)(y)= \left(f\right)_0(y)=\frac{1}{2\pi}\left(\int_{\mathbb{T}}f(x',y)dx'\right).
\end{align*}
We also use $\hat f_k(\xi)$ to denote $\hat f(k,\xi)$ to emphasize it is the Fourier transform of the $k$ mode.

We denote the standard $L^{2}$ norms by $\|\cdot\|_{2}$. We make common use of the Gevrey-$\frac{1}{s}$ norm with Sobolev correction defined by
\begin{align*}
  \|f\|_{\mathcal G^{s,\lambda,\sigma}}^2=\sum_k \int \left|\hat f_k(\eta)\right|^2 e^{2\lambda|k,\eta|^s}\langle k,\eta\rangle^{2\sigma}  d\eta.
\end{align*}
Since for most of the paper, we are taking $s$ as a fixed constant, it is normally omitted. We refer to this norm as the $\mathcal G^{s,\lambda,\sigma}$ norm and occasionally refer to the space of functions
\begin{align*}
  \mathcal G^{s,\lambda,\sigma}= \left\{f\in L^2:\|f\|_{\mathcal G^{s,\lambda,\sigma}}<\infty\right\}.
\end{align*}
It is clear that for $s=0$, we have $\|f\|_{\mathcal G^{s,\lambda,\sigma}}=\|f\|_{H^{\sigma}}$.

For $|k|=1,2, \ldots$, $\eta \in \mathbb R$ with $k\eta>0$, and $\beta\in[0,\frac{1}{3}]$, let
$$
t_{k, \eta}=\frac{2 |\eta|}{2 |k|+1},\ t_{k, \eta}^{\beta,-}=\frac{|\eta|}{|k|}-\frac{1}{8}\frac{|\eta|^{1-3\beta}}{|k|^{2-3\beta}},\ t_{k, \eta}^{\beta,+}=\frac{|\eta|}{|k|}+\frac{1}{8}\frac{|\eta|^{1-3\beta}}{|k|^{2-3\beta}},
$$
and $t_{k, \eta}=2|\eta|$. Then we define the following frequency interval for $|k|=1,2, \ldots$ and $\eta \in \mathbb R$ with $k\eta>0$
\begin{align*}
  \mathrm{I}_{k, \eta}=\left[t_{k, \eta}, t_{|k|-1, \eta}\right],\ \mathrm{I}_{k, \eta}^L=\left[t_{k, \eta}, \frac{\eta}{k}\right],\ \mathrm{I}_{k, \eta}^R=\left[\frac{\eta}{k}, t_{|k|-1, \eta}\right], \\
  \tilde{\mathrm{I}}_{k, \eta}^{\beta}=\left[t_{k, \eta}^{\beta,-}, t_{k, \eta}^{\beta,+}\right],\ \tilde{\mathrm{I}}_{k, \eta}^{\beta,L}=\left[t_{k, \eta}^{\beta,-}, \frac{\eta}{k}\right],\ \tilde{\mathrm{I}}_{k, \eta}^{\beta,R}=\left[\frac{\eta}{k}, t_{k, \eta}^{\beta,+}\right] .
\end{align*}     
When no confusion can arise, we will simplify $t_{k, \eta}^{\beta,\pm}$, $\tilde{\mathrm{I}}_{k, \eta}^{\beta}$, $\tilde{\mathrm{I}}_{k, \eta}^{\beta,L}$, and $\tilde{\mathrm{I}}_{k, \eta}^{\beta,R}$ to $t_{k, \eta}^{\pm}$, $\tilde{\mathrm{I}}_{k, \eta}$, $\tilde{\mathrm{I}}_{k, \eta}^{L}$, and $\tilde{\mathrm{I}}_{k, \eta}^{R}$ respectively. We can see all these points and intervals are defined on the non-negative real axis. Through this manuscript, for any frequency $(k,\eta)$ with $\eta k>0$, we call $\frac{\eta}{k}$ the corresponding critical time, $\mathrm{I}_{k, \eta}$ the critical time region, and $\tilde{\mathrm{I}}_{k, \eta}$ the very critical time region.

For $\eta \ge 0$, we define $E(\eta)\in\mathbb Z$ to be the integer part. 

For a statement $Q, \mathbf{1}_{Q}$ or $\chi^{Q}$ will denote the function that equals 1 if $Q$ is true and 0 otherwise.

\section{Proof of Theorem \ref{Thm: main}}
In this section, we present the key propositions and complete the proof of Theorem \ref{Thm: main} by admitting those propositions.

\subsection{Coordinate transform}\label{sec-coor}
We will use the same change of coordinates as in \cite{BMV2016} which allows us to simultaneously `mod out' by the evolution of the time-dependent background shear flow and treat the mixing of this background shear as a perturbation of the Couette flow (in particular, to understand the nonlinear effect of the Orr mechanism).

The change of coordinates used is $(t,x,y)\to (t,z,v)$, where
 $z(t,x,y)=x-tv(t,y)$ and $v(t,y)$ satisfies
 \begin{align}\label{eq-v}
   (\pa_t-\nu\pa_{yy})\left(t(v(t,y)-y)\right)=  u^{(1)}_0(t,y),
 \end{align}
with initial data $\lim\limits_{t\to 0}t(v(t,y)-y)=0$, where $ u^{(1)}_0 (t,y)=\frac{1}{2\pi}\int_{\mathbb{T}}u^{(1)}(t,x,y)dx$. 

Define the following quantities
\begin{align}
\bar q(t,v(t,y))&=v(t,y)-y,\label{eq: barq}\\
v'(t,v(t,y))&=(\pa_yv)(t,y),\label{eq: v'}\\
v''(t,v(t,y))&=(\pa_{yy}v)(t,y),\label{eq: v''}\\
[\pa_tv](t,v(t,y))&=(\pa_tv)(t,y),\label{eq: v_t}\\
f(t,z(t,x,y),v(t,y))&=\om(t,x,y),\label{eq: om}\\
\phi(t,z(t,x,y),v(t,y))&=\psi(t,x,y),\label{eq: psi}\\
\tu(t,z(t,x,y),v(t,y))&=u^{(1)}(t,x,y).\label{eq: U^x}
\end{align}
Thus we get
\begin{align}\label{eq: phi f}
  \Delta_t\phi\eqdef\pa_{zz}\phi+(v')^2(\pa_v-t\pa_z)^2\phi+v''(\pa_v-t\pa_z)\phi=f,
\end{align}
and
\begin{align}\label{eq: f1}
  &\pa_tf+[\pa_tv]\pa_vf-\nu v''t\pa_zf+v'\na_{z,v}^{\bot}P_{\neq}\phi\cdot\na_{z,v}f=\nu\Delta_tf,
\end{align}
where $\na_{z,v}^{\bot}=(-\pa_v,\pa_z)$, $\na_{z,v}=(\pa_z,\pa_v)$ and $P_{\neq}\phi=\phi-\phi_0$, $\tu_0(t,v)=\frac{1}{2\pi}\int_{\mathbb{T}_{2\pi}}\tu(t,z,v)dz$. 

We also obtain that 
\begin{align}\label{eq: tu_0}
  \pa_t\tu_0+[\pa_t v]\pa_v\tu_0+ \left(v'\na^{\bot}_{z,v}P_{\neq 0}\phi\cdot\na \tu\right)_0=\nu\Delta_t\tu_0.
\end{align}

Define the auxiliary function
\begin{align*}
  q(t,v)=\frac{1}{t}(\tu_0(t,v)-\bar q(t,v)),
\end{align*}
which implies that
\begin{align*}
&[\pa_tv]=q+\nu v'',\\
&v'\pa_v\bar q(t,v)=v'(t,v)-1,\\
&\pa_t\bar q+[\pa_tv]\pa_v\bar q=[\pa_tv],\\
&v'\pa_vv'=v''=\Delta_t\bar q,
\end{align*}
and that $q$ satisfies
\begin{align}\label{eq: q}
  \pa_tq+\frac{2q}{t}+q\pa_vq=-\frac{v'}{t} \left(\na_{z,v}^{\bot}P_{\neq}\phi\cdot\na_{z,v}\tu\right)_0 +\nu(v')^2\pa_{vv}q.
\end{align}

If we denote $h=v'-1$, we get that
\begin{align*}
  \pa_th+q\pa_v h=\frac{-f_0-h}{t}+\nu\widetilde{\Delta}_th.
\end{align*}
Let $\bar{h}=\frac{-f_0-h}{t}$, thus we obtain that
\begin{align*}
  \pa_t\bar{h}+q\pa_v\bar{h}=-\frac{2}{t}\bar{h}+\frac{v'}{t} \left(\na_{z,v}^{\bot}P_{\neq}\phi\cdot\na_{z,v}f\right)_0 +\nu\widetilde{\Delta}_t\bar{h}.
\end{align*}

It gives that
\begin{align}\label{eq: f-new}
\pa_tf+\bar u\cdot\na_{z,v}f=\nu\widetilde{\Delta}_tf,
\end{align}
where 
\begin{align*}
  \bar u(t,z,v)=\left(\begin{aligned}0\\q\end{aligned}\right)+v'\na_{z,v}^{\bot}P_{\neq}\phi
\end{align*}
and $\widetilde{\Delta}_tf=\pa_{zz}f+(v')^2(\pa_v-t\pa_z)^2f$.

By changing the coordinates we deduce our problem by studying the following system: 
\begin{align}\label{eq: main f}
&\left\{
\begin{aligned}
&\pa_tf+\bar u\cdot\na_{z,v}f=\nu\widetilde{\Delta}_tf, \\
&\bar u(t,z,v)=\left(\begin{aligned}0\\q\end{aligned}\right)+v'\na_{z,v}^{\bot}P_{\neq}\phi,\\
&\Delta_t\phi=f, \quad v''=v'\pa_vv', \quad h=v'-1,
\end{aligned}\right.\\
\label{eq: coor}
&\left\{
\begin{aligned}
&\pa_tq+\frac{2q}{t}+q\pa_vq=-\frac{v'}{t} \left(\na_{z,v}^{\bot}P_{\neq}\phi\cdot\na_{z,v}\tu\right)_0+\nu(v')^2\pa_{vv}q,\\
&\pa_t\bar{h}+\frac{2\bar{h}}{t}+q\pa_v\bar{h}=\frac{v'}{t} \left(\na_{z,v}^{\bot}P_{\neq}\phi\cdot\na_{z,v}f\right)_0+\nu(v')^2\pa_{vv}\bar{h},\\
&\pa_th+q\pa_v h=\bar{h}+\nu(v')^2\pa_{vv}h,\\
&\tu=-v'(\pa_v-t\pa_z)\phi.
\end{aligned}
\right.
\end{align}
\subsection{Main energy estimate}\label{sec-main-energy}
In light of the previous section, our goal is to control the solution to the main system \eqref{eq: main f} and the coordinate system \eqref{eq: coor} uniformly in a suitable norm as $t\to\infty$. The key idea we use for this is the carefully designed time-dependent norm written as:
\begin{align*}
  \|\rmA(t,\nabla)f\|_{L^2}^2=\sum_k\int_\eta \left|\rmA_k(t,\eta)\hat f_k(t,\eta)\right|^2d\eta.
\end{align*}
The multiplier $\rmA(t,\na)$ has several components
\begin{align*}
 \rmA_k(t,\eta)=e^{\mathbf{1}_{k \neq 0}c_M \nu^{\frac{1}{3}}t}e^{\lambda(t)|k,\eta|^s}\langle k,\eta\rangle^\sigma W_k(t,\eta)G_k(t,\eta)\mathfrak M_k(t,\eta),
\end{align*}
where $W_k(t,\eta)$, $G_k(t,\eta)$, and $\mathfrak M_k(t,\eta)$ are defined in \eqref{def-weight-W}, \eqref{def-weight-G}, and \eqref{def-weight-M} respectively. Through this paper, we take $c_M=\frac{1}{8}$.

We also introduce a special multiplier $\rmA^{\mathrm R}(t,\na)$ for the coordinate functions:
\begin{align*}
  \rmA^{\mathrm R}(t,\eta)=e^{\lambda(t)|\eta|^s}\langle \eta\rangle^\sigma W^{\mathrm R}(t,\eta)G_0(t,\eta),
\end{align*}
where  $W^{\rm R}(t,\eta)$ is defined in \eqref{def-weight-WR}.

We also introduce a lower-order norm
\begin{align*}
  \|A^\gamma(t,\nabla)f\|_{L^2}^2=\sum_k\int_\eta \left|A_k^\gamma(t,\eta)\hat f_k(t,\eta)\right|^2d\eta,
\end{align*}
where
\begin{equation}\label{eq-def-Agamma}
  A_k^\gamma(t,\eta)=\left\{
    \begin{array}{ll}
      (1+c_M \nu^{\frac{1}{3}}t)e^{c_M \nu^{\frac{1}{3}}t}e^{\lambda(t)|k,\eta|^s}\langle k,\eta\rangle^\gamma\mathfrak M_k(t,\eta),&\text{ for }k\neq0,\\
     e^{\lambda(t)|\eta|^s}\langle \eta\rangle^\gamma,&\text{ for }k=0.
    \end{array}
  \right.
\end{equation}
Here $7\le\gamma\le \sigma-4$. We also use $A_0^\gamma$ to denote the zero mode restriction of $A^\gamma$, and use $A_{\neq}^\gamma=A^\gamma-A_0^\gamma$ to denote the non-zero mode restriction of $A^\gamma$.

The index $\lambda(t)$ is the main Gevrey-$\frac{1}{s}$ regularity and is chosen to satisfy
\begin{align*}
  \lambda(t)=\frac{3}{4}\lambda_0+\frac{1}{4}\lambda_1, \ t\le 1,
\end{align*}
and
\begin{align*}
  \dot \lambda(t)=\frac{d}{dt}\lambda(t)=-\frac{\delta_\lambda}{\langle t\rangle^{2c_1}}(1+\lambda(t)),\ t>1,
\end{align*}
where $\delta_\lambda\approx \lambda_0-\lambda_1$ is a small parameter that ensures $\lambda(t)>\frac{1}{2}(\lambda_0+\lambda_1)$. Through this paper, we take $c_1=\frac{5}{8}$.

We define our higher energy:
\begin{align*}
  \mathcal E^{s}(t)=\frac{1}{2}\|\rmA(t,\nabla)f\|_{L^2}^2+\f12\|A^{\gamma}f\|_{L^2}^2+\mathcal E^{s}_v(t),
\end{align*}
where
\begin{align*}
  \mathcal E^{s}_v(t)=\nu^{\beta}t^3\left\|\frac{\rmA}{\langle\pa_v \rangle^s}  q\right\|_{ L^2}^2(t)+t^4\|A^{\gamma}q\|_{L^2}(t)+\nu^{\beta}t^3\left\|\frac{\rmA}{\langle\pa_v \rangle^s}  \bar h\right\|_{ L^2}^2(t)+\varepsilon\nu^{\beta}\|\rmA^{\mathrm R} h\|_{ L^2}^2(t).
\end{align*}

By well-posedness theory for two-dimensional Navier–Stokes equations in Gevrey spaces we may safely ignore the time interval (say) $[0,1]$ by further restricting the size of the initial data. That is, we have the following lemma.
\begin{lemma}\label{lem-energy-01}
  For $\varepsilon>0$, $\nu>0$, $\beta\in[0,\frac{1}{3}]$, $s(\beta)=\frac{1-3\beta}{2-3\beta}$, and $\sigma\ge 11$, there exists $\grave \varepsilon>0$ such that if $\left\|u_{in}\right\|_{L^2}^2+\|\om_{in}\|_{\mathcal{G}^{s(\beta),\lambda_0,\sigma}}^2\le {\grave \varepsilon}^2\nu^{2\beta}$, then $\mathcal E^{s}(1)\le {\varepsilon}^2\nu^{2\beta}$, and $\sup\limits_{t\in[0,1]}\|h(t)\|_{L^\infty}\le \frac{1}{2}$.
\end{lemma}
\begin{proof}The proof of this lemma is given in Section \ref{sec-appD}.
\end{proof}

The goal is next to prove by a continuity argument that this energy $\mathcal E^{s}(t)$ (together with some related quantities) is uniformly bounded for all time if $\varepsilon$ is sufficiently small. We define the following controls referred to in the sequel as the bootstrap hypotheses for $t\ge 1$.

\noindent {\bf Higher regularity: main system}
\begin{equation}\label{eq-boot-hf}
  \begin{aligned}    
   &\|\rmA f(t)\|_{L^2}^2+\int^t_1\nu^{\frac{1}{3}}\|\rmA f(t')\|_{L^2}^2 +\nu \||\Delta_L|^{1/2}\rmA f(t')\|_{L^2}^2\\
  &\qquad\qquad\qquad\qquad\qquad\qquad+\mathrm{CK}_{\lambda}(t')+\mathrm{CK}_{W}(t')+\mathrm{CK}_{G}(t')dt'\le 64\varepsilon^2\nu^{2\beta},     
  \end{aligned}
\end{equation}
where the $\mathrm{CK}$ stands for ``Cauchy–Kovalevskaya'',
\begin{align*}
  \mathrm{CK}_{\lambda}(t) =& \left|\dot\lambda(t)\right| \sum_k\left\||k,\eta|^{\frac{s}{2}}\rmA_k \hat f_k(t,\eta)\right\|_{L^2_\eta}^2,\ \mathrm{CK}_{W}(t) = \sum_k\left\|\sqrt{\frac{\pa_t w_k(t,\iota(k,\eta))}{w_k(t,\iota(k,\eta))}} \rmA_k\hat f_k(t,\eta)\right\|_{L^2_\eta}^2,\\
  \mathrm{CK}_{G}(t) =& \sum_k\left\|\sqrt{\frac{\pa_t g(t,\iota(k,\eta))}{g(t,\iota(k,\eta))}} \rmA_k \hat f_k(t,\eta)\right\|_{L^2_\eta}^2.
\end{align*}
Here the functions $w_k(t,\eta),\iota(k,\eta), g(t,\eta)$ are defined in \eqref{eq-def-wk2}, \eqref{eq-def-iota}, \eqref{eq-toy-weak2} (see Section \ref{sec-srw} and Section \ref{sec-wrg} for more details).

\noindent {\bf Higher regularity: coordinate system}
\begin{equation}\label{eq-boot-hq}
  \begin{aligned}    
   &t^{3}\left\|\frac{\rmA}{\langle\pa_v \rangle^s} q(t)\right\|_{L^2}^2+\int^t_1\nu {t'}^{3}\left\|\frac{\rmA}{\langle\pa_v \rangle^s}\pa_vq(t')\right\|_{L^2}^2+{t'}^{2}\left\|\frac{\rmA}{\langle\pa_v \rangle^s} q(t')\right\|_{L^2}^2\\
  &\qquad\qquad\qquad\qquad\qquad\qquad+\mathrm{CK}_{\lambda}^{v,1}(t')+\mathrm{CK}_{W}^{v,1}(t')+\mathrm{CK}_{G}^{v,1}(t')dt'\le 64\varepsilon^2\nu^{\beta},     
  \end{aligned}
\end{equation}

\begin{equation}\label{eq-boot-hbh}
  \begin{aligned}    
   &t^{3}\left\|\frac{\rmA}{\langle\pa_v \rangle^s} \bar h(t)\right\|_{L^2}^2+\int^t_1\nu{t'}^{3} \left\|\frac{\rmA}{\langle\pa_v \rangle^s}\pa_v\bar h(t')\right\|_{L^2}^2+{t'}^{2} \left\|\frac{\rmA}{\langle\pa_v \rangle^s} \bar h(t')\right\|_{L^2}^2\\
  &\qquad\qquad\qquad\qquad\qquad\qquad+\mathrm{CK}_{\lambda}^{v,2}(t')+\mathrm{CK}_{W}^{v,2}(t')+\mathrm{CK}_{G}^{v,2}(t')dt'\le 64\varepsilon^2\nu^{\beta},     
  \end{aligned}
\end{equation}

\begin{equation}\label{eq-boot-hh}
  \begin{aligned}    
   &\|\rmA^{\mathrm R} h(t)\|_{L^2}^2+\int^t_1\nu \|\rmA^{\mathrm R}\pa_v h(t')\|_{L^2}^2+\mathrm{CCK}_\lambda^{v,1}(t')+\mathrm{CCK}_W^{v,1}(t')+\mathrm{CCK}_G^{v,1}(t')\\
  &\qquad\qquad\qquad\qquad\qquad\qquad+\mathrm{CCK}_\lambda^{v,2}(t')+\mathrm{CCK}_W^{v,2}(t')+\mathrm{CCK}_G^{v,2}(t')dt'\le 64\varepsilon \nu^{\beta},     
  \end{aligned}
\end{equation}
where
\begin{align*}
  &\mathrm{CK}_{\lambda}^{v,1}(t) =  t^{3}\left|\dot\lambda(t)\right| \left\||\eta|^{\frac{s}{2}}\frac{\rmA_0}{\langle \eta\rangle^s}\hat q(t,\eta)\right\|_{L^2_\eta}^2,\  \mathrm{CK}_{\lambda}^{v,2}(t) = t^{3}\left|\dot\lambda(t)\right| \left\||\eta|^{\frac{s}{2}}\frac{\rmA_0}{\langle \eta\rangle^s}\hat{\bar h}(t,\eta)\right\|_{L^2_\eta}^2,\\
  &\mathrm{CK}_{W}^{v,1}(t) =  t^3\left\|\sqrt{\frac{\pa_t w_0(t,\eta)}{w_0(t,\eta)}} \frac{\rmA_0}{\langle \eta\rangle^s}\hat q(t,\eta)\right\|_{L^2_\eta}^2,\ \mathrm{CK}_{W}^{v,2}(t) = t^3\left\|\sqrt{\frac{\pa_t w_0(t,\eta)}{w_0(t,\eta)}} \frac{\rmA_0}{\langle \eta\rangle^s}\hat{\bar h}(t,\eta)\right\|_{L^2_\eta}^2,\\
  &\mathrm{CK}_{G}^{v,1}(t) =  t^3\left\|\sqrt{\frac{\pa_t g(t,\eta)}{g(t,\eta)}} \frac{\rmA_0}{\langle \eta\rangle^s}\hat q(t,\eta)\right\|_{L^2_\eta}^2 ,\ \mathrm{CK}_{G}^{v,2}(t) = t^3\left\|\sqrt{\frac{\pa_t g(t,\eta)}{g(t,\eta)}} \frac{\rmA_0}{\langle \eta\rangle^s}\hat{\bar h}(t,\eta)\right\|_{L^2_\eta}^2,\\
  &\mathrm{CCK}_{\lambda}^{v,1}(t) =  \left|\dot\lambda(t)\right| \left\||\eta|^{\frac{s}{2}}\rmA^{\mathrm R}\hat h(t,\eta)\right\|_{L^2_\eta}^2,\  \mathrm{CCK}_{\lambda}^{v,2}(t) = \left|\dot\lambda(t)\right| \left\||\eta|^{\frac{s}{2}}\frac{\rmA^{\mathrm R}}{\langle \eta\rangle}\widehat{v''}(t,\eta)\right\|_{L^2_\eta}^2,\\
  &\mathrm{CCK}_{W}^{v,1}(t) = \left\|\sqrt{\frac{\pa_t w_0(t,\eta)}{w_0(t,\eta)}} \rmA^{\mathrm R}\hat h(t,\eta)\right\|_{L^2_\eta}^2,\ \mathrm{CCK}_{W}^{v,2}(t) =  \left\|\sqrt{\frac{\mathfrak w(\nu,t,\eta)\pa_t w_0(t,\eta)}{w_0(t,\eta)}} \frac{\rmA^{\mathrm R}}{\langle \eta\rangle}\widehat{v''}(t,\eta)\right\|_{L^2_\eta}^2,\\
  &\mathrm{CCK}_{G}^{v,1}(t) = \left\|\sqrt{\frac{\pa_t g(t,\eta)}{g(t,\eta)}} \rmA^{\mathrm R}\hat h(t,\eta)\right\|_{L^2_\eta}^2 ,\ \mathrm{CCK}_{G}^{v,2}(t) = \left\|\sqrt{\frac{\mathfrak w(\nu,t,\eta)\pa_t g(t,\eta)}{g(t,\eta)}} \frac{\rmA^{\mathrm R}}{\langle \eta\rangle}\widehat{v''}(t,\eta)\right\|_{L^2_\eta}^2,
\end{align*}
and $\mathfrak w$ is defined in \eqref{eq-def-crr-w}.

\noindent {\bf Lower regularity: strong enhanced dissipation}
\begin{equation}\label{eq-boot-lfn0}
  \begin{aligned}    
   &\|A^\gamma f_{\neq}(t)\|_{L^2}^2+\int^t_1\nu^{\frac{1}{3}}\|A^\gamma f_{\neq}(t')\|_{L^2}^2 +\nu \||\Delta_L|^{1/2}A^\gamma f_{\neq}(t')\|_{L^2}^2\\
   &\qquad\qquad\qquad\qquad\qquad\qquad\qquad+{t'}^{-2c_1}\||\nabla|^{\frac{s}{2}}A^\gamma f_{\neq}(t')\|_{L^2}^2dt'\le 64\varepsilon^2\nu^{2\beta}.   
  \end{aligned}
\end{equation}

\noindent {\bf Lower regularity: decay of the zero mode}
\begin{equation}\label{eq-boot-lq}
  \begin{aligned}    
   &t^{4}\|A^{\gamma} q(t)\|_{L^2}^2+\int^t_1\nu {t'}^{4}\|A^{\gamma}\pa_vq(t')\|_{L^2}^2\\
  &\qquad\qquad\qquad\qquad\qquad\qquad+{t'}^{-2c_1}{t'}^{4}\||\nabla|^{\frac{s}{2}}A^{\gamma} q(t')\|_{L^2}^2dt'\le 64\varepsilon^2\nu^{2\beta},     
  \end{aligned}
\end{equation}
\begin{equation}\label{eq-boot-lf0}
  \begin{aligned}    
   &\left\|A^\gamma_0f_0(t)\right\|_{L^2}^2+\nu t\left\|A^\gamma_0\pa_vf_0(t)\right\|_{L^2}^2+\int^t_1\nu \|A^\gamma_0\pa_vf_0(t')\|_{L^2}^2+\nu^2t' \|A^\gamma_0\pa_vf_0(t')\|_{L^2}^2\\
  &\qquad  +{t'}^{-2c_1}\||\nabla|^{\frac{s}{2}}A^\gamma_0f_0(t')\|_{L^2}^2+{t'}^{-2c_1}\nu t'\||\nabla|^{\frac{s}{2}}A^\gamma_0\pa_vf_0(t')\|_{L^2}^2dt'\le 64\varepsilon^2\nu^{2\beta}.     
  \end{aligned}
\end{equation}

Let $I^*$ be the connected set of times $t\ge1$ such that the bootstrap hypotheses \eqref{eq-boot-hf}–\eqref{eq-boot-lf0} are all satisfied. We will work on regularized solutions for which we know $\mathcal E^{s}(t)$ takes values continuously in time, and hence $I^*$ is a closed interval $[1,T^*]$ with $T^*\ge1$. The bootstrap is complete if we show that $I^*$ is also open, which is the purpose of the following proposition, the proof of which constitutes the majority of this work.

\begin{proposition}\label{prop-boot}
  For $\sigma\ge 11$, $\nu>0$, and $7\le\gamma\le\sigma-4$, there exist $0<\varepsilon_0,\nu_0<1$, such that for all $0<\nu\le \nu_0$ and $0<\varepsilon\le \varepsilon_0$, such that if on $[1,T^*]$ the bootstrap hypotheses \eqref{eq-boot-hf}–\eqref{eq-boot-lf0} hold, then for any $t\in [1,T^*]$, we have the following properties:

\noindent {\bf1. Vorticity boundedness:}
\begin{equation}\label{eq-est-boot-hf}
  \begin{aligned}    
   &\|\rmA f(t)\|_{L^2}^2+\int^t_1\nu^{\frac{1}{3}}\|\rmA f_{\neq}(t')\|_{L^2}^2 +\nu \||\Delta_L|^{1/2}\rmA f(t')\|_{L^2}^2\\
  &\qquad\qquad\qquad\qquad\qquad\qquad+\mathrm{CK}_{\lambda}(t')+\mathrm{CK}_{W}(t')+\mathrm{CK}_{G}(t')dt'\le 32\varepsilon^2\nu^{2\beta}.
  \end{aligned}
\end{equation}

\noindent {\bf2. Control of coordinates system:}
\begin{equation}\label{eq-est-boot-hq}
  \begin{aligned}    
   &t^{3}\left\|\frac{\rmA}{\langle\pa_v \rangle^s} q(t)\right\|_{L^2}^2+\int^t_1\nu {t'}^{3}\left\|\frac{\rmA}{\langle\pa_v \rangle^s}\pa_vq(t')\right\|_{L^2}^2+{t'}^{2}\left\|\frac{\rmA}{\langle\pa_v \rangle^s} q(t')\right\|_{L^2}^2\\
  &\qquad\qquad\qquad\qquad\qquad\qquad+\mathrm{CK}_{\lambda}^{v,1}(t')+\mathrm{CK}_{W}^{v,1}(t')+\mathrm{CK}_{G}^{v,1}(t')dt'\le 32\varepsilon^2\nu^{\beta},     
  \end{aligned}
\end{equation}

\begin{equation}\label{eq-est-boot-hbh}
  \begin{aligned}    
   &t^{3}\left\|\frac{\rmA}{\langle\pa_v \rangle^s} \bar h(t)\right\|_{L^2}^2+\int^t_1\nu{t'}^{3} \left\|\frac{\rmA}{\langle\pa_v \rangle^s}\pa_v\bar h(t')\right\|_{L^2}^2+{t'}^{2} \left\|\frac{\rmA}{\langle\pa_v \rangle^s} \bar h(t')\right\|_{L^2}^2\\
  &\qquad\qquad\qquad\qquad\qquad\qquad+\mathrm{CK}_{\lambda}^{v,2}(t')+\mathrm{CK}_{W}^{v,2}(t')+\mathrm{CK}_{G}^{v,2}(t')dt'\le 32\varepsilon^2\nu^{\beta},     
  \end{aligned}
\end{equation}

\begin{equation}\label{eq-est-boot-hh}
  \begin{aligned}    
   &\|\rmA^{\mathrm R} h(t)\|_{L^2}^2+\int^t_1\nu \|\rmA^{\mathrm R}\pa_v h(t')\|_{L^2}^2+\mathrm{CCK}_\lambda^{v,1}(t')+\mathrm{CCK}_W^{v,1}(t')+\mathrm{CCK}_G^{v,1}(t')\\
  &\qquad\qquad\qquad\qquad\qquad\qquad+\mathrm{CCK}_\lambda^{v,2}(t')+\mathrm{CCK}_W^{v,2}(t')+\mathrm{CCK}_G^{v,2}(t')dt'\le 32\varepsilon\nu^{\beta}.     
  \end{aligned}
\end{equation}

\noindent {\bf 3. Strongly enhanced dissipation}
\begin{equation}\label{eq-est-boot-lfn0}
  \begin{aligned}    
   &\|A^\gamma f_{\neq}(t)\|_{L^2}^2+\int^t_1\nu^{\frac{1}{3}}\|A^\gamma f_{\neq}(t')\|_{L^2}^2 +\nu \||\Delta_L|^{1/2}A^\gamma f_{\neq}(t')\|_{L^2}^2\\
   &\qquad\qquad\qquad\qquad\qquad\qquad\qquad+{t'}^{-2c_1}\||\nabla|^{\frac{s}{2}}A^\gamma f_{\neq}(t')\|_{L^2}^2dt'\le 32\varepsilon^2\nu^{2\beta}.   
  \end{aligned}
\end{equation}

\noindent {\bf 4. Decay of the zero mode:}
\begin{equation}\label{eq-est-boot-lq}
  \begin{aligned}    
   &t^{4}\|A^{\gamma} q(t)\|_{L^2}^2+\int^t_1\nu {t'}^{4}\|A^{\gamma}\pa_vq(t')\|_{L^2}^2\\
  &\qquad\qquad\qquad\qquad\qquad\qquad+{t'}^{-2c_1}{t'}^{4}\||\pa_v|^{\frac{s}{2}}A^{\gamma} q(t')\|_{L^2}^2dt'\le 32\varepsilon^2\nu^{2\beta},     
  \end{aligned}
\end{equation}
\begin{equation}\label{eq-est-boot-lf0}
  \begin{aligned}    
   &\left\|A^\gamma_0f_0(t)\right\|_{L^2}^2+\nu t\left\|A^\gamma_0\pa_vf_0(t)\right\|_{L^2}^2+\int^t_1\bigg(\nu \|A^\gamma_0\pa_vf_0(t')\|_{L^2}^2+\nu^2t' \|A^\gamma_0\pa_vf_0(t')\|_{L^2}^2\\
  &\qquad  +{t'}^{-2c_1}\||\pa_v|^{\frac{s}{2}}A^\gamma_0f_0(t')\|_{L^2}^2+{t'}^{-2c_1}\nu t'\||\pa_v|^{\frac{s}{2}}A^\gamma_0\pa_vf_0(t')\|_{L^2}^2\bigg)dt'\le 32\varepsilon^2\nu^{2\beta}.     
  \end{aligned}
\end{equation}
From which it follows that $T^*=+\infty$.
\end{proposition}

The remainder of the paper is devoted to proving Proposition \ref{prop-boot}. It is natural to compute the time evolution of $\|\rmA(t,\nabla)f\|_{L^2}^2$. From the definition of $\rmA$, we have
\begin{align*}
  \pa_t \rmA_k(t,\eta)=&\left(\mathbf{1}_{k \neq 0}c_M \nu^{\frac{1}{3}}+\dot\lambda(t)|k,\eta|^s-\frac{\pa_tw_k(t,\iota(k,\eta))}{w_k(t,\iota(k,\eta))}-\frac{\pa_tg(t,\iota(k,\eta))}{g(t,\iota(k,\eta))}-\frac{\pa_t\mathfrak m_k(t,\eta)}{\mathfrak m_k(t,\eta)}\right)\rmA_k(t,\eta).
\end{align*}
Then we deduce from  \eqref{eq: main f} that
\begin{align*}
  \frac{1}{2}\frac{d}{dt}\|\rmA f\|_{L^2}^2=&\int \rmA f (\pa_t \rmA)f dzdv+\int \rmA f \rmA\pa_tfdzdv\\
  =&-\mathrm{CK}_{\lambda}(t)-\mathrm{CK}_{W}(t)-\mathrm{CK}_{G}(t)+c_M\nu^{\frac{1}{3}}\|\rmA f_{\neq}\|_{L^2}^2-\sum_{k\neq0}\left\|\sqrt{\frac{\pa_t\mathfrak m_k(t, \eta)}{\mathfrak m_k(t, \eta)}} \rmA_k \hat f_k(t,\eta)\right\|_{L^2_\eta}^2\\
  &-\int \rmA f \rmA(\bar u\cdot\na_{z,v}f)dzdv+\nu\int \rmA f \rmA\widetilde{\Delta}_tfdzdv.
\end{align*}
We write 
\begin{equation}\label{eq-dissipation-error}
  \begin{aligned}    
    &\nu\int \rmA f\rmA(\widetilde{\Delta}_t f)dzdv\\
  =&\nu\int \rmA f\rmA(\Delta_L f)dzdv-\nu\int \rmA f\rmA \left((1-(v')^2)(\pa_v-t\pa_z)^2f\right)dzdv\\
  =&-\nu \left\||\Delta_L|^{\frac{1}{2}}\rmA f\right\|_{L^2}^2-\nu\int \rmA f\rmA \left((1-(v')^2)(\pa_v-t\pa_z)^2f\right)dzdv\\
  =&-\nu \left\||\Delta_L|^{\frac{1}{2}}\rmA f\right\|_{L^2}^2+{\mathrm E},  
  \end{aligned}
\end{equation}
where $\rm E$ is the diffusion error term. 
By integrating by parts, we have
\begin{align*}
  &\int \rmA f\rmA(\bar u\cdot\nabla f)dzdv\\
  =&-\frac{1}{2}\int\nabla \cdot \bar u| \rmA f|^2dzdv+\int \rmA f[\rmA(\bar u\cdot\nabla f)-\bar u\cdot\nabla \rmA f]dzdv.
\end{align*}
Note that the relative velocity is note divergence-free:
\begin{align*}
  \nabla \cdot \bar u=\pa_vq+\pa_vv'\pa_z\phi_{\neq}=\pa_vq+\pa_v h\pa_z\phi_{\neq}.
\end{align*}
Therefore, by Lemma \ref{lem-elliptic-low-1}, Sobolev embedding, and the bootstrap hypotheses,
\begin{align*}
  \left|\int\nabla_{z,v}\cdot \bar u| \rmA f|^2dzdv\right|\le& \|\nabla \cdot \bar u\|_{L^\infty}\|\rmA f\|_{L^2}^2\\
  \lesssim& \left(\|q\|_{H^2_v}+\|\phi_{\neq}\|_{H^3_{z,v}}+\|h\|_{H^2_v}\|\phi_{\neq}\|_{H^3_{z,v}}\right)\|\rmA f\|_{L^2}^2\\
  \lesssim& \left(\|q\|_{H^2_v}+\frac{\|f_{\neq}\|_{H^5_{z,v}}}{t^2}\right)\|\rmA f\|_{L^2}^2\lesssim\frac{\varepsilon^3\nu^{3\beta}}{t^2}.
\end{align*}

To handle the commutator, $\int \rmA f[\rmA(\bar u\cdot\nabla f)-\bar u\cdot\nabla \rmA f]dzdv$, we use a paraproduct decomposition. Precisely, we define three main contributions: transport, reaction, and remainder:
\begin{align}\label{eq-decom-f}
  \int \rmA f[\rmA(\bar u\cdot\nabla f)-\bar u\cdot\nabla \rmA f]dzdv=\frac{1}{2\pi}\sum_{N\ge8}\mathrm{T}_N+\frac{1}{2\pi}\sum_{N\ge8}{\mathrm R}_N+\frac{1}{2\pi}\mathcal R,
\end{align}
where
\begin{align*}
  {\mathrm T}_N=&2\pi\int \rmA f[\rmA(\bar u_{<N/8}\cdot\nabla_{z,v} f_N)-\bar u_{<N/8}\cdot\nabla_{z,v}\rmA f_N]dzdv,\\
  {\mathrm R}_N=&2\pi\int \rmA f[\rmA(\bar u_{N}\cdot\nabla_{z,v} f_{<N/8})-\bar u_{N}\cdot\nabla_{z,v}\rmA f_{<N/8}]dzdv,\\
  \mathcal R=&2\pi\sum_{N\in \mathbb D}\sum_{N/8\le N'\le 8N}\int \rmA f[\rmA(\bar u_{N}\cdot\nabla_{z,v} f_{N'})-\bar u_{N}\cdot\nabla_{z,v}\rmA f_{N'}]dzdv.
\end{align*}
Here $\mathbb D=\{\frac{1}{2},1,2,4,\dots,2^j,\dots\}$, and $f_N$ denotes the $N$-th Littlewood-Paley projection and $f_{<N}$ means the Littlewood-Paley projection onto frequencies less than $N$ (See Appendix \ref{sec-decompo} for the Fourier analysis conventions we are taking). Formally, the paraproduct decomposition \eqref{eq-decom-f} represents a kind of `linearization' for the evolution of higher frequencies around the lower frequencies. The terminology `reaction' is borrowed from Mouhot and Villani \cite{MouhotVillani2011}.

Controlling the reaction contribution in \eqref{eq-decom-f}  is the subject of Section \ref{sec-reaction}, in which we prove:
\begin{proposition}\label{pro-reaction}
    Under the bootstrap hypotheses, for $\varepsilon$ sufficiently small, it holds that,
  \begin{equation}\label{eq-est-reaction}
    \begin{aligned}    
 \sum_{N\ge8}\left|{\mathrm R}_N\right|
 \lesssim&\varepsilon\mathrm{CK}_\lambda+\varepsilon\mathrm{CK}_W+\varepsilon\mathrm{CK}_G+\frac{\varepsilon^3 \nu^{3\beta}}{t^2}\\
& +\varepsilon \nu^{\beta}\mathrm{CK}_\lambda^{v,1}+\varepsilon \nu^{\beta}\mathrm{CK}_W^{v,1}+\varepsilon \nu^{\beta}\mathrm{CK}_G^{v,1}+\varepsilon^2\nu^{\beta}\mathrm{CCK}_{\lambda}^{v,1}\\
 &+\varepsilon\left\|\left\langle \frac{\pa_v}{t\pa_z} \right\rangle^{-1}\Delta_L\left(|\nabla|^{\frac{s}{2}}t^{-c_1}+\sqrt{\frac{\mathfrak w\pa_t w }{w }} +\sqrt{\frac{\mathfrak w\pa_t g }{g}} \right) \rmA\phi_{\neq}\right\|_{L^2}^2.     
    \end{aligned}
  \end{equation}
\end{proposition}

Controlling the transport and the remainder contribution in \eqref{eq-decom-f} and the dissipation error term is the subject of Section \ref{sec-transport}, in which we prove:
\begin{proposition}\label{pro-transport}
    Under the bootstrap hypotheses, for $\varepsilon$ sufficiently small, it holds that,
  \begin{equation}\label{eq-est-transport}
    \begin{aligned}    
 \int^{T^*}_1\sum_{N\ge8}\left|{\mathrm T}_{N}(t)\right|+\left|\mathcal R(t)\right|+ \left|{\mathrm E}(t)\right|dt\lesssim \varepsilon^3\nu^{2\beta}.   
    \end{aligned}
  \end{equation}
\end{proposition}

\begin{proposition}\label{pro-elliptic-high}
    Under the bootstrap hypotheses, for $\varepsilon$ sufficiently small, it holds that,
  \begin{equation}\label{eq-est-sharp-ellip}
    \begin{aligned}    
 &\left\|\left\langle \frac{\pa_v}{t\pa_z} \right\rangle^{-1}\Delta_L\left(|\nabla|^{\frac{s}{2}}t^{-c_1}+\sqrt{\frac{\mathfrak w\pa_t w }{w }} +\sqrt{\frac{\mathfrak w\pa_t g }{g}} \right)\rmA \phi_{\neq}\right\|_{L^2}^2\\
 \lesssim&\mathrm{CK}_\lambda+\mathrm{CK}_W+\mathrm{CK}_G+\varepsilon^2\nu^{\beta}\sum_{i=1}^2 \left(\mathrm{CCK}_{\lambda}^{v,i}+\mathrm{CCK}^{v,i}_W+\mathrm{CCK}^{v,i}_G\right) .     
    \end{aligned}
  \end{equation}
\end{proposition}

Combing Proposition \ref{pro-reaction}, Proposition \ref{pro-transport}, and Proposition \ref{pro-elliptic-high}, we get \eqref{eq-est-boot-hf}. 

The proof of \eqref{eq-est-boot-hq} can be found in Section \ref{sec-hq}.

The proof of \eqref{eq-est-boot-hbh} can be found in Section \ref{sec-hbh}.

The proof of \eqref{eq-est-boot-hh} can be found in Section \ref{sec-hh}.

The proof of \eqref{eq-est-boot-lfn0} can be found in Section \ref{sec-lfn0}.

The proof of \eqref{eq-est-boot-lq} can be found in Section \ref{sec-lq}.

The proof of \eqref{eq-est-boot-lf0} can be found in Section \ref{sec-lf0}.

Thus we complete the proof of Proposition \ref{prop-boot}.
\subsection{Conclusion of the proof}
By Proposition \ref{prop-boot} we have a global uniform bound on $\mathcal E^{s}(t)$, and therefore the uniform bounds on 
\begin{equation}\label{eq-est-Gev-f}
  \begin{aligned}    
    &\|f(t)\|_{\mathcal G^{s,\lambda(t),\sigma}}^2+e^{c_M \nu^{\frac{1}{3}}t}\|f_{\neq}(t)\|_{\mathcal G^{s,\lambda(t),\sigma}}^2 +e^{c_M \nu^{\frac{1}{3}}t}t^4\|P_{\neq}\phi(t)\|_{\mathcal G^{s,\lambda(t),\gamma}}^2\\
  &+t^4\|q(t)\|_{\mathcal G^{s,\lambda(t),\gamma}}^2+\varepsilon\nu^{\beta}\|h(t)\|_{\mathcal G^{s,\lambda(t),\sigma}}^2\lesssim\varepsilon^2\nu^{2\beta}.  
  \end{aligned}
\end{equation}

Define $\lambda_{\infty}=\lim_{t\to\infty}\lambda(t)$. In order to complete the proof of Theorem \ref{Thm: main}, we undo the change of coordinates in $v$, switching to the coordinates $(z,y)$. Writing 
\begin{align*}
  \tilde \omega(t,z,y)=f(t,z,v)=\omega(t,x,y),\quad\tilde \psi(t,z,y)=\phi(t,z,v)=\psi(t,x,y),
\end{align*}
one derives from \eqref{eq-nl-om-chan}, as in Section \ref{sec-coor}, that
\begin{align*}
  &\pa_t\tilde \omega(t,z,y)+\nabla^{\perp}_{z,y}P_{\neq}\tilde \psi\cdot\nabla_{z,y}\tilde\omega(t,z,y)\\
  =&\nu \pa_z^2\tilde \omega(t, z,y)+\nu(\pa_yv)^2(\pa_y-t\pa_z)^2\tilde \omega(t, z,y).
\end{align*}
Similarly, denote $\tilde U(t,z,y)=\tilde u(t,z,v)=u^{(1)}(t,x,y)$, we deduce from \eqref{eq-nl-NS} that
\begin{align}
  \pa_tu_0^{(1)}+ \left( \na^{\bot}_{z,y}P_{\neq }\tilde \phi\cdot\nabla_{z,y} \tilde U\right)_0=\nu\pa_y^2u_0^{(1)}.
\end{align}

We may follow the argument in Section 2.3 of \cite{BM2015}, apply Lemma \ref{lem-inverse} and Lemma \ref{lem-compo} to deduce from \eqref{eq-est-Gev-f} that the estimates on $f,\phi$ imply estimates on $\tilde \omega,\tilde \psi$. The main issue is inverting the coordinate transform $y=y(t,v)$ and deducing good Gevrey regularity estimates on $v(t,y)-y$ and $v-y(t,v)$. The only difference from \cite{BM2015} is in establishing the $L^2$ control on $v(t,y)-y$, which here requires a straightforward estimate on the forced heat equation \eqref{eq-v}. We omit the details for brevity and conclude that, for $\varepsilon$ sufficiently small, there exists some $\lambda_{3}\in(\lambda_1,\lambda_{\infty})$ such that
\begin{align}\label{eq-est-Gev-om}
  \|\tilde \omega(t)\|_{\mathcal G^{s,\lambda_{3},\sigma}}^2+e^{c_M \nu^{\frac{1}{3}}t}\|\tilde \omega_{\neq}(t)\|_{\mathcal G^{s,\lambda_{3},\sigma}}^2+e^{c_M \nu^{\frac{1}{3}}t}t^4\|\tilde \psi_{\neq}(t)\|_{\mathcal G^{s,\lambda_{3},\gamma}}^2\lesssim\varepsilon^2\nu^{2\beta}.
\end{align}
This completes the proof of Theorem \ref{Thm: main}.

\section{Toy model and nonlinear growth}
In this section, we estimate the nonlinear growth through some toy models. 
\subsection{Formal derivation of the toy model}
From Section 2, we see that the basic challenge to the proof of Theorem \ref{Thm: main} is controlling the regularity and size of solutions to \eqref{eq: main f}. Since we must pay regularity to deduce decay on the velocity $\bar u$, it is natural to consider the frequency interactions in the product $\bar u\cdot\nabla f$ with the frequencies of $\bar u$ much larger than $f$, which corresponds to the ``reaction'' term. This leads us to study a simpler model
\begin{align*}
  \pa_t f-\nu\widetilde{\Delta}_tf=-\bar u\cdot\nabla f_{lo},
\end{align*}
where $f_{lo}$ is defined by $\hat f_{lo}(t,k,\eta)=\hat f(t,k,\eta)\mathrm{1}_{|k|\le 1,|\eta|\le C}$. As we see from \eqref{eq: main f}, $u$ consists of several terms, however, let us focus on the term we think should be the worst and also use the approximation $v'\approx 1$ and $v''\approx 0$, further reducing the problem:
\begin{align*}
  \pa_t f-\nu \Delta_Lf=\pa_vP_{\neq}\phi\pa_z f_{lo},\quad \Delta_L\phi=f,
\end{align*} 
and on the Fourier side:
\begin{align*}
  \partial_{t} \hat{f}(t, l, \eta)&+\nu\left(l^{2}+(\eta-l t)^{2}\right) \hat{f}(t, l, \eta)\\
  &=\int_{|\xi-\eta|\le C} \sum_{|k-l|= 1} \frac{\pm\xi}{k^{2}+(\xi-k t)^{2}} \hat{f}(t, k, \xi) \hat{f}_{lo}(t, \pm 1, \eta-\xi) d \xi.
\end{align*}
We think that $\hat{f}_{lo}(t, \pm1,\eta)$ behaves as the linear solution \eqref{eq: Lin-sol} which is restricted to the lower frequency, and has enhanced dissipation, thus $|\hat{f}_{lo}|\approx \mathrm{1}_{|\eta|\le C}\kappa \nu^\beta e^{-c\nu t^{3}}$, where $\kappa \nu^\beta$ stands for the size of the perturbation. Next, we replace $f_{lo}$ by 
\begin{align*}
\mathrm{1}_{k=\pm1,|\eta|\le C}\kappa \nu^\beta e^{-c\nu^{\frac{1}{3}}t},
\end{align*}
and regard $\hat{f}(t, k, \xi)$ as $\hat{f}(t, k, \eta)$ since $|\xi-\eta|<C$, and rewrite the equation as
\begin{align*}
  &\partial_{t} \hat{f}(t, l, \eta)+\nu\left(l^{2}+(\eta-l t)^{2}\right) \hat{f}(t, l, \eta)=\kappa \nu^\beta e^{-c\nu^{\frac{1}{3}}t}\sum_{|k-l|=1,k\neq0} \frac{\eta(l-k)}{k^{2}+(\eta-k t)^{2}} \hat{f}(t, k, \eta).
\end{align*}
As time advances this system of ODEs will go through resonances or the Orr critical time $t=\frac{\eta}{k}$, at which time the $k$ mode strongly forces the $k\pm 1$ modes, meanwhile, the $k\pm 1$ modes have only a weak effect on $k$ mode. Indeed, for $t\in \mathrm{I}_{k, \eta}$ with $|k|\lesssim |\eta|^{\frac{1}{2}}$, it holds that $(k\pm1)^{2}+(\eta-(k\pm1) t)^{2}\approx \frac{\eta^2}{k^2}$. Therefore we formally have for $t\in \mathrm I_{k,\eta}$
\begin{align*}
    \partial_{t} \hat{f}(t, k\pm1, \eta)\approx -\nu\frac{\eta^2}{k^2}\hat{f}(t, k\pm1, \eta)+\kappa \nu^\beta e^{-c\nu^{\frac{1}{3}}t}\frac{\pm\eta}{k^{2}+(\eta-k t)^{2}} \hat{f}(t, k, \eta),\\
    \partial_{t} \hat{f}(t, k, \eta)\approx -\nu\left(k^{2}+(\eta-k t)^{2}\right) \hat{f}(t, k, \eta)+\kappa \nu^\beta e^{-c\nu^{\frac{1}{3}}t}\sum_{\pm} \frac{k^2}{\pm\eta} \hat{f}(t, k\pm1, \eta).
\end{align*}
Here the $k$ mode is the resonant mode and the $k\pm 1$ modes are the non-resonant modes.

Next, we will carefully analyze the resonance effect of the resonant mode on the non-resonant modes and further simplify the system. We have the following observations (here we usually take $\bm{\eta,k\ge0}$ but the observations are also valid for the case $\eta,k<0$. Note that modes where $\eta k<0$ do not have resonance for positive times):
\begin{itemize}
  \item When $t\gg \nu^{-\frac{1}{3}}$, the enhanced dissipation will offer an extremely small coefficient which makes the resonance effect weaker. So we focus on the region of time $t\lesssim \nu^{-\frac{1}{3}}$. During this region of time, $e^{-c\nu^{\frac{1}{3}}t}\approx 1$.
  \item For $t\in \mathrm{I}_{k,\eta}$, the strength of viscous effect on the $k\pm1$ mode is
  $$\nu\left((k\pm1)^{2}+\big(\eta-(k\pm1) t\big)^{2}\right)\approx \nu \frac{|\eta|^2}{|k|^2}\approx \nu t^2\lesssim \nu^{\frac{1}{3}},$$
  which is weak.

  \item The rapid growth of $\hat{f}(t, k\pm1, \eta)$ happens when $|t-\frac{\eta}{k}|\lesssim 1$. If $\frac{|\eta|^{1-3\beta}}{|k|^{2-3\beta}}> 1$ and $t\approx \nu^{-\frac{1}{3}}$, then the strength of resonance effect could be $\nu^\beta e^{-c\nu^{\frac{1}{3}}t}\frac{|\eta|}{|k|^2}\approx\left(\nu^{\frac{1}{3}}t\right)^{3\beta}e^{-c\nu^{\frac{1}{3}}t}\frac{|\eta|^{1-3\beta}}{|k|^{2-3\beta}}\gtrsim 1$. We can see that this effect is much stronger than the viscous effect, so we call it the strong resonance. 

  \item There are two other cases where the resonance effect is still stronger than the viscous effect. For $|t-\frac{\eta}{k}|\approx \frac{|\eta|^{1-3\beta}}{|k|^{2-3\beta}}$ with $k$ such that $\frac{|\eta|^{1-3\beta}}{|k|^{2-3\beta}}> 1$ and $t\approx \nu^{-\frac{1}{3}}$, it holds that
\begin{align*}
  \nu^\beta e^{-c\nu^{\frac{1}{3}}t} \frac{|\eta|}{k^{2}+(\eta-k t)^{2}}\approx \frac{|k|^{2-3\beta}}{|\eta|^{1-3\beta}}\gtrsim |k|t^{-1}\approx |k|\nu^{\frac{1}{3}}\gtrsim \nu^{\frac{1}{3}}.
\end{align*}
  And, for $|t-\frac{\eta}{k}|\approx 1$ with $k$ such that $\frac{|\eta|^{1-3\beta}}{|k|^{2-3\beta}}<1$ and $\frac{|\eta|}{|k|^2}>1$, it holds that
\begin{align*}
  \nu^\beta e^{-c\nu^{\frac{1}{3}}t} \frac{|\eta|}{k^{2}+(\eta-k t)^{2}}\approx \nu^\beta \frac{|\eta|}{|k|^2}\gtrsim \nu^{\frac{1}{3}}.
\end{align*}
We call these two cases weak resonance.
\item For both strong and weak resonance, the scenario we are most concerned with is a high-to-low cascade in which the $k$ mode has a strong effect at time $\frac{\eta}{k}$ that excites the $k-1$ mode which has a strong effect at time $\frac{\eta}{k-1}$ that excites the $k-2$ mode and so on. The echoes observed experimentally in \cite{YD2002,YDO2005} arise from an effect similar to this \cite{VMW1998,Vanneste2001}.
\end{itemize}
From the above observations, we deduce the following two toy models to reflect the strong and weak resonance effects respectively:
\begin{equation}\label{eq-toy-strong}
  \left\{
    \begin{array}{l}
    \pa_t f_{\mathrm {NR}}^{st}=\kappa\frac{|\eta|^{1-3\beta}}{|k|^{2-3\beta}}\frac{1}{(1+|t-\frac{\eta}{k}|^2)}f_{\mathrm {R}}^{st},\\
      \pa_tf_{\mathrm {R}}^{st}=\kappa\frac{|k|^{2-3\beta}}{|\eta|^{1-3\beta}}f_{\mathrm {NR}}^{st},
    \end{array}
  \right.\text{ for }t\in \tilde{\mathrm{I}}_{k,\eta},
\end{equation}
\begin{equation}\label{eq-toy-weak}
  \f{\pa_t f^{we}(t,\eta)}{f^{we}(t,\eta)}=\left\{
    \begin{array}{l}
      \frac{\frac{|\eta|^{1-3\beta}}{|k|^{2-3\beta}}}{\left(\frac{|\eta|^{1-3\beta}}{|k|^{2-3\beta}}\right)^2+(t-\frac{\eta}{k})^2},\text{ for }t\in \mathrm{I}_{k,\eta},\ \frac{|\eta|^{1-3\beta}}{|k|^{2-3\beta}}\ge1;\\
      \frac{\kappa \nu^\beta e^{-c\nu^{\frac{1}{3}}t}\frac{|\eta|}{|k|^2}}{1+(t-\frac{\eta}{k})^2},\text{ for } t\in \mathrm{I}_{k,\eta},\ \frac{|\eta|^{1-3\beta}}{|k|^{2-3\beta}}<1.
    \end{array}
  \right.
\end{equation}

In \eqref{eq-toy-strong}, we think of $f_{\mathrm {R}}^{st}$ as being the resonant mode ($k$ mode) and $f_{\mathrm {NR}}^{st}$ being the non-resonant modes ($k\pm1$ mode). We drop the viscous terms as the viscous effect is weak than the resonance effect.  Our goal is to catch the possible worst growth, so we appropriately magnify the growth factors. In the evolution equation of $f_{\mathrm {NR}}^{st}$, we slightly magnify the factor from $\kappa \nu^\beta e^{-c\nu^{\frac{1}{3}}t}\frac{|\eta|}{k^{2}+(\eta-k t)^{2}}$ to $\kappa\frac{|\eta|^{1-3\beta}}{|k|^{2-3\beta}}\frac{1}{(1+|t-\frac{\eta}{k}|^2)}$. While in the evolution equation of $f_{\mathrm {R}}^{st}$, we significantly magnify the factor from $\kappa \nu^\beta e^{-c\nu^{\frac{1}{3}}t}\frac{k^2}{|\eta|}$ to $\kappa\frac{|k|^{2-3\beta}}{|\eta|^{1-3\beta}}$. However, $\kappa\frac{|k|^{2-3\beta}}{|\eta|^{1-3\beta}}$ is still less than the minimum of $\kappa\frac{|\eta|^{1-3\beta}}{|k|^{2-3\beta}}\frac{1}{(1+|t-\frac{\eta}{k}|^2)}$ during $t\in\tilde{\mathrm{I}}_{k,\eta}$.

In \eqref{eq-toy-weak}, we also ignore the viscous terms. For the weak resonance case, the effect of resonant mode on the non-resonant modes is not much stronger than the one of non-resonant modes on the resonant mode, so here we use $f^{we}$ to denote all the modes, without distinguishing whether they are resonant modes or not. For $t\in \mathrm{I}_{k,\eta}$ with $\ \frac{|\eta|^{1-3\beta}}{|k|^{2-3\beta}}\ge1$, we use the factor $ \frac{\frac{|\eta|^{1-3\beta}}{|k|^{2-3\beta}}}{\left(\frac{|\eta|^{1-3\beta}}{|k|^{2-3\beta}}\right)^2+(t-\frac{\eta}{k})^2}$ instead of $ \frac{\frac{|\eta|^{1-3\beta}}{|k|^{2-3\beta}}}{1+(t-\frac{\eta}{k})^2}$. This is because this weak resonance only happens on $t\in \mathrm{I}_{k,\eta}\setminus\tilde{\mathrm{I}}_{k,\eta}$, and we extend the definition of this toy model to the whole $\mathrm{I}_{k,\eta}$.
\subsection{Strong resonance and weight $W$}\label{sec-srw}
\subsubsection{Construction of $w$}\label{sss-w}
Keeping with the intuition from the derivation of \eqref{eq-toy-strong}, in this part we will think of $\eta$ as a fixed parameter and time-varying. 

This method was developed by Bedrossian-Masmoudi in \cite{BM2015}. A key feature of this method is how the toy model is used to construct a norm that precisely matches the estimated worst-case behavior that the reaction terms create, done by choosing $w_k(t,\eta)$ to be an approximate solution to \eqref{eq-toy-strong}. Inspired by the weight function constructed in \cite{BM2015}, we introduce the weight functions $w_{\mathrm{R}}$ and $w_{\mathrm{NR}}$ to match behavior for the resonant and non-resonant mode respectively.

Recall the definitions of $t_{k, \eta}$, $t_{k, \eta}^{\pm}$, $\mathrm{I}_{k, \eta}$, $\mathrm{I}_{k, \eta}^L$, $\ \mathrm{I}_{k, \eta}^R$, $\tilde{\mathrm{I}}_{k, \eta}$, $\tilde{\mathrm{I}}_{k, \eta}^{L}$, and $\tilde{\mathrm{I}}_{k, \eta}^{R}$ that given in Section \ref{sec-notation}. For $k\eta>0$, $|\eta|\geq 1$ and $|k|=1,...,E(|\eta|^{s(\beta)})$ with $s(\beta)=\frac{1-3\beta}{2-3\beta}$, we define
\begin{align*}
w_{\mathrm{NR}}(t,\eta)=w_{\mathrm{R}}(t,\eta)=1,\quad \forall t\geq t_{1,|\eta|}^+,
\end{align*}
for $t\in \tilde{\mathrm{I}}_{k, \eta}$ with $|k|\geq 1$, 
\begin{align*}
w_{\mathrm{NR}}(t,\eta)=\left(\frac{|k|^{2-3\beta}}{|\eta|^{1-3\beta}}\left[1+a_{k,\eta}\left|t-\frac{\eta}{k}\right|\right]\right)^{C_1\kappa}w_{\mathrm{NR}}(t_{k,\eta}^+,\eta),\quad &\forall t\in \tilde{\mathrm{I}}_{k, \eta}^{R},\\
w_{\mathrm{R}}(t,\eta)=\frac{|k|^{2-3\beta}}{|\eta|^{1-3\beta}}\left[1+a_{k,\eta}\left|t-\frac{\eta}{k}\right|\right]w_{\mathrm{NR}}(t,\eta),\quad &\forall t\in \tilde{\mathrm{I}}_{k, \eta}^{R},\\
w_{\mathrm{NR}}(t,\eta)=\left(1+a_{k,\eta}\left|t-\frac{\eta}{k}\right|\right)^{-1-C_1\kappa}w_{\mathrm{NR}}\left(\frac{\eta}{k},\eta\right),\quad &\forall t\in \tilde{\mathrm{I}}_{k, \eta}^{L},\\
w_{\mathrm{R}}(t,\eta)=\frac{|k|^{2-3\beta}}{|\eta|^{1-3\beta}}\left[1+a_{k,\eta}\left|t-\frac{\eta}{k}\right|\right]w_{\mathrm{NR}}\left(t,\eta\right),\quad &\forall t\in \tilde{\mathrm{I}}_{k, \eta}^{L},
\end{align*}
and for $t\in [t_{k,\eta}^+,t_{|k|-1,|\eta|}^-]$ with $|k|\geq 2$,
\begin{align*}
w_{\mathrm{NR}}(t,\eta)=w_{\mathrm{R}}(t,\eta)=w_{\mathrm{NR}}(t_{k,\eta}^+,\eta)=w_{\mathrm{NR}}(t_{|k|-1,|\eta|}^-,\eta)=w_{\mathrm{R}}(t_{k,\eta}^+,\eta)=w_{\mathrm{R}}(t_{|k|-1,|\eta|}^-,\eta),
\end{align*}
where
\begin{align*}
a_{k,\eta}=8\left(1-\frac{|k|^{2-3\beta}}{|\eta|^{1-3\beta}}\right).
\end{align*}
The choice of $a_{k,\eta}$ was made to ensure for $|k|\geq 1$ that
\begin{align*}
w_{\mathrm{NR}}(t_{k,\eta}^+,\eta)=w_{\mathrm{R}}(t_{k,\eta}^+,\eta)=w_{\mathrm{NR}}(t_{k,\eta}^-,\eta)\left(\frac{|\eta|^{1-3\beta}}{|k|^{2-3\beta}}\right)^{1+2C_1\kappa}=w_{\mathrm{R}}(t_{k,\eta}^-,\eta)\left(\frac{|\eta|^{1-3\beta}}{|k|^{2-3\beta}}\right)^{1+2C_1\kappa}.
\end{align*}
Finally, we take $w_{\mathrm{NR}}(t,\eta)=w_{\mathrm{R}}(t,\eta)=w_{\mathrm{NR}}(t_{E(|\eta|^{s(\beta)}),|\eta|}^-,\eta)=w_{\mathrm{R}}(t_{E(|\eta|^{s(\beta)}),|\eta|}^-,\eta)$ for $t\in[0,t_{E(|\eta|^{s(\beta)}),|\eta|}^-]$. And for $|\eta|<1$, we define $w_{\mathrm{NR}}(t,\eta)=w_{\mathrm{R}}(t,\eta)=1$ for $t\in[0,+\infty]$.

For the remainder of the paper, we fix $\kappa$ such that $1+2C_1\kappa=2$.

One can easily check that $w_{\mathrm{NR}}$ and $w_{\mathrm{R}}$ defined above satisfy
\begin{equation}\label{eq-toy-strong2}
\left\{\begin{aligned}
\pa_t w_{\mathrm{NR}}&\approx \ka\frac{|\eta|^{1-3\beta}}{|k|^{2-3\beta}}\frac{1}{(1+|t-\frac{\eta}{k}|^2)}w_{\mathrm{R}},\\
\pa_tw_{\mathrm{R}}&\approx \ka\frac{|k|^{2-3\beta}}{|\eta|^{1-3\beta}}w_{\mathrm{NR}},
\end{aligned}
\right.
\end{equation}
which is consistent with \eqref{eq-toy-strong}.

Note that we always have $0\le a_{k,\eta}<8$, but that $a_{k,\eta}$ approaches $0$ when $k$ approaches $E(|\eta|^{s(\beta)})$. This will present minor technical difficulties in the sequel since this implies that $\pa_tw$ vanishes near this time, hence a loss of the lower bounds in \eqref{eq-toy-strong2}. See Lemma \ref{lem-evo-w-lntype}.

We remark that $s(\beta)=\frac{1-3\beta}{2-3\beta}\in[0,\frac{1}{2}]$, and $1-s(\beta)=\frac{1}{2-3\beta}\in[\frac{1}{2},1]$. In the rest of this manuscript, we will abbreviate $s(\beta)$ to $s$.

Next, we define for $k\eta\ge0$ that
\begin{equation}\label{eq-def-wk1}
  w_k(t,\eta)=\left\{
    \begin{array}{ll}
      w_k(t_{E(|\eta|^s),|\eta|},\eta),&t<t_{E(|\eta|^s),|\eta|},\\
      w_{\mathrm{NR}}(t,\eta),&t\in [t_{E(|\eta|^s),|\eta|},2|\eta|]\setminus \mathrm{I}_{k,\eta},\\
      w_{\mathrm{R}}(t,\eta),&t\in \mathrm{I}_{k,\eta},\\
      1,&t>2|\eta|,
    \end{array}
  \right.
\end{equation}
and for $k\eta<0$ that
\begin{equation}\label{eq-def-wk2}
  w_k(t,\eta)=w_0(t,\eta)=\left\{
    \begin{array}{ll}
      w_k(t_{E(|\eta|^s),|\eta|},\eta),&t<t_{E(|\eta|^s),|\eta|},\\
      w_{\mathrm{NR}}(t,\eta),&t\in [t_{E(|\eta|^s),|\eta|},2|\eta|],\\
      1,&t>2|\eta|.
    \end{array}
  \right.
\end{equation}
Since $w_{\mathrm{R}}$ and $w_{\mathrm{NR}}$ agree at the end-points of $\mathrm{I}_{k,\eta}$, $w_k(t,\eta)$ is Lipschitz continuous in time. The different behaviors of $w_k$ on the critical time interval and non-critical time interval are important in establishing the energy estimates. We also remark that, here we also define $w_k(t,\eta)$ for the case $k\eta<0$ which is different to \cite{BM2015} and \cite{BMV2016}.

\subsubsection{Definition of $W$}
Let us first introduce a transition function between the horizontal  frequency $k$ and the vertical frequency $\eta$:
\begin{equation}\label{eq-def-iota}
  \iota(k,\eta)=\left\{
    \begin{array}{ll}
      k,&|k|>|\eta|,\\
      \eta,&|k|\le|\eta|.
    \end{array}
  \right.
\end{equation}
\begin{lemma}
  It holds that
  \begin{align*}
  \big||\iota(k,\eta)|-|\iota(m,\xi)|\big|\le \left| k-m,\eta-\xi\right|.
\end{align*}
\end{lemma}
\begin{proof}
 For $(|\eta|\ge|k|,|\xi|\ge|m|)$ or $(|\eta|\le|k|,|\xi|\le|m|)$, there is nothing to prove. If $(|\eta|\le|k|,|\xi|\ge|m|)$, we have
\begin{align*}
  \big||\iota(k,\eta)|-|\iota(m,\xi)|\big|=\big||k|-|\xi|\big|,
\end{align*}
if $|\xi|\ge|k|$, then 

\begin{equation}\label{eq-comp-m-k-xi-eta}
  \big||\iota(k,\eta)|-|\iota(m,\xi)|\big|=\big||k|-|\xi|\big|\le\left\{
    \begin{array}{ll}
      \big||\eta|-|\xi|\big|\le|\eta-\xi|,&\text{ for }|\xi|\ge|k|,\\
      \big||k|-|m|\big|\le|k-m|,&\text{ for }|\xi|\le|k|.
    \end{array}
  \right.
\end{equation}
The estimates for $(|\eta|\ge|k|,|\xi|\le|m|)$ are similar. 
\end{proof}
We define the following multiplier to deal with the strong resonance in the nonlinear interaction:
\begin{align}\label{def-weight-W}
  W_k(t,\eta)=\frac{1}{w_k(t,\iota(k,\eta))}.
\end{align}
The main reason for using the transition function $\iota(k,\eta)$ is to balance the derivative in $z$ and $v$. A similar transition function is used in \cite{MasmoudiZhao2019}, also see \cite{BM2015} for a different way with the same purpose. 
Note that for $|\eta|< |k|$, we have
\begin{align*}
  \frac{1}{w_k(t,\iota(k,\eta))}=\frac{1}{w_k(t,k)}=\f{1}{w_k(t,-k)}.
\end{align*}

Besides, we also introduce a special multiplier $W^R$ which will be used in the energy estimates for the coordinate functions. We define 
\begin{equation}
  w^{\mathrm R}(t,\eta)=\left\{
    \begin{array}{ll}
      w^{\mathrm R}(t_{E(|\eta|^s),|\eta|},\eta),&t<t_{E(|\eta|^s),|\eta|},\\
      w_{\mathrm{R}}(t,\eta),&t\in [t_{E(|\eta|^s),|\eta|},2|\eta|],\\
      1,&t>2|\eta|,
    \end{array}
  \right.
\end{equation}
and 
\begin{align}\label{def-weight-WR}
  W^{\mathrm R}(t,\eta)=\frac{1}{w^{\mathrm R}(t,\eta)}.
\end{align}

\subsection{Weak resonance and $G$}\label{sec-wrg}
Similar to the construction of $w$, we now introduce the weight function $g$ as an approximate solution to \eqref{eq-toy-weak}. Here we also think of $\eta$ as a fixed parameter and time-varying.

Let $g(t,\eta)$ be the solution of the following ODE for $|\eta|\ge1$:
\begin{equation}\label{eq-toy-weak2}
  \pa_t g(t,\eta)=\left\{
    \begin{array}{l}
      \frac{\kappa \frac{|\eta|^{1-3\beta}}{|k|^{2-3\beta}}}{\left(\frac{|\eta|^{1-3\beta}}{|k|^{2-3\beta}}\right)^2+(t-\frac{|\eta|}{|k|})^2}g(t,\eta),\text{ for }t\in \mathrm{I}_{k,\eta},\ 1\le|k|\le E(|\eta|^s);\\
      \frac{\kappa\mathfrak g(\nu,\eta)\langle \nu^{\frac{1}{3}}\frac{\eta}{k}\rangle^{-\frac{3}{2}}\nu^\beta\frac{|\eta|}{|k|^2}}{1+(t-\frac{|\eta|}{|k|})^2}g(t,\eta),\text{ for }t\in \mathrm{I}_{k,\eta},\ E(|\eta|^s)+1\le|k|\le E(|\eta|),
    \end{array}
  \right.
\end{equation}
where  $\mathfrak g(\nu,\eta)=\langle\nu^{-\frac{1}{3}}|\eta|^{s-1}\rangle^{3\beta}$. We take the boundary condition that
\begin{align*}
  g(t,\eta)|_{t=2|\eta|}=1,\  g(t,\eta)|_{t=t_{|k|,|\eta|}}=g(t_{|k|,|\eta|},\eta)\text{ for }1\le|k|\le E(|\eta|),
\end{align*}
we also let
\begin{align*}
  g(t,\eta)=1,\ \forall t\ge 2|\eta|,\ g(t,\eta)=g(t_{E(|\eta|),|\eta|},\eta),\ \forall t\le t_{E(|\eta| ),|\eta|}.
\end{align*}
For $|\eta|<1$, we define $g(t,\eta)=1$ for $t\in[0,+\infty]$. It is easy to check that $g(t,-\eta)=g(t,\eta)$.

Compared to \eqref{eq-toy-weak}, here we replace $e^{-c\nu^{\frac{1}{3}}t}$ to $\langle \nu^{\frac{1}{3}}\frac{\eta}{k}\rangle^{-\frac{3}{2}}$, and extend the growing time interval from $[|\eta|^{\frac{1}{2}},2|\eta|]$ to $[1,2|\eta|]$ owing to some technical reasons (See Lemma \ref{lem-g-growth} and Section \ref{sec-react-nrnr},respectively). We introduce $\mathfrak g(\nu,\eta)$ to ensure that the value of $\pa_t g(t,\eta)$ does not change drastically on $t=t_{E(|\eta|^s,|\eta|)}$ for $|\eta|^s\le \nu^{-\frac{1}{3}}$, and this property will be used in the proof of Lemma \ref{lem-w-wgl} and Lemma \ref{lem-g-gl}.
\begin{remark}
  For $|\eta|\ge1$, let $k_1=E(|\eta|^s)$, $k_2=E(|\eta|^s)+1$, and $t_1\in \mathrm{I}_{k_1,\eta}$, $t_2\in \mathrm{I}_{k_2,\eta}$. It is clear that $\frac{|\eta|^{1-3\beta}}{|k_i|^{2-3\beta}}\approx 1$. If $|\eta|^{1-s}\le \nu^{-\frac{1}{3}}$, we have $\mathfrak g(\nu,\eta)\approx \nu^{-\beta}|\eta|^{-3\beta(1-s)}$, $\langle \nu^{\frac{1}{3}}\frac{\eta}{k_i}\rangle^{-\frac{3}{2}}\approx 1$, and
  \begin{align*}
  \frac{\pa_t g(t_1,\eta)}{g(t_1,\eta)}=\frac{\kappa \frac{|\eta|^{1-3\beta}}{|k_1|^{2-3\beta}}}{\left(\frac{|\eta|^{1-3\beta}}{|k_1|^{2-3\beta}}\right)^2+(t-\frac{|\eta|}{|k_1|})^2}\approx \frac{1}{1+\left|\frac{|\eta|}{|k_1|}-t\right|^2}\\
  \frac{\pa_t g(t_2,\eta)}{g(t_2,\eta)}=\frac{\kappa\mathfrak g(\nu,\eta)\langle \nu^{\frac{1}{3}}\frac{\eta}{k_2}\rangle^{-\frac{3}{2}}\nu^\beta\frac{|\eta|}{|k_2|^2}}{1+(t-\frac{|\eta|}{|k_2|})^2}\approx \frac{1}{1+\left|\frac{|\eta|}{|k_2|}-t\right|^2}
  \end{align*}
  If $|\eta|^{1-s}> \nu^{-\frac{1}{3}}$, we have $\mathfrak g(\nu,\eta)\approx 1$, $\langle \nu^{\frac{1}{3}}\frac{|\eta|}{|k_i|}\rangle^{-\frac{3}{2}}\approx \nu^{-\frac{1}{2}}|\eta|^{-\frac{3}{2}(1-s)}$, and
    \begin{align*}
  \frac{\pa_t g(t_1,\eta)}{g(t_1,\eta)}=\frac{\kappa \frac{|\eta|^{1-3\beta}}{|k_1|^{2-3\beta}}}{\left(\frac{|\eta|^{1-3\beta}}{|k_1|^{2-3\beta}}\right)^2+(t-\frac{|\eta|}{|k_1|})^2}\approx \frac{1}{1+\left|\frac{|\eta|}{|k_1|}-t\right|^2}\\
  \frac{\pa_t g(t_2,\eta)}{g(t_2,\eta)}=\frac{\kappa\mathfrak g(\nu,\eta)\langle \nu^{\frac{1}{3}}\frac{\eta}{k_2}\rangle^{-\frac{3}{2}}\nu^\beta\frac{|\eta|}{|k_2|^2}}{1+(t-\frac{|\eta|}{|k_2|})^2}\approx \frac{(\nu^{\frac{1}{3}}t)^{-\frac{3}{2}+3\beta}}{1+\left|\frac{\eta}{k_2}-t\right|^2}.
  \end{align*}
So there is a big gap for the case $|\eta|^{1-s}> \nu^{-\frac{1}{3}}$. 

This gap generates from the fact that the magnification of the coefficient for $w_{\mathrm{NR}}$ and $w_{\mathrm{R}}$ is too large. And if we use the definition of $w_{\mathrm{NR}}$ and $w_{\mathrm{R}}$ that are coincident with the ones in \cite{BM2015}, this gap is inevitable. However, for such a case, enhanced dissipation has occurred. In Lemma \ref{lem-w-wgl}, we introduce the correction function $\mathfrak w$ to eliminate this gap.
\end{remark}

From the above definition, we have for $t\in \mathrm{I}_{k,\eta},\ 1\le|k|\le E(|\eta|^s)$
\begin{align}\label{eq-toy-g-evo1}
  g(t,\eta)=&e^{\kappa  \left[\arctan \left(\left(\frac{|\eta|^{1-3\beta}}{|k|^{2-3\beta}}\right)^{-1}\left(t-\frac{|\eta|}{|k|}\right)\right)+\arctan \left(\left(\frac{|\eta|^{1-3\beta}}{|k|^{2-3\beta}}\right)^{-1}\frac{|\eta|}{(2|k|+1)|k|}\right)\right]}g(t_{|k|,|\eta|},\eta),
\end{align}
and for $t\in \mathrm{I}_{k,\eta},\ E(|\eta|^s)+1\le|k|\le E(|\eta|)$
\begin{align}\label{eq-toy-g-evo2}
  g(t,\eta)=&e^{\kappa\mathfrak g(\nu,\eta) \langle \nu^{\frac{1}{3}}\frac{\eta}{k}\rangle^{-\frac{3}{2}}\nu^\beta\frac{|\eta|}{|k|^2}\left[\arctan \left(t-\frac{|\eta|}{|k|}\right)+\arctan \left(\frac{|\eta|}{(2|k|+1)|k|}\right)\right]}g(t_{|k|,|\eta|},\eta).
\end{align}

Then we define the following multiplier to deal with the weak resonance in the nonlinear interaction:
\begin{align}\label{def-weight-G}
  G_k(t,\eta)=\frac{1}{g(t,\iota(k,\eta))}.
\end{align}

\subsection{Enhanced dissipation and $\mathfrak M$}
In the previous works \cite{BMV2016,BVW2018,MasmoudiZhao2019}, the authors prove the enhanced dissipation in a weak sense. In Theorem \ref{Thm: main}, the enhanced dissipation we got is in a strong sense. The difference is that we introduce a weight function to extract the enhanced dissipation.

We define $\mathfrak m_k(t,\eta)$ as follows:
\begin{itemize}
  \item[(1)] if $k=0$, $\mathfrak m_k(t,\eta)=1$;
  \item[(2)] if $k\neq0$, $\frac{\pa_t \mathfrak m_k(t,\eta)}{\mathfrak m_k(t,\eta)}=\frac{\nu^{\frac{1}{3}}}{1+\nu^{\frac{2}{3}}(\eta-kt)^2}$, $\mathfrak m_k(0,\eta)=1$. 
\end{itemize}
Direct calculation yields
\begin{align*}
  \mathfrak m_k(t,\eta)=\exp \left (\frac{1}{k}\left[\arctan \left(\nu^{\frac{1}{3}} k t-\nu^{\frac{1}{3}} \eta\right)+\arctan \left(\nu^{\frac{1}{3}} \eta\right)\right]\right ), \text{ for }k \neq 0.
\end{align*}

Then we define the following multiplier:
\begin{align}\label{def-weight-M}
  \mathfrak M_k(t,\eta)=\frac{1}{\mathfrak m_k(t,\eta)}.
\end{align}

It is clear that
\begin{align*}
\mathbf{1}_{k \neq 0}\left|\frac{1}{k}\left[\arctan \left(\nu^{\frac{1}{3}} k t-\nu^{\frac{1}{3}} \eta\right)+\arctan \left(\nu^{\frac{1}{3}} \eta\right)\right]\right| \leq \pi,
\end{align*}
then
\begin{align*}
  e^{-\pi} \le \mathfrak M_k(t,\eta)=\frac{1}{\mathfrak m_k(t, \eta)}\le e^{\pi}, \text{ for }(k, \eta) \in \mathbb{Z} \times \mathbb{R},
\end{align*}
and $\mathfrak M_k(t,\eta)$ is a ghost type weight \cite{Alinhac2001}.

The following lemma shows the origin of enhanced dissipation.
\begin{lemma}\label{lem-nu13}
  If $k\neq0$, then there holds
  \begin{align}
    \nu(\eta-kt)^2+\frac{\pa_t\mathfrak m_k(t, \eta)}{\mathfrak m_k(t, \eta)}\ge \frac{1}{2}\nu^{\frac{1}{3}}.
  \end{align}
\end{lemma}
\begin{proof}
  If $|\eta-kt|\ge \nu^{-\frac{1}{3}}$, we have
  \begin{align*}
    \nu(\eta-kt)^2=\nu^{\frac{1}{3}}\left(\nu^{\frac{1}{3}}|\eta-kt|\right)^2\ge \frac{1}{2}\nu^{\frac{1}{3}}.
  \end{align*}
  On the other hand, if $|\eta-kt|< \nu^{-\frac{1}{3}}$, it holds that
  \begin{align*}
    \frac{\pa_t\mathfrak m_k(t, \eta)}{\mathfrak m_k(t, \eta)}=\frac{\nu^{\frac{1}{3}}}{1+\nu^{\frac{2}{3}}(\eta-kt)^2}\ge \frac{1}{2}\nu^{\frac{1}{3}}.
  \end{align*}
\end{proof}
\subsection{Total growth of $w_k$ and $g$}
The following lemma shows that the toy models predict a growth of high frequencies which amounts to a loss of Gevrey-$\frac{1}{s}$ regularity, which is the primary origin of the setting in Theorem \ref{Thm: main}. The primary purpose of the toy models \eqref{eq-toy-strong} and \eqref{eq-toy-weak} is to provide precise mode-by-mode information about how this loss can occur.

\begin{lemma}[Growth of $w$]\label{lem-grow-w}
  For $|\eta|\ge1$, there exists $\mu=2(2-3\beta)(1+2C_1\kappa)$ such that
  \begin{align}\label{eq-grow-w}
    \frac{w_k(2|\eta|,\eta)}{w_k(0,\eta)}=\frac{1}{w_k(0,\eta)}=\frac{1}{w_k(t_{E(|\eta|^s,|\eta|)},\eta)}\approx \frac{1}{|\eta|^{\frac{\mu s}{4}}}e^{\frac{\mu}{2}|\eta|^s}.
  \end{align}
\end{lemma}
\begin{proof}
  We have
  \begin{align*}
    \frac{1}{w_k(t_{E(|\eta|^s,|\eta|)},\eta)}=&\prod_{k=1}^{E(|\eta|^s)}\left(\frac{|\eta|^{1-3\beta}}{|k|^{2-3\beta}}\right)^{1+2C_1\kappa}=\left(\frac{|\eta|^{(1-3\beta)\mathcal N}}{(\mathcal N!)^{2-3\beta}}\right)^{(1+2C_1\kappa)}. 
  \end{align*}
  Here $\mathcal N=E(|\eta|^s)$. By Stirling's approximation, we have $\mathcal N!\approx\sqrt{2\pi \mathcal N}\left(\frac{\mathcal N}{e}\right)^\mathcal N$, and
  \begin{align*}
    \frac{|\eta|^{(1-3\beta)\mathcal N}}{(\mathcal N!)^{2-3\beta}}\approx \mathcal N^{-\frac{2-3\beta}{2}}e^{(2-3\beta)\mathcal N}\left(\frac{|\eta|^{1-3\beta}}{\mathcal N^{2-3\beta}}\right)^\mathcal N\approx |\eta|^{-\frac{1-3\beta}{2}}e^{(2-3\beta)|\eta|^s}.
  \end{align*}
  Therefore,
  \begin{align*}
    \frac{1}{w_k(t_{E(|\eta|^s,|\eta|)},\eta)}\approx \frac{1}{|\eta|^{\frac{(1-3\beta)(1+2C_1\kappa)}{2}}}e^{(2-3\beta)(1+2C_1\kappa)|\eta|^s}.
  \end{align*}
\end{proof}
\begin{lemma}[Growth of $g$]\label{lem-g-growth}
  For $|\eta|\ge1$, there exists $\tilde \mu=C_2\kappa$ such that
  \begin{align}\label{eq-est-g-exp}
    &\frac{g(2|\eta|,\eta)}{g(0,\eta)}=\frac{1}{g(0,\eta)}=\frac{1}{g(t_{E(|\eta|),\eta},\eta)}\le e^{\tilde \mu|\eta|^{s}}.
  \end{align}
This gives us for any $t>0$ that,
\begin{align*}
  1\le \frac{1}{g(t,\eta)}\le e^{\tilde \mu|\eta|^{s}}.
\end{align*}
\end{lemma}
\begin{proof}
Recalling \eqref{eq-toy-g-evo1} and \eqref{eq-toy-g-evo2}, we have
\begin{align*}
  g(t_{|k|-1,|\eta|},\eta)=e^{\kappa F(k,\eta)}g(t_{|k|,|\eta|},\eta),&\text{ for }1\le|k|\le E(|\eta|^s),\\
  g(t_{|k|-1,|\eta|},\eta)=e^{\kappa\tilde F(k,\eta)}g(t_{|k|,|\eta|},\eta),&\text{ for }E(|\eta|^s)+1\le|k|\le E(|\eta|),
\end{align*}
where
\begin{align*}
  F(k,\eta)=& \arctan \left(\left(\frac{|\eta|^{1-3\beta}}{|k|^{2-3\beta}}\right)^{-1}\frac{|\eta|}{(2|k|-1)|k|}\right)+\arctan \left(\left(\frac{|\eta|^{1-3\beta}}{|k|^{2-3\beta}}\right)^{-1}\frac{|\eta|}{(2|k|+1)|k|}\right),\\
  \tilde F(k,\eta)=&\mathfrak g(\nu,\eta)\langle \nu^{\frac{1}{3}}\frac{\eta}{k}\rangle^{-\frac{3}{2}}\nu^\beta\frac{|\eta|}{|k|^2}\left[\arctan \left(\frac{|\eta|}{(2|k|-1)|k|}\right)+\arctan \left(\frac{|\eta|}{(2|k|+1)|k|}\right)\right].
\end{align*}
It follows that
\begin{align*}
  \frac{g(2|\eta|,\eta)}{g(t_{E(|\eta|),\eta},\eta)}=&\prod_{k=1}^{E(|\eta|^s)}e^{\kappa F(k,\eta)} \times\prod_{|k|=E(|\eta|^s)+1}^{E(|\eta|)}e^{\kappa \tilde F(k,\eta)}\\
  \le& e^{\pi\kappa|\eta|^{s}+\pi\kappa\sum_{|k|=E(|\eta|^s)+1}^{E(|\eta|)}\mathfrak g(\nu,\eta)\langle \nu^{\frac{1}{3}}\frac{\eta}{k}\rangle^{-\frac{3}{2}}\nu^\beta\frac{|\eta|}{|k|^2}}.
\end{align*}
Therefore, it suffices to estimate
\begin{align*}
  \sum_{|k|=E(|\eta|^s)+1}^{E(|\eta|)}\mathfrak g(\nu,\eta)\langle \nu^{\frac{1}{3}}\frac{\eta}{k}\rangle^{-\frac{3}{2}}\nu^\beta\frac{|\eta|}{|k|^2}.
\end{align*}
Recall that $s=\frac{1-3\beta}{2-3\beta}$ and $\mathfrak g(\nu,\eta)= \langle\nu^{-\frac{1}{3}}|\eta|^{s-1}\rangle^{3\beta}$. If $\nu^{\frac{1}{3}}|\eta|\le |\eta|^s$, we have $|\eta|\le \nu^{-\frac{1}{3(1-s)}}$, $\mathfrak g(\nu,\eta)\approx \nu^{-\beta}|\eta|^{-3\beta(1-s)}$, and $\langle \nu^{\frac{1}{3}}\frac{\eta}{k}\rangle^{-\frac{3}{2}}\approx 1$. Then 
\begin{align*}
  \sum_{|k|=E(|\eta|^s)+1}^{E(|\eta|)}\mathfrak g(\nu,\eta)\langle \nu^{\frac{1}{3}}\frac{\eta}{k}\rangle^{-\frac{3}{2}}\nu^\beta\frac{|\eta|}{|k|^2}\approx& \sum_{|k|=E(|\eta|^s)+1}^{E(|\eta|)}\frac{|\eta|^{1-3\beta(1-s)}}{k(k+1)} \\
  \lesssim&|\eta|^{1-s-3\beta(1-s)}=|\eta|^{(1-s)(1-3\beta)}=|\eta|^s.
\end{align*}
If $|\eta|^s\le\nu^{\frac{1}{3}}|\eta|\le |\eta|$, we have $|\eta|\ge \nu^{-\frac{1}{3(1-s)}}$, $\mathfrak g(\nu,\eta)\approx 1$, and
\begin{align*}
  &\sum_{|k|=E(|\eta|^s)+1}^{E(|\eta|)}\mathfrak g(\nu,\eta)\langle \nu^{\frac{1}{3}}\frac{\eta}{k}\rangle^{-\frac{3}{2}}\nu^\beta\frac{|\eta|}{|k|^2}\\
  \approx & \sum_{|k|=E(|\eta|^s)+1}^{E(\nu^{\frac{1}{3}}|\eta|)}\frac{\nu^{\beta-\frac{1}{2}}|\eta|^{-\frac{1}{2}}}{\sqrt{k}+\sqrt{k+1}}+\sum_{E(\nu^{\frac{1}{3}}|\eta|)}^{E(|\eta|)}\frac{\nu^{\beta}|\eta|}{k(k+1)}\\
  \lesssim&\nu^{\beta-\frac{1}{3}}\lesssim |\eta|^s.
\end{align*}
Accordingly, for both cases, it holds that
\begin{align}\label{eq-sum-s-g}
  \sum_{|k|=E(|\eta|^s)+1}^{E(|\eta|)}\mathfrak g(\nu,\eta)\langle \nu^{\frac{1}{3}}\frac{\eta}{k}\rangle^{-\frac{3}{2}}\nu^\beta\frac{|\eta|}{|k|^2}\lesssim |\eta|^s.
\end{align}

Combing the above estimates, we arrive at
\begin{align*}
   \frac{1}{g(0,\eta)}\le e^{C_2\kappa|\eta|^{s}}.
 \end{align*} 
\end{proof}
\subsection{Properties of $w$, $g$, $\mathfrak m$, $W$, $G$,  and $\mathfrak M$}
This subsection proves some important and useful properties of $w$, $g$, $\mathfrak m$, $W$, $G$,  and $\mathfrak M$. 
The following lemma expresses the well-separation of critical times and is used several times in the sequel.
\begin{lemma}\label{lem-separate}
  Let $\xi$, $\eta\in\mathbb R$ be such that there exists some $\al\ge1$ with $\frac{1}{\al}|\xi|\le|\eta|\le \al|\xi|$, $\xi\eta>0$, and let $m,k\in\mathbb Z$ be such that $m\xi,k\eta>0$. Then, if $t\in \mathrm{I}_{m,\xi}\cap \mathrm{I}_{k,\eta}$ (note that $m\approx k$), at least one of the following holds:
  \begin{itemize}
    \item[(a)] $k=m$ and $t\in \tilde {\mathrm{I}}_{m,\xi}\cap \tilde{\mathrm{I}}_{k,\eta}$ (almost same interval);\\
    \item[(b)] $k=m$ and $|t-\frac{\xi}{m}|\ge \frac{1}{10\al}\frac{|\xi|^{1-3\beta}}{|m|^{2-3\beta}}$ and $|t-\frac{\eta}{k}|\ge \frac{1}{10\al}\frac{|\eta|^{1-3\beta}}{|k|^{2-3\beta}}$ (far from strong resonance);\\
    \item[(c)] $k=m$ and $|\xi-\eta|\ge C_\al \left|\frac{\eta}{k}\right|^{1-3\beta}$ (weak well-separated);\\
    \item[(d)] $|t-\frac{\xi}{m}|\ge \frac{1}{10\al}\frac{|\xi|}{|m|^{2}}$ and $|t-\frac{\eta}{k}|\ge \frac{1}{10\al}\frac{|\eta|}{|k|^{2}}$ (far from weak resonance);\\
    \item[(e)] $|\xi-\eta|\ge C_\al \frac{\eta}{k}$ (strong well-separated).
  \end{itemize}
  Moreover, if $t\in \tilde{\mathrm{I}}_{m,\xi}\cap {\mathrm{I}}_{k,\eta}$, then at least one of the following holds:
    \begin{itemize}
    \item[(a')] $k=m$;\\
    \item[(b')] $|\xi-\eta|\ge C_\al \left|\frac{\eta}{k}\right|$.
  \end{itemize}
\end{lemma}
\begin{proof}
Recall the definition of $\mathrm{I}_{k,\eta}$. It follows from $t\in \mathrm{I}_{m,\xi}\cap \mathrm{I}_{k,\eta}$ that
\begin{align*}
  \frac{|m|}{|k|}=\frac{|tm|}{|tk|}=\frac{|\xi|}{|\eta|}\frac{|tm|}{|\xi|}\frac{|\eta|}{|tk|}\approx 1.
\end{align*}
Therefore $|m|\approx |k|$. As $m\xi,k\eta>0$ and $\frac{1}{\al}\xi\le\eta\le \al\xi$ for $\al\ge1$, we have $m\approx k$.

  For $k=m$ and both (a), (b) are false. Without loss of generality, we assume that $\left|t-\frac{\xi}{m}\right|<\frac{1}{10\al}\frac{|\xi|^{1-3\beta}}{|m|^{2-3\beta}}$, which means that $t\in\tilde{\mathrm{I}}_{m,\xi}$. As (a) is false, we have $t\notin \tilde{\mathrm{I}}_{k,\eta}$. It follows that
  \begin{align*}
    |\xi-\eta|\ge |m| \left|\frac{\xi}{m}-\frac{\eta}{m}\right|\ge |m|\left|\frac{\eta}{m}-t\right|-|m|\left|t-\frac{\xi}{m}\right|\ge \frac{|m|}{8}\frac{|\eta|^{1-3\beta}}{|m|^{2-3\beta}}-\frac{|m|}{10\al}\frac{|\xi|^{1-3\beta}}{|m|^{2-3\beta}}\gtrsim \left|\frac{\eta}{k}\right|^{1-3\beta}.
  \end{align*}
  Suppose $m\neq k$ and (d) are false, one can show that (e) holds by using the above argument. The proof for (a'), and (b') is similar. 
\end{proof}
\subsubsection{Properties of $w$ and $W$}
The following property follows directly from the definitions of $w_{\mathrm{R}}$ and $w_{\mathrm{NR}}$.
\begin{lemma}\label{lem-evo-w-lntype}
  For $t \in \tilde{\mathrm{I}}_{k, \eta}$ with $|k|\le \frac{1}{2}|\eta|^s$, we have:
\begin{align*}
\frac{\partial_{t} w_{\mathrm{NR}}(t, \eta)}{w_{\mathrm{NR}}(t, \eta)}\approx \frac{a_{k,\eta}}{(1+\left|t-\frac{|\eta|}{|k|}\right|)} \approx \frac{1}{(1+\left|t-\frac{|\eta|}{|k|}\right|)} \approx \frac{\partial_{t} w_{\mathrm{R}}(t, \eta)}{w_{\mathrm{R}}(t, \eta)}.
\end{align*}

For $t \in \tilde{\mathrm{I}}_{k, \eta}$ with $\frac{1}{2}|\eta|^s<|k|\le E(|\eta|^s)$, we have
\begin{align*}
  \frac{w_{\mathrm{NR}}(t,\eta)}{w_{\mathrm{R}}(t,\eta)}=\frac{\frac{|\eta|^{1-3\beta}}{|k|^{2-3\beta}}}{1+a_{k,\eta}\left|t-\frac{|\eta|}{|k|}\right|}\lesssim 1.
\end{align*}
\end{lemma}

The following two lemmas show that although the toy model neglected interactions in $\eta$ and $\xi$, $w(t,\eta)$ with $w(t,\xi)$ can still be compared effectively.
\begin{lemma}\label{lem-om-react-s}
  For all $\eta,\xi\in\mathbb R$ and $t\ge 1$, we have
  \begin{align*}
    \frac{w_{\mathrm{NR}}(t,\xi)}{w_{\mathrm{NR}}(t,\eta)}\lesssim \langle \xi-\eta\rangle^{1+2C_1\kappa}e^{C\mu |\xi-\eta|^s}.
  \end{align*}
  where $C$ is a constant independent of $\beta$ and $\kappa$.
\end{lemma}
\begin{proof}
  Without loss of generality, we assume that $|\xi|\le|\eta|$. And we will prove
  \begin{align}\label{eq-w-xi-eta}
   \langle \xi-\eta\rangle^{-(1+2C_1\kappa)}e^{-\mu |\xi-\eta|^s} \lesssim\frac{w_{\mathrm{NR}}(t,\xi)}{w_{\mathrm{NR}}(t,\eta)}\lesssim \langle \xi-\eta\rangle^{1+2C_1\kappa}e^{\mu |\xi-\eta|^s}.
  \end{align}
  If $|\eta|< 1$ or $t\ge2|\eta|$, we have $w_{\mathrm{NR}}(t,\xi)=w_{\mathrm{NR}}(t,\xi)=1$ for $\forall t\ge1$, and there is nothing to prove. If $|\xi|< 1\le |\eta|$, then we have $|\eta|\le1+|\eta-\xi|$ and hence we have from \eqref{eq-grow-w} that
\begin{align*}
  1\le\frac{w_{\mathrm{NR}}(t,\xi)}{w_{\mathrm{NR}}(t,\eta)}=\frac{1}{w_{\mathrm{NR}}(t,\eta)}\lesssim \frac{1}{|\eta|^{\frac{\mu s}{4}}}e^{\frac{\mu}{2}|\eta|^s}\lesssim e^{\mu |\xi-\eta|^s}.
\end{align*}
If $|\xi|\le \frac{1}{2}|\eta|$ or $\xi\eta<0$, it holds that $|\xi|\le |\xi-\eta|$ and $|\eta|\le 2|\xi-\eta|$. Thus
\begin{align*}
   e^{-\mu |\xi-\eta|^s}\le e^{-\frac{\mu}{2}|\xi|^s}\lesssim w_{\mathrm{NR}}(t,\xi)\le\frac{w_{\mathrm{NR}}(t,\xi)}{w_{\mathrm{NR}}(t,\eta)}\le\frac{1}{w_{\mathrm{NR}}(t,\eta)}\lesssim \frac{1}{|\eta|^{\frac{\mu s}{4}}}e^{\frac{\mu}{2}|\eta|^s}\le e^{\mu |\xi-\eta|^s}.
\end{align*}

Therefore, we only need to consider the case $|\eta|,|\xi|\ge1$  such that $\frac{|\eta|}{2}\le|\xi|\le |\eta|$, $\xi\eta>0$, and $t\le 2|\xi|$. Without loss of generality, we assume that $1\le\xi\le\eta$. We discuss the following 4 cases:
  \begin{itemize}
    \item[Case 1:] $t\le \min(t_{E(|\xi|^s),|\xi|},t_{E(|\eta|^s),|\eta|})$;
    \item[Case 2:] $\min(t_{E(|\xi|^s),|\xi|},t_{E(|\eta|^s),|\eta|})\le t\le \max(t_{E(|\xi|^s),|\xi|},t_{E(|\eta|^s),|\eta|})$;
    \item[Case 3:] $\max(t_{E(|\xi|^s),|\xi|},t_{E(|\eta|^s),|\eta|})\le t\le 2\xi$;
    \item[Case 4:] $2\xi\le t\le 2\eta$.
  \end{itemize}

For Case 1, $t\le\min(t_{E(|\xi|^s),|\xi|},t_{E(|\eta|^s),|\eta|})$. We deduce from \eqref{eq-grow-w} that
\begin{align*}
  \frac{w_{\mathrm{NR}}(t,\xi)}{w_{\mathrm{NR}}(t,\eta)}\approx \left(\frac{|\xi|}{|\eta|}\right)^{\frac{\mu s}{4}}e^{\frac{\mu}{2}(|\eta|^s-|\xi|^s)},
\end{align*}
which gives \eqref{eq-w-xi-eta} with \eqref{inq-s2}.

For Case 2, $\min(t_{E(|\xi|^s),|\xi|},t_{E(|\eta|^s),|\eta|})<t<\max(t_{E(|\xi|^s),|\xi|},t_{E(|\eta|^s),|\eta|})$. We have
\begin{align*}
  w_{\mathrm{NR}}(\min(t_{E(|\xi|^s),|\xi|},t_{E(|\eta|^s),|\eta|}),\xi)\le w_{\mathrm{NR}}(t,\xi)\le w_{\mathrm{NR}}(\max(t_{E(|\xi|^s),|\xi|},t_{E(|\eta|^s),|\eta|}),\xi),\\
  w_{\mathrm{NR}}(\min(t_{E(|\xi|^s),|\xi|},t_{E(|\eta|^s),|\eta|}),\eta)\le w_{\mathrm{NR}}(t,\eta)\le w_{\mathrm{NR}}(\max(t_{E(|\xi|^s),|\xi|},t_{E(|\eta|^s),|\eta|}),\eta).
\end{align*}
It follows that
\begin{align*}
  \frac{w_{\mathrm{NR}}(t,\xi)}{w_{\mathrm{NR}}(t,\eta)}\le& \frac{w_{\mathrm{NR}}(\max(t_{E(|\xi|^s),|\xi|},t_{E(|\eta|^s),|\eta|}),\xi)}{w_{\mathrm{NR}}(\min(t_{E(|\xi|^s),|\xi|},t_{E(|\eta|^s),|\eta|}),\eta)}\\
  =&\frac{w_{\mathrm{NR}}(\max(t_{E(|\xi|^s),|\xi|},t_{E(|\eta|^s),|\eta|}),\xi)}{w_{\mathrm{NR}}(\min(t_{E(|\xi|^s),|\xi|},t_{E(|\eta|^s),|\eta|}),\xi)}\frac{w_{\mathrm{NR}}(\min(t_{E(|\xi|^s),|\xi|},t_{E(|\eta|^s),|\eta|}),\xi)}{w_{\mathrm{NR}}(\min(t_{E(|\xi|^s),|\xi|},t_{E(|\eta|^s),|\eta|}),\eta)}\\
  \le&\frac{w_{\mathrm{NR}}(\min(t_{E(|\xi|^s),|\xi|},t_{E(|\eta|^s),|\eta|}),\xi)}{w_{\mathrm{NR}}(\min(t_{E(|\xi|^s),|\xi|},t_{E(|\eta|^s),|\eta|}),\eta)},
\end{align*}
and
\begin{align*}
  \frac{w_{\mathrm{NR}}(t,\xi)}{w_{\mathrm{NR}}(t,\eta)}\ge& \frac{w_{\mathrm{NR}}(\max(t_{E(|\xi|^s),|\xi|},t_{E(|\eta|^s),|\eta|}),\xi)}{w_{\mathrm{NR}}(\min(t_{E(|\xi|^s),|\xi|},t_{E(|\eta|^s),|\eta|}),\eta)}\\
  =&\frac{w_{\mathrm{NR}}(\min(t_{E(|\xi|^s),|\xi|},t_{E(|\eta|^s),|\eta|}),\xi)}{w_{\mathrm{NR}}(\max(t_{E(|\xi|^s),|\xi|},t_{E(|\eta|^s),|\eta|}),\xi)}\frac{w_{\mathrm{NR}}(\max(t_{E(|\xi|^s),|\xi|},t_{E(|\eta|^s),|\eta|}),\xi)}{w_{\mathrm{NR}}(\max(t_{E(|\xi|^s),|\xi|},t_{E(|\eta|^s),|\eta|}),\eta)}\\
  \ge&\frac{w_{\mathrm{NR}}(\max(t_{E(|\xi|^s),|\xi|},t_{E(|\eta|^s),|\eta|}),\xi)}{w_{\mathrm{NR}}(\max(t_{E(|\xi|^s),|\xi|},t_{E(|\eta|^s),|\eta|}),\eta)}.
\end{align*}
Then the upper bound part reduces to Case 2, and the lower bound part reduces to Case 3.

For Case 3, $\max(t_{E(|\xi|^s),|\xi|},t_{E(|\eta|^s),|\eta|})\le t\le 2|\xi|$. Let $j,l\ge1$ such that $t\in \mathrm{I}_{l,\xi}\cap \mathrm{I}_{j,\eta}$, we have that $l\le j\lesssim l$ and 
\begin{align*}
  1\le l\le E(|\xi|^s),\quad 1\le j\le E(|\eta|^s).
\end{align*}
For this case, it is easy to check that
\begin{align*}
  \frac{w_{\mathrm{NR}}(t,\xi)}{w_{\mathrm{NR}}(t,\eta)}\gtrsim 1.
\end{align*}
To estimate the upper bound, we divide Case 3 into five sub-cases.

\noindent{\bf Case 3.1.} Recall Lemma \ref{lem-separate}, we first consider the case $l=j$ and $t\in \tilde{\mathrm{I}}_{l,\xi}\cap\tilde{\mathrm{I}}_{j,\eta}$ (Case (a) in Lemma \ref{lem-separate}), which means that $t\in \tilde{\mathrm{I}}_{l,\xi}^R\cap\tilde{\mathrm{I}}_{j,\eta}^R$ or $t\in \tilde{\mathrm{I}}_{l,\xi}^L\cap\tilde{\mathrm{I}}_{j,\eta}^L$ or $t\in \tilde{\mathrm{I}}_{l,\xi}^R\cap\tilde{\mathrm{I}}_{j,\eta}^L$. For these cases, $|\xi-\eta|\le \frac{1}{8}\frac{|\xi|^{1-3\beta}+|\eta|^{1-3\beta}}{|j|^{1-3\beta}}$ is very small.

From the definition of $w_{\mathrm{NR}}$, we have
  \begin{align*}
    w_{\mathrm{NR}}(t,\eta)=&\prod_{k=1}^{j-1}\left(\frac{|\eta|^{1-3\beta}}{|k|^{2-3\beta}}\right)^{-(1+2C_1\kappa)}\left(\frac{|j|^{2-3\beta}}{|\eta|^{1-3\beta}}\left[1+a_{j,\eta}\left|t-\frac{|\eta|}{|j|}\right|\right]\right)^{C_1\kappa}\text{ for }t\in \tilde{\mathrm{I}}_{j,\eta}^R,\\
    w_{\mathrm{NR}}(t,\eta)=&\prod_{k=1}^{j-1}\left(\frac{|\eta|^{1-3\beta}}{|k|^{2-3\beta}}\right)^{-(1+2C_1\kappa)}\left(\frac{|j|^{2-3\beta}}{|\eta|^{1-3\beta}}\right)^{C_1\kappa}\left(1+a_{j,\eta}\left|t-\frac{|\eta|}{|j|}\right|\right)^{-1-C_1\kappa}\text{ for }t\in \tilde{\mathrm{I}}_{j,\eta}^L.
  \end{align*}
If $t\in \tilde{\mathrm{I}}_{l,\xi}^R\cap\tilde{\mathrm{I}}_{j,\eta}^R$ with $j=l$, as $|\eta|\ge|\xi|$ and $t\ge \frac{\eta}{j},\frac{\xi}{j}$, we must have $\left|t-\frac{\eta}{j}\right|\le \left|t-\frac{\xi}{j}\right|$. We have
\begin{align*}
  \frac{w_{\mathrm{NR}}(t,\xi)}{w_{\mathrm{NR}}(t,\eta)}=&\left(\frac{|\eta| }{|\xi| }\right)^{(j-1)(1-3\beta)(1+2C_1\kappa)}\left(\frac{\frac{|j|^{2-3\beta}}{|\xi|^{1-3\beta}}\left[1+a_{j,\xi}\left|t-\frac{\xi}{j}\right|\right]}{\frac{|j|^{2-3\beta}}{|\eta|^{1-3\beta}}\left[1+a_{j,\eta}\left|t-\frac{\eta}{j}\right|\right]}\right)^{C_1\kappa}.
\end{align*}
It is easy to check that $1\lesssim \frac{w_{\mathrm{NR}}(t,\xi)}{w_{\mathrm{NR}}(t,\eta)}$, and we only concern the upper bound part. As $j\le |\xi|^s$, we have
\begin{equation}\label{eq-est-eta-xi-step}
  \begin{aligned}    
      \left(\frac{|\eta| }{|\xi| }\right)^{(j-1)(1-3\beta)(1+2C_1\kappa)}=&\left(1+\frac{\big||\eta|-|\xi|\big| }{|\xi| }\right)^{\frac{|\xi| }{\big||\eta|-|\xi|\big| }\frac{\big||\eta|-|\xi|\big| }{|\xi| }(j-1)(1-3\beta)(1+2C_1\kappa)}\\
  \lesssim&e^{(1-3\beta)(1+2C_1\kappa)\frac{|\eta -\xi|}{|\xi| }|\xi|^s}\le e^{(1-3\beta)(1+2C_1\kappa)|\eta -\xi|^s},
  \end{aligned}
\end{equation}
and
\begin{align*}
  &\left(\frac{\frac{|j|^{2-3\beta}}{|\xi|^{1-3\beta}}\left[1+a_{j,\xi}\left|t-\frac{\xi}{j}\right|\right]}{\frac{|j|^{2-3\beta}}{|\eta|^{1-3\beta}}\left[1+a_{j,\eta}\left|t-\frac{\eta}{j}\right|\right]}\right)^{C_1\kappa}=\left(\frac{|\eta|}{|\xi|}\right)^{C_1\kappa(1-3\beta)}\left(\frac{\left[1+a_{j,\xi}\left|t-\frac{\xi}{j}\right|\right]}{\left[1+a_{j,\eta}\left|t-\frac{\eta}{j}\right|\right]}\right)^{C_1\kappa}\\
  \le&\left(\frac{|\eta|}{|\xi|}\right)^{C_1\kappa(1-3\beta)}\left(1+\left|a_{j,\xi} |t-\frac{\xi}{j} |-a_{j,\eta} |t-\frac{\eta}{j} |\right|\right)^{C_1\kappa}\\
  \le&\left(\frac{|\eta|}{|\xi|}\right)^{C_1\kappa(1-3\beta)}\left(1+|a_{j,\xi}-a_{j,\eta}| |t-\frac{\xi}{j} |+a_{j,\eta}\left| |t-\frac{\xi}{j} |-|t-\frac{\eta}{j} |\right|\right)^{C_1\kappa}\\
  \lesssim&\left(\frac{|\eta|}{|\xi|}\right)^{C_1\kappa(1-3\beta)}\left(1+8j^{2-3\beta}\left|\frac{1}{|\eta|^{1-3\beta}}-\frac{1}{|\xi|^{1-3\beta}}\right|\frac{|\xi|^{1-3\beta}}{|j|^{2-3\beta}}+a_{j,\eta}\frac{|\eta-\xi|}{j}\right)^{C_1\kappa}\\
  \lesssim&\left(\frac{|\eta|}{|\xi|}\right)^{C_1\kappa(1-3\beta)}\left(1+\frac{|\xi-\eta|}{|\xi|}+a_{j,\eta}\frac{|\eta-\xi|}{j}\right)^{C_1\kappa}\lesssim \langle \xi-\eta\rangle^{C_1\kappa}e^{C_1\kappa|\eta -\xi|^s}.
\end{align*}
Which gives \eqref{eq-w-xi-eta}.

If $t\in \tilde{\mathrm{I}}_{l,\xi}^L\cap\tilde{\mathrm{I}}_{j,\eta}^L$ with $j=l$, as $|\eta|\ge|\xi|$ and $t\le \frac{\eta}{j},\frac{\xi}{j}$, we must have $\left|t-\frac{\eta}{j}\right|\ge \left|t-\frac{\xi}{j}\right|$. We have
\begin{align*}
  \frac{w_{\mathrm{NR}}(t,\xi)}{w_{\mathrm{NR}}(t,\eta)}=&\left(\frac{|\eta| }{|\xi| }\right)^{(j-1)(1-3\beta)(1+2C_1\kappa)+(1-3\beta)C_1\kappa}\left(\frac{\left[1+a_{j,\xi}\left|t-\frac{\xi}{j}\right|\right]}{\left[1+a_{j,\eta}\left|t-\frac{\eta}{j}\right|\right]}\right)^{-1-C_1\kappa}.
\end{align*}
It holds that
\begin{align*}
  1\lesssim&\left(\frac{\left[1+a_{j,\xi}\left|t-\frac{\xi}{j}\right|\right]}{\left[1+a_{j,\eta}\left|t-\frac{\eta}{j}\right|\right]}\right)^{-1-C_1\kappa}=\left(\frac{\left[1+a_{j,\eta}\left|t-\frac{\eta}{j}\right|\right]}{\left[1+a_{j,\xi}\left|t-\frac{\xi}{j}\right|\right]}\right)^{1+C_1\kappa}\\
  \le&\left(1+\left|a_{j,\xi} |t-\frac{\xi}{j} |-a_{j,\eta} |t-\frac{\eta}{j} |\right|\right)^{1+C_1\kappa}\\
  \le&\left(1+|a_{j,\xi}-a_{j,\eta}| |t-\frac{\xi}{j} |+a_{j,\eta}\left| |t-\frac{\xi}{j} |-|t-\frac{\eta}{j} |\right|\right)^{1+C_1\kappa}\\
  \lesssim&\left(1+8j^{2-3\beta}\left|\frac{1}{|\eta|^{1-3\beta}}-\frac{1}{|\xi|^{1-3\beta}}\right|\frac{|\xi|^{1-3\beta}}{|j|^{2-3\beta}}+a_{j,\eta}\frac{|\eta-\xi|}{j}\right)^{1+C_1\kappa}\\
  \lesssim&\left(1+\frac{|\xi-\eta|}{|\xi|}+a_{j,\eta}\frac{|\eta-\xi|}{j}\right)^{1+C_1\kappa}\lesssim \langle \xi-\eta\rangle^{(1+C_1\kappa)}.
\end{align*}
Combining the above estimate with \eqref{eq-est-eta-xi-step}, we also have \eqref{eq-w-xi-eta}.

If $t\in \tilde{\mathrm{I}}_{l,\xi}^L\cap\tilde{\mathrm{I}}_{j,\eta}^R$ with $j=l$, then $\frac{\xi}{j}\le t\le \frac{\eta}{j}$, and
\begin{align*}
   \left|t-\frac{\xi}{j}\right|,\left|t-\frac{\eta}{j}\right|\le\left|\frac{\eta-\xi}{j}\right|.
 \end{align*} 
It follows that
\begin{align*}
  \frac{w_{\mathrm{NR}}(t,\xi)}{w_{\mathrm{NR}}(t,\eta)}=&\left(\frac{|\eta| }{|\xi| }\right)^{(j-1)(1-3\beta)(1+2C_1\kappa)+(1-3\beta)C_1\kappa}\left[1+a_{j,\xi}\left|t-\frac{\xi}{j}\right|\right]^{C_1\kappa}\left[1+a_{j,\eta}\left|t-\frac{\eta}{j}\right|\right]^{1+C_1\kappa}\\
  \le&\left(\frac{|\eta| }{|\xi| }\right)^{(j-1)(1-3\beta)(1+2C_1\kappa)+(1-3\beta)C_1\kappa}\left[1+8 \frac{\eta-\xi}{j}\right]^{1+2C_1\kappa}
\end{align*}

\noindent{\bf Case 3.2.} Next, we consider the case $l=j$ and $\left|t-\frac{\xi}{l}\right|\ge \frac{1}{20}\frac{\xi^{1-3\beta}}{l^{2-3\beta}}$ and $\left|t-\frac{\eta}{j}\right|\ge \frac{1}{20}\frac{\eta^{1-3\beta}}{j^{2-3\beta}}$ (Case (b) in Lemma \ref{lem-separate}). In this case, we have
\begin{align*}
  \left(1+a_{j,\eta}\left|t-\frac{\eta}{j}\right|\right)\approx \frac{\eta^{1-3\beta}}{j^{2-3\beta}}\text{ for }t\in\tilde{\mathrm{I}}_{j,\eta}\text{ and } \left|t-\frac{\eta}{j}\right|\ge \frac{1}{20}\frac{\eta^{1-3\beta}}{j^{2-3\beta}}\\
  \left(1+a_{l,\xi}\left|t-\frac{\xi}{l}\right|\right)\approx \frac{\xi^{1-3\beta}}{l^{2-3\beta}}\text{ for }t\in\tilde{\mathrm{I}}_{l,\xi}\text{ and } \left|t-\frac{\xi}{l}\right|\ge \frac{1}{20}\frac{\xi^{1-3\beta}}{l^{2-3\beta}},
\end{align*}
it follows that
\begin{equation}
  w_{\mathrm{NR}}(t,\xi)\approx\left\{
    \begin{array}{ll}
      w_{\mathrm{NR}}(t^+_{l,\xi},\xi)&\text{ for }t\in[\frac{\xi}{l}+\frac{1}{20}\frac{\xi^{1-3\beta}}{l^{2-3\beta}}, \frac{\xi}{l}+\frac{1}{8}\frac{\xi^{1-3\beta}}{l^{2-3\beta}}]\\
      w_{\mathrm{NR}}(t^-_{l,\xi},\xi)&\text{ for }t\in[\frac{\xi}{l}-\frac{1}{8}\frac{\xi^{1-3\beta}}{l^{2-3\beta}}, \frac{\xi}{l}-\frac{1}{20}\frac{\xi^{1-3\beta}}{l^{2-3\beta}}],
    \end{array}
  \right.
\end{equation}
\begin{equation}
  w_{\mathrm{NR}}(t,\eta)\approx\left\{
    \begin{array}{ll}
      w_{\mathrm{NR}}(t^+_{j,\eta},\eta)&\text{ for }t\in[\frac{\eta}{j}+\frac{1}{20}\frac{\eta^{1-3\beta}}{j^{2-3\beta}}, \frac{\eta}{j}+\frac{1}{8}\frac{\eta^{1-3\beta}}{j^{2-3\beta}}]\\
      w_{\mathrm{NR}}(t^-_{j,\eta},\eta)&\text{ for }t\in[\frac{\eta}{j}-\frac{1}{8}\frac{\eta^{1-3\beta}}{j^{2-3\beta}}, \frac{\eta}{j}-\frac{1}{20}\frac{\eta^{1-3\beta}}{j^{2-3\beta}}].
    \end{array}
  \right.   
\end{equation}
Then, we have for $\frac{\xi}{l}< t<\frac{\eta}{j}$ that
\begin{align*}
  |\xi-\eta|=j|\frac{\xi}{j}-\frac{\eta}{j}|=j \left(|t-\frac{\xi}{l}|+|t-\frac{\eta}{j}|\right)\gtrsim \left(\frac{\eta}{j}\right)^{1-3\beta},
\end{align*}
and
\begin{align*}
  \frac{w_{\mathrm{NR}}(t,\xi)}{w_{\mathrm{NR}}(t,\eta)}=\frac{w_{\mathrm{NR}}(t^+_{l,\xi},\xi)}{w_{\mathrm{NR}}(t^-_{j,\eta},\eta)}=& \left(\frac{|\eta| }{|\xi| }\right)^{(j-1)(1-3\beta)(1+2C_1\kappa)}\frac{1}{\left(\frac{|j|^{2-3\beta}}{|\eta|^{1-3\beta}}\right)^{1+2C_1\kappa}}\\
  \lesssim&\left(\frac{|\eta| }{|\xi| }\right)^{(j-1)(1-3\beta)(1+2C_1\kappa)}|\xi-\eta|^{1+2C_1\kappa},
\end{align*}
for $\frac{\xi}{l}\le\frac{\eta}{j}<t$ that
\begin{align*}
  \frac{w_{\mathrm{NR}}(t,\xi)}{w_{\mathrm{NR}}(t,\eta)}=\frac{w_{\mathrm{NR}}(t^+_{l,\xi},\xi)}{w_{\mathrm{NR}}(t^+_{j,\eta},\eta)}=& \left(\frac{|\eta| }{|\xi| }\right)^{(j-1)(1-3\beta)(1+2C_1\kappa)},
\end{align*}
and for $t<\frac{\xi}{l}\le\frac{\eta}{j}$ that
\begin{align*}
  \frac{w_{\mathrm{NR}}(t,\xi)}{w_{\mathrm{NR}}(t,\eta)}=\frac{w_{\mathrm{NR}}(t^-_{l,\xi},\xi)}{w_{\mathrm{NR}}(t^-_{j,\eta},\eta)}=& \left(\frac{|\eta| }{|\xi| }\right)^{j(1-3\beta)(1+2C_1\kappa)}.
\end{align*}
This gives \eqref{eq-w-xi-eta} with \eqref{eq-est-eta-xi-step}.

\noindent{\bf Case 3.3.} The rest case with $j=l$ is $|\xi-\eta|\gtrsim |\frac{\eta}{j}|^{1-3\beta}$ (Case (c) in Lemma \ref{lem-separate}). In this case, we have
\begin{align*}
  \frac{w_{\mathrm{NR}}(t,\xi)}{w_{\mathrm{NR}}(t,\eta)}\lesssim& \left(\frac{|\eta| }{|\xi| }\right)^{(j-1)(1-3\beta)(1+2C_1\kappa)}\left(\frac{|\eta|^{1-3\beta}}{|j|^{2-3\beta}}\right)^{1+2C_1\kappa}\\
  \lesssim &\left(\frac{|\eta| }{|\xi| }\right)^{(j-1)(1-3\beta)(1+2C_1\kappa)}|\xi-\eta|^{1+2C_1\kappa}.
\end{align*}
This gives \eqref{eq-w-xi-eta} with \eqref{eq-est-eta-xi-step}.

\noindent{\bf Case 3.4.} Next, we consider the case $l=j-1$. If $t\in \mathrm{I}_{j,\eta}^L$, then $t_{j-1,\xi}\le \frac{\eta}{j}$. If  $t\in \mathrm{I}_{j-1,\xi}^R$, then $\frac{\xi}{j-1}<t_{l-1,\eta}$. In either one of these cases, we deduce that $\frac{\xi}{l^2}\lesssim \frac{\eta-\xi}{l}$ and thus
\begin{align*}
  \frac{w_{\mathrm{NR}}(t,\xi)}{w_{\mathrm{NR}}(t,\eta)}\le& \left(\frac{|\eta| }{|\xi| }\right)^{(j-2)(1-3\beta)(1+2C_1\kappa)}\left(\frac{|\eta|^{1-3\beta}}{|j-1|^{2-3\beta}}\right)^{1+2C_1\kappa}\left(\frac{|\eta|^{1-3\beta}}{|j|^{2-3\beta}}\right)^{1+2C_1\kappa}\\
  \lesssim &\left(\frac{|\eta| }{|\xi| }\right)^{(j-2)(1-3\beta)(1+2C_1\kappa)}\left(\frac{|\xi-\eta|^{1-3\beta}}{|j|}\right)^{2(1+2C_1\kappa)} \lesssim |\xi-\eta|^{2(1-3\beta)(1+2C_1\kappa)}e^{\frac{\mu}{2} |\xi-\eta|^s}\\
  \lesssim&\langle \xi-\eta\rangle^{(1+2C_1\kappa)}e^{\mu |\xi-\eta|^s}.
\end{align*}
In the last inequality, we use the following argument. If $\beta\ge \frac{1}{6}$, then $2(1-3\beta)(1+2C_1\kappa)\le(1+2C_1\kappa)$. If $\beta\le \frac{1}{6}$, we have $s\ge \frac{1}{3}$, then
\begin{align*}
  |\xi-\eta|^{2(1-3\beta)(1+2C_1\kappa)}e^{\frac{\mu}{2} |\xi-\eta|^s}\lesssim e^{\mu |\xi-\eta|^s}.
\end{align*}

If $t\in t\in \mathrm{I}_{j,\eta}^R\cap \mathrm{I}_{j-1,\xi}^L$, there are only two possibilities, one is $\frac{\xi}{l^2}\lesssim \frac{\eta-\xi}{l}$, the other is $t-\frac{\eta}{j}\gtrsim \frac{\eta}{j^2}$ and $t-\frac{\xi}{j}\gtrsim \frac{\xi}{j^2}$, which gives us that
\begin{align*}
  w_{\mathrm{NR}}(t,\xi)=w_{\mathrm{NR}}(t^+_{l,\xi},\xi),\quad w_{\mathrm{NR}}(t,\eta)=w_{\mathrm{NR}}(t^+_{l,\eta},\eta).
\end{align*}
then we have
\begin{align*}
  \frac{w_{\mathrm{NR}}(t,\xi)}{w_{\mathrm{NR}}(t,\eta)}=\left(\frac{|\eta| }{|\xi| }\right)^{(j-1)(1-3\beta)(1+2C_1\kappa)}\lesssim e^{\mu |\xi-\eta|^s}.
\end{align*}

\noindent{\bf Case 3.5.} For $j\ge l+2$, we have $\frac{\eta-\xi}{j(j-l)}\approx \frac{\eta}{j^2}$ and
\begin{align}\label{eq-dist-j-l}
  j-l\le \frac{|lt-\xi|+|\xi-\eta|+|\eta-jt|}{t}\lesssim 1+\frac{|\xi-\eta|}{t}.
\end{align}
Here we use the fact that 
\begin{align*}
  \frac{|lt-\xi|}{t}=\frac{l|t-\frac{\xi}{l}|}{t}\le \frac{l \frac{\xi}{l^2}}{t}\approx \frac{\xi}{jt}\approx 1.
\end{align*}
It follows from the definition of $w_{\mathrm{NR}}$ that
\begin{align*}
  \frac{w_{\mathrm{NR}}(t,\xi)}{w_{\mathrm{NR}}(t,\eta)}\le& \left(\frac{|\eta| }{|\xi| }\right)^{(l-1)(1-3\beta)(1+2C_1\kappa)}\prod_{k=l}^{j}\left(\frac{|\eta|^{1-3\beta}}{|k|^{2-3\beta}}\right)^{1+2C_1\kappa}.
\end{align*}
It holds that
\begin{align*}
  \prod_{k=l}^{j}\left(\frac{|\eta|^{1-3\beta}}{|k|^{2-3\beta}}\right)=& \eta^{(j-l+1)(1-3\beta)}\left(\frac{(l-1)!}{j!}\right)^{2-3\beta}\\
  \lesssim&e^{(j-l+1)(2-3\beta)}\left( \frac{l-1}{j}\right)^{(l-\frac{1}{2})(2-3\beta)}\frac{ \eta^{(j-l+1)(1-3\beta)}}{j^{(j-l+1)(2-3\beta)}}.
\end{align*}
It is easy to check that
\begin{align*}
  \left( \frac{l-1}{j}\right)^{(l-\frac{1}{2})(2-3\beta)}=&\left(1- \frac{j-l+1}{j}\right)^{\frac{j}{j-l+1}\frac{j-l+1}{j}(l-\frac{1}{2})(2-3\beta)}\approx e^{-\frac{l-\frac{1}{2}}{j}(2-3\beta)(j-l+1)},
\end{align*}
and
\begin{align*}
  \frac{ \eta^{(j-l+1)(1-3\beta)}}{j^{(j-l+1)(2-3\beta)}}\lesssim& \left(\frac{|\xi-\eta|^{1-3\beta}}{j(j-l)^{1-3\beta}}\right)^{j-l+1}\le \left(\frac{|\xi-\eta|^{1-3\beta}}{(j-l)^{2-3\beta}}\right)^{j-l+1}.
\end{align*}
For fixed $|\xi-\eta|$, function $\left(\frac{|\xi-\eta|^{1-3\beta}}{(j-l)^{2-3\beta}}\right)^{j-l+1}$ reaches its maximum at $\ln(|j-l|)=\frac{\ln(|\xi-\eta|^{1-3\beta})}{2-3\beta}-1$, which means that
\begin{align*}
  \frac{ \eta^{(j-l+1)(1-3\beta)}}{j^{(j-l+1)(2-3\beta)}}\lesssim& \left(\frac{|\xi-\eta|^{1-3\beta}}{(j-l)^{2-3\beta}}\right)^{j-l+1}\lesssim e^{(2-3\beta)(j-l+1)}\lesssim e^{\frac{\mu}{2}|\xi-\eta|^s}.
\end{align*}
Combining the above estimates, we have
\begin{align*}
  \frac{w_{\mathrm{NR}}(t,\xi)}{w_{\mathrm{NR}}(t,\eta)}\le& \left(\frac{|\eta| }{|\xi| }\right)^{(l-1)(1-3\beta)(1+2C_1\kappa)}e^{\frac{(j-l+1)^2}{j}(2-3\beta)}e^{\frac{\mu}{2}|\xi-\eta|^s}\\
  \lesssim&e^{(1-3\beta)(1+2C_1\kappa)|\eta -\xi|^s}e^{\frac{j|\xi-\eta|^2}{|\xi|^2}(2-3\beta)}e^{\frac{\mu}{2}|\xi-\eta|^s}\lesssim e^{C\mu|\xi-\eta|^s}.
\end{align*}

At last, we turn to Case 4, $2\xi\le t\le 2\eta$. If $t\ge \eta+\frac{1}{8}|\eta|^{1-3\beta}$, we have
\begin{align*}
  \frac{w_{\mathrm{NR}}(t,\xi)}{w_{\mathrm{NR}}(t,\eta)}=1,
\end{align*}
and there is nothing to prove. If $t\le \eta+\frac{1}{8}|\eta|^{1-3\beta}$, then $\eta\le 4|\eta-\xi|$, therefore
\begin{align*}
  \frac{w_{\mathrm{NR}}(t,\xi)}{w_{\mathrm{NR}}(t,\eta)}=\frac{1}{w_{\mathrm{NR}}(t,\eta)}\lesssim \frac{1}{|\eta|^{\frac{\mu s}{4}}}e^{\frac{\mu}{2}|\eta|^s}\lesssim e^{\mu |\xi-\eta|^s}.
\end{align*}
This completes the proof.
\end{proof}
A consequence of Lemma \ref{lem-om-react-s} is the following, which allows to easily exchanging $W_k(\eta)$ for $W_l(\xi)$.
\begin{corol}\label{cor-W}
    Let $\eta,\xi\in\mathbb R$ and $t\ge 1$, we have that
    \begin{itemize}
      \item for ($t\notin\tilde{\mathrm{I}}_{k,\eta}$), or ($t\in\tilde{\mathrm{I}}_{k,\eta}$, $\frac{1}{2}|\eta|^s<|k|\le E(|\eta|^s)$),
          \begin{align}\label{eq-est-Wkl1}
      \frac{W_k(\eta)}{W_l(\xi)}\lesssim \langle \xi-\eta,k-l\rangle^{1+2C_1\kappa}e^{C\mu |\xi-\eta,k-l|^s};
    \end{align}
    \item for ($t\in\tilde{\mathrm{I}}_{k,\eta},t\notin \tilde{\mathrm{I}}_{l,\xi}$, $|k|\le \frac{1}{2}|\eta|^s$),
     \begin{align}\label{eq-est-Wkl2}
      \frac{W_k(\eta)}{W_l(\xi)}\lesssim \frac{\frac{|\eta|^{1-3\beta}}{|k|^{2-3\beta}}}{\left[1+\left|t-\frac{\eta}{k}\right|\right]}\langle \xi-\eta,k-l\rangle^{1+2C_1\kappa}e^{C\mu |\xi-\eta,k-l|^s};
    \end{align} 
    \item for ($t\notin\tilde{\mathrm{I}}_{k,\eta},t\in \tilde{\mathrm{I}}_{l,\xi}$, $|l|\le \frac{1}{2}|\xi|^s$),
     \begin{align}\label{eq-est-Wkl3}
      \frac{W_k(\eta)}{W_l(\xi)}\lesssim \frac{\left[1+\left|t-\frac{\xi}{l}\right|\right]}{\frac{|\xi|^{1-3\beta}}{|l|^{2-3\beta}}}\langle \xi-\eta,k-l\rangle^{1+2C_1\kappa}e^{C\mu |\xi-\eta,k-l|^s};
    \end{align} 
        \item for $t\in\tilde{\mathrm{I}}_{k,\eta}\cap \tilde{\mathrm{I}}_{l,\xi}$, $|k|\le \frac{1}{2}|\eta|^s$, and $\frac{1}{\al}|\xi|\le|\eta|\le\al|\xi|$ for some $\al\ge1$,
     \begin{align}\label{eq-est-Wkl4}
      \frac{W_k(\eta)}{W_l(\xi)}\lesssim \frac{\left[1+\left|t-\frac{\xi}{l}\right|\right]}{\left[1+\left|t-\frac{\eta}{k}\right|\right]}\langle \xi-\eta\rangle^{1+2C_1\kappa}e^{C\mu |\xi-\eta|^s};
    \end{align} 
          \item for ($k=l$), or ($t\in\tilde{\mathrm{I}}_{k,\eta}\cap {\mathrm{I}}_{m,\xi}$, with $m\neq k$)
          \begin{align}\label{eq-est-Wkl}
      \frac{W_k(\eta)}{W_l(\xi)}\lesssim \langle \xi-\eta\rangle^{2+2C_1\kappa}e^{C\mu |\xi-\eta|^s}.
    \end{align}
    \end{itemize}
\end{corol}
\begin{proof}
  The estimates \eqref{eq-est-Wkl1}-\eqref{eq-est-Wkl4} follow directly from the definition of $W$, Lemma \ref{lem-evo-w-lntype} and Lemma \ref{lem-om-react-s}. We only give the proof for \eqref{eq-est-Wkl}. 

For the case $t\in\tilde{\mathrm{I}}_{k,\eta}\cap {\mathrm{I}}_{m,\xi}$, with $m\neq k$. Then we by Lemma \ref{lem-separate} we have
  \begin{align*}
    \frac{\frac{|\eta|^{1-3\beta}}{|k|^{2-3\beta}}}{\left[1+\left|t-\frac{\eta}{k}\right|\right]}\le \frac{\eta}{k}\lesssim |\xi-\eta|.
  \end{align*}

For the case $k=l$. If $\eta\xi<0$, or $t\notin\tilde{\mathrm{I}}_{k,\eta}$, or $t\in\tilde{\mathrm{I}}_{k,\eta}\cap   \tilde{\mathrm{I}}_{l,\xi}$, there is nothing to prove. If $t\in\tilde{\mathrm{I}}_{k,\eta}$ and $t\notin {\mathrm{I}}_{l,\xi}$, there must exists $m\neq k$ such that $t\in {\mathrm{I}}_{m,\xi}$. If $t\in\tilde{\mathrm{I}}_{k,\eta}\cap \left({\mathrm{I}}_{l,\xi}\setminus \tilde{\mathrm{I}}_{l,\xi}\right)$, we have $\xi\approx \eta$, then by Lemma \ref{lem-separate} there must holds $|t-\frac{\xi}{m}|\ge \frac{1}{10\al}\frac{|\xi|^{1-3\beta}}{|m|^{2-3\beta}}$ or $|\xi-\eta|\gtrsim \left|\frac{\eta}{k}\right|^{1-3\beta}$. It follows that
  \begin{align*}
    \frac{\frac{|\eta|^{1-3\beta}}{|k|^{2-3\beta}}}{\left[1+\left|t-\frac{\eta}{k}\right|\right]}\lesssim \langle \xi-\eta\rangle,
  \end{align*}
which gives \eqref{eq-est-Wkl}.
\end{proof}

\begin{corol}\label{cor-WR}
    Under the same assumption to Corollary \ref{cor-W}, we have
      \begin{align}
          \frac{W^{\mathrm R}(\eta)}{W^{\mathrm R}(\xi)}\lesssim\langle \xi-\eta\rangle^{2+2C_1\kappa} e^{C\mu|\eta-\xi|^s}.
    \end{align}
\end{corol}

\begin{lemma}\label{lem-w-wgl}
For $t\in\tilde{\mathrm{I}}_{k,\eta}\cap \mathrm{I}_{l,\xi}$ such that $\frac{1}{\al}|\xi|\le|\eta|\le\al|\xi|$ for some $\al\ge1$ we have
\beq\label{eq-w-wgl}
\begin{aligned}
&  \sqrt{\frac{\mathfrak w(\nu,t,\eta)\pa_t w_j(t,\eta)}{w_j(t,\eta)}}\\
&\lesssim_{\alpha} \left(\sqrt{\frac{\mathfrak w(\nu,t,\xi)\pa_t w_m(t,\xi)}{w_m(t,\xi)}}+\sqrt{\frac{\mathfrak w(\nu,t,\xi)\pa_t g(t,\xi)}{g(t,\xi)}}+t^{-c_1}|\xi|^{\frac{s}{2}}\right)\langle \xi-\eta\rangle,
\end{aligned}
\eeq
where 
\begin{equation}\label{eq-def-crr-w}
  \mathfrak w(\nu,t,\eta)=\left\{
    \begin{array}{ll}
      \langle\nu^{\frac{1}{3}}|\eta|^{1-s}\rangle^{-\frac{3}{2}+3\beta},& t\ge t_{E(|\eta|^s),|\eta|},\\
      1,&t< t_{E(|\eta|^s),|\eta|}.
    \end{array}
  \right.
\end{equation}
\end{lemma}
\begin{proof}
We divided the proof into 4 cases.

Case 1, $|\xi|^{1-s}\lesssim t\le t_{E(|\xi|^s),|\xi|}$. In this case, it holds that $\frac{|\xi|^{1-3\beta}}{|l|^{2-3\beta}}\approx \frac{|\eta|^{1-3\beta}}{|k|^{2-3\beta}}\approx 1$. If $|\eta|^{1-s}\le \nu^{-\frac{1}{3}}$, we have $\mathfrak w(\nu,t,\eta)\approx 1$,  $\mathfrak w(\nu,t,\xi)= 1$, $\mathfrak g(\nu,\eta)\approx \nu^{-\beta}|\eta|^{-3\beta(1-s)}$, $\langle \nu^{\frac{1}{3}}\frac{\eta}{k}\rangle^{-\frac{3}{2}}\approx 1$, and
\begin{align*}
  \sqrt{\frac{\mathfrak w(\nu,t,\eta)\pa_t w_j(t,\eta)}{w_j(t,\eta)}}&\approx \sqrt{\frac{a_{k,\eta}}{1+a_{k,\eta}\left|\frac{\eta}{k}-t\right|}},\\
  \sqrt{\frac{\mathfrak w(\nu,t,\xi)\pa_t g(t,\xi)}{g(t,\xi)}}&=\sqrt{\frac{\kappa\mathfrak g(\nu,\eta)\langle \nu^{\frac{1}{3}}\frac{\eta}{k}\rangle^{-\frac{3}{2}}\nu^\beta\frac{|\eta|}{|k|^2}}{1+(t-\frac{|\eta|}{|k|})^2}}\approx \sqrt{\frac{1}{1+\left|\frac{\xi}{l}-t\right|^2}}.
\end{align*}
If $|\eta|^{1-s}> \nu^{-\frac{1}{3}}$, we have $\mathfrak w(\nu,t,\eta)\approx \nu^{-\frac{1}{2}+\beta}|\eta|^{3(-\frac{1}{2}+\beta)(1-s)}\approx (\nu^{\frac{1}{3}}t)^{-\frac{3}{2}+3\beta}$, $\mathfrak g(\nu,\eta)\approx 1$, $\langle \nu^{\frac{1}{3}}\frac{\xi}{l}\rangle^{-\frac{3}{2}}\approx \nu^{-\frac{1}{2}}|\xi|^{-\frac{3}{2}(1-s)}\approx (\nu^{\frac{1}{3}}t)^{-\frac{3}{2}}$, and
\begin{align*}
  \sqrt{\frac{\mathfrak w(\nu,t,\eta)\pa_t w_j(t,\eta)}{w_j(t,\eta)}}&\approx \sqrt{\frac{a_{k,\eta}(\nu^{\frac{1}{3}}t)^{-\frac{3}{2}+3\beta}}{1+a_{k,\eta}\left|\frac{\eta}{k}-t\right|}},\\
  \sqrt{\frac{\mathfrak w(\nu,t,\xi)\pa_t g(t,\xi)}{g(t,\xi)}}
 &\approx \sqrt{\frac{(\nu^{\frac{1}{3}}t)^{-\frac{3}{2}+3\beta}}{1+\left|\frac{\xi}{l}-t\right|^2}}.
\end{align*}
For both cases, we have 
  \begin{align*}
     \sqrt{\frac{\mathfrak w(\nu,t,\eta)\pa_t w_j(t,\eta)}{w_j(t,\eta)}}\sqrt{\frac{g(t,\xi)}{\mathfrak w(\nu,t,\xi)\pa_t g(t,\xi)}}\approx \sqrt{a_{k,\eta}\frac{1+\left|\frac{\xi}{l}-t\right|^2}{1+a_{k,\eta}\left|\frac{\eta}{k}-t\right|}} \lesssim 1+\left|\frac{\xi}{l}-t\right|.
  \end{align*}
So it suffices to show that 
$$\left|\frac{\xi}{l}-t\right|\lesssim\langle \xi-\eta\rangle.$$
If $l=k$ (case (a') in Lemma \ref{lem-separate}), it is clear that
\begin{align*}
  \left|\frac{\xi}{l}-t\right|\lesssim \left|\frac{\xi-\eta}{l}\right|+\left|\frac{\eta}{k}-t\right|\lesssim \left|\frac{\xi-\eta}{l}\right|+\frac{|\eta|^{1-3\beta}}{|k|^{2-3\beta}}\lesssim \langle \xi-\eta\rangle.
\end{align*}
If $l\neq k$, by Lemma \ref{lem-separate} it holds that $|\xi-\eta|\gtrsim \left|\frac{\xi}{l}\right|$ (case (b') in Lemma \ref{lem-separate}), then
\begin{align*}
   \left|\frac{\xi}{l}-t\right|\le \left|\frac{\xi}{l}\right| \lesssim \left|\xi-\eta\right|.
\end{align*}

  Case 2, $t_{E(|\xi|^s),|\xi|}\le t\le2 |\xi|^{1-s}$. In this case, it also holds that $\frac{|\xi|^{1-3\beta}}{|l|^{2-3\beta}}\approx \frac{|\eta|^{1-3\beta}}{|k|^{2-3\beta}}\approx 1$, then
  \begin{align*}
     \sqrt{\frac{\mathfrak w(\nu,t,\eta)\pa_t w_j(t,\eta)}{w_j(t,\eta)}}\sqrt{\frac{g(t,\xi)}{\mathfrak w(\nu,t,\xi)\pa_t g(t,\xi)}}\approx \sqrt{a_{k,\eta}\frac{\left(\frac{|\xi|^{1-3\beta}}{|l|^{2-3\beta}}\right)^2+\left|\frac{\xi}{l}-t\right|^2}{\frac{|\xi|^{1-3\beta}}{|l|^{2-3\beta}}\left(1+a_{k,\eta}\left|\frac{\eta}{k}-t\right|\right)}} \lesssim 1+\left|\frac{\xi}{l}-t\right|.
  \end{align*}
We can get the result in the same way as in Case 1.

Case 3, $t\ge2|\xi|$. In this case, we have $t\approx |\eta|\approx |\xi|$. It follows from $\tilde{\mathrm{I}}_{k,\eta}$ that $2|\xi|\le t\le |\eta|+\frac{1}{8}|\eta|^{1-3\beta}$ and
\begin{align*}
  |\xi-\eta|\ge \frac{7}{8} |\eta|\gtrsim t.
\end{align*}
Therefore, is clear that
\begin{align*}
  \sqrt{\frac{\pa_t w_j(t,\eta)}{w_j(t,\eta)}}\lesssim1\lesssim t^{-c_1}|\xi|^{\frac{s}{2}}\langle\xi-\eta\rangle.
\end{align*}

Case 4, $2|\xi|^{1-s}\le t\le 2|\xi|$. If $t\in\tilde{\mathrm{I}}_{l,\xi}$, it holds that
  \begin{align*}
     \sqrt{\frac{\pa_t w_j(t,\eta)}{w_j(t,\eta)}}\sqrt{\frac{w_m(t,\xi)}{\pa_t w_m(t,\xi)}}\approx \sqrt{\frac{\left(1+\left|\frac{\xi}{l}-t\right|\right)}{\left(1+\left|\frac{\eta}{k}-t\right|\right)}}.
  \end{align*}
By using Lemma \ref{lem-separate} we get
\begin{equation}\label{eq-est-sep-lk}
  \frac{\left(1+\left|\frac{\xi}{l}-t\right|\right)}{\left(1+\left|\frac{\eta}{k}-t\right|\right)}\lesssim\left\{
    \begin{array}{ll}
      1+\frac{\left|\frac{\xi-\eta}{l}\right|}{ 1+\left|\frac{\eta}{k}-t\right| }\le  \langle\xi-\eta\rangle,& \text{ if Case (a') holds};\\
      1+\frac{|\xi|^{1-3\beta}}{|l|^{2-3\beta}}\lesssim \langle\xi-\eta\rangle,& \text{ if Case (b') holds},
    \end{array}
  \right.
\end{equation}
which gives \eqref{eq-w-wgl}.

If $t\in \left(\mathrm{I}_{l,\xi}\setminus\tilde{\mathrm{I}}_{l,\xi}\right)$, we have $\left|\frac{\xi}{l}-t\right|\gtrsim \frac{|\xi|^{1-3\beta}}{|l|^{2-3\beta}}$ and
  \begin{align*}
     &\sqrt{\frac{\pa_t w_j(t,\eta)}{w_j(t,\eta)}}\sqrt{\frac{g(t,\xi)}{\pa_t g(t,\xi)}}\approx \sqrt{\frac{\left(\frac{|\xi|^{1-3\beta}}{|l|^{2-3\beta}}\right)^2+\left|\frac{\xi}{l}-t\right|^2}{\frac{|\xi|^{1-3\beta}}{|l|^{2-3\beta}}\left(1+\left|\frac{\eta}{k}-t\right|\right)}}\lesssim\sqrt{\frac{1+\left|\frac{\xi}{l}-t\right|}{\frac{|\xi|^{1-3\beta}}{|l|^{2-3\beta}} }} \sqrt{\frac{1+\left|\frac{\xi}{l}-t\right|}{ 1+\left|\frac{\eta}{k}-t\right| }}.
  \end{align*}
By using Lemma \ref{lem-separate} again, we have
\begin{equation}
 \sqrt{\frac{1+\left|\frac{\xi}{l}-t\right|}{\frac{|\xi|^{1-3\beta}}{|l|^{2-3\beta}} }}\lesssim\left\{
    \begin{array}{ll}
      \sqrt{1+\left|\frac{\xi}{l}-t\right|}\lesssim \langle\xi-\eta\rangle^{\frac{1}{2}},& \text{ if }k\neq l;\\
      1,& \text{ if }k=l \text{ and } \left|\frac{\xi}{l}-t\right|\le\frac{|\xi|^{1-3\beta}}{|l|^{2-3\beta}};\\
      \sqrt{1+ \frac{|\xi-\eta|}{|l|}}\lesssim\langle\xi-\eta\rangle^{\frac{1}{2}},& \text{ if }k=l \text{ and } \left|\frac{\xi}{l}-t\right|>\frac{|\xi|^{1-3\beta}}{|l|^{2-3\beta}}.
    \end{array}
  \right.
\end{equation}
Combing this inequality and \eqref{eq-est-sep-lk}, we deduce
  \begin{align*}
     &\sqrt{\frac{\pa_t w_j(t,\eta)}{w_j(t,\eta)}}\sqrt{\frac{g(t,\xi)}{\pa_t g(t,\xi)}}\lesssim\langle \xi-\eta\rangle.
  \end{align*}
\end{proof}
\begin{lemma}\label{lem-tran-W}
  Let $1\le t\le 2\max(|\xi|,|\eta|,|k|,|m|)$, then
  \begin{align*}
    \left| \frac{W_k(t,\eta)}{W_m(t,\xi)}-1\right|\frac{|m,\xi|}{t^2}\lesssim& \frac{\langle k-m,\xi-\eta\rangle^{2+2C_1\kappa}|m,\xi|^s}{t^2}e^{C\mu|k-m,\xi-\eta|^s}\\
    &+\frac{|\xi-\eta,k-m|^{2+2C_1\kappa}|m,\xi|^s}{t^{1+s}}e^{C\mu|\xi-\eta,k-m|^s}(1-\mathbf1_{k=m,\beta\ge \frac{1}{6}})\\
    &+\frac{|\xi-\eta |^{2+2C_1\kappa}}{t}e^{C\mu|\xi-\eta |^s}\mathbf1_{k=m,\beta\ge \frac{1}{6}}\\
    &+\langle \xi-\eta\rangle^{1+2C_1\kappa}e^{C\mu |\xi-\eta|^s}\sqrt{\frac{\pa_t g(t,\iota(m,\xi))}{g(t,\iota(m,\xi))}}\sqrt{\frac{\pa_t g(t,\iota(k,\eta))}{g(t,\iota(k,\eta))}}\\
    &+\langle \xi-\eta\rangle^{1+2C_1\kappa}e^{C\mu |\xi-\eta|^s}\sqrt{\frac{\pa_t w_m(t,\iota(m,\xi))}{w_m(t,\iota(m,\xi))}}\sqrt{\frac{\pa_t w_k(t,\iota(k,\eta))}{w_k(t,\iota(k,\eta))}}.
\end{align*}
\end{lemma}
\begin{proof}
  Recall that
\begin{align*}
  W_k(t,\eta)=\frac{1}{w_k(t,\iota(k,\eta))}.
\end{align*}
We mainly discuss the case where $|\eta|\ge |k|$ and $|\xi|\ge |m|$, and there is no essential difference for the cases where $|\eta|< |k|$ or $|\xi|< |m|$.

For $|\eta|\ge |k|$ and $|\xi|\ge |m|$, we know that $\frac{|m,\xi|}{t^2}\approx \frac{|\xi|}{t^2}$ and
  \begin{align*}
    \left| \frac{W_k(t,\eta)}{W_m(t,\xi)}-1\right|=&\left| \frac{w_m(t,\iota(m,\xi))}{w_k(t,\iota(k,\eta))}-1\right|=\left| \frac{w_m(t,\xi)}{w_k(t,\eta)}-1\right|.
  \end{align*}

From Lemma \ref{lem-om-react-s} and Corollary \ref{cor-W} we have
\begin{align*}
  \frac{w_m(t,\xi)}{w_k(t,\eta)}\lesssim |\eta|^{1-3\beta}\langle \xi-\eta,k-l\rangle^{1+2C_1\kappa}e^{C\mu |\xi-\eta,k-l|^s},
\end{align*}
then there is nothing to prove for 
\begin{align*}
  \langle k-m,\xi-\eta\rangle\gtrsim\left(|k|+|m|+|\eta|+|\xi|\right)^{1-s}.
\end{align*}
So we assume
  \begin{align}\label{eq-assum-W}
    |k-m|+|\xi-\eta|\le \frac{1}{100}\left(|k|+|m|+|\eta|+|\xi|\right)^{1-s}\le\frac{1}{100}\left(|k|+|m|+|\eta|+|\xi|\right),
  \end{align}
and divide the proof into different cases.

Case 1, $t\le \min(t_{E(|\xi|^s),|\xi|},t_{E(|\eta|^s),|\eta|})$. We have 
\begin{align*}
  \left| \frac{w_m(t,\xi)}{w_k(t,\eta)}-1\right|=\left|\frac{w_{\mathrm{NR}}(0,\xi)}{w_{\mathrm{NR}}(0,\eta)}-1\right|.
\end{align*}

We claim that $|E(|\eta|^s)-E(|\xi|^s)|\le 1$. If $|\eta|\ge|\xi|$, from the assumption \eqref{eq-assum-W}, we have
\begin{align*}
  |\xi-\eta|\le \frac{1}{100}\left(2|\eta|+2|\xi|\right)^{1-s}\le \frac{1}{20} |\xi|^{1-s}.
\end{align*}
Then we have
\begin{align*}
   (|\eta|^s-|\xi|^s)=s\int^{|\eta|}_{|\xi|}\frac{1}{x^{1-s}}dx\le \frac{s|\xi-\eta|}{|\xi|^{1-s}}\le \frac{s}{20}<1,
\end{align*}
which means that $|E(|\eta|^s)-E(|\xi|^s)|\le 1$. If $|\eta|\le|\xi|$, one can get the same result.

If $E(|\eta|^s)=E(|\xi|^s)$, then we have
\begin{align*}
  &\left|\frac{w_m(0,\xi)}{w_k(0,\eta)}-1\right|=\left|\left(\frac{|\eta|}{|\xi|}\right)^{E(|\eta|^s)(1-3\beta)(1+2C_1\kappa)}-1\right|\\
  =&\left|\left(1+\frac{|\eta|-|\xi|}{|\xi|}\right)^{E(|\eta|^s)(1-3\beta)(1+2C_1\kappa)}-1\right|\\
  \lesssim&E(|\eta|^s)(1-3\beta)(1+2C_1\kappa)\frac{|\eta-\xi|}{|\xi|}\lesssim\frac{|\eta-\xi|}{|\xi|^{1-s}}.
\end{align*}
Next, if $E(|\eta|^s)=E(|\xi|^s)+1$, we have $|\xi|^s\le E(|\eta|^s)\le |\eta|^s$. It is clear that $\frac{|\eta|}{|\xi|}\ge 1$, and
\begin{align*}
  &\left|\frac{w_m(0,\xi)}{w_k(0,\eta)}-1\right|=\left|\left(\frac{|\eta|}{|\xi|}\right)^{E(|\xi|^s)(1-3\beta)(1+2C_1\kappa)}\left(\frac{|\eta|^{1-3\beta}}{E(|\eta|^s)^{2-3\beta}}\right)^{(1+2C_1\kappa)}-1\right|\\
  \le&\left|\left(\frac{|\eta|}{|\xi|}\right)^{E(|\xi|^s)(1-3\beta)(1+2C_1\kappa)}\left(\frac{|\eta|^{1-3\beta}}{|\xi|^{1-3\beta}}\right)^{(1+2C_1\kappa)}-1\right|\\
  =&\left|\left(1+\frac{|\eta-\xi|}{|\xi|}\right)^{E(|\eta|^s)(1-3\beta)(1+2C_1\kappa)}-1\right|\\
  \lesssim&E(|\eta|^s)(1-3\beta)(1+2C_1\kappa)\frac{|\eta-\xi|}{|\xi|}\lesssim\frac{|\eta-\xi|}{|\xi|^{1-s}}.
\end{align*}
If $E(|\eta|^s)=E(|\xi|^s)-1$, we have $|\eta|^s\le E(|\xi|^s)\le |\xi|^s$. Thus $\frac{|\eta|}{|\xi|}\le 1$, and
\begin{align*}
  &\left|\frac{w_m(0,\xi)}{w_k(0,\eta)}-1\right|=\left|\left(\frac{|\eta|}{|\xi|}\right)^{E(|\eta|^s)(1-3\beta)(1+2C_1\kappa)}\left(\frac{E(|\xi|^s)^{2-3\beta}}{|\xi|^{1-3\beta}}\right)^{(1+2C_1\kappa)}-1\right|\\
  \le&\left|\left(\frac{|\eta|}{|\xi|}\right)^{E(|\eta|^s)(1-3\beta)(1+2C_1\kappa)}\left(\frac{|\eta|^{1-3\beta}}{|\xi|^{1-3\beta}}\right)^{(1+2C_1\kappa)}-1\right|\\
  =&\left|\left(1-\frac{|\eta-\xi|}{|\xi|}\right)^{E(|\xi|^s)(1-3\beta)(1+2C_1\kappa)}-1\right|\lesssim\frac{|\eta-\xi|}{|\xi|^{1-s}}.
\end{align*}

Case 2, $\min(t_{E(|\xi|^s),|\xi|},t_{E(|\eta|^s),|\eta|})\le t\le \max(t_{E(|\xi|^s),|\xi|},t_{E(|\eta|^s),|\eta|})$. From the assumption \eqref{eq-assum-W}, we must have
\begin{itemize}
  \item if $|\eta|\ge|\xi|$ and $E(|\xi|^s)=E(|\eta|^s)$, then $t\in (\mathrm{I}_{E(|\xi|^s),|\xi|}^L\setminus\tilde{\mathrm{I}}_{E(|\xi|^s), |\xi|}^{L})$;\\
  \item if $|\eta|\ge|\xi|$ and $E(|\xi|^s)=E(|\eta|^s)-1$, then $t\in (\mathrm{I}_{E(|\eta|^s),|\eta|}^R\setminus\tilde{\mathrm{I}}_{E(|\eta|^s), |\eta|}^{R})$;\\
  \item if $|\eta|\le|\xi|$ and $E(|\xi|^s)=E(|\eta|^s)$, then $t\in (\mathrm{I}_{E(|\eta|^s),|\eta|}^L\setminus\tilde{\mathrm{I}}_{E(|\eta|^s), |\eta|}^{L})$;\\
  \item if $|\eta|\le|\xi|$ and $E(|\xi|^s)=E(|\eta|^s)+1$, then $t\in (\mathrm{I}_{E(|\xi|^s),|\xi|}^R\setminus\tilde{\mathrm{I}}_{E(|\xi|^s), |\xi|}^{R})$.  
\end{itemize}
Thus we have the same result as Case 1.

Case 3, $\max(t_{E(|\xi|^s),|\xi|},t_{E(|\eta|^s),|\eta|})\le t\le 2\min(|\xi|,|\eta|)$. For this case, $t\gtrsim |\eta|^{1-s}$. During this time period, there must be some $j,l$ such that $t\in \mathrm{I}_{l,\xi}\cap \mathrm{I}_{j,\eta}$ and 
\begin{align*}
  1\le |l|\le E(|\xi|^s),\quad 1\le |j|\le E(|\eta|^s).
\end{align*}
As $|\xi-\eta|\le \frac{|\xi|^{1-s}}{20}$, we have
\begin{align*}
  |\xi-\eta|\le \frac{1}{10}\frac{|\xi|}{|l|}.
\end{align*}
Then we have $\big||l|-|j|\big|\le 1$. 

Case 3.1, $|\xi|\le|\eta|$, we have  $|l|\le |j|$.

Case 3.1.1, $j=l$ and $t\in (\mathrm{I}_{j,\eta}^L\setminus\tilde{\mathrm{I}}_{j, \eta}^{L})\cap (\mathrm{I}_{l,\xi}^L\setminus\tilde{\mathrm{I}}_{l, \xi}^{L})$, or $t\in (\mathrm{I}_{j,\eta}^R\setminus\tilde{\mathrm{I}}_{j, \eta}^{R})\cap (\mathrm{I}_{l,\xi}^R\setminus\tilde{\mathrm{I}}_{l, \xi}^{R})$, or $|l|=|j|-1$ and $t\in (\mathrm{I}_{j,\eta}^R\setminus\tilde{\mathrm{I}}_{j, \eta}^{R}) \cap (\mathrm{I}_{l,\xi}^L\setminus\tilde{\mathrm{I}}_{l, \xi}^{L})$. We have
\begin{align*}
  \frac{w_m(t,\xi)}{w_k(t,\eta)}=\left(\frac{|\eta| }{|\xi| }\right)^{(j)(1-3\beta)(1+2C_1\kappa)},
\end{align*}
and there is nothing to prove.

If $|l|=|j|-1$, and not in Case 3.1.1. It is impossible. Thus we only need to consider $|l|=|j|$.

Case 3.1.2, $j=l$, and $t\in (\mathrm{I}_{j,\eta}^R\setminus\tilde{\mathrm{I}}_{j, \eta}^{R})\cap (\mathrm{I}_{l,\xi}^L\setminus\tilde{\mathrm{I}}_{l, \xi}^{L})$. In this case, we have $|\xi-\eta|\gtrsim \frac{1}{10}\left(\frac{|\xi|}{|l|}\right)^{1-3\beta}$. It follows that
\begin{align*}
  &\left|\frac{w_m(t,\xi)}{w_k(t,\eta)}-1\right|=\left|\left(\frac{|\eta|}{|\xi|}\right)^{(j-1)(1-3\beta)(1+2C_1\kappa)}\left(\frac{|\eta|^{1-3\beta}}{|j|^{2-3\beta}}\right)^{1+2C_1\kappa}-1\right|\\
  \le& \left(\frac{|\eta|^{1-3\beta}}{|j|^{2-3\beta}}\right)^{1+2C_1\kappa}\left|\left(\frac{|\eta|}{|\xi|}\right)^{(j-1)(1-3\beta)(1+2C_1\kappa)} -1\right|+\left(\frac{|\eta|^{1-3\beta}}{|j|^{2-3\beta}}\right)^{1+2C_1\kappa}-1\\
  \lesssim&|\xi-\eta|^{1+2C_1\kappa}\frac{|\eta-\xi|}{|\xi|^{1-s}}+\frac{|\xi-\eta|^{1+2C_1\kappa}}{|j|^{1+2C_1\kappa}}.
\end{align*}
One can see that 
\begin{align*}
  \frac{|\xi-\eta|^{1+2C_1\kappa}}{|j|^{1+2C_1\kappa}}\frac{|\xi|}{t^2}\lesssim \frac{|\xi-\eta|^{1+2C_1\kappa}}{t}.
\end{align*}

Case 3.1.3, $j=l$, $t\in \tilde{\mathrm{I}}_{j,\eta}$ and $t\notin \tilde{\mathrm{I}}_{l,\xi}$.

Case 3.1.3.1, $|\xi-\eta|\le \frac{1}{20}\left(\frac{|\xi|}{|l|}\right)^{1-3\beta}$. As $|\eta|\ge|\xi|$, we have $t\in \tilde{\mathrm{I}}_{j,\eta}^R\cap (\mathrm{I}_{l,\xi}^R\setminus\tilde{\mathrm{I}}_{l, \xi}^{R})$. It also holds that $\left|t-\frac{\eta}{j}\right|\approx \frac{|\eta|^{1-3\beta}}{|j|^{2-3\beta}}$ and $\left|t-\frac{\xi}{l}\right|\approx \frac{|\xi|^{1-3\beta}}{|l|^{2-3\beta}}$. Thus we have
\begin{align*}
  \left|\frac{w_m(t,\xi)}{w_k(t,\eta)}-1\right|\le1+\frac{w_m(t,\xi)}{w_k(t,\eta)}.
\end{align*}
As $\left|t-\frac{\eta}{j}\right|\approx \frac{|\eta|^{1-3\beta}}{|j|^{2-3\beta}}$, there is no difference between $j=k$ and $j\neq k$. By Lemma \ref{lem-evo-w-lntype} and Corollary \ref{cor-W}, we have 
\begin{align*}
  1+\frac{w_m(t,\xi)}{w_k(t,\eta)}\lesssim\langle \xi-\eta\rangle^{1+2C_1\kappa}e^{C\mu |\xi-\eta|^s}.
\end{align*}
For this case, we also have
\begin{align*}
  \frac{|\xi|}{t^2}=&\sqrt{\frac{\frac{|\xi|^{1-3\beta}}{|l|^{2-3\beta}}}{\left(\frac{|\xi|^{1-3\beta}}{|l|^{2-3\beta}}\right)^2+\left|t-\frac{\xi}{l}\right|^2}}\sqrt{\frac{\frac{|\eta|^{1-3\beta}}{|j|^{2-3\beta}}}{\left(\frac{|\eta|^{1-3\beta}}{|j|^{2-3\beta}}\right)^2+\left|t-\frac{\eta}{j}\right|^2}}\frac{1}{\sqrt{\frac{\frac{|\xi|^{1-3\beta}}{|l|^{2-3\beta}}}{\left(\frac{|\xi|^{1-3\beta}}{|l|^{2-3\beta}}\right)^2+\left|t-\frac{\xi}{l}\right|^2}}\sqrt{\frac{\frac{|\eta|^{1-3\beta}}{|j|^{2-3\beta}}}{\left(\frac{|\eta|^{1-3\beta}}{|j|^{2-3\beta}}\right)^2+\left|t-\frac{\eta}{j}\right|^2}}}\frac{|\xi|}{t^2}\\
  \lesssim&\sqrt{\frac{\frac{|\xi|^{1-3\beta}}{|l|^{2-3\beta}}}{\left(\frac{|\xi|^{1-3\beta}}{|l|^{2-3\beta}}\right)^2+\left|t-\frac{\xi}{l}\right|^2}}\sqrt{\frac{\frac{|\eta|^{1-3\beta}}{|j|^{2-3\beta}}}{\left(\frac{|\eta|^{1-3\beta}}{|j|^{2-3\beta}}\right)^2+\left|t-\frac{\eta}{j}\right|^2}}\frac{|\xi|^{2-3\beta}}{|l|^{2-3\beta}t^2}\\
  \lesssim&\sqrt{\frac{\frac{|\xi|^{1-3\beta}}{|l|^{2-3\beta}}}{\left(\frac{|\xi|^{1-3\beta}}{|l|^{2-3\beta}}\right)^2+\left|t-\frac{\xi}{l}\right|^2}}\sqrt{\frac{\frac{|\eta|^{1-3\beta}}{|j|^{2-3\beta}}}{\left(\frac{|\eta|^{1-3\beta}}{|j|^{2-3\beta}}\right)^2+\left|t-\frac{\eta}{j}\right|^2}}\frac{1}{t^{3\beta}}=\sqrt{\frac{\pa_t g(t,\xi)}{g(t,\xi)}}\sqrt{\frac{\pa_t g(t,\eta)}{g(t,\eta)}}\frac{1}{t^{3\beta}}.
\end{align*}

Case 3.1.3.2, $|\xi-\eta|\ge \frac{1}{20}\left(\frac{|\xi|}{|l|}\right)^{1-3\beta}$. We also know that $t\in(\mathrm{I}_{l,\xi}^R\setminus\tilde{\mathrm{I}}_{l, \xi}^{R})$. The most troublesome case is $j=k$, thus $w_k(t,\eta)=w_{\mathrm{R}}(t,\eta)$ and
\begin{align*}
  \frac{w_m(t,\xi)}{w_k(t,\eta)}\le \left(\frac{|\eta|}{|\xi|}\right)^{(j-1)(1-3\beta)(1+2C_1\kappa)}\left(\frac{|\eta|^{1-3\beta}}{|j|^{2-3\beta}}\right)^{2+2C_1\kappa}.
\end{align*}
Thus we have
\begin{align*}
  &\left|\frac{w_m(t,\xi)}{w_k(t,\eta)}-1\right|\le\left|\left(\frac{|\eta|}{|\xi|}\right)^{(j-1)(1-3\beta)(1+2C_1\kappa)}\left(\frac{|\eta|^{1-3\beta}}{|j|^{2-3\beta}}\right)^{2+2C_1\kappa}-1\right|\\
  \le& \left(\frac{|\eta|^{1-3\beta}}{|j|^{2-3\beta}}\right)^{2+2C_1\kappa}\left|\left(\frac{|\eta|}{|\xi|}\right)^{(j-1)(1-3\beta)(1+2C_1\kappa)} -1\right|+\left(\frac{|\eta|^{1-3\beta}}{|j|^{2-3\beta}}\right)^{2+2C_1\kappa}-1\\
  \lesssim&|\xi-\eta|^{2+2C_1\kappa}\frac{|\eta-\xi|}{|\xi|^{1-s}}+\frac{|\xi-\eta|^{2+2C_1\kappa}}{|j|^{2+2C_1\kappa}},
\end{align*}
which is similar to Case 3.1.2.

Case 3.1.4, $j=l$, $t\notin \tilde{\mathrm{I}}_{j,\eta}$ and $t\in \tilde{\mathrm{I}}_{l,\xi}$.

Case 3.1.4.1, $|\xi-\eta|\le \frac{1}{20}\left(\frac{|\xi|}{|l|}\right)^{1-3\beta}$. As $|\eta|\ge|\xi|$, we have $t\in (\mathrm{I}_{j,\eta}^L\setminus\tilde{\mathrm{I}}_{j, \eta}^{L})\cap \tilde{\mathrm{I}}_{l,\xi}^R$. It also holds that $\left|t-\frac{\eta}{j}\right|\approx \frac{|\eta|^{1-3\beta}}{|j|^{2-3\beta}}$ and $\left|t-\frac{\xi}{l}\right|\approx \frac{|\xi|^{1-3\beta}}{|l|^{2-3\beta}}$. Then similar to Case 3.1.3.1, we have
\begin{align*}
  \left|\frac{w_m(t,\xi)}{w_k(t,\eta)}-1\right|\frac{|\xi|}{t^2} \le\langle \xi-\eta\rangle^{1+2C_1\kappa}e^{C\mu |\xi-\eta|^s}\sqrt{\frac{\pa_t g(t,\xi)}{g(t,\xi)}}\sqrt{\frac{\pa_t g(t,\eta)}{g(t,\eta)}}\frac{1}{t^{3\beta}}.
\end{align*}

Case 3.1.4.2, $|\xi-\eta|\ge \frac{1}{20}\left(\frac{|\xi|}{|l|}\right)^{1-3\beta}$.  If $l\neq m$, thus $w_m(t,\xi)=w_{\mathrm{NR}}(t,\xi)$ and
\begin{align*}
  1\le\frac{w_m(t,\xi)}{w_k(t,\eta)}=\frac{w_{\mathrm{NR}}(t,\xi)}{w_k(t,\eta)}\le \left(\frac{|\eta|}{|\xi|}\right)^{(j-1)(1-3\beta)(1+2C_1\kappa)}\left(\frac{|\eta|^{1-3\beta}}{|j|^{2-3\beta}}\right)^{1+2C_1\kappa}.
\end{align*}
Thus similar to Case 3.1.2, we have
\begin{align}\label{eq-est-3.1.4.2}
  &\left|\frac{w_m(t,\xi)}{w_k(t,\eta)}-1\right|\frac{|\xi|}{t^2}
  \lesssim|\xi-\eta|^{1+2C_1\kappa}\frac{|\eta-\xi||\xi|^s}{t^2}+\frac{|\xi-\eta|^{1+2C_1\kappa}}{t}.
\end{align}

{If $l= m$, $w_m(t,\xi)=w_{\mathrm{R}}(t,\xi)$. We have for  $t\in \tilde{\mathrm{I}}_{l, \xi}^{R}$ 
\begin{align*}
  \frac{w_m(t,\xi)}{w_k(t,\eta)}= \left(\frac{|\eta|}{|\xi|}\right)^{(j-1)(1-3\beta)(1+2C_1\kappa)}\left(\frac{|\eta|^{1-3\beta}}{|j|^{2-3\beta}}\right)^{1+2C_1\kappa} \left(\frac{|l|^{2-3\beta}}{|\xi|^{1-3\beta}}\left[1+a_{l,\xi}\left|t-\frac{\xi}{l}\right|\right]\right)^{1+C_1\kappa}\ge1,
\end{align*}
and for  $t\in \tilde{\mathrm{I}}_{l, \xi}^{L}$ that
\begin{align*}
  \frac{w_m(t,\xi)}{w_k(t,\eta)}=& \left(\frac{|\eta|}{|\xi|}\right)^{(j-1)(1-3\beta)(1+2C_1\kappa)}\left(\frac{|\eta|^{1-3\beta}}{|j|^{2-3\beta}}\right)^{1+2C_1\kappa} \left(\frac{|l|^{2-3\beta}}{|\xi|^{1-3\beta}}\right)^{1+C_1\kappa}\left(\left[1+a_{l,\xi}\left|t-\frac{\xi}{l}\right|\right]\right)^{-C_1\kappa}\\
  \ge& \left(\frac{|\eta|}{|\xi|}\right)^{(j-1)(1-3\beta)(1+2C_1\kappa)}\left(\frac{|\eta|^{1-3\beta}}{|j|^{2-3\beta}}\right)^{1+2C_1\kappa} \left(\frac{|l|^{2-3\beta}}{|\xi|^{1-3\beta}}\right)^{1+2C_1\kappa}\ge1.
\end{align*}

Therefore, it holds that
\begin{align*}
  1\le \frac{w_m(t,\xi)}{w_k(t,\eta)}=\frac{w_{\mathrm{R}}(t,\xi)}{w_k(t,\eta)}\le \frac{w_{\mathrm{NR}}(t,\xi)}{w_k(t,\eta)},
\end{align*}
and
\begin{align*}
  \left|\frac{w_{\mathrm{R}}(t,\xi)}{w_k(t,\eta)}-1\right|\le\left|\frac{w_{\mathrm{NR}}(t,\xi)}{w_k(t,\eta)}-1\right|.
\end{align*}
As a result, the estimate \eqref{eq-est-3.1.4.2} also holds for the case $l=m$.}

Case 3.1.5, $j=l$, $t\in \tilde{\mathrm{I}}_{j,\eta}\cap\tilde{\mathrm{I}}_{l,\xi}$.

Case 3.1.5.1, $|\xi-\eta|\le \frac{1}{20}\left(\frac{|\xi|}{|l|}\right)^{1-3\beta}$. For all the case that $j=k$ or $j\neq k$ and $l=m$ or $l\neq m$, we have
\begin{align*}
  \left|\frac{w_m(t,\xi)}{w_k(t,\eta)}-1\right|\lesssim \frac{\frac{|\eta|^{1-3\beta}}{|j|^{2-3\beta}}}{\left[1+a_{j,\eta}\left|t-\frac{\eta}{j}\right|\right]}\langle \xi-\eta\rangle^{1+2C_1\kappa}e^{C\mu |\xi-\eta|^s}.
\end{align*}

If $\frac{1}{2}|\eta|^s<|j|\le E(|\eta|^s)$, then $\frac{|\eta|^{1-3\beta}}{|j|^{2-3\beta}}\approx1$. Recall that $t\in \tilde{\mathrm{I}}_{j,\eta}\cap\tilde{\mathrm{I}}_{l,\xi}$ and $\left|t-\frac{\eta}{j}\right|\lesssim \frac{|\eta|^{1-3\beta}}{|j|^{2-3\beta}}$ and $\left|t-\frac{\xi}{l}\right|\lesssim \frac{|\xi|^{1-3\beta}}{|l|^{2-3\beta}}$. Then similar to Case 3.1.3.1, we have
\begin{align*}
  \left|\frac{w_m(t,\xi)}{w_k(t,\eta)}-1\right|\frac{|\xi|}{t^2}\lesssim \langle \xi-\eta\rangle^{1+2C_1\kappa}e^{C\mu |\xi-\eta|^s}\sqrt{\frac{\pa_t g(t,\xi)}{g(t,\xi)}}\sqrt{\frac{\pa_t g(t,\eta)}{g(t,\eta)}}\frac{1}{t^{3\beta}}.
\end{align*}

If $|j|\le \frac{1}{2} |\eta|^s$, then $a_{j,\eta}\approx a_{l,\xi}\approx 1$. Thus we have
\begin{align*}
  \frac{\frac{|\eta|^{1-3\beta}}{|j|^{2-3\beta}}}{\left[1+\left|t-\frac{\eta}{j}\right|\right]}\lesssim&\sqrt{\frac{1}{\left[1+\left|t-\frac{\eta}{j}\right|\right]}}\sqrt{\frac{1}{\left[1+\left|t-\frac{\xi}{l}\right|\right]}}\langle\xi-\eta\rangle\frac{|\eta|^{1-3\beta}}{|j|^{2-3\beta}}\\
  \lesssim&\sqrt{\frac{\pa_t w(t,\xi)}{w(t,\xi)}}\sqrt{\frac{\pa_t w(t,\eta)}{w(t,\eta)}}\langle\xi-\eta\rangle\frac{|\eta|^{1-3\beta}}{|j|^{2-3\beta}},
\end{align*}
and
\begin{align*}
  \left|\frac{w_m(t,\xi)}{w_k(t,\eta)}-1\right|\frac{|\xi|}{t^2}\lesssim \langle \xi-\eta\rangle^{1+2C_1\kappa}e^{C\mu |\xi-\eta|^s}\sqrt{\frac{\pa_t w(t,\xi)}{w(t,\xi)}}\sqrt{\frac{\pa_t w(t,\eta)}{w(t,\eta)}}\frac{1}{t^{3\beta}}.
\end{align*}

Case 3.1.5.2, $|\xi-\eta|\ge \frac{1}{20}\left(\frac{|\xi|}{|l|}\right)^{1-3\beta}$. The most troublesome case is $j=k$, thus $w_k(t,\eta)=w_{\mathrm{R}}(t,\eta)$. If $l\neq m$, we have
\begin{align}\label{eq-est-3.1.5.2}
  \frac{w_m(t,\xi)}{w_k(t,\eta)}\le \left(\frac{|\eta|}{|\xi|}\right)^{(j-1)(1-3\beta)(1+2C_1\kappa)}\left(\frac{|\eta|^{1-3\beta}}{|j|^{2-3\beta}}\right)^{2+2C_1\kappa}.
\end{align}
Thus similar to Case 3.1.3.2, we have
\begin{align*}
  &\left|\frac{w_m(t,\xi)}{w_k(t,\eta)}-1\right|\frac{|\xi|}{t^2} \lesssim|\xi-\eta|^{2+2C_1\kappa}\frac{|\eta-\xi||\xi|^s}{t^2}+\frac{|\xi-\eta|^{2+2C_1\kappa}}{t}.
\end{align*}

If $l=m$, we still have \eqref{eq-est-3.1.5.2}, but in some special case $\frac{w_m(t,\xi)}{w_k(t,\eta)}\le 1$. However, if $\frac{w_m(t,\xi)}{w_k(t,\eta)}\le 1$, we also have
\begin{align*}
  \left|\frac{w_m(t,\xi)}{w_k(t,\eta)}-1\right|\frac{|\xi|}{t^2} \le \frac{|\xi|}{t^2} \le \frac{|\xi|^{1-3\beta}}{|l|^{2-3\beta}}\frac{|\xi|}{t^2}\lesssim \frac{|\xi-\eta|}{t}.
\end{align*}

Case 3.2, $|\xi|\ge|\eta|$, we have  $|l|\ge |j|$.

Case 3.2.1, $j=l$ and $t\in (\mathrm{I}_{j,\eta}^L\setminus\tilde{\mathrm{I}}_{j, \eta}^{L})\cap (\mathrm{I}_{l,\xi}^L\setminus\tilde{\mathrm{I}}_{l, \xi}^{L})$, or $t\in (\mathrm{I}_{j,\eta}^R\setminus\tilde{\mathrm{I}}_{j, \eta}^{R})\cap (\mathrm{I}_{l,\xi}^R\setminus\tilde{\mathrm{I}}_{l, \xi}^{R})$, or $|l|=|j|+1$ and $t\in (\mathrm{I}_{j,\eta}^L\setminus\tilde{\mathrm{I}}_{j, \eta}^{L}) \cap (\mathrm{I}_{l,\xi}^R\setminus\tilde{\mathrm{I}}_{l, \xi}^{R})$. We have
\begin{align*}
  \frac{w_m(t,\xi)}{w_k(t,\eta)}=\left(\frac{|\eta| }{|\xi| }\right)^{(j)(1-3\beta)(1+2C_1\kappa)},
\end{align*}
and there is nothing to prove.

If $|l|=|j|-1$, and not in Case 3.1.1. It is impossible. Thus we only need to consider $|l|=|j|$.

Case 3.2.2, $j=l$. Following the calculation in Case 3.1.2, 3.1.3, 3.1.4, and 3.1.5, we can get similar estimates. 

Case 4, $t\ge 2\min(|\xi|,|\eta|)$. From the assumption \eqref{eq-assum-W}, we have $t\ge \max(t_{1,|\eta|}^+,t_{1,|\xi|}^+)$, then $\left|\frac{w_m(t,\xi)}{w_k(t,\eta)}-1\right|=0$.

As a conclusion, we complete the proof for the case where $|\eta|\ge |k|$ and $|\xi|\ge |m|$.

If $|\eta|\le |k|$ and $|\xi|\ge |m|$, we also only need to consider
  \begin{align*}
    |k-m|+|\xi-\eta|\le \frac{1}{100}\left(|k|+|m|+|\eta|+|\xi|\right)^{1-s}\le\frac{1}{100}\left(|k|+|m|+|\eta|+|\xi|\right).
  \end{align*}
Which means that
  \begin{align*}
    |k-m|+|\xi-\eta|\le\frac{1}{50}\left(|k|+|\xi|\right).
  \end{align*}
From \eqref{eq-comp-m-k-xi-eta}, it is clear that
\begin{align*}
  \big||k|-|\xi|\big|\le|k-m|+|\xi-\eta|\le\frac{1}{50}\left(|k|+|\xi|\right).
\end{align*}
We also have 
\begin{align*}
   (|k|^s-|\xi|^s)=s\int^{k}_{|\xi|}\frac{1}{x^{1-s}}dx\le \frac{s\big||k|-|\xi|\big|}{|\xi|^{1-s}}\le \frac{s}{20}<1.
\end{align*}

Thus $|E(|k|^s)-E(|\xi|^s)|\le 1$. And for $t\in \mathrm{I}_{l,\xi}\cap \mathrm{I}_{j,k}$ we also have $\big||l|-|j|\big|\le 1$. So the proof is the same as in the previous case. 

The cases ($|\eta|\ge |k|$ and $|\xi|\le |m|$) and ($|\eta|\le |k|$ and $|\xi|\le |m|$) can also be handled in the same way.
\end{proof}

\subsubsection{Properties of $g$ and $G$}
\begin{lemma}\label{lem-g-exp}
  For all $\xi,\eta\in \mathbb R$ and $t\ge 1$, it holds that
\begin{align}\label{eq-est-gg}
\frac{g(t, \xi)}{g(t, \eta)} \lesssim e^{C\tilde\mu|\eta-\xi|^{s}}.
\end{align}
\end{lemma}
\begin{proof}
  Without loss of generality, we assume that $|\xi|\le|\eta|$. It is enough to prove
  \begin{align*}
     e^{-C\tilde\mu|\eta-\xi|^{s}}\lesssim\frac{g(t, \xi)}{g(t, \eta)}\lesssim e^{C\tilde\mu|\eta-\xi|^{s}}.
  \end{align*}
   If $|\eta|< 1$, we have $g(t,\xi)=g(t,\eta)=1$, so there is nothing to prove. If $|\xi|<1\le|\eta|$, then we have $|\eta|\le 1+|\eta-\xi|$ and hence we have from \eqref{eq-est-g-exp}
  \begin{align*}
    1\lesssim\frac{g(t, \xi)}{g(t, \eta)}=\frac{1}{g(t, \eta)}\le \frac{1}{g(0, \eta)}\lesssim e^{\tilde\mu|\eta|^{s}}\lesssim e^{\tilde\mu|\eta-\xi|^{s}}.
   \end{align*} 
  If $|\xi|\le \frac{|\eta|}{2}$, then it holds that $|\eta|\le 2|\eta-\xi|$ and $|\xi|\le|\eta-\xi|$. Then we have
    \begin{align*}
     e^{-C \kappa|\eta-\xi|^{s}}\lesssim e^{-C \kappa|\xi|^{s}}\lesssim g(t, \xi)\le\frac{g(t, \xi)}{g(t, \eta)}\le\frac{1}{g(t, \eta)}\lesssim e^{\tilde\mu|\eta|^{s}}\lesssim e^{\tilde\mu|\eta-\xi|^{s}}.
   \end{align*} 
  So we only need to focus on the case $\min(|\xi|,|\eta|)> 1$ and $\frac{|\eta|}{2}\le |\xi|\le |\eta|$. 

We first introduce
\begin{align*}
  F(k,\eta)=& \arctan \left(\left(\frac{|\eta|^{1-3\beta}}{|k|^{2-3\beta}}\right)^{-1}\frac{|\eta|}{(2|k|-1)|k|}\right)+\arctan \left(\left(\frac{|\eta|^{1-3\beta}}{|k|^{2-3\beta}}\right)^{-1}\frac{|\eta|}{(2|k|+1)|k|}\right),\\
  \tilde F(k,\eta)=&\mathfrak g(\nu,\eta)\langle \nu^{\frac{1}{3}}\frac{\eta}{k}\rangle^{-\frac{3}{2}}\nu^\beta\frac{|\eta|}{|k|^2}\left[\arctan \left(\frac{|\eta|}{(2|k|-1)|k|}\right)+\arctan \left(\frac{|\eta|}{(2|k|+1)|k|}\right)\right].
\end{align*}
  We consider the following 6 cases:
  \begin{itemize}
    \item[Case 1:] $t\le \min(t_{E(|\xi|),|\xi|},t_{E(|\eta|),|\eta|})$;
    \item[Case 2:] $\min(t_{E(|\xi|),|\xi|},t_{E(|\eta|),|\eta|})\le t\le \max(t_{E(|\xi|),|\xi|},t_{E(|\eta|),|\eta|})$;
    \item[Case 3:] $\max(t_{E(|\xi|),|\xi|},t_{E(|\eta|),|\eta|})\le t\le \min(t_{E(|\xi|^s),|\xi|},t_{E(|\eta|^s),|\eta|})$;
    \item[Case 4:] $\min(t_{E(|\xi|^s),|\xi|},t_{E(|\eta|^s),|\eta|})\le t\le \max(t_{E(|\xi|^s),|\xi|},t_{E(|\eta|^s),|\eta|})$;
    \item[Case 5:] $\max(t_{E(|\xi|^s),|\xi|},t_{E(|\eta|^s),|\eta|})\le t\le 2|\xi|$;
    \item[Case 6:] $2|\xi|\le t\le 2|\eta|$.
  \end{itemize}

For Case 1, $t\le \min(t_{E(|\xi|),|\xi|},t_{E(|\eta|),|\eta|})$. Under the assumption of Case 1, we have
\begin{align*}
  \frac{g(t, \xi)}{g(t, \eta)}=\frac{g(0, \xi)}{g(0,\eta)}=\exp \left(\mathcal G_1(\eta)-\mathcal G_1(\xi)\right)\exp \left(\mathcal G_2(\eta)-\mathcal G_2(\xi)\right),
\end{align*}
with
\begin{align*}
  \mathcal G_1(\eta)=\kappa\sum_{k=1}^{E(|\eta|^s)}F(k,\eta),\text{ and }\mathcal G_2(\eta)=\kappa\sum_{E(|\eta|^s)+1}^{E(|\eta|)}\tilde F(k,\eta),
\end{align*}
where $F(k,\eta)$ and $\tilde F(k,\eta)$ are given in Lemma \ref{lem-g-growth}.

It is clear that
\begin{align*}
  &\exp \left(\mathcal G_1(\eta)-\mathcal G_1(\xi)\right)\\
  \le&\exp \left(\kappa \sum_{k=E(|\xi|^s)+1}^{E(|\eta|^s)} F(k,\eta)\right)\exp \left(\kappa \sum_{k=1}^{E(|\xi|^s)} \big(F(k,\eta)-F(k,\xi)\big)\right).
\end{align*}

From the definition of $F(k,\eta)$, we have
\begin{align*}
  0\le& F(k,\eta)-F(k,\xi)\\
  \le&\left|\arctan \left(\left(\frac{|\xi|^{1-3\beta}}{|k|^{2-3\beta}}\right)^{-1}\frac{|\xi|}{(2|k|-1)|k|}\right)-\arctan \left(\left(\frac{|\eta|^{1-3\beta}}{|k|^{2-3\beta}}\right)^{-1}\frac{|\eta|}{(2|k|-1)|k|}\right)\right|\\
  &+\left|\arctan \left(\left(\frac{|\xi|^{1-3\beta}}{|k|^{2-3\beta}}\right)^{-1}\frac{|\xi|}{(2|k|+1)|k|}\right)-\arctan \left(\left(\frac{|\eta|^{1-3\beta}}{|k|^{2-3\beta}}\right)^{-1}\frac{|\eta|}{(2|k|+1)|k|}\right)\right|\\
  \lesssim&\beta \frac{|k|^{3\beta}|\xi-\eta|}{|\xi|^{1+3\beta}},
\end{align*}
and then
\begin{align}\label{eq-est-gg-s}
  \sum_{k=1}^{E(|\xi|^s)} |F(k,\xi)-F(k,\eta)|\lesssim  \frac{\beta |\xi-\eta|}{|\xi|^{(1-s)(1+3\beta)}}\le \beta |\xi-\eta|^{s}|\xi|^{-3\beta(1-s)}\le \beta |\xi-\eta|^{s}.
\end{align}
Therefore, by using \eqref{inq-s2}, we have 
\begin{align*}
  \exp \left(\mathcal G_1(\eta)-\mathcal G_1(\xi)\right)\lesssim \exp \left(C\kappa(|\eta|^{s}-|\xi|^s)+C\kappa\beta|\xi-\eta|^s\right)\lesssim\exp \left(C\kappa |\xi-\eta|^s\right).
\end{align*}
Similarly, we have
\begin{align*}
  &\exp \left(\mathcal G_2(\eta)-\mathcal G_2(\xi)\right)\\
  \le&\exp \left(-\kappa \sum_{k=E(|\xi|^s)+1}^{E(|\eta|^s)} \tilde F(k,\xi)\right)\exp \left(\kappa \sum_{k=E(|\eta|^s)+1}^{E(|\xi|)} \big(\tilde F(k,\eta)-\tilde F(k,\xi)\big)\right)\exp \left(\kappa \sum_{k=E(|\xi|)+1}^{E(|\eta|)} \tilde F(k,\eta)\right).
\end{align*}
Be careful that we do not assume $E(|\xi|)\ge E(|\eta|^s)$.

We write
\begin{align*}
  &\left|\tilde F(k,\eta)-\tilde F(k,\xi)\right|\\
  \le&\left|\mathfrak g(\nu,\xi)\langle \nu^{\frac{1}{3}}\frac{\xi}{k}\rangle^{-\frac{3}{2}}\nu^\beta\frac{|\xi|}{|k|^2}\arctan \left(\frac{|\xi|}{(2|k|-1)|k|}\right)-\mathfrak g(\nu,\eta)\langle \nu^{\frac{1}{3}}\frac{\eta}{k}\rangle^{-\frac{3}{2}}\nu^\beta\frac{|\eta|}{|k|^2}\arctan \left(\frac{|\eta|}{(2|k|-1)|k|}\right)\right|\\
  &+\left|\mathfrak g(\nu,\xi)\langle \nu^{\frac{1}{3}}\frac{\xi}{k}\rangle^{-\frac{3}{2}}\nu^\beta\frac{|\xi|}{|k|^2}\arctan \left(\frac{|\xi|}{(2|k|+1)|k|}\right)-\mathfrak g(\nu,\eta)\langle \nu^{\frac{1}{3}}\frac{\eta}{k}\rangle^{-\frac{3}{2}}\nu^\beta\frac{|\eta|}{|k|^2}\arctan \left(\frac{|\eta|}{(2|k|+1)|k|}\right)\right|.
\end{align*}
It is clear that
\begin{align*}
  \mathfrak g(\nu,\xi)-\mathfrak g(\nu,\eta)&=\langle\nu^{-\frac{1}{3}}|\xi|^{s-1}\rangle^{3\beta}-\langle\nu^{-\frac{1}{3}}|\eta|^{s-1}\rangle^{3\beta}\\
  &\approx 3\beta \frac{\nu^{-\frac{1}{3}}(1-s)|\eta|^{s-2}|\xi-\eta|}{\left(1+\nu^{-\frac{1}{3}}|\eta|^{s-1}\right)^{1-3\beta}}
  \lesssim\left(\nu^{-\frac{1}{3}}|\eta|^{s-1}\right)^{3\beta}\frac{|\xi-\eta|}{|\eta|},\\
  \langle \nu^{\frac{1}{3}}\frac{\xi}{k}\rangle^{-\frac{3}{2}}-\langle \nu^{\frac{1}{3}}\frac{\eta}{k}\rangle^{-\frac{3}{2}}&\lesssim(1+\nu^{\frac{1}{3}}\frac{|\xi|}{|k|})^{-\frac{5}{2}}\nu^{\frac{1}{3}}\frac{|\xi-\eta|}{|k|},
  \end{align*}
 and
 \begin{align*}
  \arctan \left(\frac{|\xi|}{(2|k|-1)|k|}\right)-\arctan \left(\frac{|\eta|}{(2|k|-1)|k|}\right)
  \lesssim\frac{|k|^2|\xi-\eta|}{|k|^4+|\eta|^2}.
\end{align*}
Then we have
\begin{align*}
  &\left|\tilde F(k,\eta)-\tilde F(k,\xi)\right|\\
  \le&\langle \nu^{\frac{1}{3}}\frac{\eta}{k}\rangle^{-\frac{3}{2}}|\eta|^{-3\beta(1-s)}\frac{|\xi-\eta|}{|k|^2}+\langle\nu^{-\frac{1}{3}}|\eta|^{s-1}\rangle^{3\beta}(1+\nu^{\frac{1}{3}}\frac{|\xi|}{|k|})^{-\frac{5}{2}}\nu^{\frac{1}{3}+\beta}\frac{|\xi-\eta||\eta|}{|k|^3}\\
  &+\langle\nu^{-\frac{1}{3}}|\eta|^{s-1}\rangle^{3\beta}(1+\nu^{\frac{1}{3}}\frac{|\xi|}{|k|})^{-\frac{3}{2}}\nu^{\beta}\frac{|\xi-\eta|}{|k|^2}+\langle\nu^{-\frac{1}{3}}|\eta|^{s-1}\rangle^{3\beta}(1+\nu^{\frac{1}{3}}\frac{|\xi|}{|k|})^{-\frac{3}{2}}\nu^{\beta}\frac{|\eta||\xi-\eta|}{|k|^4+|\eta|^2}\\
  \lesssim&\frac{|\xi-\eta|^{2s}}{|k|^2}+\langle\nu^{-\frac{1}{3}}|\eta|^{s-1}\rangle^{3\beta}(1+\nu^{\frac{1}{3}}\frac{|\xi|}{|k|})^{-\frac{5}{2}}\nu^{\beta}\frac{|\xi-\eta|}{|k|^2}\\
  &+\langle\nu^{-\frac{1}{3}}|\eta|^{s-1}\rangle^{3\beta}(1+\nu^{\frac{1}{3}}\frac{|\xi|}{|k|})^{-\frac{3}{2}}\nu^{\beta}\frac{|\xi-\eta|}{|k|^2}+\langle\nu^{-\frac{1}{3}}|\eta|^{s-1}\rangle^{3\beta}(1+\nu^{\frac{1}{3}}\frac{|\xi|}{|k|})^{-\frac{3}{2}}\nu^{\beta}\frac{|\eta||\xi-\eta|}{|k|^4+|\eta|^2}.
\end{align*}
Here we use the fact that
\begin{align*}
  |\eta|^{-3\beta(1-s)}\frac{|\xi-\eta|}{|k|^2}=|\eta|^{2s-1}\frac{|\xi-\eta|}{|k|^2}\lesssim\frac{|\xi-\eta|^{2s}}{|k|^2}.
\end{align*}
Similar to \eqref{eq-sum-s-g}, we have
\begin{align}\label{eq-est-g-s-1-part}
  \sum_{k=E(|\eta|^s)+1}^{E(|\xi|)}\frac{|\xi-\eta|^{2s}}{|k|^2}+\langle\nu^{-\frac{1}{3}}|\eta|^{s-1}\rangle^{3\beta}(1+\nu^{\frac{1}{3}}\frac{|\xi|}{|k|})^{-\frac{5}{2}}\nu^{\beta}\frac{|\xi-\eta|}{|k|^2}\lesssim \frac{|\xi-\eta|}{|\eta|^{1-s}}\le |\xi-\eta|^s.
\end{align}
So we only need to focus on $\langle\nu^{-\frac{1}{3}}|\eta|^{s-1}\rangle^{3\beta}(1+\nu^{\frac{1}{3}}\frac{|\xi|}{|k|})^{-\frac{3}{2}}\nu^{\beta}\frac{|\eta||\xi-\eta|}{|k|^4+|\eta|^2}$. If $|\eta|^{1-s}\le \nu^{-\frac{1}{3}}$, we have
\begin{align*}
  &\sum_{k=E(|\eta|^s)+1}^{E(|\xi|)} \langle\nu^{-\frac{1}{3}}|\eta|^{s-1}\rangle^{3\beta}(1+\nu^{\frac{1}{3}}\frac{|\xi|}{|k|})^{-\frac{3}{2}}\nu^{\beta}\frac{|\eta||\xi-\eta|}{|k|^4+|\eta|^2}\\
  \approx&\sum_{k=E(|\eta|^s)+1}^{E(|\xi|)} \frac{|\eta|^{2s}|\xi-\eta|}{|k|^4+|\eta|^2} \\
  \lesssim&\int_{|\eta|^s}^{|\xi|^{\frac{1}{2}}} \frac{|\eta|^{2s}|\xi-\eta|}{x^4+|\eta|^2} dx+\int_{|\eta|^{\frac{1}{2}}}^{|\xi|}\frac{|\eta|^{2s}|\xi-\eta|}{x^4+|\eta|^2} dx\\
  \lesssim&\frac{|\xi-\eta|}{|\eta|^{\frac{3}{2}-2s}}\lesssim\frac{|\xi-\eta|^s}{|\eta|^{\frac{1}{2}-s}}\lesssim|\xi-\eta|^s.
\end{align*}
If $|\eta|^{1-s}> \nu^{-\frac{1}{3}}$, we have
\begin{align*}
  &\sum_{k=E(|\eta|^s)+1}^{E(|\xi|)} \langle\nu^{-\frac{1}{3}}|\eta|^{s-1}\rangle^{3\beta}(1+\nu^{\frac{1}{3}}\frac{|\xi|}{|k|})^{-\frac{3}{2}}\nu^{\beta}\frac{|\eta||\xi-\eta|}{|k|^4+|\eta|^2}\\
  \lesssim&\sum_{k=E(|\eta|^s)+1}^{E(|\xi|)} \frac{|\eta|^{1-3\beta}|x|^{3\beta}|\xi-\eta|}{|k|^4+|\eta|^2} \\
   \lesssim&\int_{|\eta|^s}^{|\xi|^{\frac{1}{2}}}\frac{|\eta|^{1-3\beta}|x|^{3\beta}|\xi-\eta|}{|x|^4+|\eta|^2}dx+\int_{|\eta|^{\frac{1}{2}}}^{|\xi|}\frac{|\eta|^{1-3\beta}|x|^{3\beta}|\xi-\eta|}{|x|^4+|\eta|^2}dx\\
   \lesssim& \int_{|\eta|^s}^{|\xi|^{\frac{1}{2}}}\frac{|\eta|^{1-3\beta}|x|^{3\beta}|\xi-\eta|}{|\xi|^2}dx+\int_{|\eta|^{\frac{1}{2}}}^{|\xi|}\frac{|\eta|^{1-3\beta}|x|^{3\beta}|\xi-\eta|}{|x|^4}dx\\
   \lesssim&  \frac{|\eta|^{1-3\beta}|\xi|^{\frac{1+3\beta}{2}}|\xi-\eta|}{|\xi|^2}+\frac{|\eta|^{1-3\beta} |\xi-\eta|}{|\xi|^{\frac{3-3\beta}{2}}}\lesssim  \frac{|\xi-\eta|}{|\xi|^{\frac{1+3\beta}{2}}}\lesssim |\xi-\eta|^s.
\end{align*}
Here we use \eqref{eq-sum-s-g} and the fact $\frac{1+3\beta}{2}=1-s+\frac{3\beta s}{2}$. 

Then we have 
\begin{align}\label{eq-est-gg-s1s}
  \sum_{k=E(|\eta|^s)+1}^{E(|\xi|)} |\tilde F(k,\xi)-\tilde F(k,\eta)|\lesssim \frac{|\xi-\eta|}{|\eta|^{1-s}} \lesssim|\xi-\eta|^s.
\end{align}

It follows from \eqref{inq-s2} that
\begin{align*}
  \sum_{k=E(|\xi|^s)+1}^{E(|\eta|^s)} \left|\tilde F(k,\xi)\right|\lesssim |\xi-\eta|^s.
\end{align*}
 If $|\eta|^{1-s}\le \nu^{-\frac{1}{3}}$, we have
 \begin{align}\label{eq-est-gg-1s-1}
   \sum_{k=E(|\xi|)+1}^{E(|\eta|)} \tilde F(k,\eta)\lesssim\sum_{k=E(|\xi|)+1}^{E(|\eta|)} \frac{|\eta|^{2s}}{|k|^{2}}\lesssim \frac{|\xi-\eta|}{|\eta|^{2-2s}}\lesssim |\xi-\eta|^s.
 \end{align}
 If $|\eta|^{1-s}> \nu^{-\frac{1}{3}}$, we have 
\begin{equation}\label{eq-est-gg-1s-2}
  \begin{aligned}  
 &\sum_{k=E(|\xi|)+1}^{E(|\eta|)} \tilde F(k,\eta)\lesssim\sum_{k=E(|\xi|)+1}^{E(|\eta|)} \frac{|\eta|^{1-3\beta}}{|k|^{2-3\beta}} \lesssim \frac{|\eta|^{1-3\beta}|\xi-\eta|}{|\xi|^{2-3\beta}}\lesssim \frac{|\xi-\eta|}{|\xi|}\lesssim |\xi-\eta|^s.   
  \end{aligned}
\end{equation}
Combing the above estimates, we have
\begin{align*}
 e^{-C\kappa|\xi-\eta|^s}\le\frac{g(t, \xi)}{g(t, \eta)}=\frac{g(0, \xi)}{g(0, \eta)}\le e^{C\kappa|\xi-\eta|^s}.
\end{align*}

For Case 2, $\min(t_{E(|\xi|),|\xi|},t_{E(|\eta|),|\eta|})\le t\le \max(t_{E(|\xi|),|\xi|},t_{E(|\eta|),|\eta|})$. 

If $t_{E(|\xi|),|\xi|}\le t\le t_{E(|\eta|),|\eta|}$, we have 
\begin{align*}
  g(t,\eta)=g(t_{E(|\eta|),|\eta|},\eta)=g(0,\eta),
\end{align*}
and
\begin{align*}
  g(t_{E(|\eta|),|\eta|},\xi)\ge g(t,\xi)\ge g(t_{E(|\xi|),|\xi|},\eta)=g(0,\xi).
\end{align*}
Then, we have
\begin{align*}
  e^{-C\kappa|\xi-\eta|^s}\le\frac{g(0, \xi)}{g(0, \eta)}\le\frac{g(t, \xi)}{g(t, \eta)}\le \frac{g(t_{E(|\eta|),|\eta|}, \xi)}{g(t_{E(|\eta|),|\eta|}, \eta)}.
\end{align*}
Similarly, if $t_{E(|\eta|),|\eta|}\le t\le t_{E(|\xi|),|\xi|}$ we have
\begin{align*}
  g(t,\xi)=g(t_{E(|\xi|),|\xi|},\xi)=g(0,\eta),
\end{align*}
and
\begin{align*}
  g(t_{E(|\xi|),|\xi|},\eta)\ge g(t,\eta)\ge g(t_{E(|\eta|),|\eta|},\eta)=g(0,\eta).
\end{align*}
\begin{align*}
  e^{C\kappa|\xi-\eta|^s}\ge\frac{g(0, \xi)}{g(0, \eta)}\ge\frac{g(t, \xi)}{g(t, \eta)}\ge \frac{g(t_{E(|\xi|),|\xi|}, \xi)}{g(t_{E(|\xi|),|\xi|}, \eta)}.
\end{align*}
For both cases, one side has been given in Case 1, and the other side will be given in Case 3.

For Case 3, $\max(t_{E(|\xi|),|\xi|},t_{E(|\eta|),|\eta|})\le t\le \min(t_{E(|\xi|^s),|\xi|},t_{E(|\eta|^s),|\eta|})$. During this time period, there must be some $j,l$ such that $t\in \mathrm{I}_{l,\xi}\cap \mathrm{I}_{j,\eta}$ and 
\begin{align*}
  E(|\xi|^s)+1\le |l|\le E(|\xi|),\quad E(|\eta|^s)+1\le |j|\le E(|\eta|).
\end{align*}
As $|\xi|\le|\eta|$, it is clear that $E(|\xi|^s)\le E(|\eta|^s)$.

We have
\begin{align*}
  g(t,\eta)=&\exp \left(-\kappa \sum_{k=1}^{E(|\xi|^s)}F(k,\eta)-\kappa \sum_{k=E(|\xi|^s)+1}^{E(|\eta|^s)} F(k,\eta)-\kappa \sum_{k=E(|\eta|^s)+1}^{j-1}\tilde F(k,\eta)\right)\\
  &\times e^{\kappa \mathfrak g(\nu,\eta)\langle \nu^{\frac{1}{3}}\frac{\eta}{j}\rangle^{-\frac{3}{2}}\nu^\beta\frac{|\eta|}{|j|^2}\left[\arctan \left(t-\frac{|\eta|}{|j|}\right)-\arctan \left(\frac{|\eta|}{(2|j|-1)|j|}\right)\right]}.
\end{align*}
If $l\le E(|\eta|^s)$, we have
\begin{align*}
  g(t,\xi)=&\exp \left(-\kappa \sum_{k=1}^{E(|\xi|^s)}F(k,\xi)\right) \exp \left(-\kappa \sum_{k=E(|\xi|^s)+1}^{l-1}\tilde F(k,\xi)\right)\\
  &\times e^{\kappa\mathfrak g(\nu,\xi) \langle \nu^{\frac{1}{3}}\frac{\xi}{l}\rangle^{-\frac{3}{2}}\nu^\beta\frac{|\xi|}{|l|^2}\left[\arctan \left(t-\frac{|\xi|}{|l|}\right)-\arctan \left(\frac{|\xi|}{(2|l|-1)|l|}\right)\right]}
\end{align*}
and if $l\ge E(|\eta|^s)+1$, we have
\begin{align*}
  g(t,\xi)=&\exp \left(-\kappa \sum_{k=1}^{E(|\xi|^s)}F(k,\xi)-\kappa \sum_{k=E(|\xi|^s)+1}^{E(|\eta|^s)}\tilde F(k,\xi)-\kappa \sum_{k=E(|\eta|^s)+1}^{l-1}\tilde F(k,\xi)\right)\\
  &\times e^{\kappa \mathfrak g(\nu,\xi)\langle \nu^{\frac{1}{3}}\frac{\xi}{l}\rangle^{-\frac{3}{2}}\nu^\beta\frac{|\xi|}{|l|^2}\left[\arctan \left(t-\frac{|\xi|}{|l|}\right)-\arctan \left(\frac{|\xi|}{(2|l|-1)|l|}\right)\right]}.
\end{align*}
Recall that $|\xi|\le|\eta|$. For $k\le E(|\xi|^s)$, we have 
\begin{align*}
  F(k,\xi)\le F(k,\eta),
\end{align*}
For $E(|\xi|^s)+1\le k\le E(|\eta|^s)$, we have
\begin{align*}
  \tilde F(k,\xi)\le F(k,\eta).
\end{align*}
Indeed, it holds that
\begin{align*}
  \frac{|\xi|}{|k|^{2-3\beta}}\le |\xi|^{3\beta}\le|\eta|^{3\beta}
\end{align*}
and then
\begin{align*}
  \frac{|\xi|}{(2|k|+1)|k|}\le \left(\frac{|\eta|^{1-3\beta}}{|k|^{2-3\beta}}\right)^{-1}\frac{|\eta|}{(2|k|+1)|k|}
\end{align*}
and then
\begin{align*}
   \mathfrak g(\nu,\xi)\langle \nu^{\frac{1}{3}}\frac{\xi}{k}\rangle^{-\frac{3}{2}}\nu^\beta\frac{|\xi|}{|k|^2}\arctan \left(\frac{|\xi|}{(2|k|+1)|k|}\right)\le \arctan \left(\left(\frac{|\eta|^{1-3\beta}}{|k|^{2-3\beta}}\right)^{-1}\frac{|\eta|}{(2|k|+1)|k|}\right)
\end{align*}
For $k\ge E(|\eta|^s)+1$, we do not know $\tilde F(k,\xi)$ and $\tilde F(k,\eta)$ which is bigger.  

For the upper bound estimate. We have
\begin{align*}
  \frac{ g(t,\xi)}{ g(t,\eta)}
  \lesssim&\exp \left(\kappa \sum_{k=1}^{E(|\xi|^s)}\left(F(k,\eta)-F(k,\xi)\right)\right)\exp \left(\kappa \sum_{k=E(|\xi|^s)+1}^{E(|\eta|^s)} \left(F(k,\eta)-\tilde F(k,\xi)\right)\right)\\
  &\times\exp \left(\kappa \sum_{k=E(|\eta|^s)+1}^{l-1}\left|\tilde F(k,\eta)-\tilde F(k,\xi)\right|\right)\exp \left(\kappa \sum_{k=l}^{j}\tilde F(k,\eta)\right)\\
  \lesssim& \exp \left(C\kappa\beta |\xi-\eta|^s\right)\exp \left(C \kappa \left(E(|\eta|^s)-E(|\xi|^s)\right)\right)\exp \left(C \kappa |\xi-\eta|^s\right)\exp \left(C \kappa |l-j|\frac{|\eta|^{1-3\beta}}{|l|^{2-3\beta}}\right)
\end{align*}
Here the problem is to estimate $\sum_{k=l}^{j}\tilde F(k,\eta)$. If $|\eta|^{1-s}\le \nu^{-\frac{1}{3}}$, we have
\begin{align*}
  \sum_{k=l}^{j}\tilde F(k,\eta)\lesssim (l-j+1)\frac{|\eta|^{1-3\beta}}{|l|^{2-3\beta}}.
\end{align*}
If $|\eta|^{1-s}< \nu^{-\frac{1}{3}}$, we have
\begin{align*}
  \sum_{k=l}^{j}\tilde F(k,\eta)\lesssim (l-j+1)\frac{|\eta|^{2s}}{|l|^{2}}.
\end{align*}

By the fact that $|\xi-\eta|\le |\xi|\le|\eta|$, $|l|\gtrsim |\eta|^s$, and $t\approx\frac{\eta}{l}$ we have
\begin{align*}
  |l-j|\le \frac{|lt-\xi|+|\xi-\eta|+|\eta-jt|}{t}\lesssim 1+\frac{|\xi-\eta|}{t}.
\end{align*}
Here we use the fact that
\begin{align*}
  \frac{|lt-\xi|}{t}=\frac{l|t-\frac{\xi}{l}|}{t}\le \frac{l \frac{\xi}{l^2}}{t}\approx \frac{\xi}{jt}\approx 1.
\end{align*}
Then we have
\begin{align*}
  |l-j+1|\frac{|\eta|^{1-3\beta}}{|l|^{2-3\beta}}\lesssim 1+\frac{|\xi-\eta|}{|\eta|^{3\beta}|j|^{1-3\beta}}\lesssim1+\frac{|\xi-\eta|}{|\eta|^{1-s}}\lesssim1+|\xi-\eta|^{(1-s)(1-3\beta)}=1+|\xi-\eta|^s,
\end{align*}
\begin{align*}
  |l-j|\frac{|\eta|^{2s}}{|l|^{2}}\lesssim 1+\frac{|\eta|^s}{|l|}\frac{|\xi-\eta|}{|\eta|^{1-s}}\lesssim 1+|\xi-\eta|^s.
\end{align*}
Which shows that
\begin{align*}
  \frac{ g(t,\xi)}{ g(t,\eta)}\lesssim e^{C\kappa|\xi-\eta|^s}.
\end{align*}
From the analysis, it is clear that
\begin{align*}
  e^{-C\kappa|\xi-\eta|^s}\le\frac{ g(t,\xi)}{ g(t,\eta)}.
\end{align*}

For Case 4, $\min(t_{E(|\xi|^s),|\xi|},t_{E(|\eta|^s),|\eta|})\le t\le \max(t_{E(|\xi|^s),|\xi|},t_{E(|\eta|^s),|\eta|})$. The proof is similar to Case 2.
We know that
\begin{align*}
  g(t,\xi)\le g(\max(t_{E(|\xi|^s),|\xi|},t_{E(|\eta|^s),|\eta|}),\xi),\\
  g(t,\eta)\ge g(\min(t_{E(|\xi|^s),|\xi|},t_{E(|\eta|^s),|\eta|}),\eta), 
\end{align*}
then we have
\begin{align*}
  \frac{g(t,\xi)}{g(t,\eta)}\le& \frac{g(\max(t_{E(|\xi|^s),|\xi|},t_{E(|\eta|^s),|\eta|}),\xi)}{g(\min(t_{E(|\xi|^s),|\xi|},t_{E(|\eta|^s),|\eta|}),\eta)}\\
  =&\frac{g(\max(t_{E(|\xi|^s),|\xi|},t_{E(|\eta|^s),|\eta|}),\xi)}{g(\min(t_{E(|\xi|^s),|\xi|},t_{E(|\eta|^s),|\eta|}),\xi)}\frac{g(\min(t_{E(|\xi|^s),|\xi|},t_{E(|\eta|^s),|\eta|}),\xi)}{g(\min(t_{E(|\xi|^s),|\xi|},t_{E(|\eta|^s),|\eta|}),\eta)}.
\end{align*}
It is clear that
\begin{align*}
  \frac{g(\max(t_{E(|\xi|^s),|\xi|},t_{E(|\eta|^s),|\eta|}),\xi)}{g(\min(t_{E(|\xi|^s),|\xi|},t_{E(|\eta|^s),|\eta|}),\xi)}\le e^{C\kappa|\xi-\eta|^s},
\end{align*}
so it follows from the result of Case 3 that
\begin{align*}
  \frac{g(t,\xi)}{g(t,\eta)}\le e^{C\kappa   \left(|\eta|^s-|\xi|^s\right) }\le e^{C\kappa |\xi-\eta|^s}.
\end{align*}
For the same reason, one can deduce that
\begin{align*}
  e^{-C\kappa |\xi-\eta|^s}\le\frac{g(t,\xi)}{g(t,\eta)}.
\end{align*}

For Case 5, $\max(t_{E(|\xi|^s),|\xi|},t_{E(|\eta|^s),|\eta|})\le t\le 2|\xi|$. For this case, $t\gtrsim |\eta|^{1-s}$. The proof is similar to Case 3. 
During this time period, there must be some $j,l$ such that $t\in \mathrm{I}_{l,\xi}\cap \mathrm{I}_{j,\eta}$ and 
\begin{align*}
  1\le |l|\le E(|\xi|^s),\quad 1\le |j|\le E(|\eta|^s).
\end{align*}
As $|\xi|\le|\eta|$, it is clear that $E(|\xi|^s)\le E(|\eta|^s)$ and $l\le j$.

We have
\begin{align*}
  g(t,\eta)=&\exp \left(-\kappa \sum_{k=1}^{l-1}F(k,\eta)\right)\exp \left(-\kappa \sum_{k=l}^{j-1} F(k,\eta)\right)\\
  &\times e^{\kappa \left[\arctan \left(\left(\frac{|\eta|^{1-3\beta}}{|j|^{2-3\beta}}\right)^{-1}(t-\frac{|\eta|}{|j|})\right)-\arctan \left(\left(\frac{|\eta|^{1-3\beta}}{|j|^{2-3\beta}}\right)^{-1}\frac{|\eta|}{(2|j|-1)|j|}\right)\right]}.
\end{align*}
and
\begin{align*}
  g(t,\xi)=&\exp \left(-\kappa \sum_{k=1}^{l-1}F(k,\xi)\right) e^{\kappa \left[\arctan \left(\left(\frac{|\xi|^{1-3\beta}}{|l|^{2-3\beta}}\right)^{-1}(t-\frac{|\xi|}{|l|})\right)-\arctan \left(\left(\frac{|\xi|^{1-3\beta}}{|l|^{2-3\beta}}\right)^{-1}\frac{|\xi|}{(2|l|-1)|l|}\right)\right]}.
\end{align*}
Then we have
\begin{align*}
  \frac{ g(t,\xi)}{ g(t,\eta)}
  \lesssim&\exp \left(\kappa \sum_{k=1}^{l-1}\left(F(k,\eta)-F(k,\xi)\right)\right)\exp \left(\kappa \sum_{k=l}^{j}F(k,\eta)\right)\\
  \lesssim& \exp \left(C\kappa\beta |\xi-\eta|^s\right)\exp \left(C\kappa |l-j|\right).
\end{align*}
We have
\begin{align*}
  |l-j|\le \frac{|tl-\xi|+|\xi-\eta|+|\eta-tj|}{t}\lesssim 1+\frac{|\xi-\eta|}{t}\lesssim |\xi-\eta|^s.
\end{align*}
Then it holds that
\begin{align*}
   \frac{ g(t,\xi)}{ g(t,\eta)}\lesssim e^{C\kappa|\xi-\eta|^s} .
\end{align*}

For Case 6, $2|\xi|\le t\le 2|\eta|$, it is really similar to Case 2.

We know that
\begin{align*}
  g(2|\xi|,\eta)\le g(t,\eta)\le g(2|\eta|,\eta)= 1,\quad g(t,\xi)=g(2|\xi|,\xi)=1.
\end{align*}
Then we have
\begin{align*}
  1=\frac{ g(2|\eta|,\xi)}{ g(2|\eta|,\eta)}\le\frac{ g(t,\xi)}{ g(t,\eta)}\le \frac{ g(2|\xi|,\xi)}{ g(2|\xi|,\eta)}\lesssim e^{C\kappa|\xi-\eta|^s}.
\end{align*}
\end{proof}
\begin{lemma}\label{lem-g-gl}
For $t\in\mathrm{I}_{k,\eta}\cap \mathrm{I}_{l,\xi}$ such that $|k|\le |\eta|$ and $\frac{1}{\al}|\xi|\le|\eta|\le\al|\xi|$ for some $\al\ge1$ we have
\begin{align*}
  \sqrt{\frac{\mathfrak w(\nu,t,\eta)\pa_t g(t,\eta)}{g(t,\eta)}}\lesssim_{\alpha} \left(\sqrt{\frac{\mathfrak w(\nu,t,\xi)\pa_t g(t,\xi)}{g(t,\xi)}}+t^{-c_1}|\xi|^{\frac{s}{2}}\right)\langle \xi-\eta\rangle.
\end{align*}
\end{lemma}
\begin{proof}
Case 1, $t\ge2|\xi|$. In this case, we have $t\approx |\xi|\approx |\eta|$, $|k|\approx|l|\approx1$, and
\begin{align*}
  \sqrt{\frac{\mathfrak w(\nu,t,\eta)\pa_t g(t,\eta)}{g(t,\eta)}}\approx\sqrt{ \frac{\mathfrak w(\nu,t,\eta)|\eta|^{1-3\beta}}{\left(|\eta|^{1-3\beta}\right)^2+(t-\frac{|\eta|}{|k|})^2}}.
\end{align*}
If $\left|\frac{|\eta|}{|k|}-t\right|\le \frac{1}{2}\frac{|\eta| }{|k|^{2 }}$, thus $t\le |\eta|+\frac{1}{2}|\eta|$, we have 
\begin{align*}
  \langle\xi-\eta\rangle\gtrsim |\xi|\gtrsim t,
\end{align*}
and then
\begin{align*}
 \sqrt{ \frac{\mathfrak w(\nu,t,\eta)|\eta|^{1-3\beta}}{\left(|\eta|^{1-3\beta}\right)^2+(t-\frac{|\eta|}{|k|})^2}}\lesssim 1\lesssim t^{-c_1}|\xi|^{\frac{s}{2}}\langle\xi-\eta\rangle.
\end{align*}
If $\left|\frac{|\eta|}{|k|}-t\right|> \frac{1}{2}\frac{|\eta| }{|k|^{2 }}$, we have 
\begin{align*}
 \sqrt{ \frac{\mathfrak w(\nu,t,\eta)|\eta|^{1-3\beta}}{\left(|\eta|^{1-3\beta}\right)^2+(t-\frac{|\eta|}{|k|})^2}}\lesssim |\eta|^{-\frac{1}{2}(1+3\beta)}\approx |\xi|^{\frac{s}{2}} |\eta|^{-\frac{1}{2}(1+3\beta+s)} \lesssim t^{-c_1}|\xi|^{\frac{s}{2}}\langle\xi-\eta\rangle.
\end{align*}

Case 2, $2|\xi|^{1-s}\le t\le 2|\xi|$, if $t\in \left(\mathrm{I}_{l,\xi}\setminus\tilde{\mathrm{I}}_{l,\xi}\right)$, we deduce from Lemma \ref{lem-separate} that
  \begin{align*}
    \sqrt{\frac{\mathfrak w(\nu,t,\eta)\pa_t g(t,\eta)}{g(t,\eta)}}\sqrt{\frac{g(t,\xi)}{\mathfrak w(\nu,t,\xi)\pa_t g(t,\xi)}}
    &\approx \sqrt{\frac{\left(\frac{|\xi|^{1-3\beta}}{|l|^{2-3\beta}}\right)^2+\left|\frac{\xi}{l}-t\right|^2}{\left(\frac{|\eta|^{1-3\beta}}{|k|^{2-3\beta}}\right)^2+\left|\frac{\eta}{k}-t\right|^2}}\\
    &\lesssim\sqrt{\frac{\left(1+\left|\frac{\xi}{l}-t\right|^2\right)}{\left(1+\left|\frac{\eta}{k}-t\right|^2\right)}} \lesssim\langle \xi-\eta\rangle.
  \end{align*}

If $t\in \tilde{\mathrm{I}}_{l,\xi}$, then we have
  \begin{align*}
    \sqrt{\frac{\mathfrak w(\nu,t,\eta)\pa_t g(t,\eta)}{g(t,\eta)}}\sqrt{\frac{g(t,\xi)}{\mathfrak w(\nu,t,\xi)\pa_t g(t,\xi)}}\approx \sqrt{\frac{\left(\frac{|\xi|^{1-3\beta}}{|l|^{2-3\beta}}\right)^2}{\left(\frac{|\eta|^{1-3\beta}}{|k|^{2-3\beta}}\right)^2+\left|\frac{\eta}{k}-t\right|^2}}\lesssim1\lesssim\langle \xi-\eta\rangle.
  \end{align*}

Case 3, $1\le t\le 2|\xi|^{1-s}$. If $|l|\le |\xi|$, it is easy to show that
  \begin{align*}
    \sqrt{\frac{\mathfrak w(\nu,t,\eta)\pa_t g(t,\eta)}{g(t,\eta)}}\sqrt{\frac{g(t,\xi)}{\mathfrak w(\nu,t,\xi)\pa_t g(t,\xi)}}\approx \sqrt{\frac{\left(1+\left|\frac{\xi}{l}-t\right|^2\right)}{\left(1+\left|\frac{\eta}{k}-t\right|^2\right)}} \lesssim\langle \xi-\eta\rangle.
  \end{align*}

  If $|l|> |\xi|$, which means that $t\in \mathrm{I}_{k,\eta}$ but $t<t_{E(|\xi|),|\xi|}$, we have $t\approx|\eta|\approx|\xi|\approx 1$, then
  \begin{align*}
    \sqrt{\frac{\mathfrak w(\nu,t,\eta)\pa_t g(t,\eta)}{g(t,\eta)}}\lesssim 1\lesssim \frac{|\xi|^s}{t^{2c_1}}\langle\eta-\xi\rangle.
  \end{align*}
  \end{proof}

\begin{lemma}\label{lem-trans-G-s}
  Let $t\le 2\max(|\xi|,|\eta|,|k|,|m|)$, then
  \begin{align}\label{eq-est-G}
    \left| \frac{G_k(t,\eta)}{G_m(t,\xi)}-1\right|\lesssim \Big(\frac{\langle k-m,\xi-\eta\rangle}{\left(|k|+|m|+|\eta|+|\xi|\right)^{1-s}}+\frac{\langle t\rangle\langle \xi-\eta\rangle}{|\xi|+|\eta|+|k|+|m|}\Big)e^{C\kappa|\xi-\eta,k-m|^s}.
\end{align}
\end{lemma}
\begin{proof}
Recall that
\begin{align*}
  G_k(t,\eta)=\frac{1}{g(t,\iota(k,\eta))}.
\end{align*}

For the same reason to Lemma \ref{lem-tran-W}, we only discus the case $|\eta|\ge |k|$ and $|\xi|\ge |m|$, and  the cases where $|\eta|< |k|$ or $|\xi|< |m|$ can be treated in the same way.

For $|\eta|\ge |k|$ and $|\xi|\ge |m|$, we know that
  \begin{align*}
    \left| \frac{G_k(t,\eta)}{G_m(t,\xi)}-1\right|=&\left| \frac{g(t,\iota(m,\xi))}{g(t,\iota(k,\eta))}-1\right|=\left| \frac{g(t,\xi)}{g(t,\eta)}-1\right|.
  \end{align*}
From Lemma \ref{lem-g-exp}, we have 
\begin{align*}
  \frac{g(t, \xi)}{g(t, \eta)}\lesssim e^{C \kappa|\eta-\xi|^{s}},
\end{align*}
so we only need to study the case
  \begin{align*}
    |k-m|+|\xi-\eta|\le\frac{1}{100}\left(|k|+|m|+|\eta|+|\xi|\right)^{1-s}\le \frac{1}{100}\left(|k|+|m|+|\eta|+|\xi|\right).
  \end{align*}
Which means that $|\xi-\eta|\le \frac{|\eta|+|\xi|}{50}$ and $|\xi-\eta|\le \frac{|\xi|^{1-s}}{20}$. We discuss the following 6 cases:
  \begin{itemize}
    \item[Case 1:] $t\le \min(t_{E(|\xi|),|\xi|},t_{E(|\eta|),|\eta|})$;
    \item[Case 2:] $\min(t_{E(|\xi|),|\xi|},t_{E(|\eta|),|\eta|})\le t\le \max(t_{E(|\xi|),|\xi|},t_{E(|\eta|),|\eta|})$;
    \item[Case 3:] $\max(t_{E(|\xi|),|\xi|},t_{E(|\eta|),|\eta|})\le t\le \min(t_{E(|\xi|^s),|\xi|},t_{E(|\eta|^s),|\eta|})$;
    \item[Case 4:] $\min(t_{E(|\xi|^s),|\xi|},t_{E(|\eta|^s),|\eta|})\le t\le \max(t_{E(|\xi|^s),|\xi|},t_{E(|\eta|^s),|\eta|})$;
    \item[Case 5:] $\max(t_{E(|\xi|^s),|\xi|},t_{E(|\eta|^s),|\eta|})\le t\le 2|\xi|$;
    \item[Case 6:] $2|\xi|\le t\le 2|\eta|$.
  \end{itemize}

For Case 1, $t\le \min(t_{E(|\xi|),|\xi|},t_{E(|\eta|),|\eta|})$. We have from the proof of Case 1 of Lemma \ref{lem-g-exp} that
\begin{align*}
  e^{-C\kappa \frac{|\xi-\eta|}{|\eta|^{1-s}}}\le\frac{ g(t,\xi)}{ g(t,\eta)}\le e^{C\kappa \frac{|\xi-\eta|}{|\eta|^{1-s}}}.
\end{align*}
Therefore, we have
\begin{align*}
  \left|\frac{g(t, \xi)}{g(t, \eta)}-1\right|\lesssim \frac{|\xi-\eta|}{|\eta|^{1-s}}e^{C\kappa|\xi-\eta|^s}.
\end{align*}

For Case 2, $\min(t_{E(|\xi|),|\xi|},t_{E(|\eta|),|\eta|})\le t\le \max(t_{E(|\xi|),|\xi|},t_{E(|\eta|),|\eta|})$. By using the same argument as in the proof of Lemma \ref{lem-g-exp}, the estimate of this case can be reduced to Case 1 and Case 3.

For Case 3, $\max(t_{E(|\xi|),|\xi|},t_{E(|\eta|),|\eta|})\le t\le \min(t_{E(|\xi|^s),|\xi|},t_{E(|\eta|^s),|\eta|})$. During this time period, there are $j$ and $l$ such that $t\in \mathrm{I}_{l,\xi}\cap \mathrm{I}_{j,\eta}$ and 
\begin{align*}
  E(|\xi|^s)+1\le |l|\le E(|\xi|),\quad E(|\eta|^s)+1\le |j|\le E(|\eta|).
\end{align*}

Case 3.1, $|\xi|\le|\eta|$. It is clear that $E(|\xi|^s)\le E(|\eta|^s)$. Then if $\frac{1}{g(t,\eta)}\ge\frac{1}{g(t,\xi)}$, we have
\begin{align*}
  \frac{ g(t,\xi)}{ g(t,\eta)}
  \le&\exp \left(\kappa \sum_{k=1}^{E(|\xi|^s)}\left(F(k,\eta)-F(k,\xi)\right)\right)\exp \left(\kappa \sum_{k=E(|\xi|^s)+1}^{E(|\eta|^s)} \left(F(k,\eta)-\tilde F(k,\xi)\right)\right)\\
  &\times\exp \left(\kappa \sum_{k=E(|\eta|^s)+1}^{l-1}\left|\tilde F(k,\eta)-\tilde F(k,\xi)\right|\right)\exp \left(\kappa \sum_{k=l}^{j}\tilde F(k,\eta)\right),
\end{align*}
if $\frac{1}{g(t,\eta)}\le\frac{1}{g(t,\xi)}$
we have
\begin{align*}
  \frac{ g(t,\xi)}{ g(t,\eta)}
  \ge&\exp \left(\kappa \sum_{k=1}^{E(|\xi|^s)}\left(F(k,\eta)-F(k,\xi)\right)\right)\exp \left(\kappa \sum_{k=E(|\xi|^s)+1}^{E(|\eta|^s)} \left(F(k,\eta)-\tilde F(k,\xi)\right)\right)\\
  &\times\exp \left(-\kappa \sum_{k=E(|\eta|^s)+1}^{l}\left|\tilde F(k,\eta)-\tilde F(k,\xi)\right|\right)\exp \left(\kappa \sum_{k=l+1}^{j-1}\tilde F(k,\eta)\right).
\end{align*}
From the proof in Lemma \ref{lem-g-exp}, we know that
\begin{align*}
  \sum_{k=1}^{E(|\xi|^s)} |F(k,\xi)-F(k,\eta)|\lesssim  \frac{\beta |\xi-\eta|}{|\xi|^{(1-s)(1+3\beta)}}\le  \frac{\beta |\xi-\eta|}{|\xi|^{1-s}}.
\end{align*}
\begin{align*}
  \sum_{k=E(|\xi|^s)+1}^{E(|\eta|^s)} \left|F(k,\eta)-\tilde F(k,\xi)\right|\lesssim |\eta|^s-|\xi|^s\lesssim \frac{s |\xi-\eta|}{|\xi|^{1-s}}.
\end{align*}

\begin{align*}
    &\sum_{k=E(|\eta|^s)+1}^{l}\left|\tilde F(k,\eta)-\tilde F(k,\xi)\right|\le\sum_{k=E(|\eta|^s)+1}^{E(|\xi|)} |\tilde F(k,\xi)-\tilde F(k,\eta)|   \lesssim \frac{|\xi-\eta|}{|\eta|^{1-s}}.
\end{align*}
\begin{align*}
  \sum_{k=l}^{j}\left|\tilde F(k,\eta)\right|\lesssim \sum_{k=l}^{j} \frac{|\eta|^{1-3\beta}}{|k|^{2-3\beta}}\lesssim \frac{|\eta|^{1-3\beta}}{|l|^{2-3\beta}}|l-j|\text{ or }\frac{|\eta|^{2s}}{|l|^2}|l-j|.
\end{align*}
We know that if $|\xi-\eta|\ge \frac{1}{10}\frac{|\xi|}{|l|}$, as
\begin{align*}
  |l-j|\le \frac{|tl-\xi|+|\xi-\eta|+|\eta-tj|}{t}\lesssim 1+\frac{|\xi-\eta|}{t}\lesssim\frac{|\xi-\eta|}{t},
\end{align*}
then
\begin{align*}
  \frac{|\eta|^{1-3\beta}}{|l|^{2-3\beta}}|l-j|\lesssim \frac{|\xi-\eta|}{|l|^{1-3\beta}|\eta|^{3\beta}}\lesssim \frac{|\xi-\eta|}{|\eta|^{1-s}},\\
  |l-j|\frac{|\eta|^{2s}}{|l|^{2}}\lesssim  \frac{|\eta|^s}{|l|}\frac{|\xi-\eta|}{|\eta|^{1-s}}\lesssim\frac{|\xi-\eta|}{|\eta|^{1-s}}.
\end{align*}
If $|\xi-\eta|\le \frac{1}{10}\frac{|\xi|}{|l|}$, it will be more delicate. If $j=l$, then it could be easily treated. We have
\begin{align*}
  &\left|\mathfrak g(\nu,\xi)\langle \nu^{\frac{1}{3}}\frac{\xi}{l}\rangle^{-\frac{3}{2}}\nu^\beta\frac{|\xi|}{|l|^2}\arctan \left(\frac{|\xi|}{(2|l|-1)|l|}\right)-\mathfrak g(\nu,\xi)\langle \nu^{\frac{1}{3}}\frac{\xi}{l}\rangle^{-\frac{3}{2}}\nu^\beta\frac{|\xi|}{|l|^2}\arctan \left(t-\frac{|\xi|}{|l|}\right)\right.\\
  &\left.-\mathfrak g(\nu,\eta)\langle \nu^{\frac{1}{3}}\frac{\eta}{j}\rangle^{-\frac{3}{2}}\nu^\beta\frac{|\eta|}{|j|^2}\arctan \left(\frac{|\eta|}{(2|j|-1)|j|}\right)+\mathfrak g(\nu,\eta)\langle \nu^{\frac{1}{3}}\frac{\eta}{j}\rangle^{-\frac{3}{2}}\nu^\beta\frac{|\eta|}{|j|^2}\arctan \left(t-\frac{|\eta|}{|j|}\right)\right|\\
  \le&\left|\mathfrak g(\nu,\xi)\langle \nu^{\frac{1}{3}}\frac{\xi}{l}\rangle^{-\frac{3}{2}}\nu^\beta\frac{|\xi|}{|l|^2}\arctan \left(\frac{|\xi|}{(2|l|-1)|l|}\right)-\mathfrak g(\nu,\eta)\langle \nu^{\frac{1}{3}}\frac{\eta}{j}\rangle^{-\frac{3}{2}}\nu^\beta\frac{|\eta|}{|j|^2}\arctan \left(\frac{|\eta|}{(2|j|-1)|j|}\right)\right|\\
  &+\left|\mathfrak g(\nu,\xi)\langle \nu^{\frac{1}{3}}\frac{\xi}{l}\rangle^{-\frac{3}{2}}\nu^\beta\frac{|\xi|}{|l|^2}\arctan \left(t-\frac{|\xi|}{|l|}\right)-\mathfrak g(\nu,\eta)\langle \nu^{\frac{1}{3}}\frac{\eta}{j}\rangle^{-\frac{3}{2}}\nu^\beta\frac{|\eta|}{|j|^2}\arctan \left(t-\frac{|\eta|}{|j|}\right)\right|\\
  \lesssim& \frac{|\xi-\eta|}{|\eta|^{1-s}}+\min \left(1,\frac{|\eta-\xi|}{|l|}\right).
\end{align*}
Here we use the fact that 
\begin{align*}
  \mathfrak g(\nu,\xi)\langle \nu^{\frac{1}{3}}\frac{\xi}{l}\rangle^{-\frac{3}{2}}\nu^\beta\frac{|\xi|}{|l|^2}\lesssim  \frac{|\xi|^{1-3\beta}}{|l|^{2-3\beta}}+\frac{|\xi|^{2s}}{|l|^2}\lesssim 1,
\end{align*}
and
\begin{align*}
  \left|\arctan \left(t-\frac{|\xi|}{|l|}\right)-\arctan \left(t-\frac{|\eta|}{|l|}\right)\right|\lesssim \min \left(1,\frac{|\eta-\xi|}{|l|}\right).
\end{align*}
If $|j|=|l|+1$. Recall that $|\xi-\eta|\ge \frac{1}{10}\frac{|\xi|}{|l|}$, we have 
\begin{align*}
  \left|t-\frac{2|\eta|}{2|j|-1}\right|+\left|t-\frac{2|\xi|}{2|j|-1}\right|=\frac{2|\eta-\xi|}{2|j|-1}\le \frac{1}{10}\frac{|\xi|}{|j|(|j|-1)},
\end{align*}
and
\begin{align*}
  \left|t-\frac{|\eta|}{|j|}\right|\approx\left|t-\frac{|\xi|}{|l|}\right|\approx\frac{|\eta|}{|j|^2}.
\end{align*}
Then, we have
\begin{align*}
  &\mathfrak g(\nu,\eta)\langle \nu^{\frac{1}{3}}\frac{\eta}{j}\rangle^{-\frac{3}{2}}\nu^\beta\frac{|\eta|}{|j|^2}\left|\arctan \left(\frac{2|\eta|}{2|j|-1}-\frac{|\eta|}{|j|}\right)- \arctan \left(t-\frac{|\eta|}{|j|}\right)\right| \lesssim\min \left(1,\frac{|\eta-\xi|}{|l|}\right).
\end{align*}
In the same way, we have
\begin{align*}
  &\mathfrak g(\nu,\xi)\langle \nu^{\frac{1}{3}}\frac{\xi}{l}\rangle^{-\frac{3}{2}}\nu^\beta\frac{|\xi|}{|l|^2}\left| \arctan \left(\frac{2|\xi|}{2|j|-1}-\frac{|\xi|}{|j-1|}\right)- \arctan \left(t-\frac{|\xi|}{|j|-1}\right)\right|\lesssim \min \left(1,\frac{|\eta-\xi|}{|l|}\right).
\end{align*}
From the fact that $|e^a-1|\lesssim |a|e^a$, we have that
\begin{align*}
  \left|\frac{g(t, \xi)}{g(t, \eta)}-1\right|\lesssim \frac{|\xi-\eta|}{|\eta|^{1-s}}e^{C\kappa|\xi-\eta|^s}+\frac{|\xi-\eta|}{|l|}e^{C\kappa|\xi-\eta|^s}.
\end{align*}

Case 3.2, $|\xi|>|\eta|$. The same to Case 3.1. Here the estimate we need is 
\begin{align*}
  \frac{g(t, \xi)}{g(t, \eta)}\ge e^{-C\kappa \left(\frac{|\xi-\eta|}{|\eta|^{1-s}}+\min \left(1,\frac{|\eta-\xi|}{|l|}\right)\right)}.
\end{align*}

For Case 4, $\min(t_{E(|\xi|^s),|\xi|},t_{E(|\eta|^s),|\eta|})\le t\le \max(t_{E(|\xi|^s),|\xi|},t_{E(|\eta|^s),|\eta|})$.

This can be treated in the same way as Case 4 of Lemma \ref{lem-g-exp}.

For Case 5, $\max(t_{E(|\xi|^s),|\xi|},t_{E(|\eta|^s),|\eta|})\le t\le 2|\xi|$.

For this case, we have
\begin{align*}
  1\le |l|\le E(|\xi|^s),\quad 1\le |j|\le E(|\eta|^s).
\end{align*}
As $|\xi-\eta|\le \frac{|\xi|^{1-s}}{20}$, we have
\begin{align*}
  |\xi-\eta|\le \frac{1}{10}\frac{|\xi|}{|l|}.
\end{align*}
Then we have $|l|=|j|$ or $|j|=|l|+1$. Then we can treat it in a similar way to the previous case. For example, if $|j|=|l|+1$, we have
\begin{align*}
  \left|t-\frac{2|\eta|}{2|j|-1}\right|+\left|t-\frac{2|\xi|}{2|j|-1}\right|=\frac{2|\eta-\xi|}{2|j|-1}\le \frac{1}{10}\frac{|\xi|}{|j|(|j|-1)},
\end{align*}
and
\begin{align*}
  \left|t-\frac{|\eta|}{|j|}\right|\approx\left|t-\frac{|\xi|}{|l|}\right|\approx\frac{|\eta|}{|j|^2}.
\end{align*}
Then
\begin{align*}
  &\left|\arctan \left(\left(\frac{|\eta|^{1-3\beta}}{|j|^{2-3\beta}}\right)^{-1}\left(t-\frac{|\eta|}{|j|}\right)\right)-\arctan \left(\left(\frac{|\eta|^{1-3\beta}}{|j|^{2-3\beta}}\right)^{-1}\left(\frac{2|\eta|}{2|j|-1}-\frac{|\eta|}{|j|}\right)\right)\right|\\
  \lesssim&\frac{\frac{|j|^{2-3\beta}}{|\eta|^{1-3\beta}}\frac{|\eta-\xi|}{|j|}}{1+\left|\frac{\eta}{j}\right|^{6\beta}}\lesssim \frac{|\eta-\xi||j|^{1+3\beta}}{|\eta|^{1+3\beta}}\lesssim \frac{|\eta-\xi|}{|\eta|^{(1-s)(1+3\beta)}}.
\end{align*}
Similarly,
\begin{align*}
  &\left|\arctan \left(\left(\frac{|\xi|^{1-3\beta}}{|l|^{2-3\beta}}\right)^{-1}\left(t-\frac{|\xi|}{|l|}\right)\right)-\arctan \left(\left(\frac{|\xi|^{1-3\beta}}{|l|^{2-3\beta}}\right)^{-1}\left(\frac{2|\xi|}{2|l|-1}-\frac{|\xi|}{|l|}\right)\right)\right| \lesssim \frac{|\eta-\xi|}{|\eta|^{(1-s)(1+3\beta)}}.
\end{align*}

If $|l|=|j|$, we have
\begin{align*}
  &\left|\arctan \left(\left(\frac{|\xi|^{1-3\beta}}{|l|^{2-3\beta}}\right)^{-1}\left(t-\frac{|\xi|}{|l|}\right)\right)-\arctan \left(\left(\frac{|\eta|^{1-3\beta}}{|j|^{2-3\beta}}\right)^{-1}\left(t-\frac{|\eta|}{|j|}\right)\right)\right| \\
  \lesssim&\min \left(1,\frac{|\eta-\xi|}{|l|}\right)
\end{align*}

For Case 6, $2|\xi|\le t\le 2|\eta|$, we have 
\begin{align*}
  \left|\frac{g(t, \xi)}{g(t, \eta)}-1\right|=&\left|\exp \left(\kappa\left[\arctan \left(\frac{1}{|\eta|^{1-3\beta}}|\eta|\right)-\arctan \left(\frac{1}{|\eta|^{1-3\beta}}\left(t-|\eta|\right)\right)\right]  \right)-1\right|\\
  \lesssim& \frac{\frac{|\eta-\xi|}{|\eta|^{1-3\beta}}}{1+|\eta|^{6\beta}}e^{\frac{\frac{|\eta-\xi|}{|\eta|^{1-3\beta}}}{1+|\eta|^{6\beta}}}\lesssim \frac{\langle k-m,\xi-\eta\rangle}{\left(|k|+|m|+|\eta|+|\xi|\right)^{1-s}}e^{C\kappa|k-m,\xi-\eta|^s}.
\end{align*}

This completes the proof.
\end{proof}
\subsection{Product lemma}
The norm defined by the Fourier multiplier $\rmA(t,\na)$ is not an algebra due to the discrepancy between resonant and non-resonant modes. However, we have the following product lemma.
\begin{lemma}\label{lem-product-1}
  Given $f=f(z,v)$, $h=h(v)$, it holds for $0\le \alpha_1$, $0\le\alpha_2\le1$ that
  \begin{align*}
  \left\||\nabla|^{\alpha_1} \langle\nabla\rangle^{-\alpha_2} \rmA(fh)\right\|_{L^2}\lesssim \left\|A^\gamma f\right\|_{L^2} \left\||\pa_v|^{\alpha_1} \langle\pa_v\rangle^{-\alpha_2} \rmA^{\mathrm R}h\right\|_{L^2}+\left\|A^\gamma h\right\|_{L^2} \left\||\nabla|^{\alpha_1} \langle\nabla\rangle^{-\alpha_2} \rmA f\right\|_{L^2}.
  \end{align*}
  \end{lemma}
  \begin{proof}
    We write
\begin{align*}
  &\left\||\nabla|^{\alpha_1} \langle\nabla\rangle^{-\alpha_2} \rmA(fh)\right\|_{L^2}^2\\
  \le&\frac{1}{2\pi}\sum_k\int \left( \mathrm 1_{|k,\xi|\ge|\eta-\xi|}+\mathrm 1_{|k,\xi|<|\eta-\xi|}\right) |k,\eta|^{2\alpha_1}\langle k,\eta\rangle^{-2\alpha_2}\left|\rmA_k(t,\eta)\right|^2 \left|\hat f_k(\xi)\right|^2\left|\hat h(\eta-\xi)\right|^2d\xi d\eta\\
  =&I_\ge+I_<.
\end{align*}

We first consider $I_\ge$. On the support of the integrand, it holds that $|k,\xi|\ge \frac{1}{2}|k,\eta|$.  We deduce from \eqref{inq-s2} and \eqref{inq-s3} that
\begin{align*}
  e^{\lambda|k,\eta|^s}\le e^{\lambda|k,\xi|^s+c\lambda|\eta-\xi|^s}.
\end{align*}
We write
\begin{align*}
  \rmA_k(\eta)=&\rmA_k(\xi)\frac{\rmA_k(\eta)}{\rmA_k(\xi)}=\rmA_k(\xi)e^{\lambda|k,\eta|^s-\lambda|k,\xi|^s}\frac{\langle k,\eta\rangle^\sigma}{\langle k,\xi\rangle^\sigma}\frac{W_k(\eta)}{W_k(\xi)}\frac{G_k(\eta)}{G_k(\xi)}\frac{\mathfrak M_k(\eta)}{\mathfrak M_k(\xi)}.
\end{align*}
It follows from \eqref{eq-est-Wkl} and \eqref{eq-est-gg} that
\begin{align*}
  e^{\lambda|k,\eta|^s-\lambda|k,\xi|^s}\frac{\langle k,\eta\rangle^\sigma}{\langle k,\xi\rangle^\sigma}\frac{W_k(\eta)}{W_k(\xi)}\frac{G_k(\eta)}{G_k(\xi)}\frac{\mathfrak M_k(\eta)}{\mathfrak M_k(\xi)} \lesssim \langle \xi-\eta\rangle^{2+2C_1\kappa} e^{c\lambda|\eta-\xi|^s}.
\end{align*}
Then by using Lemma \ref{lem-product}, we have
\begin{align*}
  I_\ge\lesssim& \sum_k\int |k,\xi|^{2\alpha_1}\langle k,\xi\rangle^{-2\alpha_2}\left|\rmA_k(t,\xi)\right|^2 \left|\hat f_k(\xi)\right|^2\langle \xi-\eta\rangle^{2+2C_1\kappa} e^{c\lambda|\eta-\xi|^s}\left|\hat h(\eta-\xi)\right|^2d\xi d\eta\\
  \lesssim& \left\|h\right\|_{\mathcal{G}^{s,c\lambda,5}} \left\||\nabla|^{\alpha_1} \langle\nabla\rangle^{-\alpha_2} \rmA f\right\|_{L^2}\le \left\|A^\gamma h\right\|_{L^2} \left\||\nabla|^{\alpha_1} \langle\nabla\rangle^{-\alpha_2} \rmA f\right\|_{L^2}.
\end{align*}

The estimate for $I_<$ is the same.
  \end{proof}
\begin{lemma}\label{lem-product-2}
  Writing $(v')^2-1=(v'-1)^2+2(v'-1)=h^2+2h$ and $v''=\pa_v(v'-1)+(v'-1)\pa_v(v'-1)=\pa_vh+h\pa_vh$, by the bootstrap hypotheses on $h$, it holds that
  \begin{align*}
  \left\|\rmA^{\mathrm R}\left((v')^2-1\right)\right\|_{L^2}\lesssim&\left\|\rmA^{\mathrm R}h\right\|_{L^2}+\left\|\rmA^{\mathrm R}h\right\|_{L^2}^2,\\
  \left\| \frac{\rmA^{\mathrm R}}{\langle\pa_v\rangle}v''\right\|_{L^2}\lesssim&\left\|\rmA^{\mathrm R}h\right\|_{L^2}+\left\|\rmA^{\mathrm R}h\right\|_{L^2}^2.
  \end{align*}
\end{lemma}
\begin{proof}
  This lemma follows from Lemma \ref{lem-product-1} and Corollary \ref{cor-WR}.
\end{proof}
\section{Elliptic estimates}

The purpose of this section is to provide a thorough analysis of $\Delta_t$. In particular, in this section, we prove Proposition \ref{pro-elliptic-high}.

\subsection{Lossy estimate}
The following lemma from \cite{BM2015} shows that by paying regularity, one can still deduce the same decay from $\Delta_t^{-1}$ as from $\Delta_L^{-1}$. The terminology ``lossy''  refers to the fact that one needs to have a weaker norm on the left-hand side of \eqref{eq-est-ell-loss} relative to the right-hand side. Since the regularity of $v''$ is one order lower than $f$, the following lemma shows that to obtain the point-wise time decay, we need to pay three derivatives.

\begin{lemma}\label{lem-elliptic-low-1}
  Under the bootstrap hypotheses, for $\varepsilon$ sufficiently small, it holds that for $\sigma'\le \sigma-1$ that,
  \begin{align}\label{eq-est-ell-loss}
    \| \langle\pa_z\rangle^4\phi_{\neq}\|_{\mathcal G^{s,\lambda,\sigma'-2}}\lesssim \frac{1}{\langle t\rangle^2} \|f_{\neq}\|_{\mathcal G^{s,\lambda,\sigma'}}.
  \end{align}
\end{lemma}
\begin{proof}
  We have
  \begin{align*}
    \|\langle\pa_z\rangle^4\phi_{\neq}\|_{\mathcal G^{s,\lambda,\sigma'-2}}^2=&\sum_{k\neq0}\|\langle k\rangle^2e^{\lambda(t)|k,\eta|^s}\langle k,\eta\rangle^{\sigma'-2} \hat \phi_k\|_{L^2}^2\\
    =&\sum_{k\neq0}\int_\eta \langle k\rangle^4e^{2\lambda(t)|k,\eta|^s} \langle k,\eta\rangle^{2\sigma'-4}  |\hat\phi_k(\eta)|^2 d\eta\\
    \le&\sum_{k\neq0}\int_\eta e^{2\lambda(t)|k,\eta|^s} \frac{\langle k\rangle^8\langle k,\eta\rangle^{2\sigma'-4}}{\left(k^2+(\eta-kt)^2\right)^2} |\widehat {\Delta_L\phi}_k(\eta)|^2 d\eta\\
    \lesssim&\sum_{k\neq0}\int_\eta e^{2\lambda(t)|k,\eta|^s} \frac{\langle k\rangle^8\langle k,\eta\rangle^{2\sigma'}}{\left(k^2+\eta^2\right)^2\left(k^2+(\eta-kt)^2\right)^2} |\widehat {\Delta_L\phi}_k(\eta)|^2 d\eta\\
    \lesssim&\sum_{k\neq0}\int_\eta e^{2\lambda(t)|k,\eta|^s} \frac{\langle k\rangle^8\langle k,\eta\rangle^{2\sigma'}}{k^8\left(1+\left(\frac{\eta}{k}\right)^2\right)^2\left(1+(\frac{\eta}{k}-t)^2\right)^2} |\widehat {\Delta_L\phi}_k(\eta)|^2 d\eta\\
    \lesssim&\frac{1}{\langle t\rangle^4}\sum_{k\neq0}\int_\eta e^{2\lambda(t)|k,\eta|^s} \frac{\langle k\rangle^8\langle k,\eta\rangle^{2\sigma'}}{k^8} |\widehat {\Delta_L\phi}_k(\eta)|^2 d\eta\\
    \lesssim&\frac{1}{\langle t\rangle^4} \|\Delta_L\phi_{\neq}\|_{\mathcal G^{s,\lambda,\sigma'}}^2.
  \end{align*}
  We write $\Delta_t$ as a perturbation of $\Delta_L$ via
  \begin{align}\label{eq-DeltaL}
    \Delta_L \phi_{\neq}=f_{\neq}+\left(1-(v')^2\right)(\pa_v-t\pa_z)^2\phi_{\neq}-v''(\pa_v-t\pa_z)\phi_{\neq}.
  \end{align}
It follows that
\begin{align*}
  \|\Delta_L\phi_{\neq}\|_{\mathcal G^{s,\lambda,\sigma'}}\lesssim &\|f_{\neq}\|_{\mathcal G^{s,\lambda,\sigma'}}+\left\| \left(1-(v')^2\right)(\pa_v-t\pa_z)^2\phi_{\neq}\right\|_{\mathcal G^{s,\lambda,\sigma'}}+\left\| v''(\pa_v-t\pa_z)\phi_{\neq}\right\|_{\mathcal G^{s,\lambda,\sigma'}}\\
  \lesssim&\|f_{\neq}\|_{\mathcal G^{s,\lambda,\sigma'}}+\left\|\rmA^{\mathrm R} h\right\|_{L^2}\|\Delta_L\phi_{\neq}\|_{\mathcal G^{s,\lambda,\sigma'}}\\
  \lesssim&\|f_{\neq}\|_{\mathcal G^{s,\lambda,\sigma'}}+\varepsilon^{\frac{1}{2}\beta}\nu^{\frac{1}{2}\beta}\|\Delta_L\phi_{\neq}\|_{\mathcal G^{s,\lambda,\sigma'}}.
\end{align*}
This completes the proof.
\end{proof}

\begin{corol}\label{cor-elliptic-low-2}
  Under the same assumption of Lemma \ref{lem-elliptic-low-1}, it holds for $0\le\alpha\le2$ that
    \begin{align}\label{eq-est-ell-loss-2}
    \| |\nabla_L|^\alpha\phi_{\neq}\|_{\mathcal G^{s,\lambda,\sigma'-2+\alpha}}\lesssim \frac{1}{\langle t\rangle^{2-\alpha}} \|f_{\neq}\|_{\mathcal G^{s,\lambda,\sigma'}}.
  \end{align}
\end{corol}
\subsection{Precision estimates}
Now we turn to the proof of Proposition \ref{pro-elliptic-high}, announced in Section \ref{sec-main-energy}. If $\Delta_t$  were simply $\Delta_L$ then the estimate would be trivial. 
\begin{proof}[Proof of Proposition \ref{pro-elliptic-high}]
Since the coefficients depend only on $v$, $\Delta_t\phi=f$ decouples mode-by-mode in the $z$ frequencies. Hence, we essentially prove a mode-by-mode analog of \eqref{eq-est-sharp-ellip} and then sum.

We define the multipliers,
\begin{align*}
  \mathcal M_1(t,l,\xi)&=\left\langle \frac{\xi}{lt} \right\rangle^{-1}|l,\xi|^{\frac{s}{2}}t^{-c_1}\rmA_l(\xi)P_{\neq},\\
   \mathcal M_2(t,l,\xi)&=\left\langle \frac{\xi}{lt} \right\rangle^{-1}\sqrt{\frac{\mathfrak w(\nu,t,\iota(l,\xi))\pa_t w_l(t,\iota(l,\xi)) }{w_l(t,\iota(l,\xi)) }} \rmA_l(\xi)P_{\neq},\\
   \mathcal M_3(t,l,\xi)&=\left\langle \frac{\xi}{lt} \right\rangle^{-1}\sqrt{\frac{\mathfrak w(\nu,t,\iota(l,\xi))\pa_t g(t,\iota(l,\xi)) }{g(t,\iota(l,\xi))}} \rmA_l(\xi)P_{\neq}.
\end{align*}
  Recall \eqref{eq-DeltaL} that
    \begin{align*}
    \Delta_L \phi_{\neq}=f_{\neq}+\left(1-(v')^2\right)(\pa_v-t\pa_z)^2\phi_{\neq}-v''(\pa_v-t\pa_z)\phi_{\neq}.
  \end{align*}
Clearly,
\begin{align*}
  \sum_{i=1,2,3}\left\|\mathcal M_i f\right\|_{L^2}^2\lesssim t^{-2c_1}\left\||\nabla|^{\frac{s}{2}}P_{\neq} \rmA f\right\|_{L^2}^2+\left\|\sqrt{\frac{\mathfrak w\pa_t w }{w }} \rmA f\right\|_{L^2}^2+\left\|\sqrt{\frac{\mathfrak w\pa_t g }{g}} \rmA f\right\|_{L^2}^2,
\end{align*}
and hence the proposition would be trivial if $\Delta_L \phi=f$.

Define 
\begin{align}\label{def-T1T2}
  {\mathrm T}^1=\left(1-(v')^2\right)(\pa_v-t\pa_z)^2\phi_{\neq},\quad {\mathrm T}^2=v''(\pa_v-t\pa_z)\phi_{\neq}.
\end{align}
and divide each via a paraproduct decomposition in the $v$ variable only
\begin{align*}
  {\mathrm T}^1=&\sum_{N\ge8}\left(1-(v')^2\right)_N(\pa_v-t\pa_z)^2\phi_{<N/8}\\
  &+\sum_{N\ge8}\left(1-(v')^2\right)_{<N/8}(\pa_v-t\pa_z)^2\phi_{N}\\
  &+\sum_{N\in\mathbb D}\sum_{N'\sim N}\left(1-(v')^2\right)_{N'}(\pa_v-t\pa_z)^2\phi_{N}\\
  =&{\mathrm T}^1_{\mathrm {HL}}+{\mathrm T}^1_{\mathrm {LH}}+{\mathrm T}^1_{\mathcal R}
\end{align*}
\begin{align*}
  {\mathrm T}^2=&\sum_{N\ge8}v''_N(\pa_v-t\pa_z)^2\phi_{<N/8}+\sum_{N\ge8}v''_{<N/8}(\pa_v-t\pa_z)^2\phi_{N}+\sum_{N\in\mathbb D}\sum_{N'\sim N}v''_{N'}(\pa_v-t\pa_z)^2\phi_{N}\\
  =&{\mathrm T}^2_{\mathrm {HL}}+{\mathrm T}^2_{\mathrm {LH}}+{\mathrm T}^2_{\mathcal R}
\end{align*}
We only need to show that $\mathcal M_i$ applying on ${\mathrm T}^1$ and ${\mathrm T}^2$ could be absorbed on the left-hand side of \eqref{eq-est-sharp-ellip}. Each step has several complications, dealt with and discussed below.

\subsubsection{Low-high interactions} 
Since ${\mathrm T}^1_{\mathrm {LH}}$ contains more derivatives on $\phi$ than ${\mathrm T}^2_{\mathrm {LH}}$, the former is strictly harder so we treat only ${\mathrm T}^1_{\mathrm {LH}}$. In what follows we use the shorthand
\begin{align*}
  \widehat{\mathcal V}(\xi)=\widehat{\left(1-(v')^2\right)}(\xi).
\end{align*}
We write
\begin{align*}
  \widehat{\mathcal M_1{\mathrm T}^1_{\mathrm {LH}}}(l,\xi)=-\frac{1}{2\pi}\sum_{N\ge8}\int_{\xi'} \mathcal M_1(t,l,\xi)\widehat{\mathcal V}(\xi-\xi')_{<N/8}(\xi'-lt)^2\hat\phi_l(\xi')_{N}d\xi'.
\end{align*}
On the support of the integrand (see Appendix \ref{sec-decompo}), it holds that
\begin{align*}
  \big||l,\xi|-|l,\xi'|\big|\le |\xi-\xi'|\le \frac{6}{32}|\xi'|\le \frac{6}{32}|l,\xi'|.
\end{align*}
then we have
\begin{align*}
  |\xi-\xi'|\le \frac{1}{5}|\xi'|,\quad |\xi|\approx|\xi'|
\end{align*}
and
\begin{align*}
  |l,\xi|^{\frac{s}{2}}\approx |l,\xi'|^{\frac{s}{2}},\quad\left\langle \frac{\xi}{lt} \right\rangle^{-1}\approx\left\langle \frac{\xi'}{lt} \right\rangle^{-1}.
\end{align*}
From Lemma \ref{lem-dis-s}, we deduce that
\begin{align*}
  e^{\lambda|l,\xi|^s}\le e^{\lambda|l,\xi'|^s+c\lambda|\xi-\xi'|^s}.
\end{align*}
Then by using \eqref{eq-est-Wkl} and \eqref{eq-est-gg} we have
\begin{equation}\label{eq-est-ell-lh}
  \begin{aligned}    
   \rmA_l(\xi)=&\rmA_l(\xi')\frac{\rmA_l(\xi)}{\rmA_l(\xi')}=\rmA_l(\xi')e^{\lambda|l,\xi|^s-\lambda|l,\xi'|^s}\frac{\langle l,\xi\rangle^\sigma}{\langle l,\xi'\rangle^\sigma}\frac{W_l(\xi)}{W_l(\xi')}\frac{G_l(\xi)}{G_l(\xi')}\frac{\mathfrak M_l(\xi)}{\mathfrak M_l(\xi')}\\
  \lesssim&\rmA_l(\xi')\langle \xi-\xi'\rangle^{2+2C_1\kappa} e^{c\lambda|\xi-\xi'|^s},   
  \end{aligned}
\end{equation}
and
\begin{align*}
  \left|\widehat{\mathcal M_1{\mathrm T}^1_{\mathrm {LH}}}(l,\xi)\right|\lesssim&\sum_{N\ge8}\int_{\xi'} \langle \xi-\xi'\rangle^{2+2C_1\kappa} e^{c\lambda|\xi-\xi'|^s} \left| \widehat{\mathcal V}(\xi-\xi')_{<N/8}\right|\\
  &\qquad\qquad\qquad\times\left\langle \frac{\xi'}{lt} \right\rangle^{-1}|l,\xi'|^{\frac{s}{2}}t^{-c_1}(\xi'-lt)^2\rmA_l(\xi')\left|\hat\phi_l(\xi')_{N}\right|d\xi'.
\end{align*}
Then \eqref{eq-decomp-sum}, Lemma \ref{lem-conv-Young}, and Lemma \ref{lem-product-2} imply 
\begin{align*}
  \left\|\mathcal M_1{\mathrm T}^1_{\mathrm {LH}}\right\|_{L^2}^2=&\sum_{l\neq0}\left\|\mathcal M_1{\mathrm T}^1_{\mathrm {LH}}(l)\right\|_{L^2}^2\\
  \lesssim&\sum_{l\neq 0}\sum_{N\ge8}\left\|\left(1-(v')^2\right)_{<N/8}\right\|_{\mathcal G^{s,\lambda,\gamma}}^2 \left\|\mathcal M_1\Delta_LP_{\neq}(\phi_l)_N\right\|_{L^2}^2\\
  \lesssim&\varepsilon\nu^{\beta} \left\|\mathcal M_1\Delta_L\phi_{\neq}\right\|_{L^2}^2.
\end{align*}

Now we turn to $\mathcal M_2{\mathrm T}^1_{\mathrm {LH}}$. From Lemma \ref{lem-w-wgl}, we have
\begin{align*}
  \left|\widehat{\mathcal M_2{\mathrm T}^1_{\mathrm {LH}}}(l,\xi)\right|\lesssim&\sum_{N\ge8}\int_{\xi'}\left\langle \frac{\xi}{lt} \right\rangle^{-1}\sqrt{\frac{\mathfrak w\pa_t w_l(t,\iota(l,\xi)) }{w_l(t,\iota(l,\xi)) }}\rmA_l(\xi) \left|\widehat{\mathcal V}(\xi-\xi')_{<N/8}\right|(\xi'-lt)^2 \left|\hat\phi_l(\xi')_{N}\right|d\xi'\\
  \lesssim&\sum_{N\ge8}\int_{\xi'}\left\langle \frac{\xi'}{lt} \right\rangle^{-1} \left(\sqrt{\frac{\mathfrak w\pa_t w_l(t,\iota(l,\xi')) }{w_l(t,\iota(l,\xi')) }}+\sqrt{\frac{\mathfrak w\pa_t g(t,\iota(l,\xi')) }{g(t,\iota(l,\xi')) }}+t^{-c_1}|l,\xi'|^{\frac{s}{2}}\right)\langle\xi-\xi'\rangle\\
  &\qquad\qquad\times e^{c\lambda|\xi-\xi'|^s} \left|\widehat{\mathcal V}(\xi-\xi')_{<N/8}\right|(\xi'-lt)^2\rmA_l(\xi') \left|\hat\phi_l(\xi')_{N}\right|d\xi',
\end{align*}
and then
\begin{align*}
  \left\|\mathcal M_2{\mathrm T}^1_{\mathrm {LH}}\right\|_{L^2}^2=&\sum_{l\neq0}\left\|\mathcal M_2{\mathrm T}^1_{\mathrm {LH}}(l)\right\|_{L^2}^2\\
  \lesssim&\sum_{l\neq 0}\sum_{N\ge8}\left\|A^\gamma\left(1-(v')^2\right)_{<N/8}\right\|_{L^2}^2 \\
  &\qquad\qquad\times\left(\left\|\mathcal M_1\Delta_LP_{\neq}(\phi_l)_N\right\|_{L^2}^2+\left\|\mathcal M_2\Delta_LP_{\neq}(\phi_l)_N\right\|_{L^2}^2+\left\|\mathcal M_3\Delta_LP_{\neq}(\phi_l)_N\right\|_{L^2}^2\right)\\
  \lesssim&\varepsilon\nu^{\beta}\left(\left\|\mathcal M_1\Delta_L\phi_{\neq}\right\|_{L^2}^2+\left\|\mathcal M_2\Delta_L\phi_{\neq}\right\|_{L^2}^2+\left\|\mathcal M_3\Delta_L\phi_{\neq}\right\|_{L^2}^2\right).
\end{align*}
The estimate for $\mathcal M_3{\mathrm T}^1_{\mathrm {LH}}$ is similar. From Lemma \ref{lem-g-gl}, we have
\begin{align*}
  \left|\widehat{\mathcal M_3{\mathrm T}^1_{\mathrm {LH}}}(l,\xi)\right|\lesssim&\sum_{N\ge8}\int_{\xi'}\left\langle \frac{\xi}{lt} \right\rangle^{-1}\sqrt{\frac{\mathfrak w\pa_t g(t,\iota(l,\xi)) }{g(t,\iota(l,\xi)) }}\rmA_l(\xi) \left|\widehat{\mathcal V}(\xi-\xi')_{<N/8}\right|(\xi'-lt)^2 \left|\hat\phi_l(\xi')_{N}\right|d\xi'\\
  \lesssim&\sum_{N\ge8}\int_{\xi'}\left\langle \frac{\xi'}{lt} \right\rangle^{-1} \left(\sqrt{\frac{\mathfrak w\pa_t g(t,\iota(l,\xi')) }{g(t,\iota(l,\xi')) }}+t^{-c_1}|l,\xi'|^{\frac{s}{2}}\right)\langle\xi-\xi'\rangle\\
  &\qquad\qquad\times e^{c\lambda|\xi-\xi'|^s} \left|\widehat{\mathcal V}(\xi-\xi')_{<N/8}\right|(\xi'-lt)^2\rmA_l(\xi') \left|\hat\phi_l(\xi')_{N}\right|d\xi',
\end{align*}
and then
\begin{align*}
  \left\|\mathcal M_3{\mathrm T}^1_{\mathrm {LH}}\right\|_{L^2}^2=&\sum_{l\neq0}\left\|\mathcal M_3{\mathrm T}^1_{\mathrm {LH}}(l)\right\|_{L^2}^2\lesssim\varepsilon\nu^{\beta}\left(\left\|\mathcal M_1\Delta_L\phi_{\neq}\right\|_{L^2}^2+\left\|\mathcal M_3\Delta_L\phi_{\neq}\right\|_{L^2}^2\right).
\end{align*}

The estimate for ${\mathrm T}^2_{\mathrm {LH}}$ is similar.
\subsubsection{High-low interactions}
We only focus on ${\mathrm T}^2_{\mathrm {HL}}$ as ${\mathrm T}^1_{\mathrm {HL}}$ is similar and easier. We break $\mathcal M_1 {\mathrm T}^2_{\mathrm {HL}}$ into two cases:
\begin{align*}
  &\widehat{\mathcal M_1{\mathrm T}^2_{\mathrm {HL}}}(l,\xi)\\
  =&-\frac{i}{2\pi}\sum_{N\ge8}\int_{\xi'} \left[\mathbf 1_{|l|\ge \frac{1}{16}|\xi|}+\mathbf 1_{|l|< \frac{1}{16}|\xi|}\right]  \mathcal M_1(t,l,\xi)\widehat{v''}(\xi-\xi')_{N}(\xi'-lt)\hat\phi_l(\xi')_{<N/8}d\xi'\\
  =&\widehat{\mathcal M_1{\mathrm T}^{2,z}_{\mathrm {HL}}}(l,\xi)+\widehat{\mathcal M_1{\mathrm T}^{2,v}_{\mathrm {HL}}}(l,\xi).
\end{align*}
For $\mathcal M_1{\mathrm T}^{2,z}_{\mathrm {HL}}$, as $|l|$ big. If $\frac{1}{16}|\xi|\le|l|\le16|\xi|$, by using \eqref{inq-s3} we have
\begin{align*}
  |l,\xi|^s\le c|l,\xi'|^s+c|\xi-\xi'|^s,
\end{align*}
and if $16|\xi|\le|l|$, by using \eqref{inq-s2} we have
\begin{align*}
  |l,\xi|^s\le |l,\xi'|^s+c|\xi-\xi'|^s.
\end{align*}
We also have
\begin{align*}
   \left\langle \frac{\xi}{lt} \right\rangle^{-1}\sim \left\langle \frac{\xi'}{lt} \right\rangle^{-1}\sim 1.
\end{align*}
As $\langle l,\xi\rangle\approx\langle l,\xi'\rangle$, we can treat this term as a Low-High term and get
\begin{align*}
  \left|\widehat{\mathcal M_1{\mathrm T}^{2,z}_{\mathrm {HL}}}(l,\xi)\right|\lesssim&\sum_{N\ge8}\int_{\xi'}\left\langle \frac{\xi'}{lt} \right\rangle^{-1}|l|^{\frac{s}{2}}t^{-c_1}e^{c\lambda|\xi-\xi'|^s} \left|\widehat{v''}(\xi-\xi')_{N}\right||\xi'-lt|\rmA_l(\xi') \left|\hat\phi_l(\xi')_{<N/8}\right|d\xi',
\end{align*}
and
\begin{align*}
  \left\|\mathcal M_1{\mathrm T}^{2,z}_{\mathrm {HL}}\right\|_{L^2}^2 \lesssim&\sum_{l\neq 0}\sum_{N\ge8}\left\|v''_{N}\right\|_{\mathcal G^{s,\lambda,\gamma}}^2 \left\|\mathcal M_1\Delta_LP_{\neq}(\phi_l)_{N/8}\right\|_{L^2}^2\\
  \lesssim&\varepsilon\nu^{\beta} \left\|\mathcal M_1\Delta_L\phi_{\neq}\right\|_{L^2}^2.
\end{align*}
Next, we consider $\mathcal M_1{\mathrm T}^{2,v}_{\mathrm {HL}}$. On the support of the integrand, it holds that
\begin{align*}
  |l,\xi'|\le \frac{1}{4}|\xi-\xi'|,
\end{align*}
and then by \eqref{inq-s2} 
\begin{align*}
  |l,\xi|^s\le c|l,\xi'|^s+|\xi-\xi'|^s,\quad .
\end{align*}

As $16|l|\le |\xi|\approx|\xi-\xi'|$, we have
\begin{align*}
  \left\langle \frac{\xi}{lt} \right\rangle^{-1}\approx \left\langle \frac{\xi-\xi'}{lt} \right\rangle^{-1}.
\end{align*}
Similar to \eqref{eq-est-ell-lh}, we have
\begin{align*}
     \rmA_l(\xi)\lesssim&\rmA^{\mathrm R}(\xi-\xi')e^{ c_M \nu^{\frac{1}{3}}t}\langle l,\xi'\rangle^{2+2C_1\kappa} e^{c\lambda|l,\xi'|^s}. 
\end{align*}
Therefore,
\begin{align*}
  &\left|\widehat{\mathcal M_1{\mathrm T}^{2,v}_{\mathrm {HL}}}(l,\xi)\right|\\
  \lesssim&\sum_{N\ge8}\int_{\xi'}\left\langle \frac{\xi-\xi'}{lt} \right\rangle^{-1}|\xi-\xi'|^{\frac{s}{2}}t^{-c_1}\rmA^{\mathrm R}(\xi-\xi')\left|\widehat{v''}(\xi-\xi')_{N}\right|\\
  &\qquad\qquad\qquad\qquad\qquad\times\langle l,\xi'\rangle^{2+2C_1\kappa}e^{c_M \nu^{\frac{1}{3}}t}e^{c\lambda|l,\xi'|^s}|\xi'-lt| \left|\hat\phi_l(\xi')_{<N/8}\right|d\xi'\\
  \lesssim&\sum_{N\ge8}\int_{\xi'} \frac{1}{\left\langle \frac{\xi-\xi'}{lt} \right\rangle \langle lt\rangle}|\xi-\xi'|^{\frac{s}{2}}t^{-c_1}\rmA^{\mathrm R}(\xi-\xi')\left|\widehat{v''}(\xi-\xi')_{N}\right|\\
  &\qquad\qquad\qquad\qquad\qquad\times\langle l,\xi'\rangle^{2+2C_1\kappa}e^{c_M \nu^{\frac{1}{3}}t}e^{c\lambda|l,\xi'|^s}|\xi'-lt|\langle lt\rangle \left|\hat\phi_l(\xi')_{<N/8}\right|d\xi'\\
  \lesssim&\sum_{N\ge8}\int_{\xi'} |\xi-\xi'|^{\frac{s}{2}}t^{-c_1}\frac{\rmA^{\mathrm R}(\xi-\xi')}{\langle \xi-\xi'\rangle}\left|\widehat{v''}(\xi-\xi')_{N}\right|\\
  &\qquad\qquad\qquad\qquad\qquad\times\langle l,\xi'\rangle^{2+2C_1\kappa}e^{c_M \nu^{\frac{1}{3}}t}e^{c\lambda|l,\xi'|^s}|\xi'-lt|\langle lt\rangle \left|\hat\phi_l(\xi')_{<N/8}\right|d\xi'.
\end{align*}
Then by using Corollary \ref{cor-elliptic-low-2} we have
\begin{align*}
  \left\|\mathcal M_1{\mathrm T}^{2,v}_{\mathrm {HL}}\right\|_{L^2}^2\lesssim&\sum_{l\neq 0}\sum_{N\ge8} t^{-2c_1}\left\||\pa_v|^{\frac{s}{2}}\frac{\rmA^{\mathrm R}}{\langle\pa_v\rangle} v''_{N}\right\|_{L^2}^2 \left\|e^{c_M \nu^{\frac{1}{3}}t}e^{\lambda|\nabla|^s}|\nabla|^{6}|\nabla_L|tP_{\neq}\phi _{N/8}\right\|_{L^2}^2\\
  \lesssim&\varepsilon^2\nu^{2\beta} \mathrm{CCK}_{\lambda}^{v,2}.
\end{align*}

Accordingly, 
\begin{align*}
   \left\|\mathcal M_1{\mathrm T}^2_{\mathrm {HL}}\right\|_{L^2}^2\lesssim \varepsilon\nu^{\beta} \left\|\mathcal M_1\Delta_L\phi_{\neq}\right\|_{L^2}^2+\varepsilon^2\nu^{2\beta} \mathrm{CCK}^2_{\lambda}.
\end{align*}

Then we turn to $\mathcal M_2{\mathrm T}^2_{\mathrm {HL}}$.
\begin{align*}
   \widehat{\mathcal M_2{\mathrm T}^2_{\mathrm {HL}}}(l,\xi)=&-\frac{i}{2\pi}\sum_{N\ge8}\int_{\xi'} \left[\mathbf 1_{|l|\ge \frac{1}{16}|\xi|}+\mathbf 1_{|l|< \frac{1}{16}|\xi|}\right]  \mathcal M_2(t,l,\xi)\widehat{v''}(\xi-\xi')_{N}(\xi'-lt)\hat\phi_l(\xi')_{<N/8}d\xi'\\
  =&\widehat{\mathcal M_2{\mathrm T}^{2,z}_{\mathrm {HL}}}(l,\xi)+\widehat{\mathcal M_2{\mathrm T}^{2,v}_{\mathrm {HL}}}(l,\xi).
\end{align*}
For $\mathcal M_2{\mathrm T}^{2,z}_{\mathrm {HL}}$, on the support of the integrand, it holds from \eqref{inq-s2} and \eqref{inq-s3} that
\begin{align*}
  |\xi|\approx|\xi-\xi'|,\ |l,\xi|^s\le |l,\xi'|^s+c|\xi-\xi'|^s,\ \left\langle \frac{\xi}{lt} \right\rangle^{-1}\sim \left\langle \frac{\xi'}{lt} \right\rangle^{-1}\sim 1.
\end{align*}
If $\frac{1}{16}|\xi|\le|l|\le 16|\xi|$, we have $|\xi|\approx|\xi-\xi'|\approx|l|$, and $\pa_t w_l(t,\iota(l,\xi))\neq0$ only for $|l|^{1-s}\lesssim t\lesssim |l|$, therefore
\begin{align*}
  \sqrt{\frac{\mathfrak w\pa_t w_l(t,\iota(l,\xi)) }{w_l(t,\iota(l,\xi))}}\lesssim 1\lesssim t^{-c_1}|l|^{\frac{s}{2}}|\xi-\xi'|.
\end{align*}
If $|l|\ge 16|\xi|$, we have $\iota(l,\xi)=\iota(l,\xi')=l$ and
\begin{align*}
  \sqrt{\frac{\mathfrak w\pa_t w_l(t,\iota(l,\xi)) }{w_l(t,\iota(l,\xi))}}=\sqrt{\frac{\mathfrak w\pa_t w_l(t,\iota(l,\xi')) }{w_l(t,\iota(l,\xi'))}}.
\end{align*}
It follows that
\begin{align*}
 & \left|\widehat{\mathcal M_2{\mathrm T}^{2,z}_{\mathrm {HL}}}(l,\xi)\right|\\
 &\lesssim\sum_{N\ge8}\int_{\xi'}\left\langle \frac{\xi}{lt} \right\rangle^{-1}\mathbf 1_{|l|\ge \frac{1}{16}|\xi|}\sqrt{\frac{\mathfrak w\pa_t w_l(t,\iota(l,\xi)) }{w_l(t,\iota(l,\xi))}}\rmA_l(\xi) \left|\widehat{v''}(\xi-\xi')_{N}\right|\\
  &\qquad\qquad\qquad\qquad\qquad\qquad\qquad\qquad\qquad\qquad\qquad\times|\xi'-lt| \left|\hat\phi_l(\xi')_{<N/8}\right|d\xi'\\
&  \lesssim\sum_{N\ge8}\int_{\xi'}\left\langle \frac{\xi'}{lt} \right\rangle^{-1} e^{c\lambda|\xi-\xi'|^s}\langle \xi-\xi'\rangle^{3+2C_1\kappa} \left|\widehat{v''}(\xi-\xi')_{N}\right|\\
  &\qquad\qquad\qquad\times \left(t^{-c_1}|l|^{\frac{s}{2}}+ \sqrt{\frac{\mathfrak w\pa_t w_l(t,\iota(l,\xi')) }{w_l(t,\iota(l,\xi'))}}\right)\rmA_l(\xi')|\xi'-lt| \left|\hat\phi_l(\xi')_{<N/8}\right|d\xi'.
\end{align*}
Then we have
\begin{align*}
  \left\|\mathcal M_2{\mathrm T}^{2,z}_{\mathrm {HL}}\right\|_{L^2}^2=&\sum_{l\neq0}\left\|\mathcal M_2{\mathrm T}^{2,z}_{\mathrm {HL}}(l)\right\|_{L^2}^2\lesssim\varepsilon\nu^{\beta}\left(\left\|\mathcal M_1\Delta_L\phi_{\neq}\right\|_{L^2}^2+\left\|\mathcal M_2\Delta_L\phi_{\neq}\right\|_{L^2}^2\right).
\end{align*}

Next, we consider $\mathcal M_2{\mathrm T}^{2,v}_{\mathrm {HL}}$. on the support of the integrand, it holds $|\xi|,|\xi-\xi'|\ge|l|$,
\begin{align*}
  |\xi|\approx|\xi-\xi'|,\ |l,\xi|^s\le c|l,\xi'|^s+|\xi-\xi'|^s,\ \left\langle \frac{\xi}{lt} \right\rangle^{-1}\approx \left\langle \frac{\xi-\xi'}{lt} \right\rangle^{-1}.
\end{align*}
By using Lemma \ref{lem-w-wgl}, we have for $\sigma\ge10$,
\begin{align*}
  &\left|\widehat{\mathcal M_2{\mathrm T}^{2,v}_{\mathrm {HL}}}(l,\xi)\right|\\
  \lesssim&\sum_{N\ge8}\int_{\xi'}\left\langle \frac{\xi-\xi'}{lt} \right\rangle^{-1}\left(\sqrt{\frac{\mathfrak w \pa_t w_0(t,\xi-\xi')}{w_0(t,\xi-\xi')}}+\sqrt{\frac{\mathfrak w \pa_t g(t,\xi-\xi')}{g(t,\xi-\xi')}}+t^{-c_1}|\xi-\xi'|^{\frac{s}{2}}\right)\\
  &\qquad\qquad\times A^{\mathrm R}(\xi-\xi') \left|\widehat{v''}(\xi-\xi')_{N} \right||\xi'-lt|\langle\xi'\rangle^{3+2C_1\kappa}e^{c_M \nu^{\frac{1}{3}}t}e^{c\lambda|l,\xi'|^s} \left|\hat\phi_l(\xi')_{<N/8}\right|d\xi'\\
  \lesssim&\sum_{N\ge8}\int_{\xi'} \left(\sqrt{\frac{\mathfrak w \pa_t w_0(t,\xi-\xi')}{w_0(t,\xi-\xi')}}+\sqrt{\frac{\mathfrak w \pa_t g(t,\xi-\xi')}{g(t,\xi-\xi')}}+t^{-c_1}|\xi-\xi'|^{\frac{s}{2}}\right) \frac{\rmA^{\mathrm R}(\xi-\xi')}{\langle \xi-\xi'\rangle} \left|\widehat{v''}(\xi-\xi')_{N}\right|\\
  &\qquad\qquad\times|\xi'-lt|\langle\xi'\rangle^{3+2C_1\kappa}e^{c_M \nu^{\frac{1}{3}}t}e^{c\lambda|l,\xi'|^s}\langle lt\rangle \left|\hat\phi_l(\xi')_{<N/8}\right|d\xi'.
\end{align*}

Then by using Corollary \ref{cor-elliptic-low-2}  we have
\begin{align*}
  \left\|\mathcal M_2{\mathrm T}^{2,v}_{\mathrm {HL}}\right\|_{L^2}^2=&\sum_{l\neq0}\left\|\mathcal M_2{\mathrm T}^{2,v}_{\mathrm {HL}}(l)\right\|_{L^2}^2\lesssim\varepsilon^2\nu^{2\beta} \left(\mathrm{CCK}^{v,2}_{\lambda}+\mathrm{CCK}^{v,2}_{W}+\mathrm{CCK}^{v,2}_{G}\right).
\end{align*}

The estimate for $\mathcal M_3{\mathrm T}^2_{\mathrm {HL}}$ is similar and we omit the details.

In a conclusion, we have
\begin{align*}
   &\left\|\mathcal M_1{\mathrm T}^2_{\mathrm {HL}}\right\|_{L^2}^2+\left\|\mathcal M_2{\mathrm T}^2_{\mathrm {HL}}\right\|_{L^2}^2+\left\|\mathcal M_3{\mathrm T}^2_{\mathrm {HL}}\right\|_{L^2}^2\\
   \lesssim& \varepsilon\nu^{\beta} \left(\left\|\mathcal M_1\Delta_L\phi_{\neq}\right\|_{L^2}^2+\left\|\mathcal M_2\Delta_L\phi_{\neq}\right\|_{L^2}^2+\left\|\mathcal M_3\Delta_L\phi_{\neq}\right\|_{L^2}^2\right)\\
   &+\varepsilon^2\nu^{2\beta}\left(\mathrm{CCK}^{v,2}_{\lambda}+\mathrm{CCK}^{v,2}_{W}+\mathrm{CCK}^{v,2}_{G}\right).
\end{align*}

\subsubsection{Remainders}
The last terms to consider are ${\mathrm T}^1_{\mathcal R}$ and ${\mathrm T}^2_{\mathcal R}$. In these terms, powers of $\pa_v$ can be split evenly between the two factors. However, the same is not true of $l$. For this reason, we treat both remainders as Low-High terms. The difference between ${\mathrm T}^1_{\mathcal R}$ and ${\mathrm T}^2_{\mathcal R}$ here is insignificant since it is straightforward to gain $\langle\pa_v\rangle^{-1}$ for $v''$. Hence we focus only on ${\mathrm T}^1_{\mathcal R}$.

Begin with $\mathcal M_1{\mathrm T}^1_{\mathcal R}$ and divide into two cases based on the relative size of $\xi'$ and $l$,
\begin{align*}
  &\left|\widehat{\mathcal M_1{\mathrm T}^1_{\mathcal R}}(l,\xi)\right|\\
  &\lesssim\sum_{N\in\mathbb D}\sum_{N'\sim N} \int_{\xi'} \left[\mathbf 1_{|l|\ge 100 |\xi'|}+\mathbf 1_{|l|<100|\xi'|}\right]  \mathcal M_1(t,l,\xi) \left|\widehat{\mathcal V}(\xi-\xi')_{N'}\right||\xi'-lt|^2 \left|\hat\phi_l(\xi')_{N}\right|  d\xi'\\
  &=\left|\widehat{\mathcal M_1{\mathrm T}^{1,z}_{\mathcal R}}(l,\xi)\right|+\left|\widehat{\mathcal M_1{\mathrm T}^{1,v}_{\mathcal R}}(l,\xi)\right|.
\end{align*}
Consider first $\left|\widehat{\mathcal M_1{\mathrm T}^{1,z}_{\mathcal R}}(l,\xi)\right|$. On the support of the integrand,
\begin{align*}
  \big||l,\xi|-|l,\xi'|\big|\le |\xi-\xi'|\le \frac{3N'}{2}\le 12N \le \frac{24}{100}\left|l,\xi'\right|,\ |l,\xi|\lesssim|l,\xi'|,\ \left\langle \frac{\xi}{lt} \right\rangle^{-1}\approx\left\langle \frac{\xi'}{lt} \right\rangle^{-1}\approx1.
\end{align*}
By $|\xi|\le |\xi-\xi'|+|\xi'| \le |l|$ we have that,
\begin{align*}
  W_l(t,\xi)=W_l(t,\xi')=\frac{1}{w_l(t,l)},\ G_l(t,\xi)=G_l(t,\xi')=\frac{1}{g(t,l)},
\end{align*}
and then
\begin{align*}
  \left|\widehat{\mathcal M_1{\mathrm T}^{1,z}_{\mathcal R}}(l,\xi)\right|\lesssim& \sum_{N\in\mathbb D}\sum_{N'\sim N}\int_{\xi'} \mathbf 1_{|l|\ge 100|\xi'|}  \left\langle \frac{\xi'}{lt} \right\rangle^{-1}|l,\xi'|^{\frac{s}{2}}t^{-c_1}\rmA_l(\xi) \left|\widehat{\mathcal V}(\xi-\xi')_{N'}\right||\xi'-lt|^2 \left|\hat\phi_l(\xi')_{N}\right| d\xi'\\
  \lesssim&\sum_{N\in\mathbb D}\sum_{N'\sim N} \int_{\xi'} \mathbf 1_{|l|\ge 100|\xi'|} \left|\widehat{\mathcal V}(\xi-\xi')_{N'}\right|\mathcal M_1(t,l,\xi')|\xi'-lt|^2 \left|\hat\phi_l(\xi')_{N}\right|d\xi'.
\end{align*}
Then we have
\begin{align*}
  \left\|\mathcal M_1{\mathrm T}^{1,z}_{\mathcal R}\right\|_{L^2}^2\lesssim&\varepsilon\nu^{\beta} \left\|\mathcal M_1\Delta_L\phi_{\neq}\right\|_{L^2}^2.
\end{align*}
Turn now to $\mathcal M_1{\mathrm T}^{1,v}_{\mathcal R}$. On the support of the integrand, we have
\begin{align*}
  |\xi-\xi'|\le 24|\xi'|\le 24|l,\xi'|,\quad |l,\xi'|\le 101|\xi'|\le2424 \langle \xi-\xi'\rangle,\ |l,\xi|\lesssim |l,\xi'|\approx\langle \xi-\xi'\rangle.
\end{align*}
Therefore, by \eqref{inq-s2} and \eqref{inq-s3}, there holds
\begin{align*}
  |l,\xi|^s\le\big||l,\xi'|+|\xi-\xi'|\big|^s \le |l,\xi'|^s+c|\xi-\xi'|^s.
\end{align*}
Hence, with the fact that $\left\langle \frac{\xi'}{lt} \right\rangle \left\langle \frac{\xi}{lt} \right\rangle^{-1}\lesssim \langle\xi'\rangle\lesssim \langle \xi-\xi'\rangle$, we have
\begin{align*}
 \left|\widehat{\mathcal M_1{\mathrm T}^{1,v}_{\mathcal R}}(l,\xi)\right|\lesssim& \sum_{N\in\mathbb D}\sum_{N'\sim N}\int_{\xi'}\int_{\xi'} \langle \xi-\xi'\rangle^{4+2C_1\kappa} e^{c\lambda|\xi-\xi'|^s} \left|\widehat{\mathcal V}(\xi-\xi')_{N'}\right| \\
 &\qquad\qquad\qquad\qquad\qquad\times\left\langle \frac{\xi'}{lt} \right\rangle^{-1}|l,\xi'|^{\frac{s}{2}}t^{-c_1} \rmA_l(\xi')(\xi'-lt)^2 \left|\hat\phi_l(\xi')_{N}\right|d\xi'\\
  =&\sum_{N\in\mathbb D}\sum_{N'\sim N}\int_{\xi'}\int_{\xi'} \langle \xi-\xi'\rangle^{4+2C_1\kappa} e^{c\lambda|\xi-\xi'|^s} \left|\widehat{\mathcal V}(\xi-\xi')_{N'}\right|\\
   &\qquad\qquad\qquad\qquad\qquad\times\mathcal M_1(t,l,\xi')(\xi'-lt)^2 \left|\hat\phi_l(\xi')_{N}\right|d\xi'.
\end{align*}
Then we have
\begin{align*}
  \left\|\mathcal M_1{\mathrm T}^{1,v}_{\mathcal R}\right\|_{L^2}^2\lesssim&\varepsilon\nu^{\beta} \left\|\mathcal M_1\Delta_L\phi_{\neq}\right\|_{L^2}^2.
\end{align*}

The estimates for $\mathcal M_2{\mathrm T}^1_{\mathcal R}$ and $\mathcal M_3{\mathrm T}^1_{\mathcal R}$ are similar and we omit the details.
\end{proof}

\section{Reaction}\label{sec-reaction}
In this section, we consider the high-low interaction in the nonlinear term, which is the reaction term. Here we heavily use the well-designed Fourier multiplier $\rmA(t,\na)$. 
Focus first on an individual frequency shell and divide each one into several natural pieces
\begin{align*}
  {\mathrm R}_N={\mathrm R}_N^1+{\mathrm R}_N^{\varepsilon,1}+{\mathrm R}_N^2+{\mathrm R}_N^3
\end{align*}
where
\begin{align*}
  {\mathrm R}_N^1=&\sum_{m,k\in\mathbb Z,m\neq0}\int_{\eta,\xi}\rmA_k(\eta) {\hat f}_k(\eta)\rmA_k(\eta)(k\xi-m\eta)\hat \phi_m(\xi)_N\hat f_{k-m}(\eta-\xi)_{<N/8}d\eta d\xi\\
  {\mathrm R}_N^{\varepsilon,1}=&\sum_{m,k\in\mathbb Z,m\neq0}\int_{\eta,\xi}\rmA_k(\eta) {\hat f}_k(\eta)\rmA_k(\eta)[\widehat{h\nabla^\perp \phi}]_m(\xi)_N\cdot \widehat{\nabla f}_{k-m}(\eta-\xi)_{<N/8}d\eta d\xi\\
  {\mathrm R}_N^2=&\sum_{k\in\mathbb Z}\int_{\eta,\xi}\rmA_k(\eta) {\hat f}_k(\eta)\rmA_k(\eta)\hat q(\xi)_N\widehat {\pa_vf}_{k}(\eta-\xi)_{<N/8}d\eta d\xi\\
  {\mathrm R}_N^3=&-\sum_{m,k\in\mathbb Z}\int_{\eta,\xi}\rmA_k(\eta) {\hat f}_k(\eta)\rmA_{k-m}(\eta-\xi)\hat u_m(\xi)_N\cdot\widehat {\nabla f}_{k-m}(\eta-\xi)_{<N/8}d\eta d\xi.
\end{align*}

\subsection{Main contribution}\label{sec-react-main} The main contribution comes from ${\mathrm R}_N^1$. We subdivide this integral depending on whether or not $(m, \xi)$ and/or $(k, \eta)$ are resonant as each combination requires a slightly different treatment.

Define the partition:
\begin{align*}
  1=\mathbf{1}_{t\not\in\tilde{\mathrm{I}}_{k,\eta},t\not\in\tilde{\mathrm{I}}_{m,\xi}}+\mathbf{1}_{t\in\tilde{\mathrm{I}}_{k,\eta},t\not\in\tilde{\mathrm{I}}_{m,\xi}}+\mathbf{1}_{t\not\in\tilde{\mathrm{I}}_{k,\eta},t\in\tilde{\mathrm{I}}_{m,\xi}}+\mathbf{1}_{t\in\tilde{\mathrm{I}}_{k,\eta},t\in\tilde{\mathrm{I}}_{m,\xi}}
\end{align*}
Correspondingly, denote
\begin{align*}
  {\mathrm R}_N^1=&\sum_{m,k\in\mathbb Z,m\neq0}\int_{\eta,\xi}[\mathbf{1}_{t\not\in\tilde{\mathrm{I}}_{k,\eta},t\not\in\tilde{\mathrm{I}}_{m,\xi}}+\mathbf{1}_{t\in\tilde{\mathrm{I}}_{k,\eta},t\not\in\tilde{\mathrm{I}}_{m,\xi}}+\mathbf{1}_{t\not\in\tilde{\mathrm{I}}_{k,\eta},t\in\tilde{\mathrm{I}}_{m,\xi}}+\mathbf{1}_{t\in\tilde{\mathrm{I}}_{k,\eta},t\in\tilde{\mathrm{I}}_{m,\xi}}]\\
  &\times A_k(\eta) {\hat f}_k(\eta)A_k(\eta)(k\xi-m\eta)\hat \phi_m(\xi)_N\hat f_{k-m}(\eta-\xi)_{<N/8}d\eta d\xi\\
  =&{\mathrm R}_{N;\mathrm{NR},\mathrm{NR}}^1+{\mathrm R}_{N;\mathrm{R},\mathrm{NR}}^1+{\mathrm R}_{N;\mathrm{NR},\mathrm{R}}^1+{\mathrm R}_{N;\mathrm{R},\mathrm{R}}^1.
\end{align*}

\subsubsection{Treatment of ${\mathrm R}_{N;\mathrm{NR},\mathrm{NR}}^1$}\label{sec-react-nrnr}

On the support of the integrand of ${\mathrm R}_{N}^1$, it holds that
\begin{align*}
  \frac{N}{2}\le|m,\xi|\le \frac{3N}{2},\quad |k-m,\eta-\xi|\le \frac{3N}{32}
\end{align*}
then
\begin{align*}
  \big||m,\xi|-|k,\eta|\big|\le |k-m,\eta-\xi|\le \frac{6}{32}|m,\xi|.
\end{align*}
This implies that  $|k,\eta|\approx|m,\xi|$, indeed
\begin{align*}
  \frac{26}{32}|m,\xi|\le|k,\eta|\le \frac{38}{32}|m,\xi|.
\end{align*}
Recall that $t\not\in\tilde{\mathrm{I}}_{k,\eta},t\not\in\tilde{\mathrm{I}}_{m,\xi}$. Similar to \eqref{eq-est-ell-lh}, by using Corollary \ref{cor-W}, Lemma \ref{lem-g-exp}, and Lemma \ref{lem-dis-s}, we deduce that
\begin{align*}
  \rmA_k(\eta)=&\rmA_m(\xi)\frac{\rmA_k(\eta)}{\rmA_m(\xi)}=\rmA_m(\xi)e^{(\mathbf{1}_{k \neq 0}-1)c_M \nu^{\frac{1}{3}}t}e^{\lambda|k,\eta|^s-\lambda|m,\xi|^s}\frac{\langle k,\eta\rangle^\sigma}{\langle m,\xi\rangle^\sigma}\frac{W_k(\eta)}{W_m(\xi)}\frac{G_k(\eta)}{G_m(\xi)}\frac{\mathfrak M_k(\eta)}{\mathfrak M_m(\xi)}\\
  \lesssim&\rmA_m(\xi)\langle k-m, \eta-\xi\rangle^{1+2C_1\kappa} e^{c\lambda|k-m,\eta-\xi|^s}.
\end{align*}
Therefore,
\begin{align*}
&  \left|{\mathrm R}_{N;\mathrm{NR},\mathrm{NR}}^1\right|\\\lesssim&\sum_{m,k\in\mathbb Z,m\neq0}\int_{\eta,\xi}\mathbf{1}_{t\not\in\tilde{\mathrm{I}}_{k,\eta},t\not\in\tilde{\mathrm{I}}_{m,\xi}}[\mathbf{1}_D+\mathbf{1}_{D^c}]\rmA_k(\eta) \left|{\hat f}_k(\eta)\right|\rmA_m(\xi)\frac{|m\eta-k\xi|}{m^2+(\xi-mt)^2}\\
  &\qquad\qquad\qquad\times\left|\widehat {\Delta_L\phi}_m(\xi)_N\right|\langle k-m, \eta-\xi\rangle^{1+2C_1\kappa} e^{c\lambda|k-m,\eta-\xi|^s} \left|\hat f_{k-m}(\eta-\xi)_{<N/8}\right| d\eta d\xi\\
  =&{\mathrm R}_{N;\mathrm{NR},\mathrm{NR}}^{1,D}+{\mathrm R}_{N;\mathrm{NR},\mathrm{NR}}^{1,D^c}.
\end{align*}
Here 
\begin{align*}
  D= \left\{(m,\xi): m\xi\ge 0,\ \frac{1}{2}|m|t\le |\xi|\le 2|m|t,\ |\xi|\ge 6|m|,\ t<\min(2|\eta|,2|\xi|)\right\},    
\end{align*}
and ${\mathrm R}_{N;\mathrm{NR},\mathrm{NR}}^{1,D}$ is the difficult part.

We first estimate ${\mathrm R}_{N;\mathrm{NR},\mathrm{NR}}^{1,D^c}$. For the case $m\xi<0$, we do not have resonances for positive times. In this case, if $m\neq k$ and $|\xi|\ge |m|t$ we have
\begin{align*}
  \frac{|m\eta-k\xi|}{m^2+(\xi-mt)^2}\le& \frac{|m,\xi||k-m,\eta-\xi|}{m^2+(\xi-mt)^2}\le\frac{|m,\xi||k-m,\eta-\xi|}{m^2+\xi^2+m^2t^2}\\
  \lesssim&|m,\xi|^{\frac{s}{2}}|k,\eta|^{\frac{s}{2}}\frac{|m,\xi|^{1-s}|k-m,\eta-\xi|}{\xi^2}\\
  \lesssim&|m,\xi|^{\frac{s}{2}}|k,\eta|^{\frac{s}{2}}\frac{ |k-m,\eta-\xi|}{|\xi|^s|m|t}\big\langle \frac{\xi}{mt} \big\rangle^{-1}\\
  \lesssim&\frac{1}{t^{1+s}}|m,\xi|^{\frac{s}{2}}|k,\eta|^{\frac{s}{2}}|k-m,\eta-\xi|\big\langle \frac{\xi}{mt} \big\rangle^{-1}.
\end{align*}
if $m\neq k$ and $|\xi|\le |m|t$ we have
\begin{align*}
  \frac{|m\eta-k\xi|}{m^2+(\xi-mt)^2}\le& \frac{|m,\xi||k-m,\eta-\xi|}{m^2+(\xi-mt)^2}\le\frac{|m,\xi||k-m,\eta-\xi|}{m^2+\xi^2+m^2t^2}\\
  \lesssim&|m,\xi|^{\frac{s}{2}}|k,\eta|^{\frac{s}{2}}\frac{|m,\xi|^{1-s}|k-m,\eta-\xi|}{m^2t^2}\\
  \lesssim&|m,\xi|^{\frac{s}{2}}|k,\eta|^{\frac{s}{2}}\frac{|k-m,\eta-\xi|}{|m|^{1+s}t^{1+s}}\big\langle \frac{\xi}{mt} \big\rangle^{-1}\\
  \lesssim&\frac{1}{t^{1+s}}|m,\xi|^{\frac{s}{2}}|k,\eta|^{\frac{s}{2}}|k-m,\eta-\xi|\big\langle \frac{\xi}{mt} \big\rangle^{-1}.
\end{align*}

If $m=k$ and $|\xi|\ge |m|t$, we have
\begin{align*}
  \frac{|m\eta-k\xi|}{m^2+(\xi-mt)^2}\le& \frac{|m||\xi-\eta|}{m^2+\xi^2+m^2t^2}\lesssim \frac{|m||\xi-\eta|}{|\xi|^2}\frac{|\xi|}{|m|t}\big\langle \frac{\xi}{mt} \big\rangle^{-1}\\
  \lesssim&\frac{|\xi-\eta|}{|m|t^2}\big\langle \frac{\xi}{mt} \big\rangle^{-1}.
\end{align*}
If $m=k$ and $|\xi|\le |m|t$, we have
\begin{align*}
  \frac{|m\eta-k\xi|}{m^2+(\xi-mt)^2}\le& \frac{|m||\xi-\eta|}{m^2+\xi^2+m^2t^2}  \lesssim\frac{|\xi-\eta|}{|m|t^2}\big\langle \frac{\xi}{mt} \big\rangle^{-1}.
\end{align*}

For the case $|\xi|>2|m|t$ or $|\xi|<\frac{1}{2}|m|t$ or $|\xi|< 6|m|$ or $t\ge\min(2|\eta|,2|\xi|)$, it also holds
\begin{align*}
  m^2+(\xi-mt)^2\approx m^2+\xi^2+m^2t^2.
\end{align*}
So one can get the same estimates to the case $m\xi<0$.

Therefore, by using Lemma \ref{lem-conv-Young} and the bootstrap hypotheses \eqref{eq-boot-hf}, \eqref{eq-boot-lf0}, we have
\begin{align*}
  {\mathrm R}_{N;\mathrm{NR},\mathrm{NR}}^{1,D^c}\lesssim& \frac{1}{t^{1+s}}e^{-c_M \nu^{\frac{1}{3}}t}\|A^\gamma f_{\neq}\|_{L^2} \left\||\nabla|^{\frac{s}{2}}\rmA f_{\sim N}\right\|_{L^2}\left\|\left\langle \frac{\pa_v}{t\pa_z} \right\rangle^{-1}\Delta_L|\nabla|^{\frac{s}{2}}\rmA P_{\neq}\phi_{\sim N}\right\|_{L^2}\\
  &+\frac{1}{t^{2}}\|A^\gamma f_{0}\|_{L^2}\left\||\nabla|^{\frac{s}{2}}\rmA f_{\sim N}\right\|_{L^2}\left\|\left\langle \frac{\pa_v}{t\pa_z} \right\rangle^{-1}\Delta_L|\nabla|^{\frac{s}{2}}\rmA P_{\neq}\phi_{\sim N}\right\|_{L^2}\\
  \lesssim&\varepsilon\left(\frac{1}{t^{1+s+3\beta}}+\frac{\nu^\beta}{t^2}\right)\left\||\nabla|^{\frac{s}{2}}\rmA f_{\sim N}\right\|_{L^2}\left\|\left\langle \frac{\pa_v}{t\pa_z} \right\rangle^{-1}\Delta_L|\nabla|^{\frac{s}{2}}\rmA P_{\neq}\phi_{\sim N}\right\|_{L^2}\\
  \lesssim&\varepsilon \left(t^{-2c_1}\left\||\nabla|^{\frac{s}{2}}\rmA f_{\sim N}\right\|_{L^2}^2+ t^{-2c_1}\left\|\left\langle \frac{\pa_v}{t\pa_z} \right\rangle^{-1}\Delta_L|\nabla|^{\frac{s}{2}}\rmA P_{\neq}\phi_{\sim N}\right\|_{L^2}^2\right). 
\end{align*}

Next, we turn to ${\mathrm R}_{N;\mathrm{NR},\mathrm{NR}}^{1,D}$. As $|\xi|\ge 6|m|$, then it holds that
\begin{align*}
  |\eta-\xi|\le \frac{7}{32}|\xi|,\quad |k-m|\le \frac{7}{32}|\xi|,
\end{align*}
then we have
\begin{align*}
  \frac{25}{32}|\xi|\le|\eta|\le\frac{39}{32}|\xi|.
\end{align*}
and
\begin{align*}
  |k|\le \frac{37}{96}|\xi|\le \frac{37}{75}|\eta|.
\end{align*}
Thus we have
\begin{align*}
  w_k(t,\iota(k,\eta))=w_k(t,\eta),\ w_m(t,\iota(m,\xi))=w_m(t,\xi),\\
   g(t,\iota(k,\eta))=g(t,\eta),\ g(t,\iota(m,\xi))=g(t,\xi).
\end{align*}

As $\frac{1}{2}|m|t\le |\xi|\le 2|m|t$ it holds that $\big\langle \frac{\xi}{mt} \big\rangle^{-1}\approx 1$.

Let $j,l$ be such that $t\in \mathrm{I}_{j,\eta}\cap \mathrm{I}_{l,\xi}$.  We consider the following cases depending on the time regime

Case 1, $1\le t\le\max(t_{E(|\xi|),|\xi|},t_{E(|\eta|),|\eta|})$. In this case, we have $t<2$. Thus $|\xi|\ge 2 |m|t$.

Case 2, $\max(t_{E(|\xi|),|\xi|},t_{E(|\eta|),|\eta|})\le t\le \min(t_{E(|\xi|^s),|\xi|},t_{E(|\eta|^s),|\eta|})$. In this case, we have
\begin{align*}
  1\lesssim t\lesssim |\xi|^{1-s}\approx|\eta|^{1-s}, \ |\xi|^{s}\lesssim |l|,|j|\lesssim |\xi|.
\end{align*}

Case 2.1, $m\neq k$. For this case, we have
\begin{align*}
  \frac{|m\eta-k\xi|}{m^2+(\xi-mt)^2}\lesssim \frac{|\xi|}{m^2 \left(1+(\frac{\xi}{m}-t)^2\right)}|k-m,\eta-\xi|,
\end{align*}
so we only need to estimate $\frac{|\xi|}{m^2 \left(1+(\frac{\xi}{m}-t)^2\right)}$. Recall that $\frac{1}{2}|m|t\le|\xi|\le 2 |m|t$, $m\xi\ge0$, and $t\in \mathrm{I}_{j,\eta}\cap \mathrm{I}_{l,\xi}$. It holds that $l\approx m$. For both $l=m$ and $l\neq m$, we have
\begin{align*}
  \frac{|\xi|}{m^2 \left(1+(\frac{\xi}{m}-t)^2\right)}\lesssim\frac{\frac{|\xi|}{l^2}}{\left(1+|\frac{\xi}{l}-t|\right)^2}.
\end{align*}
As $|\eta|\approx|\xi|$ then $|l|\approx|j|$, then we have
\begin{align*}
  \frac{\frac{|\xi|}{l^2}}{\left(1+|\frac{\xi}{l}-t|\right)^2}\lesssim\sqrt{\frac{\frac{|\xi|}{l^2}}{\left(1+(\frac{\xi}{l}-t)^2\right)}}\sqrt{\frac{\frac{|\eta|}{j^2}}{\left(1+(\frac{\eta}{j}-t)^2\right)}}\frac{1+|\frac{\eta}{j}-t|}{1+|\frac{\xi}{l}-t|}.
\end{align*}
By using Lemma \ref{lem-separate} we get
\begin{equation}\label{eq-est-reac-1.2.1}
  \frac{1+|\frac{\eta}{j}-t|}{1+|\frac{\xi}{l}-t|}\lesssim\left\{
    \begin{array}{ll}
      1+ \frac{|\frac{\eta-\xi}{j}|}{1+|\frac{\xi}{l}-t|}\le  \langle\eta-\xi\rangle,& \text{ if }l=j;\\
      \frac{1+|\frac{\eta}{j^2}|}{1+|\frac{\xi}{l^2}|}\lesssim 1\lesssim \langle \eta-\xi\rangle,& \text{ if Case (d) holds};\\
      1+|\frac{\eta}{j^2}|\lesssim \langle \eta-\xi\rangle,& \text{ if Case (e) holds},
    \end{array}
  \right.
\end{equation}
Therefore, 
\begin{align*}
  \frac{\frac{|\xi|}{l^2}}{\left(1+|\frac{\xi}{l}-t|\right)^2}\lesssim\sqrt{\frac{\frac{|\xi|}{l^2}}{\left(1+(\frac{\xi}{l}-t)^2\right)}}\sqrt{\frac{\frac{|\eta|}{j^2}}{\left(1+(\frac{\eta}{j}-t)^2\right)}}\langle \eta-\xi\rangle.
\end{align*}
As $t\le \min(t_{E(|\xi|^s),|\xi|},t_{E(|\eta|^s),|\eta|})$, $\mathfrak w(\nu,t,\eta)=\mathfrak w(\nu,t,\xi)=1$. In this time interval, we have for $|\eta|^{1-s}\ge\nu^{-\frac{1}{3}}$ that
\begin{align*}
  \frac{\pa_t g(t,\eta)}{g(t,\eta)}=\frac{\kappa\mathfrak g(\nu,\eta)\langle \nu^{\frac{1}{3}}\frac{\eta}{j}\rangle^{-\frac{3}{2}}\nu^\beta\frac{|\eta|}{|j|^2}}{1+(t-\frac{|\eta|}{|j|})^2}\approx \langle \nu^{\frac{1}{3}}t\rangle^{-\frac{3}{2}}\nu^{\beta} \frac{ \frac{|\eta| }{|j|^{2}}}{1+(t-\frac{\eta}{j})^2},
\end{align*}
and for $|\eta|^{1-s}\le\nu^{-\frac{1}{3}}$ that
\begin{align*}
  \frac{\pa_t g(t,\eta)}{g(t,\eta)}=\frac{\kappa\mathfrak g(\nu,\eta)\langle \nu^{\frac{1}{3}}\frac{\eta}{j}\rangle^{-\frac{3}{2}}\nu^\beta\frac{|\eta|}{|j|^2}}{1+(t-\frac{|\eta|}{|j|})^2}\approx  \frac{ \frac{|\eta|^{2s}}{|j|^{2}}}{1+(t-\frac{\eta}{j})^2}.
\end{align*}
And similarly, it holds that
\begin{equation*}
  \frac{\pa_t g(t,\xi)}{g(t,\xi)}\approx\left\{
    \begin{array}{ll}
      \langle \nu^{\frac{1}{3}}t\rangle^{-\frac{3}{2}} \frac{ \nu^{\beta}\frac{|\xi| }{|l|^{2}}}{1+(t-\frac{\xi}{l})^2},&|\eta|^{1-s}\ge\nu^{-\frac{1}{3}},\\
      \frac{ \frac{|\xi|^{2s}}{|l|^{2}}}{1+(t-\frac{\xi}{l})^2},&|\eta|^{1-s}\le\nu^{-\frac{1}{3}}.
    \end{array}
  \right.
\end{equation*}
Then we have for $|\eta|^{1-s}\ge\nu^{-\frac{1}{3}}$ that
\begin{align*}
  \frac{\frac{|\xi|}{l^2}}{\left(1+|\frac{\xi}{l}-t|\right)^2}\lesssim \nu^{-\beta} \langle \nu^{\frac{1}{3}}t\rangle^{\frac{3}{2}} \sqrt{\frac{ \pa_t g(t,\eta)}{g(t,\eta)}}\sqrt{\frac{\mathfrak w\pa_t g(t,\xi)}{g(t,\xi)}}\langle \eta-\xi\rangle,
\end{align*}
and for $|\eta|^{1-s}\le\nu^{-\frac{1}{3}}$ that
\begin{align*}
  \frac{\frac{|\xi|}{l^2}}{\left(1+|\frac{\xi}{l}-t|\right)^2}\lesssim& |\xi|^{1-2s}\sqrt{\frac{\pa_t g(t,\eta)}{g(t,\eta)}}\sqrt{\frac{\pa_t g(t,\xi)}{g(t,\xi)}}\langle \eta-\xi\rangle\\
  \le& \nu^{-\beta}\sqrt{\frac{ \pa_t g(t,\eta)}{g(t,\eta)}}\sqrt{\frac{\mathfrak w\pa_t g(t,\xi)}{g(t,\xi)}}\langle \eta-\xi\rangle.
\end{align*}

Case 2.2, $m=k$. For this case, we have 
\begin{align}\label{eq-reac-m=k}
   \frac{|m\eta-k\xi|}{m^2+(\xi-mt)^2}=\frac{\frac{1}{|m|}}{1+(\frac{\xi}{m}-t)^2}|\eta-\xi|,
\end{align}
and we only need to estimate $\frac{\frac{1}{|m|}}{1+(\frac{\xi}{m}-t)^2}$. Here we remark that the term $|\eta-\xi|$ in \eqref{eq-reac-m=k} is important in the estimate with $m=k$.

As $\frac{1}{2}|m|t\le|\xi|\le 2 |m|t$, we have
\begin{align*}
  \frac{\frac{1}{|m|}}{1+(\frac{\xi}{m}-t)^2}\lesssim \frac{\frac{1}{|l|}}{\left(1+|\frac{\xi}{l}-t|\right)^2}.
\end{align*}
Then, we have for $|\eta|^{1-s}\ge\nu^{-\frac{1}{3}}$ that
\begin{align*}
  \frac{\frac{1}{|l|}}{\left(1+|\frac{\xi}{l}-t|\right)^2}\lesssim&   \langle \nu^{\frac{1}{3}}t\rangle^{3/2}\nu^{-\beta}t^{-1} \sqrt{\frac{\pa_t g(t,\eta)}{g(t,\eta)}}\sqrt{\frac{\pa_t g(t,\xi)}{g(t,\xi)}}\langle \eta-\xi\rangle\\
  \lesssim&\nu^{-\beta}\langle\nu t \rangle^{\frac{1}{2}} \sqrt{\frac{\pa_t g(t,\eta)}{g(t,\eta)}}\sqrt{\frac{\mathfrak w\pa_t g(t,\xi)}{g(t,\xi)}}\langle \eta-\xi\rangle,
\end{align*}
and for $|\eta|^{1-s}\le\nu^{-\frac{1}{3}}$ 
\begin{align*}
  \frac{\frac{1}{|l|}}{\left(1+|\frac{\xi}{l}-t|\right)^2}\lesssim &\frac{|l|}{|\xi|^{2s}}\sqrt{\frac{\pa_t g(t,\eta)}{g(t,\eta)}}\sqrt{\frac{\pa_t g(t,\xi)}{g(t,\xi)}}\langle \eta-\xi\rangle\\
  \lesssim&|\xi|^{1-2s}\sqrt{\frac{\pa_t g(t,\eta)}{g(t,\eta)}}\sqrt{\frac{\pa_t g(t,\xi)}{g(t,\xi)}}\langle \eta-\xi\rangle\le \nu^{-\beta}\sqrt{\frac{ \pa_t g(t,\eta)}{g(t,\eta)}}\sqrt{\frac{\mathfrak w\pa_t g(t,\xi)}{g(t,\xi)}}\langle \eta-\xi\rangle.
\end{align*}

Case 3: $\min(t_{E(|\xi|^s),|\xi|},t_{E(|\eta|^s),|\eta|})\le t\le \max(t_{E(|\xi|^s),|\xi|},t_{E(|\eta|^s),|\eta|})$. In this case, we have
\begin{align*}
  t\approx |\xi|^{1-s}\approx|\eta|^{1-s}, \ |l|\approx|j|\approx|\xi|^s\approx|\eta|^s,\ \frac{|\eta|^{1-3\beta}}{|j|^{2-3\beta}}\approx\frac{|\xi|^{1-3\beta}}{|l|^{2-3\beta}}\approx1.
\end{align*}
For this time interval, the proof is very similar to Case 2. Here the expressions of $\frac{\pa_t g(t,\eta)}{g(t,\eta)}$ and $\frac{\pa_t g(t,\xi)}{g(t,\xi)}$ may be different. Without loss of generality, we assume that $t_{E(|\xi|^s),|\xi|}\le t_{E(|\eta|^s),|\eta|}$. We have
\begin{align*}
  \frac{\pa_t g(t,\xi)}{g(t,\xi)}=\frac{\kappa \frac{|\xi|^{1-3\beta}}{|l|^{2-3\beta}}}{\left(\frac{|\xi|^{1-3\beta}}{|l|^{2-3\beta}}\right)^2+(t-\frac{|\xi|}{|l|})^2}\approx \frac{1}{1+(t-\frac{|\xi|}{|l|})^2},
\end{align*}
and
\begin{equation*}
  \frac{\pa_t g(t,\eta)}{g(t,\eta)}\approx\left\{
    \begin{array}{ll}
      \langle \nu^{\frac{1}{3}}t\rangle^{-\frac{3}{2}} \frac{ \nu^{\beta}\frac{|\eta| }{|j|^{2}}}{1+(t-\frac{\eta}{j})^2},&|\eta|^{1-s}\ge\nu^{-\frac{1}{3}},\\
      \frac{1}{1+(t-\frac{\eta}{j})^2},&|\eta|^{1-s}\le\nu^{-\frac{1}{3}}.
    \end{array}
  \right.
\end{equation*}

Case 3.1, $k\neq m$. We have for $|\eta|^{1-s}\ge\nu^{-\frac{1}{3}}$ that $\mathfrak w(\nu,t,\xi)=\langle\nu^{\frac{1}{3}}|\xi|^{1-s}\rangle^{-\frac{3}{2}+3\beta}$ and
\begin{align*}
  \frac{\frac{|\xi|}{l^2}}{\left(1+|\frac{\xi}{l}-t|\right)^2}\lesssim& \langle t\rangle^{\frac{3}{2}\beta}\nu^{-\frac{1}{2}\beta} \langle \nu^{\frac{1}{3}}t\rangle^{\frac{3}{4}} \sqrt{\frac{\pa_t g(t,\eta)}{g(t,\eta)}}\sqrt{\frac{\pa_t g(t,\xi)}{g(t,\xi)}}\langle \eta-\xi\rangle\\
  \lesssim&\langle\nu^{\frac{1}{3}}|\xi|^{1-s}\rangle^{\frac{3}{4}-\frac{3}{2}\beta}\langle t\rangle^{\frac{3}{2}\beta}\nu^{-\frac{1}{2}\beta} \langle \nu^{\frac{1}{3}}t\rangle^{\frac{3}{4}} \sqrt{\frac{\pa_t g(t,\eta)}{g(t,\eta)}}\sqrt{\frac{\mathfrak w\pa_t g(t,\xi)}{g(t,\xi)}}\langle \eta-\xi\rangle\\
  \lesssim&\langle\nu^{\frac{1}{3}}t\rangle^{\frac{3}{2}-\frac{3}{2}\beta}\langle t\rangle^{\frac{3}{2}\beta}\nu^{-\frac{1}{2}\beta} \sqrt{\frac{\pa_t g(t,\eta)}{g(t,\eta)}}\sqrt{\frac{\mathfrak w\pa_t g(t,\xi)}{g(t,\xi)}}\langle \eta-\xi\rangle,
\end{align*}
and for $|\eta|^{1-s}\le\nu^{-\frac{1}{3}}$ that $\mathfrak w=1$ and
\begin{align*}
  \frac{\frac{|\xi|}{l^2}}{\left(1+|\frac{\xi}{l}-t|\right)^2}\lesssim& |\xi|^{1-2s}\sqrt{\frac{\pa_t g(t,\eta)}{g(t,\eta)}}\sqrt{\frac{\pa_t g(t,\xi)}{g(t,\xi)}}\langle \eta-\xi\rangle\le \nu^{-\beta}\sqrt{\frac{\pa_t g(t,\eta)}{g(t,\eta)}}\sqrt{\frac{\pa_t g(t,\xi)}{g(t,\xi)}}\langle \eta-\xi\rangle\\
  \lesssim& \nu^{-\beta}\sqrt{\frac{\pa_t g(t,\eta)}{g(t,\eta)}}\sqrt{\frac{\mathfrak w\pa_t g(t,\xi)}{g(t,\xi)}}\langle \eta-\xi\rangle
\end{align*}

Case 3.2, $k=m$. We have for $|\eta|^{1-s}\ge\nu^{-\frac{1}{3}}$ that
\begin{align*}
  \frac{\frac{1}{|l|}}{\left(1+|\frac{\xi}{l}-t|\right)^2}\lesssim& |\xi|^{-\frac{1}{2}}\nu^{-\frac{1}{2}\beta} \langle \nu^{\frac{1}{3}}t\rangle^{\frac{3}{4}} \sqrt{\frac{\pa_t g(t,\eta)}{g(t,\eta)}}\sqrt{\frac{\pa_t g(t,\xi)}{g(t,\xi)}}\langle \eta-\xi\rangle\\
  \lesssim& \langle\nu^{\frac{1}{3}}|\xi|^{1-s}\rangle^{\frac{3}{4}-\frac{3}{2}\beta}|\xi|^{-\frac{1}{2}}\nu^{-\frac{1}{2}\beta} \langle \nu^{\frac{1}{3}}t\rangle^{\frac{3}{4}} \sqrt{\frac{\pa_t g(t,\eta)}{g(t,\eta)}}\sqrt{\frac{\mathfrak w\pa_t g(t,\xi)}{g(t,\xi)}}\langle \eta-\xi\rangle\\
  \lesssim&\nu^{-\beta}\nu^{\frac{1}{2}}t^{\frac{1}{2}}\sqrt{\frac{\pa_t g(t,\eta)}{g(t,\eta)}}\sqrt{\frac{\mathfrak w\pa_t g(t,\xi)}{g(t,\xi)}}\langle \eta-\xi\rangle,
\end{align*}
and for $|\eta|^{1-s}\le\nu^{-\frac{1}{3}}$ that
\begin{align*}
  \frac{\frac{1}{|l|}}{\left(1+|\frac{\xi}{l}-t|\right)^2}\lesssim& \frac{1}{|l|}\sqrt{\frac{\pa_t g(t,\eta)}{g(t,\eta)}}\sqrt{\frac{\pa_t g(t,\xi)}{g(t,\xi)}}\langle \eta-\xi\rangle\lesssim\sqrt{\frac{\pa_t g(t,\eta)}{g(t,\eta)}}\sqrt{\frac{\mathfrak w\pa_t g(t,\xi)}{g(t,\xi)}}\langle \eta-\xi\rangle.
\end{align*}

Case 4, $\max(t_{E(|\xi|^s),|\xi|},t_{E(|\eta|^s),|\eta|}) \le t\le 2\min(|\xi|,|\eta|)$.

Case 4.1, $m\neq k$. Recall that $t\not\in\tilde{\mathrm{I}}_{k,\eta},t\not\in\tilde{\mathrm{I}}_{m,\xi}$ and $t\in \mathrm{I}_{j,\eta}\cap \mathrm{I}_{l,\xi}$. We have 
\begin{align*}
  \frac{|\xi|}{m^2 \left(1+(\frac{\xi}{m}-t)^2\right)}\lesssim&\sqrt{\frac{\frac{|\xi|}{l^2}}{\left(\frac{|\xi|^{1-3\beta}}{|l|^{2-3\beta}}\right)^2+(\frac{\xi}{l}-t)^2}}\sqrt{\frac{\frac{|\eta|}{j^2}}{\left(\frac{|\eta|^{1-3\beta}}{|j|^{2-3\beta}}\right)^2+(\frac{\eta}{j}-t)^2}}\langle \eta-\xi\rangle\\
  \lesssim&\langle t\rangle^{3\beta} \sqrt{\frac{\pa_t g(t,\eta)}{g(t,\eta)}}\sqrt{\frac{\pa_t g(t,\xi)}{g(t,\xi)}}\langle \eta-\xi\rangle\\
  \lesssim&\langle\nu^{\frac{1}{3}}|\xi|^{1-s}\rangle^{\frac{3}{4}-\frac{3}{2}\beta}\langle t\rangle^{3\beta} \sqrt{\frac{\pa_t g(t,\eta)}{g(t,\eta)}}\sqrt{\frac{\mathfrak w\pa_t g(t,\xi)}{g(t,\xi)}}\langle \eta-\xi\rangle\\
  \lesssim&\langle\nu^{\frac{1}{3}}t\rangle^{\frac{3}{4}-\frac{3}{2}\beta}\langle t\rangle^{3\beta} \sqrt{\frac{\pa_t g(t,\eta)}{g(t,\eta)}}\sqrt{\frac{\mathfrak w\pa_t g(t,\xi)}{g(t,\xi)}}\langle \eta-\xi\rangle.
\end{align*}

Case 4.2,  $m= k$. It holds that
\begin{align*}
  \frac{\frac{1}{|m|}}{\left(1+(\frac{\xi}{m}-t)^2\right)}\lesssim&t^{-1+3\beta}\sqrt{\frac{\frac{|\xi|^{1-3\beta}}{|l|^{2-3\beta}}}{\left(\frac{|\xi|^{1-3\beta}}{|l|^{2-3\beta}}\right)^2+(\frac{\xi}{l}-t)^2}}\sqrt{\frac{\frac{|\eta|^{1-3\beta}}{|j|^{2-3\beta}}}{{\left(\frac{|\eta|^{1-3\beta}}{|j|^{2-3\beta}}\right)^2+(\frac{\eta}{j}-t)^2}}}\langle \eta-\xi\rangle\\
  \lesssim&t^{-1+3\beta}\sqrt{\frac{\pa_t g(t,\eta)}{g(t,\eta)}}\sqrt{\frac{\pa_t g(t,\xi)}{g(t,\xi)}}\langle \eta-\xi\rangle\\
  \lesssim&\langle\nu^{\frac{1}{3}}|\xi|^{1-s}\rangle^{\frac{3}{4}-\frac{3}{2}\beta}t^{-1+3\beta}\sqrt{\frac{\pa_t g(t,\eta)}{g(t,\eta)}}\sqrt{\frac{\mathfrak w\pa_t g(t,\xi)}{g(t,\xi)}}\langle \eta-\xi\rangle.
\end{align*}
Therefore, we have for  $|\eta|^{1-s}\le \nu^{-\frac{1}{3}}$ that
\begin{align*}
  \frac{\frac{1}{|m|}}{\left(1+(\frac{\xi}{m}-t)^2\right)}\lesssim&\sqrt{\frac{\pa_t g(t,\eta)}{g(t,\eta)}}\sqrt{\frac{\mathfrak w\pa_t g(t,\xi)}{g(t,\xi)}}\langle \eta-\xi\rangle,
\end{align*}
and for $|\eta|^{1-s}\ge \nu^{-\frac{1}{3}}$ that
\begin{align*}
  \frac{\frac{1}{|m|}}{\left(1+(\frac{\xi}{m}-t)^2\right)}
  \lesssim&\langle\nu^{\frac{1}{3}}|\xi|^{1-s}\rangle^{\frac{3}{4}-\frac{3}{2}\beta}t^{-1+3\beta}\sqrt{\frac{\pa_t g(t,\eta)}{g(t,\eta)}}\sqrt{\frac{\mathfrak w\pa_t g(t,\xi)}{g(t,\xi)}}\langle \eta-\xi\rangle\\
  \lesssim&\nu^{\frac{1}{4}-\frac{1}{2}\beta} t^{\frac{3}{2}\beta-\frac{1}{4}}\sqrt{\frac{\pa_t g(t,\eta)}{g(t,\eta)}}\sqrt{\frac{\mathfrak w\pa_t g(t,\xi)}{g(t,\xi)}}\langle \eta-\xi\rangle\\
  =&(\nu^\frac{1}{2}t^{\frac{1}{2}})^{3\beta-\frac{1}{2}}\nu^{\frac{1}{2}-2\beta}\sqrt{\frac{\pa_t g(t,\eta)}{g(t,\eta)}}\sqrt{\frac{\mathfrak w\pa_t g(t,\xi)}{g(t,\xi)}}\langle \eta-\xi\rangle\\
  \le&\nu^{-\beta} \langle\nu t \rangle^{\frac{1}{2}}\sqrt{\frac{\pa_t g(t,\eta)}{g(t,\eta)}}\sqrt{\frac{\mathfrak w\pa_t g(t,\xi)}{g(t,\xi)}}\langle \eta-\xi\rangle.
\end{align*}

Combing the above estimates, we have
\begin{align*}
  {\mathrm R}_{N;\mathrm{NR},\mathrm{NR}}^{1,D}\lesssim&\left[\left(\nu^{-\beta}+\langle\nu^{\frac{1}{3}}t\rangle^{\frac{3}{2}-\frac{3}{2}\beta}\langle t\rangle^{\frac{3}{2}\beta}\nu^{-\frac{1}{2}\beta}+\langle\nu^{\frac{1}{3}}t\rangle^{\frac{3}{4}-\frac{3}{2}\beta}\langle t\rangle^{3\beta} \right) e^{-c_M \nu^{\frac{1}{3}}t}\|A^\gamma f_{\neq}\|_{L^2}\right. \\
  &+\nu^{-\beta}\langle\nu t \rangle^{\frac{1}{2}}\|A^\gamma \pa_vf_{0}\|_{L^2}\bigg]\left\|\sqrt{\frac{ \pa_t g}{g}}\rmA f_{\sim N}\right\|_{L^2}\left\|\left\langle \frac{\pa_v}{t\pa_z} \right\rangle^{-1}\Delta_L\sqrt{\frac{\mathfrak w \pa_t g}{g}}\rmA P_{\neq}\phi_{\sim N}\right\|_{L^2}\\
  \lesssim&\varepsilon \left( \left\|\sqrt{\frac{ \pa_t g}{g}}\rmA f_{\sim N}\right\|_{L^2}^2+ \left\|\left\langle \frac{\pa_v}{t\pa_z} \right\rangle^{-1}\Delta_L\sqrt{\frac{\mathfrak w \pa_t g}{g}}\rmA P_{\neq}\phi_{\sim N}\right\|_{L^2}^2\right).
\end{align*}

\subsubsection{Treatment of ${\mathrm R}_{N;\mathrm{R},\mathrm{NR}}^1$}
Recall that $t\in\tilde{\mathrm{I}}_{k,\eta},t\not\in\tilde{\mathrm{I}}_{m,\xi}$. Similar to \eqref{eq-est-ell-lh}, by using Corollary \ref{cor-W}, Lemma \ref{lem-g-exp}, and Lemma \ref{lem-dis-s}, we deduce that
\begin{align*}
  \rmA_k(\eta)=&\rmA_m(\xi)\frac{\rmA_k(\eta)}{\rmA_m(\xi)}\lesssim \rmA_m(\xi)\frac{\frac{|\eta|^{1-3\beta}}{|k|^{2-3\beta}}}{\left[1+\left|t-\frac{\eta}{k}\right|\right]}\langle k-m, \eta-\xi\rangle^{1+2C_1\kappa} e^{c\lambda|k-m,\eta-\xi|^s},
\end{align*}
and
\begin{align*}
  \left|{\mathrm R}_{N;\mathrm{R},\mathrm{NR}}^1\right|\lesssim&\sum_{m,k\in\mathbb Z,m\neq0}\int_{\eta,\xi}\mathbf{1}_{t \in\tilde{\mathrm{I}}_{k,\eta},t\not\in\tilde{\mathrm{I}}_{m,\xi}} \rmA_k(\eta) \left|{\hat f}_k(\eta)\right|\rmA_m(\xi)\frac{|m\eta-k\xi|}{m^2+(\xi-mt)^2}\frac{\frac{|\eta|^{1-3\beta}}{|k|^{2-3\beta}}}{\left[1+\left|t-\frac{\eta}{k}\right|\right]}\\
  &\qquad\qquad\qquad\times\left|\widehat {\Delta_L\phi}_m(\xi)_N\right|\langle k-m, \eta-\xi\rangle^{1+2C_1\kappa} e^{c\lambda|k-m,\eta-\xi|^s} \left|\hat f_{k-m}(\eta-\xi)_{<N/8}\right| d\eta d\xi.
\end{align*}
Then we estimate 
\begin{align*}
  \frac{|m\eta-k\xi|}{m^2+(\xi-mt)^2}\frac{\frac{|\eta|^{1-3\beta}}{|k|^{2-3\beta}}}{\left[1+\left|t-\frac{\eta}{k}\right|\right]}.
\end{align*}

As $t\in \tilde{\mathrm{I}}_{k,\eta}$, we have $\eta k>0$ and $|k|\le |\eta|^s$. From  $\frac{26}{32}|m,\xi|\le|k,\eta|\le \frac{38}{32}|m,\xi|$ and $|k-m,\eta-\xi|\le \frac{6}{32}|m,\xi|$ we can see that $|\eta|\approx |\xi|$ and $|m|\le |\xi|$. 

If $\frac{1}{2}|\eta|^s\le|k|\le E(|\eta|^s)$, then $t\approx|\xi|^{1-s}\approx|\eta|^{1-s}$, and
\begin{align*}
  \frac{\frac{|\eta|^{1-3\beta}}{|k|^{2-3\beta}}}{\left[1+\left|t-\frac{\eta}{k}\right|\right]}\approx 1.
\end{align*}
Then one can get the estimate in a similar way to Case 3 in Section \ref{sec-react-nrnr}. So we consider the following cases depending on the time regime under the assumption that $|k|<\frac{1}{2}|\eta|^s$. 

Case 1, $|\xi|\le \frac{1}{2}|m|t$. It holds that $|m|\gtrsim |k|$, $|\xi|^{1-s}\lesssim t\lesssim |\xi|$ and
\begin{align*}
  |t-\frac{\xi}{m}|\ge \frac{|\xi|}{|m|},\quad |t-\frac{\xi}{m}|\ge \frac{1}{2}t.
\end{align*}

In this case, we must have $m\neq k$. Assume $t\in \tilde{\mathrm{I}}_{k,\eta}\cap {\mathrm{I}}_{l,\xi}$ and recall that $t\not\in\tilde{\mathrm{I}}_{m,\xi}$. If $t\in \tilde{\mathrm{I}}_{k,\eta}\cap  \tilde{\mathrm{I}}_{l,\xi}$, we have
\begin{align*}
  &\frac{\frac{|\eta|^{1-3\beta}}{|k|^{2-3\beta}}}{\left[1+\left|t-\frac{\eta}{k}\right|\right]}\frac{|\xi|}{m^2+(\xi-mt)^2}\\
  \lesssim&\frac{|\eta|^{1-3\beta}}{|k|^{2-3\beta}} 
  \sqrt{\frac{1}{1+\left|t-\frac{\eta}{k}\right|}}
  \sqrt{\frac{1}{1+\left|t-\frac{\xi}{l}\right|}}
  \sqrt{\frac{1+\left|t-\frac{\xi}{l}\right|}{1+\left|t-\frac{\eta}{k}\right|}}\frac{|\xi|}{m^2+(\xi-mt)^2}\\
  \lesssim&\sqrt{\frac{\pa_t w_k(t,\eta)}{w_k(t,\eta)}}\sqrt{\frac{\pa_t w_m(t,\xi)}{w_m(t,\xi)}} \frac{\frac{\xi^{2-3\beta}}{k^{2-3\beta}}\langle\eta-\xi\rangle}{m^2t^2}\\
   \lesssim&\sqrt{\frac{\pa_t w_k(t,\eta)}{w_k(t,\eta)}}\sqrt{\frac{\pa_t w_m(t,\xi)}{w_m(t,\xi)}} \langle\eta-\xi\rangle\big\langle \frac{\xi}{mt} \big\rangle^{-1}\\
   \lesssim&\langle\nu^{\frac{1}{3}}t\rangle^{\frac{3}{4}-\frac{3}{2}\beta}\sqrt{\frac{\pa_t w_k(t,\eta)}{w_k(t,\eta)}}\sqrt{\frac{\mathfrak w\pa_t w_m(t,\xi)}{w_m(t,\xi)}}\langle\xi-\eta\rangle\big\langle \frac{\xi}{mt} \big\rangle^{-1},
\end{align*}
and if $t\in \tilde{\mathrm{I}}_{k,\eta}\cap  \left(\mathrm{I}_{l,\xi}\setminus\tilde{\mathrm{I}}_{l,\xi}\right)$,  $\frac{|\xi|^{1-3\beta}}{|l|^{2-3\beta}}\lesssim|t-\frac{\xi}{l}|\lesssim \frac{|\xi|}{|l|^2}$, we have
\begin{align*}
&\frac{\frac{|\eta|^{1-3\beta}}{|k|^{2-3\beta}}}{\left[1+\left|t-\frac{\eta}{k}\right|\right]}\frac{|\xi|}{m^2+(\xi-mt)^2}\\
  \lesssim& \sqrt{\frac{1}{1+\left|t-\frac{\eta}{k}\right|}}
  \sqrt{\frac{\frac{|\xi|^{1-3\beta}}{|l|^{2-3\beta}}}{\left[1+\left|t-\frac{\xi}{l}\right|^2\right]}}
  \sqrt{\frac{\frac{|\eta|^{1-3\beta}}{|k|^{2-3\beta}}\left[1+\left|t-\frac{\xi}{l}\right|^2\right]}{1+\left|t-\frac{\eta}{k}\right|}}
  \frac{|\xi|}{m^2+(\xi-mt)^2}\\
 \lesssim& \sqrt{\frac{1}{1+\left|t-\frac{\eta}{k}\right|}}
  \sqrt{\frac{\frac{|\xi|^{1-3\beta}}{|l|^{2-3\beta}}}{\left[1+\left|t-\frac{\xi}{l}\right|^2\right]}}
  \frac{\frac{|\xi|^{2-\frac{3}{2}\beta}}{|l|^{2-\frac{3}{2}\beta}}\langle\xi-\eta\rangle}{m^2t^2}\\
  \lesssim& \sqrt{\frac{\pa_t w_k(t,\eta)}{w_k(t,\eta)}}\sqrt{\frac{ \pa_t g(t,\xi)}{g(t,\xi)}}\langle\xi-\eta\rangle\big\langle \frac{\xi}{mt} \big\rangle^{-1}\\
   \lesssim&\langle\nu^{\frac{1}{3}}t\rangle^{\frac{3}{4}-\frac{3}{2}\beta}\sqrt{\frac{\pa_t w_k(t,\eta)}{w_k(t,\eta)}}\sqrt{\frac{\mathfrak w\pa_t g(t,\xi)}{g(t,\xi)}}\langle\xi-\eta\rangle\big\langle \frac{\xi}{mt} \big\rangle^{-1}.
\end{align*}

Case 2, $|\xi|\ge 2|m|t$. We have
\begin{align*}
  \frac{\frac{|\eta|^{1-3\beta}}{|k|^{2-3\beta}}}{\left[1+\left|t-\frac{\eta}{k}\right|\right]}\frac{|\xi|}{m^2+(\xi-mt)^2}\lesssim \frac{\frac{|\eta|^{1-3\beta}}{|k|^{2-3\beta}}}{\left[1+\left|t-\frac{\eta}{k}\right|\right]}\frac{1}{|\xi|}\lesssim\frac{\frac{|\eta|^{1-3\beta}}{|k|^{2-3\beta}}}{\left[1+\left|t-\frac{\eta}{k}\right|\right]} \frac{1}{|m|t}\big\langle \frac{\xi}{mt} \big\rangle^{-1}.
\end{align*}
For this case, we also have $k\neq m$. Then we can get a similar estimate to Case 1.

Case 3, $\frac{1}{2}|m|t\le |\xi|\le 2|m|t$. If $l\neq m$ and $m\neq k$, we have 
\begin{align*}
  \frac{|\xi|}{m^2+(\xi-mt)^2}\lesssim \frac{|\xi|}{|m|^2(\frac{|\xi|}{|m|^2})^2}\approx\frac{|\xi|}{t^2}\big\langle \frac{\xi}{mt} \big\rangle^{-1}.
\end{align*}
And the estimate is similar to Case 1.

If $l=m\neq k$, recalling that $t\not\in \tilde{\mathrm{I}}_{l,\xi}$, we have 
\begin{align*}
  &\frac{\frac{|\eta|^{1-3\beta}}{|k|^{2-3\beta}}}{\left[1+\left|t-\frac{\eta}{k}\right|\right]}\frac{|\xi|}{m^2+(\xi-mt)^2}\\
  \lesssim& \sqrt{\frac{1}{1+\left|t-\frac{\eta}{k}\right|}}
  \sqrt{\frac{\frac{|\xi|^{1-3\beta}}{|l|^{2-3\beta}}}{\left[1+\left|t-\frac{\xi}{l}\right|^2\right]}}
  \sqrt{\frac{\frac{|\eta|^{1-3\beta}}{|k|^{2-3\beta}}\left[1+\left|t-\frac{\xi}{l}\right|^2\right]}{1+\left|t-\frac{\eta}{k}\right|}}
  \frac{|\xi|}{m^2+(\xi-mt)^2}\\
 \lesssim& \sqrt{\frac{1}{1+\left|t-\frac{\eta}{k}\right|}}
  \sqrt{\frac{\frac{|\xi|^{1-3\beta}}{|l|^{2-3\beta}}}{\left[1+\left|t-\frac{\xi}{l}\right|^2\right]}}
  \frac{\frac{|\xi|^{\frac{3}{2}-\frac{3}{2}\beta}}{|l|^{3-\frac{3}{2}\beta}}\langle\xi-\eta\rangle}{1+\left|\frac{\xi}{m}-t\right|^{\frac{3}{2}}}\\
  \lesssim& \sqrt{\frac{1}{1+\left|t-\frac{\eta}{k}\right|}}
  \sqrt{\frac{\frac{|\xi|^{1-3\beta}}{|l|^{2-3\beta}}}{\left[1+\left|t-\frac{\xi}{l}\right|^2\right]}}
  \frac{\frac{|\xi|^{\frac{3}{2}-\frac{3}{2}\beta}}{|l|^{3-\frac{3}{2}\beta}}\langle\xi-\eta\rangle}{\left|\frac{|\xi|^{1-3\beta}}{|l|^{2-3\beta}}\right|^{\frac{3}{2}}}\\
   \lesssim&\sqrt{\frac{\pa_t w_k(t,\eta)}{w_k(t,\eta)}}\sqrt{\frac{\pa_t g(t,\xi)}{g(t,\xi)}}\langle\xi-\eta\rangle\big\langle \frac{\xi}{mt} \big\rangle^{-1}\left|\frac{\xi}{l}\right|^{3\beta}\\
   \lesssim&\langle\nu^{\frac{1}{3}}t\rangle^{\frac{3}{4}-\frac{3}{2}\beta}t^{3\beta}\sqrt{\frac{\pa_t w_k(t,\eta)}{w_k(t,\eta)}}\sqrt{\frac{\mathfrak w\pa_t g(t,\xi)}{g(t,\xi)}} \langle\xi-\eta\rangle\big\langle \frac{\xi}{mt} \big\rangle^{-1}.
\end{align*}

If $m=k$, then we need to estimate
\begin{align*}
  \frac{\frac{|\eta|^{1-3\beta}}{|k|^{2-3\beta}}}{\left[1+\left|t-\frac{\eta}{k}\right|\right]}\frac{\frac{1}{|m|}}{1+(\frac{\xi}{m}-t)^2}|\eta-\xi|.
\end{align*}
Recall that $t\in \tilde{\mathrm{I}}_{k,\eta}\cap \mathrm{I}_{l,\xi}$, then we have
\begin{align*}
  &\frac{\frac{|\eta|^{1-3\beta}}{|k|^{2-3\beta}}}{\left[1+\left|t-\frac{\eta}{k}\right|\right]}\frac{\frac{1}{|m|}}{1+(\frac{\xi}{m}-t)^2}\\
  \lesssim& \sqrt{\frac{1}{1+\left|t-\frac{\eta}{k}\right|}}
  \sqrt{\frac{\frac{|\xi|^{1-3\beta}}{|l|^{2-3\beta}}}{\left[1+\left|t-\frac{\xi}{l}\right|^2\right]}}
  \sqrt{\frac{\frac{|\eta|^{1-3\beta}}{|k|^{2-3\beta}}\left[1+\left|t-\frac{\xi}{l}\right|^2\right]}{1+\left|t-\frac{\eta}{k}\right|}}
  \frac{\frac{1}{|m|}}{1+(\frac{\xi}{m}-t)^2}\\
 \lesssim& \sqrt{\frac{1}{1+\left|t-\frac{\eta}{k}\right|}}
  \sqrt{\frac{\frac{|\xi|^{1-3\beta}}{|l|^{2-3\beta}}}{\left[1+\left|t-\frac{\xi}{l}\right|^2\right]}}
  \frac{\frac{|\xi|^{\frac{1}{2}-\frac{3}{2}\beta}}{|l|^{2-\frac{3}{2}\beta}}\langle\xi-\eta\rangle}{1+\left|\frac{\xi}{m}-t\right|^{\frac{3}{2}}}\\
  \lesssim& \sqrt{\frac{1}{1+\left|t-\frac{\eta}{k}\right|}}
  \sqrt{\frac{\frac{|\xi|^{1-3\beta}}{|l|^{2-3\beta}}}{\left[1+\left|t-\frac{\xi}{l}\right|^2\right]}}
  \frac{\frac{|\xi|^{\frac{1}{2}-\frac{3}{2}\beta}}{|l|^{2-\frac{3}{2}\beta}}\langle\xi-\eta\rangle}{\left|\frac{|\xi|^{1-3\beta}}{|l|^{2-3\beta}}\right|^{\frac{3}{2}}}\\
   \lesssim&\sqrt{\frac{\pa_t w_k(t,\eta)}{w_k(t,\eta)}}\sqrt{\frac{\pa_t g(t,\xi)}{g(t,\xi)}}\langle\xi-\eta\rangle\big\langle \frac{\xi}{mt} \big\rangle^{-1}\left|\frac{\xi}{l}\right|^{3\beta-1}\\
   \lesssim&\langle\nu^{\frac{1}{3}}|\xi|^{1-s}\rangle^{\frac{3}{4}-\frac{3}{2}\beta}\sqrt{\frac{\pa_t w_k(t,\eta)}{w_k(t,\eta)}}\sqrt{\frac{\mathfrak w\pa_t g(t,\xi)}{g(t,\xi)}} \langle\xi-\eta\rangle\big\langle \frac{\xi}{mt} \big\rangle^{-1}t^{3\beta-1}\\
   \lesssim&\nu^{\frac{1}{4}-\frac{1}{2}\beta}t^{-\frac{1}{4}+\frac{3}{2}\beta}\sqrt{\frac{\pa_t w_k(t,\eta)}{w_k(t,\eta)}}\sqrt{\frac{\mathfrak w\pa_t g(t,\xi)}{g(t,\xi)}} \langle\xi-\eta\rangle\big\langle \frac{\xi}{mt} \big\rangle^{-1}\\
   \lesssim&\nu^{-\beta} \langle\nu t \rangle^{\frac{1}{2}}\sqrt{\frac{\pa_t w_k(t,\eta)}{w_k(t,\eta)}}\sqrt{\frac{\mathfrak w\pa_t g(t,\xi)}{g(t,\xi)}} \langle\xi-\eta\rangle\big\langle \frac{\xi}{mt} \big\rangle^{-1}.
\end{align*}

Then similar to the estimate of ${\mathrm R}_{N;\mathrm{NR},\mathrm{NR}}^{1,D}$, we have
\begin{align*}
  \left|{\mathrm R}_{N;\mathrm{R},\mathrm{NR}}^1\right|\lesssim&\varepsilon \Bigg( \left\|\sqrt{\frac{ \pa_t w}{w}}\rmA f_{\sim N}\right\|_{L^2}^2+\left\|\sqrt{\frac{ \pa_t g}{g}}\rmA f_{\sim N}\right\|_{L^2}^2\\
  &+\left\|\left\langle \frac{\pa_v}{t\pa_z} \right\rangle^{-1}\Delta_L\sqrt{\frac{\mathfrak w \pa_t w}{w}}\rmA P_{\neq}\phi_{\sim N}\right\|_{L^2}^2+ \left\|\left\langle \frac{\pa_v}{t\pa_z} \right\rangle^{-1}\Delta_L\sqrt{\frac{\mathfrak w \pa_t g}{g}}\rmA P_{\neq}\phi_{\sim N}\right\|_{L^2}^2\Bigg).
\end{align*}

\subsubsection{Treatment of ${\mathrm R}_{N;\mathrm{NR},\mathrm{R}}^1$}

This case is very similar to ${\mathrm R}_{N;\mathrm{R},\mathrm{NR}}^1$. Here we have 
\begin{align*}
  1\le|m|\le |\xi|^s,\ |t|\le |\xi|+\frac{1}{8}|\xi|^{1-3\beta}\le \frac{9}{8}|\xi|,\ |k|\le |\eta|,\ \big\langle \frac{\xi}{mt} \big\rangle^{-1}\approx 1.
\end{align*}
It follows from Corollary \ref{cor-W} that
\begin{align*}
  \rmA_k(\eta)=&\rmA_m(\xi)\frac{\rmA_k(\eta)}{\rmA_m(\xi)}\lesssim \rmA_m(\xi)\frac{|m|^{2-3\beta}}{|\xi|^{1-3\beta}}\left[1+\left|t-\frac{\xi}{m}\right|\right]\langle k-m, \eta-\xi\rangle^{1+2C_1\kappa} e^{c\lambda|k-m,\eta-\xi|^s},
\end{align*}
and
\begin{align*}
&  \left|{\mathrm R}_{N;\mathrm{NR},\mathrm{R}}^1\right|\\\lesssim&\sum_{m,k\in\mathbb Z,m\neq0}\int_{\eta,\xi}\mathbf{1}_{t \not\in\tilde{\mathrm{I}}_{k,\eta},t\in\tilde{\mathrm{I}}_{m,\xi}} \rmA_k(\eta) \left|{\hat f}_k(\eta)\right|\rmA_m(\xi)\frac{|m\eta-k\xi|}{m^2+(\xi-mt)^2}\frac{|m|^{2-3\beta}}{|\xi|^{1-3\beta}}\left[1+\left|t-\frac{\xi}{m}\right|\right]\\
  &\qquad\qquad\qquad\times\left|\widehat {\Delta_L\phi}_m(\xi)_N\right|\langle k-m, \eta-\xi\rangle^{1+2C_1\kappa} e^{c\lambda|k-m,\eta-\xi|^s} \left|\hat f_{k-m}(\eta-\xi)_{<N/8}\right| d\eta d\xi.
\end{align*}

If $\frac{1}{2}|\xi|^s\le|m|\le E(|\xi|^s)$, then $t\approx|\xi|^{1-s}\approx|\eta|^{1-s}$, $\frac{|m|^{2-3\beta}}{|\xi|^{1-3\beta}}\approx1+\left|t-\frac{\xi}{m}\right|\approx 1$ and one can get the estimate in a similar way to Case 3 in Section \ref{sec-react-nrnr}. So we consider the following cases under the assumption that $|k|<\frac{1}{2}|\eta|^s$. 

We first consider the case $m\neq k$. We only need to estimate
\begin{align*}
  \frac{|m|^{2-3\beta}}{|\xi|^{1-3\beta}}\left[1+\left|t-\frac{\xi}{m}\right|\right]\frac{\frac{|\xi|}{m^2}}{1+(\frac{\xi}{m}-t)^2}\lesssim t^{3\beta} \frac{1}{1+\left|t-\frac{\xi}{m}\right|}.
\end{align*}

It is clear that $t\le 2|\eta|$, so $t\in {\mathrm{I}}_{j,\eta}$ for some $|j|\ge1$. By using Lemma \ref{eq-w-wgl}, we have
\begin{align*}
  \frac{1}{1+\left|t-\frac{\xi}{m}\right|}\approx&\sqrt{\frac{\pa_t w_m(t,\xi)}{w_m(t,\xi)}}\sqrt{\frac{\pa_t w_m(t,\xi)}{w_m(t,\xi)}}\\
  \lesssim&\langle\nu^{\frac{1}{3}}t\rangle^{\frac{3}{2}-3\beta}\sqrt{\frac{\mathfrak w\pa_t w_m(t,\xi)}{w_m(t,\xi)}}\sqrt{\frac{\mathfrak w\pa_t w_m(t,\xi)}{w_m(t,\xi)}}\\
  \lesssim&\langle\nu^{\frac{1}{3}}t\rangle^{\frac{3}{2}-3\beta} \left(\sqrt{\frac{\mathfrak w\pa_t w_k(t,\eta)}{w_k(t,\eta)}}+\sqrt{\frac{\mathfrak w\pa_t g(t,\eta)}{g(t,\eta)}}\right)\sqrt{\frac{\mathfrak w\pa_t w_m(t,\xi)}{w_m(t,\xi)}}.
\end{align*}

The $m=k$ case is trivial, indeed, we have
\begin{align*}
  \frac{|m|^{2-3\beta}}{|\xi|^{1-3\beta}}\left[1+\left|t-\frac{\xi}{m}\right|\right]\frac{\frac{1}{|m|}}{1+(\frac{\xi}{m}-t)^2}|\eta-\xi|\le \frac{1}{t^{1-3\beta}} \frac{1}{1+|\frac{\xi}{m}-t|}|\eta-\xi|,
\end{align*}
and
\begin{align*}
&  \frac{1}{t^{1-3\beta}}\frac{1}{1+\left|t-\frac{\xi}{m}\right|}\\
  \lesssim&\langle\nu^{\frac{1}{3}}t\rangle^{\frac{3}{2}-3\beta}t^{3\beta-1} \left(\sqrt{\frac{\mathfrak w\pa_t w_k(t,\eta)}{w_k(t,\eta)}}+\sqrt{\frac{\mathfrak w\pa_t g(t,\eta)}{g(t,\eta)}}\right)\sqrt{\frac{\mathfrak w\pa_t w_m(t,\xi)}{w_m(t,\xi)}}\\
  \lesssim&\nu^{-\beta}\langle\nu t\rangle^{\frac{1}{2}}\left(\sqrt{\frac{\mathfrak w\pa_t w_k(t,\eta)}{w_k(t,\eta)}}+\sqrt{\frac{\mathfrak w\pa_t g(t,\eta)}{g(t,\eta)}}\right)\sqrt{\frac{\mathfrak w\pa_t w_m(t,\xi)}{w_m(t,\xi)}}.
\end{align*}
Therefore, we get
\begin{align*}
  \left|{\mathrm R}_{N;\mathrm{NR},\mathrm{R}}^1\right|\lesssim&\varepsilon \Bigg( \left\|\sqrt{\frac{ \pa_t w}{w}}\rmA f_{\sim N}\right\|_{L^2}^2+\left\|\sqrt{\frac{ \pa_t g}{g}}\rmA f_{\sim N}\right\|_{L^2}^2\\
  &+\left\|\left\langle \frac{\pa_v}{t\pa_z} \right\rangle^{-1}\Delta_L\sqrt{\frac{\mathfrak w \pa_t w}{w}}\rmA P_{\neq}\phi_{\sim N}\right\|_{L^2}^2+ \left\|\left\langle \frac{\pa_v}{t\pa_z} \right\rangle^{-1}\Delta_L\sqrt{\frac{\mathfrak w \pa_t g}{g}}\rmA P_{\neq}\phi_{\sim N}\right\|_{L^2}^2\Bigg).
\end{align*}

\subsubsection{Treatment of ${\mathrm R}_{N;\mathrm{R},\mathrm{R}}^1$}

For this case, it holds that $1\le|m|\le |\xi|^s$, $1\le|k|\le |\eta|^s$, $|\xi|^{1-s}\lesssim|t|\le \frac{9}{8}|\xi|$, $\big\langle \frac{\xi}{mt} \big\rangle^{-1}\approx 1$. By using Corollary \ref{cor-W} we have
\begin{align*}
  \rmA_k(\eta)=&\rmA_m(\xi)\frac{\rmA_k(\eta)}{\rmA_m(\xi)}\lesssim \rmA_m(\xi)\frac{\left[1+\left|t-\frac{\xi}{m}\right|\right]}{\left[1+\left|t-\frac{\eta}{k}\right|\right]}\langle k-m, \eta-\xi\rangle^{1+2C_1\kappa} e^{c\lambda|k-m,\eta-\xi|^s},
\end{align*}
and
\begin{align*}
&  \left|{\mathrm R}_{N;\mathrm{R},\mathrm{R}}^1\right|\\\lesssim&\sum_{m,k\in\mathbb Z,m\neq0}\int_{\eta,\xi}\mathbf{1}_{t \in\tilde{\mathrm{I}}_{k,\eta},t \in\tilde{\mathrm{I}}_{m,\xi}} \rmA_k(\eta) \left|{\hat f}_k(\eta)\right|\rmA_m(\xi)\frac{|m\eta-k\xi|}{m^2+(\xi-mt)^2}\frac{\left[1+\left|t-\frac{\xi}{m}\right|\right]}{\left[1+\left|t-\frac{\eta}{k}\right|\right]}\\
  &\qquad\qquad\qquad\times\left|\widehat {\Delta_L\phi}_m(\xi)_N\right|\langle k-m, \eta-\xi\rangle^{1+2C_1\kappa} e^{c\lambda|k-m,\eta-\xi|^s} \left|\hat f_{k-m}(\eta-\xi)_{<N/8}\right| d\eta d\xi.
\end{align*}
If $\frac{1}{2}|\xi|^s\le|m|\le E(|\xi|^s)$, then $1+\left|t-\frac{\eta}{k}\right|\approx1+\left|t-\frac{\xi}{m}\right|\approx 1$ and one can get the estimate in a similar way to Case 3 in Section \ref{sec-react-nrnr}. So we consider the following cases under the assumption that $|k|<\frac{1}{2}|\eta|^s$. 

Case 1, $m=k$.  We have
\begin{equation}
  \begin{aligned}    
    &\frac{\left[1+\left|t-\frac{\xi}{m}\right|\right]}{\left[1+\left|t-\frac{\eta}{k}\right|\right]}\frac{|m\eta-k\xi|}{m^2+(\xi-mt)^2}=\frac{\left[1+\left|t-\frac{\xi}{m}\right|\right]}{\left[1+\left|t-\frac{\eta}{m}\right|\right]}\frac{|m||\eta-\xi|}{m^2+(\xi-mt)^2}\\
    \le& \frac{1}{\left[1+\left|t-\frac{\eta}{m}\right|\right]}\frac{1}{|m|}\frac{1}{\left[1+\left|t-\frac{\xi}{m}\right|\right]}|\xi-\eta|\\
    \le&\frac{1}{|m|}\sqrt{\frac{1}{\left[1+\left|t-\frac{\eta}{m}\right|\right]}}\sqrt{\frac{1}{\left[1+\left|t-\frac{\xi}{m}\right|\right]}}|\xi-\eta|\\
    \lesssim&\frac{1}{|m|}\sqrt{\frac{\pa_t w_m(t,\xi)}{w_m(t,\xi)}}\sqrt{\frac{\pa_t w_m(t,\eta)}{w_m(t,\eta)}}|\xi-\eta|\\
    \lesssim&\frac{1}{|m|}\langle\nu^{\frac{1}{3}}|\xi|^{1-s}\rangle^{\frac{3}{4}-\frac{3}{2}\beta}\sqrt{\frac{\mathfrak w\pa_t w_m(t,\xi)}{w_m(t,\xi)}}\sqrt{\frac{\pa_t w_m(t,\eta)}{w_m(t,\eta)}}|\xi-\eta|.
  \end{aligned}
\end{equation}
If $|\xi|^{1-s}<\nu^{-\frac{1}{3}}$, then $\langle\nu^{\frac{1}{3}}|\xi|^{1-s}\rangle^{\frac{3}{4}-\frac{3}{2}\beta}\approx 1$. If
$|\xi|^{1-s}\ge\nu^{-\frac{1}{3}}$, it holds that
\begin{align*}
  \frac{1}{|m|}\langle\nu^{\frac{1}{3}}|\xi|^{1-s}\rangle^{\frac{3}{4}-\frac{3}{2}\beta}\lesssim&\frac{1}{|m|}\nu^{\frac{1}{4}-\frac{1}{2}\beta}|\xi|^{(1-s)(\frac{3}{4}-\frac{3}{2}\beta)}\approx\frac{t}{|\xi|}\nu^{\frac{1}{4}-\frac{1}{2}\beta}|\xi|^{(1-s)(\frac{3}{4}-\frac{3}{2}\beta)}\\
  \approx&\nu^{\frac{1}{4}-\frac{1}{2}\beta}t|\xi|^{(1-s)(\frac{3}{4}-\frac{3}{2}\beta)-1}\lesssim\nu^{\frac{1}{4}-\frac{1}{2}\beta}t^{1+\frac{3}{4}-\frac{3}{2}\beta-\frac{1}{1-s}}\\
  =&\nu^{\frac{1}{4}-\frac{1}{2}\beta}t^{1+\frac{3}{4}-\frac{3}{2}\beta-2+3\beta}=\nu^{-\beta}(\nu t)^{\frac{1}{4}+\frac{1}{2}\beta}t^{\beta-\frac{1}{2}}\\
  \le&\nu^{-\beta}(\nu t)^{\frac{1}{4}+\frac{1}{2}\beta}\le\nu^{-\beta} \langle \nu t\rangle^{\frac{1}{2}}.
\end{align*}

Case 2, $n\neq k$, we only need to estimate
\begin{align*}
  \frac{\left[1+\left|t-\frac{\xi}{m}\right|\right]}{\left[1+\left|t-\frac{\eta}{k}\right|\right]}\frac{\frac{|\xi|}{m^2}}{1+(\frac{\xi}{m}-t)^2}\lesssim \frac{t}{|m|} \frac{1}{1+\left|t-\frac{\xi}{m}\right|}\frac{1}{1+\left|t-\frac{\eta}{k}\right|}.
\end{align*}
By using Lemma \ref{lem-separate}, we have $|\xi-\eta|\gtrsim \frac{\eta}{k}\approx t$. Therefore,
\begin{align*}
  \frac{t}{|m|} \frac{1}{1+\left|t-\frac{\xi}{m}\right|}\frac{1}{1+\left|t-\frac{\eta}{k}\right|}\lesssim \frac{|m,\xi|^{\frac{s}{2}}}{t^{c_1}}\frac{|k,\eta|^{\frac{s}{2}}}{t^{c_1}}\langle\xi-\eta\rangle^3.
\end{align*}

In a conclusion, we have
\begin{align*}
  \left|{\mathrm R}_{N;\mathrm{R},\mathrm{R}}^1\right|\lesssim&\varepsilon \Bigg( \left\|\sqrt{\frac{ \pa_t w}{w}}\rmA f_{\sim N}\right\|_{L^2}^2+\left\|\sqrt{\frac{ \pa_t g}{g}}\rmA f_{\sim N}\right\|_{L^2}^2\\
  &+t^{-2c_1}\left\||\nabla|^{\frac{s}{2}}\rmA f_{\sim N}\right\|_{L^2}^2+ t^{-2c_1}\left\|\left\langle \frac{\pa_v}{t\pa_z} \right\rangle^{-1}\Delta_L|\nabla|^{\frac{s}{2}}\rmA P_{\neq}\phi_{\sim N}\right\|_{L^2}^2\\
  &+\left\|\left\langle \frac{\pa_v}{t\pa_z} \right\rangle^{-1}\Delta_L\sqrt{\frac{\mathfrak w \pa_t w}{w}}\rmA P_{\neq}\phi_{\sim N}\right\|_{L^2}^2+ \left\|\left\langle \frac{\pa_v}{t\pa_z} \right\rangle^{-1}\Delta_L\sqrt{\frac{\mathfrak w \pa_t g}{g}}\rmA P_{\neq}\phi_{\sim N}\right\|_{L^2}^2\Bigg).
\end{align*}
\subsection{Corrections ${\mathrm R}_N^{\varepsilon,1}$} \label{sec-react-corr}
In this subsection, we treat ${\mathrm R}_N^{\varepsilon,1}$ which is higher order in $\varepsilon$ than ${\mathrm R}_N^{1}$.
We expand $h\phi$ with a paraproduct only in $v$:
\begin{align*}
  {\mathrm R}_N^{\varepsilon,1}=&-\frac{1}{2\pi}\sum_{L\ge 8}\ \sum_{m,k\in\mathbb Z,m\neq0}\int_{\eta,\xi,\xi'}\rmA_k(\eta){\hat f}_k(\eta)\rmA_k(\eta)\left((\eta-\xi)m-\xi'(k-m)\right)\rho_N(m,\xi)\\
  &\qquad\qquad\times \hat h(\xi'-\xi)_{<L/8}\hat\phi_m(\xi')_{L} \widehat{ f}_{k-m}(\eta-\xi)_{<N/8}d\eta d\xi d\xi'\\
  &-\frac{1}{2\pi}\sum_{L\ge 8}\ \sum_{m,k\in\mathbb Z,m\neq0}\int_{\eta,\xi,\xi'}\rmA_k(\eta){\hat f}_k(\eta)\rmA_k(\eta)\left((\eta-\xi)m-\xi'(k-m)\right)\rho_N(m,\xi)\\
  &\qquad\qquad\times \hat h(\xi'-\xi)_{L}\hat\phi_m(\xi')_{<L/8} \widehat{ f}_{k-m}(\eta-\xi)_{<N/8}d\eta d\xi d\xi'\\
  &-\frac{1}{2\pi}\sum_{L\in\mathbb D}\sum_{L'\sim L}\ \sum_{m,k\in\mathbb Z,m\neq0}\int_{\eta,\xi,\xi'}\rmA_k(\eta){\hat f}_k(\eta)\rmA_k(\eta)\left((\eta-\xi)m-\xi'(k-m)\right)\rho_N(m,\xi)\\
  &\qquad\qquad\times \hat h(\xi'-\xi)_{L'}\hat\phi_m(\xi')_{L} \widehat{ f}_{k-m}(\eta-\xi)_{<N/8}d\eta d\xi d\xi'\\
  =&{\mathrm R}_{N,\mathrm{LH}}^{\varepsilon,1}+{\mathrm R}_{N,\mathrm{HL}}^{\varepsilon,1}+{\mathrm R}_{N,\mathrm{HH}}^{\varepsilon,1}.
\end{align*}
Here $\rho_N(m,\xi)$ is a cut off function defined in Appendix \ref{sec-decompo}.

Begin first with ${\mathrm R}_{N,\mathrm{LH}}^{\varepsilon,1}$. With the support of the integrand,
\begin{align*}
  \big||m,\xi|-|k,\eta|\big|\le |k-m,\eta-\xi|\le \frac{6}{32}|m,\xi|,\\
  \big||m,\xi|-|m,\xi'|\big|\le |\xi-\xi'|\le \frac{6}{32}|\xi'|\le \frac{6}{32}|m,\xi'|.
\end{align*}
Then by using \eqref{inq-s2}, we have
\begin{align*}
  |k,\eta|^s\le |m,\xi|^s+c|m-k,\xi-\eta|^s\le|m,\xi'|^s+c|m-k,\xi-\eta|^s+c|\xi-\xi'|^s.
\end{align*}
Therefore, as $|k,\eta|\approx|m,\xi'|$, we have
\begin{align*}
&  \left|{\mathrm R}_{N,\mathrm{LH}}^{\varepsilon,1}\right|\\\lesssim&\sum_{L\ge 8}\sum_{m,k\in\mathbb Z,m\neq0}\int_{\eta,\xi,\xi'}\rmA_k(\eta) \left|{\hat f}_k(\eta)\right|\left|(\eta-\xi)m-\xi'(k-m)\right|\rho_N(m,\xi)e^{c\lambda|\xi-\xi'|^s}\\
  &\times\left|\hat h(\xi'-\xi)_{<L/8}\right| \rmA_m(\xi')\frac{W_k(\eta)}{W_m(\xi')}\left|\hat\phi_m(\xi')_{L}\right| e^{c\lambda|k-m,\eta-\xi|^s}\left|\widehat{ f}_{k-m}(\eta-\xi)_{<N/8}\right| d\eta d\xi d\xi'.
\end{align*}
We can see that
\begin{align*}
  \left|(\eta-\xi)m-\xi'(k-m)\right|\le |\eta-\xi||m|+|\xi'||k-m|.
\end{align*}
From here we may proceed analogously to the treatment of ${\mathrm R}_{N}^{1}$ with $(m,\xi')$  playing the role of $(m,\xi)$ and using the bootstrap hypotheses to deal with the low-frequency factors. We omit the details and simply conclude that the result is analogous to the estimate in Section \ref{sec-react-main}, except with an additional power of $\varepsilon^{\frac{1}{2}}\nu^{\frac{1}{2}\beta}$.

Next, we turn to ${\mathrm R}_{N,\mathrm{HL}}^{\varepsilon,1}$.
\begin{align*}
  {\mathrm R}_{N,\mathrm{HL}}^{\varepsilon,1}=&-\frac{1}{2\pi}\sum_{L\ge 8}\sum_{m,k\in\mathbb Z,m\neq0}\int_{\eta,\xi,\xi'}\rmA_k(\eta){\hat f}_k(\eta)[\mathbf1_{|m|\ge |\xi|/16}+\mathbf1_{|m|< |\xi|/16}]\\
  &\qquad\qquad\times \rmA_k(\eta)\left((\eta-\xi)m-\xi'(k-m)\right)\rho_N(m,\xi)\\
  &\qquad\qquad\times \hat h(\xi'-\xi)_{L}\hat\phi_m(\xi')_{<L/8} \widehat{ f}_{k-m}(\eta-\xi)_{<N/8}d\eta d\xi d\xi'\\
  =&{\mathrm R}_{N,\mathrm{HL}}^{\varepsilon,1,z}+{\mathrm R}_{N,\mathrm{HL}}^{\varepsilon,1,v}.
\end{align*}
For ${\mathrm R}_{N,\mathrm{HL}}^{\varepsilon,1,z}$, we have $|\xi-\xi'|\lesssim |\xi|\lesssim|m|$. It is clear from \eqref{inq-s2} that
\begin{align*}
  |k,\eta|^s\le |m,\xi|^s+c|m-k,\xi-\eta|^s.
\end{align*}
If $|m|\ge 16|\xi|$, then by \eqref{inq-s2} we have
\begin{align*}
  |k,\eta|^s\le |m,\xi'|^s+c|m-k,\xi-\eta|^s+c|\xi-\xi'|^s,
\end{align*}
and if $\frac{1}{16}|\xi|\le|m|\le 16|\xi|$, we have $|\xi-\xi'|\sim|m,\xi|$, then by \eqref{inq-s3} we have
\begin{align*}
  |k,\eta|^s\le c|m,\xi'|^s+c|m-k,\xi-\eta|^s+c|\xi-\xi'|^s.
\end{align*}
For both cases, we have
\begin{align*}
  |{\mathrm R}_{N,\mathrm{HL}}^{\varepsilon,1,z}|\lesssim&\sum_{L\ge 8}\sum_{m,k\in\mathbb Z,m\neq0}\int_{\eta,\xi,\xi'}\mathbf1_{|m|\ge |\xi|/16}\rmA_k(\eta) \left|{\hat f}_k(\eta)\right|\left|(\eta-\xi)m-\xi'(k-m)\right|\rho_N(m,\xi)\\
  &\times\left|\hat h(\xi'-\xi)_{L}\right| \rmA_m(\xi')\frac{W_k(\eta)}{W_m(\xi')}\left|\hat\phi_m(\xi')_{<L/8}\right| e^{c\lambda|k-m,\eta-\xi|^s}\left|\widehat{ f}_{k-m}(\eta-\xi)_{<N/8}\right| d\eta d\xi d\xi'.
\end{align*}
Recall that $|\xi'|\lesssim |\xi|\lesssim |m|$, we have
\begin{align*}
  \left|(\eta-\xi)m-\xi'(k-m)\right|\le |\eta-\xi||m|+|\xi'||k-m|\le |m||k-m,\eta-\xi|,
\end{align*}
and
\begin{align*}
  |m|\left|\hat\phi_m(\xi')_{<L/8}\right|=\frac{|m|}{m^2+|\xi'-mt|^2}\left|\widehat{\Delta_L\phi_m}(\xi')_{<L/8}\right|\lesssim \frac{1}{|m|t^2}\left|\widehat{\Delta_L\phi_m}(\xi')_{<L/8}\right|.
\end{align*}
Here we do not need to worry about the resonance problem. Recall that $|k,\eta|\approx|m,\xi|\lesssim|m|$. Even for $t\in \tilde{\mathrm{I}}_{k,\eta}$, it holds that
\begin{align*}
  \frac{W_k(\eta)}{W_m(\xi')}\frac{1}{|m|t^2}\lesssim& \frac{|\eta|^{1-3\beta}}{|k|^{2-3\beta}}\frac{1}{|m|t^2}\langle \xi'-\eta,k-m\rangle^{1+2C_1\kappa}e^{C\mu |\xi'-\eta,k-m|^s}\\
  \lesssim&\frac{1}{|k|^{2-3\beta}|m|^{3\beta}t^2} \langle\xi'-\xi\rangle^{1+2C_1\kappa}\langle \xi-\eta,k-m\rangle^{1+2C_1\kappa}e^{C\mu |\xi-\eta,k-m|^s+C\mu|\xi'-\xi|^s}.
\end{align*}
Therefore, we can get similar estimate to ${\mathrm R}_{N;\mathrm{NR},\mathrm{NR}}^{1,D^c}$.

For ${\mathrm R}_{N,\mathrm{HL}}^{\varepsilon,1,v}$, $|k,\eta|\approx|m,\xi|\approx|\xi|\approx|\xi-\xi'|$, and $|m,\xi'|\le \frac{1}{3}|\xi-\xi'|$, then by \eqref{inq-s2}
\begin{align*}
  |k,\eta|^s\le |\xi-\xi'|^s+c|m-k,\xi-\eta|^s+c|m,\xi'|^s.
\end{align*}
 Similar to \eqref{eq-est-ell-lh}, we have
\begin{align*}
  \rmA_k(\eta)\lesssim&\rmA^{\mathrm R}(\xi-\xi')e^{\mathbf{1}_{k \neq 0}c_M \nu^{\frac{1}{3}}t} \langle m-k,\xi-\eta\rangle^{2+2C_1\kappa} \langle m,\xi'\rangle^{2+2C_1\kappa}e^{c\lambda|m-k,\xi-\eta|^s+c\lambda|m,\xi'|^s},
\end{align*}
and
\begin{align*}
  |{\mathrm R}_{N,\mathrm{HL}}^{\varepsilon,1,v}|\lesssim&\sum_{L\ge 8}\sum_{m,k\in\mathbb Z,m\neq0}\int_{\eta,\xi,\xi'}\mathbf1_{|m|< |\xi|/16}\rmA_k(\eta) \left|{\hat f}_k(\eta)\right|\rho_N(m,\xi)\rmA^{\mathrm R}(\xi-\xi') \left|\hat h(\xi'-\xi)_{L}\right|\\
  &\qquad\qquad\times e^{\mathbf{1}_{k \neq 0}c_M \nu^{\frac{1}{3}}t}  e^{c\lambda|m,\xi'|^s}\langle m,\xi'\rangle^{3+2C_1\kappa}\left|\hat\phi_m(\xi')_{<L/8}\right|\\
  &\qquad\qquad\times e^{c\lambda|m-k,\xi-\eta|^s}\langle m-k,\xi-\eta\rangle^{3+2C_1\kappa}\left|\widehat{ f}_{k-m}(\eta-\xi)_{<N/8}\right| d\eta d\xi d\xi'.
\end{align*}
Therefore, we deduce from Lemma \ref{lem-elliptic-low-1} that
\begin{align*}
  |{\mathrm R}_{N,\mathrm{HL}}^{\varepsilon,1,v}|\lesssim& \left\| t^2e^{c_M \nu^{\frac{1}{3}}t} \phi_{\neq}\right\|_{\mathcal G^{s,\lambda,5}} \left\|f\right\|_{\mathcal G^{s,\lambda,6}} t^{-2c_1}\left\||\nabla|^{\frac{s}{2}}\rmA f_{\sim N}\right\|_{L^2}\left\||\pa_v|^{\frac{s}{2}}\rmA^{\mathrm R}h_{\sim N}\right\|_{L^2}\\
  \lesssim&\varepsilon^{\frac{3}{2}}\nu^{\beta}\left(t^{-2c_1}\left\||\nabla|^{\frac{s}{2}}\rmA f_{\sim N}\right\|_{L^2}^2+ \varepsilon \nu^{\beta}t^{-2c_1}\left\||\pa_v|^{\frac{s}{2}}\rmA^{\mathrm R}h_{\sim N}\right\|_{L^2}^2\right)
\end{align*}

At last, we consider ${\mathrm R}_{N,\mathrm{HH}}^{\varepsilon,1}$. We divide it into two cases:
\begin{align*}
  {\mathrm R}_{N,\mathrm{HH}}^{\varepsilon,1}=&-\frac{1}{2\pi}\sum_{L\in\mathbb D}\sum_{L'\sim L}\sum_{m,k\in\mathbb Z,m\neq0}\int_{\eta,\xi,\xi'}A_k(\eta){\hat f}_k(\eta)[\mathbf1_{|m|\ge 100|\xi'|}+\mathbf1_{|m|< 100|\xi'|}]\\
  &\qquad\qquad\times A_k(\eta)\left((\eta-\xi)m-\xi'(k-m)\right)\rho_N(m,\xi)\\
  &\qquad\qquad\times \hat h(\xi'-\xi)_{L'}\hat\phi_m(\xi')_{L} \widehat{ f}_{k-m}(\eta-\xi)_{<N/8}d\eta d\xi d\xi'\\
  =&{\mathrm R}_{N,\mathrm{HH}}^{\varepsilon,1,z}+{\mathrm R}_{N,\mathrm{HH}}^{\varepsilon,1,v}.
\end{align*}
For ${\mathrm R}_{N,\mathrm{HH}}^{\varepsilon,1,z}$, we have
\begin{align*}
  \big||m,\xi|-|k,\eta|\big|\le |k-m,\eta-\xi|\le \frac{6}{32}|m,\xi|,\\
  \big||m,\xi|-|m,\xi'|\big|\le |\xi-\xi'|\le 24|\xi'|\le \frac{24}{100}|m,\xi'|.
\end{align*}
then by \eqref{inq-s2} we have
\begin{align*}
  |k,\eta|^s\le |m,\xi'|^s+c|m-k,\xi-\eta|^s+c|\xi-\xi'|^s.
\end{align*}
By $|m|\ge 100|\xi'|$, it follows that $|\xi|\le \frac{1}{4}|m|$ and hence $(m,\xi)$ cannot be resonant frequency and $|\eta|\le \frac{3}{2}|k|$, which implies by $N\ge8$ that $(k,\eta)$ cannot be resonant frequency. Also using
\begin{align*}
  \left|(\eta-\xi)m-\xi'(k-m)\right|\le |\eta-\xi||m|+|\xi'||k-m|.
\end{align*}
We have
\begin{align*}
  \left|{\mathrm R}_{N,\mathrm{HH}}^{\varepsilon,1,z}\right|\lesssim&\sum_{L\in\mathbb D}\sum_{L'\sim L}\sum_{m,k\in\mathbb Z,m\neq0}\int_{\eta,\xi,\xi'}\mathbf1_{|m|\ge 100|\xi'|}A_k(\eta) \left|{\hat f}_k(\eta)\right|\rho_N(m,\xi)\\
  &\qquad\qquad\times e^{c\lambda|\xi-\xi'|^s}\langle \xi'-\xi\rangle^{1+2C_1\kappa} \left| \hat h(\xi'-\xi)_{L'}\right|A_m(\xi')|m,\xi'| \left|\hat\phi_m(\xi')_{L}\right| \\
  &\qquad\qquad\times e^{c\lambda|k-m,\xi-\eta|^s}\langle \xi-\eta,k-m\rangle^{1+2C_1\kappa} \left|\widehat{f}_{k-m}(\eta-\xi)_{<N/8}\right|d\eta d\xi d\xi'.
\end{align*}
Then we can estimate this term with the same argument to $|{\mathrm R}_{N,\mathrm{HL}}^{\varepsilon,1,z}|$. 

For ${\mathrm R}_{N,\mathrm{HH}}^{\varepsilon,1,v}$,  it holds that
\begin{align*}
  \big||m,\xi|-|k,\eta|\big|\le |k-m,\eta-\xi|\le \frac{6}{32}|m,\xi|,\\
  \frac{1}{24}|\xi'|\le |\xi-\xi'|\le 24|\xi'|\le 24|m,\xi'|,\\
  |m,\xi'|\le 101|\xi'|\le 2424|\xi-\xi'|
\end{align*}
then we have by \eqref{inq-s3} that
\begin{align*}
  |k,\eta|^s\le c|\xi-\xi'|^s+c|m,\xi'|^s+c|m-k,\xi-\eta|^s.
\end{align*}
Then for $|m|< 1/16|\xi'|$, we can estimate $|{\mathrm R}_{N,\mathrm{HH}}^{\varepsilon,1,v}|$ in the same way to $|{\mathrm R}_{N,\mathrm{HL}}^{\varepsilon,1,v}|$, and for $1/16|\xi'|\le |m|<100|\xi'|$, we can estimate $|{\mathrm R}_{N,\mathrm{HH}}^{\varepsilon,1,v}|$ in the same way to $|{\mathrm R}_{N,\mathrm{HL}}^{\varepsilon,1,z}|$. 

Combing all the above estimates, we get that
\begin{align*}
  |{\mathrm R}_{N}^{\varepsilon,1}|\lesssim&\varepsilon^{\frac{3}{2}}\nu^{\frac{1}{2}\beta}\Bigg( \left\|\sqrt{\frac{ \pa_t w}{w}}\rmA f_{\sim N}\right\|_{L^2}^2+\left\|\sqrt{\frac{ \pa_t g}{g}}\rmA f_{\sim N}\right\|_{L^2}^2+\varepsilon \nu^{\beta}t^{-2c_1}\left\||\pa_v|^{\frac{s}{2}}\rmA^{\mathrm R}h_{\sim N}\right\|_{L^2}^2\\
  &+t^{-2c_1}\left\||\nabla|^{\frac{s}{2}}\rmA f_{\sim N}\right\|_{L^2}^2+ t^{-2c_1}\left\|\left\langle \frac{\pa_v}{t\pa_z} \right\rangle^{-1}\Delta_L|\nabla|^{\frac{s}{2}}\rmA P_{\neq}\phi_{\sim N}\right\|_{L^2}^2\\
  &+\left\|\left\langle \frac{\pa_v}{t\pa_z} \right\rangle^{-1}\Delta_L\sqrt{\frac{\mathfrak w \pa_t w}{w}}\rmA P_{\neq}\phi_{\sim N}\right\|_{L^2}^2+ \left\|\left\langle \frac{\pa_v}{t\pa_z} \right\rangle^{-1}\Delta_L\sqrt{\frac{\mathfrak w \pa_t g}{g}}\rmA P_{\neq}\phi_{\sim N}\right\|_{L^2}^2\Bigg).
\end{align*}

\subsection{Zero mode reaction term ${\mathrm R}_N^2$}\label{sec-reac-R2}
We write
\begin{align*}
  {\mathrm R}_N^2=&\sum_{k\in\mathbb Z}\int_{\eta,\xi}\rmA_k(\eta){\hat f}_k(\eta)[\mathbf 1_{t\in \tilde{\mathrm{I}}_{k,\eta}}+\mathbf 1_{t\not\in \tilde{\mathrm{I}}_{k,\eta}}]\rmA_k(\eta)\hat q(\xi)_N\widehat {\pa_vf}_{k}(\eta-\xi)_{<N/8}d\eta d\xi\\
  =&{\mathrm R}_{N}^{2,\mathrm{R}}+{\mathrm R}_{N}^{2,\mathrm{NR}}.
\end{align*}

Recall that
\begin{align*}
  \frac{N}{2}\le|k,\xi|\le \frac{3N}{2},\quad |k,\eta-\xi|\le \frac{3N}{32}
\end{align*}
then
\begin{align*}
  |k|+|\eta-\xi|\le \frac{3}{16}|\xi|.
\end{align*}

We first study ${\mathrm R}_{N}^{2,\mathrm{NR}}$. Similar to \eqref{eq-est-ell-lh}, we have
\begin{align*}
  \rmA_k(\eta)\lesssim&\rmA_0(\xi)e^{\mathbf{1}_{k \neq 0}c_M \nu^{\frac{1}{3}}t} \langle k,\xi-\eta\rangle^{1+2C_1\kappa}e^{c\lambda|k,\xi-\eta|^s},
\end{align*}
and
\begin{align*}
  \left|{\mathrm R}_{N}^{2,\mathrm{NR}}\right|\lesssim& \sum_{k}\int_{\eta,\xi}\rmA_k(\eta) \left|{\hat f}_k(\eta)\right|\rmA_0(\xi) \left|\hat q(\xi)_N\right|\\
  &\qquad\qquad\times e^{\mathbf{1}_{k \neq 0}c_M \nu^{\frac{1}{3}}t}e^{c\lambda|k,\xi-\eta|^s}\langle k,\xi-\eta\rangle^{1+2C_1\kappa} \left|\widehat {f}_{k}(\eta-\xi)_{<N/8}\right|d\eta d\xi\\
  \lesssim&\sum_{k}\int_{\eta,\xi}\rmA_k(\eta) \left|{\hat f}_k(\eta)\right|\rmA_0(\xi)\frac{|\eta|^{\frac{s}{2}}|\xi|^{\frac{s}{2}}}{\langle\xi\rangle^s}\left|\hat q(\xi)_N\right|\\
  &\qquad\qquad\times e^{\mathbf{1}_{k \neq 0}c_M \nu^{\frac{1}{3}}t}e^{c\lambda|k,\xi-\eta|^s}\langle k,\xi-\eta\rangle^{1+2C_1\kappa} \left|\widehat {f}_{k}(\eta-\xi)_{<N/8}\right|d\eta d\xi\\
  \lesssim&\varepsilon\nu^{\beta}\||\nabla|^{\frac{s}{2}}\rmA f_{\sim N}\|_{L^2}\||\pa_v|^{\frac{s}{2}}\frac{\rmA}{\langle\pa_v\rangle^s}q_{\sim N}\|_{L^2}\\
  \lesssim&\varepsilon t^{-2c_1}\||\nabla|^{\frac{s}{2}}\rmA f_{\sim N}\|_{L^2}^2+\varepsilon \nu^{2\beta}t^{2c_1}\left\||\pa_v|^{\frac{s}{2}}\frac{\rmA}{\langle\pa_v\rangle^s}q_{\sim N}\right\|_{L^2}^2.
\end{align*}

For ${\mathrm R}_{N}^{2,\mathrm{R}}$, as the case $\frac{1}{2}|\eta|^s\le|k|\le |\eta|^s$ can be treated in the same way to ${\mathrm R}_{N}^{2,\mathrm{NR}}$, here we also assume that $|k|<\frac{1}{2}|\eta|^s$. We have
\begin{align*}
  \rmA_k(\eta)\lesssim&\rmA_0(\xi)\frac{\frac{|\eta|^{1-3\beta}}{|k|^{2-3\beta}}}{\left[1+\left|t-\frac{\eta}{k}\right|\right]}e^{\mathbf{1}_{k \neq 0}c_M \nu^{\frac{1}{3}}t} \langle k,\xi-\eta\rangle^{1+2C_1\kappa}e^{c\lambda|k,\xi-\eta|^s},
\end{align*}
and
\begin{align*}
  \left|{\mathrm R}_{N}^{2,\mathrm{R}}\right|\lesssim& \sum_{k}\int_{\eta,\xi}\rmA_k(\eta) \left|{\hat f}_k(\eta)\right|\rmA_0(\xi)\frac{\frac{|\eta|^{1-3\beta}}{|k|^{2-3\beta}}}{\left[1+\left|t-\frac{\eta}{k}\right|\right]} \left|\hat q(\xi)_N\right|\\
  &\qquad\qquad\times e^{\mathbf{1}_{k \neq 0}c_M \nu^{\frac{1}{3}}t}e^{c\lambda|k,\xi-\eta|^s}\langle k,\xi-\eta\rangle^{1+2C_1\kappa} \left|\widehat {f}_{k}(\eta-\xi)_{<N/8}\right|d\eta d\xi.
\end{align*}
Assume that $t\in \tilde{\mathrm{I}}_{k,\eta}\cap {\mathrm{I}}_{l,\xi}$, we have for $t\in \tilde{\mathrm{I}}_{k,\eta}\cap  \tilde{\mathrm{I}}_{l,\xi}$ that
\begin{align*}
  \frac{\frac{|\eta|^{1-3\beta}}{|k|^{2-3\beta}}}{\left[1+\left|t-\frac{\eta}{k}\right|\right]}\lesssim&\sqrt{\frac{\pa_t w_k(t,\eta)}{w_k(t,\eta)}}\sqrt{\frac{\pa_t w_0(t,\xi)}{w_0(t,\xi)}} \langle\eta-\xi\rangle \frac{|\eta|^{1-3\beta+s}}{|k|^{2-3\beta}}\frac{1}{\langle\xi\rangle^{s}}\\
  \lesssim&\sqrt{\frac{\pa_t w_k(t,\eta)}{w_k(t,\eta)}}\sqrt{\frac{\pa_t w_0(t,\xi)}{w_0(t,\xi)}} \langle\eta-\xi\rangle \frac{t^{1-3\beta+s}}{|k|^{1-s}}\frac{1}{\langle\xi\rangle^{s}},
\end{align*}
and for $t\in \tilde{\mathrm{I}}_{k,\eta}\cap \left(\mathrm{I}_{l,\xi}\setminus\tilde{\mathrm{I}}_{l,\xi}\right)$ that
\begin{align*}
  \frac{\frac{|\eta|^{1-3\beta}}{|k|^{2-3\beta}}}{\left[1+\left|t-\frac{\eta}{k}\right|\right]}\lesssim&\sqrt{\frac{\pa_t w_k(t,\eta)}{w_k(t,\eta)}}\sqrt{\frac{\pa_t g(t,\xi)}{g(t,\xi)}} \sqrt{\frac{\frac{|\eta|^{1-3\beta}}{|k|^{2-3\beta}}\left[1+\left|t-\frac{\xi}{l}\right|^2\right]}{1+\left|t-\frac{\eta}{k}\right|}}\\
  \lesssim&\sqrt{\frac{\pa_t w_k(t,\eta)}{w_k(t,\eta)}}\sqrt{\frac{\pa_t g(t,\xi)}{g(t,\xi)}}\langle\eta-\xi\rangle \frac{|\eta|^{1-\frac{3}{2}\beta+s}}{|k|^{2-\frac{3}{2}\beta}}\frac{1}{\langle\xi\rangle^{s}}\\
  \lesssim&\sqrt{\frac{\pa_t w_k(t,\eta)}{w_k(t,\eta)}}\sqrt{\frac{\pa_t g(t,\xi)}{g(t,\xi)}}\langle\eta-\xi\rangle \frac{t^{1-\frac{3}{2}\beta+s}}{|k|^{1-s}}\frac{1}{\langle\xi\rangle^{s}}.
\end{align*}
Not that  $1-\frac{3}{2}\beta+s\le \frac{3}{2}$, we have
\begin{align*}
  \left|{\mathrm R}_{N}^{2,\mathrm{R}}\right|\lesssim&\varepsilon t^{-2c_1}\||\nabla|^{\frac{s}{2}}\rmA f_{\sim N}\|_{L^2}^2+\varepsilon\left\|\sqrt{\frac{ \pa_t w}{w}}\rmA f_{\sim N}\right\|_{L^2}^2+\varepsilon \nu^{2\beta}t^{2c_1}\left\||\pa_v|^{\frac{s}{2}}\frac{\rmA}{\langle\pa_v\rangle^s}q_{\sim N}\right\|_{L^2}^2\\
  &+\varepsilon \nu^{2\beta}t^{3}\left\|\sqrt{\frac{ \pa_t w}{w}}\frac{\rmA}{\langle\pa_v\rangle^s}q_{\sim N}\right\|_{L^2}^2+\varepsilon \nu^{2\beta}t^{3}\left\|\sqrt{\frac{ \pa_t g}{g}}\frac{\rmA}{\langle\pa_v\rangle^s}q_{\sim N}\right\|_{L^2}^2.
\end{align*}

\subsection{Term ${\mathrm R}_N^3$, remainder from commutator}\label{sec-reac-R3}
For ${\mathrm R}_N^3$, we use the decay of the lower order estimate of $\bar u$. All the derivatives on $f_{<N/8}$ can be moved to $\bar u$. Indeed, by using Lemma \ref{lem-elliptic-low-1} we have
\begin{align*}
   \left|{\mathrm R}_N^3\right|\lesssim&\sum_{m,k\in\mathbb Z}\int_{\eta,\xi}\rmA_k(\eta) \left|{\hat f}_k(\eta)\right| |m,\xi| \left|\hat {\bar u}_m(\xi)_N\right| \rmA_{k-m}(\eta-\xi) \left|\hat{f}_{k-m}(\eta-\xi)_{<N/8}\right|d\eta d\xi\\
   \lesssim&\|\rmA f_{\sim N}\|_{L^2}\|{\bar u}_N\|_{H^2}\|\rmA f_{< N/8}\|_{L^2}
   \lesssim \frac{\varepsilon \nu^\beta}{t^2}\|\rmA f_{\sim N}\|_{L^2}^2+\varepsilon \nu^\beta t^2\|{\bar u}_N\|_{H^2}
   \lesssim\frac{\varepsilon \nu^\beta}{t^2}\|\rmA f_{\sim N}\|_{L^2}^2.
\end{align*}

This completes the proof of Proposition \ref{pro-reaction}.

\section{Transport term, remainder term, and dissipation error term}\label{sec-transport}
In this section, we prove Proposition \ref{pro-transport}.
We first treat the transport term
\begin{align*}
  {\mathrm T}_N=&2\pi\int \rmA f[\rmA({\bar u}_{<N/8}\cdot\nabla_{z,v} f_N)-{\bar u}_{<N/8}\cdot\nabla_{z,v}\rmA f_N]dzdv.
\end{align*}

It holds that
\begin{align*}
  &\rmA_k(\eta)-\rmA_m(\xi)=\rmA_m(\xi)\left(\frac{\rmA_k(\eta)}{\rmA_m(\xi)}-1\right)\\
  =&\rmA_m(\xi)\left(\frac{ W_k(t,\eta)}{ W_m(t,\xi)}-1\right)e^{\mathbf{1}_{k \neq 0}c_M \nu^{\frac{1}{3}}t-\mathbf{1}_{m \neq 0}c_M \nu^{\frac{1}{3}}t}e^{\lambda(t)\left(|k,\eta|^s-|m,\xi|^s\right)}\frac{\langle k,\eta\rangle^\sigma}{\langle m,\xi\rangle^\sigma} \frac{ G_k(t,\eta)}{ G_m(t,\xi)}\frac{ \mathfrak M_k(t,\eta)}{ \mathfrak M_m(t,\xi)}\\
  &+\rmA_m(\xi)(e^{\mathbf{1}_{k \neq 0}c_M \nu^{\frac{1}{3}}t-\mathbf{1}_{m \neq 0}c_M \nu^{\frac{1}{3}}t}-1)e^{\lambda(t)\left(|k,\eta|^s-|m,\xi|^s\right)}\frac{\langle k,\eta\rangle^\sigma}{\langle m,\xi\rangle^\sigma}\frac{ G_k(t,\eta)}{ G_m(t,\xi)}\frac{ \mathfrak M_k(t,\eta)}{ \mathfrak M_m(t,\xi)}\\
  &+\rmA_m(\xi)\left(e^{\lambda(t)\left(|k,\eta|^s-|m,\xi|^s\right)}-1\right)\frac{\langle k,\eta\rangle^\sigma}{\langle m,\xi\rangle^\sigma}\frac{ G_k(t,\eta)}{ G_m(t,\xi)}\frac{ \mathfrak M_k(t,\eta)}{ \mathfrak M_m(t,\xi)}\\
  &+\rmA_m(\xi)\left(\frac{\langle k,\eta\rangle^\sigma}{\langle m,\xi\rangle^\sigma}-1\right)\frac{ G_k(t,\eta)}{ G_m(t,\xi)}\frac{ \mathfrak M_k(t,\eta)}{ \mathfrak M_m(t,\xi)}\\
  &+\rmA_m(\xi)\left(\frac{ G_k(t,\eta)}{ G_m(t,\xi)}-1\right)\frac{ \mathfrak M_k(t,\eta)}{ \mathfrak M_m(t,\xi)}\\
  &+\rmA_m(\xi)\left(\frac{ \mathfrak M_k(t,\eta)}{ \mathfrak M_m(t,\xi)}-1\right).
\end{align*}

Then we split ${\mathrm T}_N$ into
\begin{align*}
  {\mathrm T}_N={\mathrm T}_{N,1}+{\mathrm T}_{N,2}+{\mathrm T}_{N,3}+{\mathrm T}_{N,4}+{\mathrm T}_{N,5}+{\mathrm T}_{N,6},
\end{align*}
where
\begin{align*}
  {\mathrm T}_{N,1}=&i\sum_{m,k\in\mathbb Z}\int_{\eta,\xi}\rmA_k(\eta){\hat f}_k(\eta)\hat {\bar u}_{k-m}(\eta-\xi)_{<N/8}\cdot (m,\xi)\rmA_m(\xi)\hat f_{m}(\xi)_{N} \\
  &\times \left(\frac{ W_k(t,\eta)}{ W_m(t,\xi)}-1\right)e^{\mathbf{1}_{k \neq 0}c_M \nu^{\frac{1}{3}}t-\mathbf{1}_{m \neq 0}c_M \nu^{\frac{1}{3}}t}e^{\lambda(t)\left(|k,\eta|^s-|m,\xi|^s\right)}\frac{\langle k,\eta\rangle^\sigma}{\langle m,\xi\rangle^\sigma} \frac{ G_k(t,\eta)}{ G_m(t,\xi)}\frac{ \mathfrak M_k(t,\eta)}{ \mathfrak M_m(t,\xi)}d\eta d\xi,
\end{align*}
\begin{align*}
  {\mathrm T}_{N,2}=&i\sum_{m,k\in\mathbb Z}\int_{\eta,\xi}\rmA_k(\eta){\hat f}_k(\eta)\hat {\bar u}_{k-m}(\eta-\xi)_{<N/8}\cdot (m,\xi)\rmA_m(\xi)\hat f_{m}(\xi)_{N} \\
  &\times (e^{\mathbf{1}_{k \neq 0}c_M \nu^{\frac{1}{3}}t-\mathbf{1}_{m \neq 0}c_M \nu^{\frac{1}{3}}t}-1)e^{\lambda(t)\left(|k,\eta|^s-|m,\xi|^s\right)}\frac{\langle k,\eta\rangle^\sigma}{\langle m,\xi\rangle^\sigma}\frac{ G_k(t,\eta)}{ G_m(t,\xi)}\frac{ \mathfrak M_k(t,\eta)}{ \mathfrak M_m(t,\xi)}d\eta d\xi,
\end{align*}
\begin{align*}
  {\mathrm T}_{N,3}=&i\sum_{m,k\in\mathbb Z}\int_{\eta,\xi}\rmA_k(\eta){\hat f}_k(\eta)\hat {\bar u}_{k-m}(\eta-\xi)_{<N/8}\cdot (m,\xi)\rmA_m(\xi)\hat f_{m}(\xi)_{N} \\
  &\times \left(e^{\lambda(t)\left(|k,\eta|^s-|m,\xi|^s\right)}-1\right)\frac{\langle k,\eta\rangle^\sigma}{\langle m,\xi\rangle^\sigma}\frac{ G_k(t,\eta)}{ G_m(t,\xi)}\frac{ \mathfrak M_k(t,\eta)}{ \mathfrak M_m(t,\xi)}d\eta d\xi,
\end{align*}
\begin{align*}
  {\mathrm T}_{N,4}=&i\sum_{m,k\in\mathbb Z}\int_{\eta,\xi}\rmA_k(\eta){\hat f}_k(\eta)\hat {\bar u}_{k-m}(\eta-\xi)_{<N/8}\cdot (m,\xi)\rmA_m(\xi)\hat f_{m}(\xi)_{N} \\
  &\times \left(\frac{\langle k,\eta\rangle^\sigma}{\langle m,\xi\rangle^\sigma}-1\right)\frac{ G_k(t,\eta)}{ G_m(t,\xi)}\frac{ \mathfrak M_k(t,\eta)}{ \mathfrak M_m(t,\xi)}d\eta d\xi,
\end{align*}
\begin{align*}
  {\mathrm T}_{N,5}=&i\sum_{m,k\in\mathbb Z}\int_{\eta,\xi}\rmA_k(\eta){\hat f}_k(\eta)\hat {\bar u}_{k-m}(\eta-\xi)_{<N/8}\cdot (m,\xi)\rmA_m(\xi)\hat f_{m}(\xi)_{N} \left(\frac{ G_k(t,\eta)}{ G_m(t,\xi)}-1\right)\frac{ \mathfrak M_k(t,\eta)}{ \mathfrak M_m(t,\xi)}d\eta d\xi,
\end{align*}
\begin{align*}
  {\mathrm T}_{N,6}=&i\sum_{m,k\in\mathbb Z}\int_{\eta,\xi}\rmA_k(\eta){\hat f}_k(\eta)\hat {\bar u}_{k-m}(\eta-\xi)_{<N/8}\cdot (m,\xi)\rmA_m(\xi)\hat f_{m}(\xi)_{N} \left(\frac{ \mathfrak M_k(t,\eta)}{ \mathfrak M_m(t,\xi)}-1\right)d\eta d\xi.
\end{align*}
\subsection{Term ${\mathrm T}_{N,1}$}
Here we may face a special case, $m=0$ and $k\neq0$, that we need to use $\hat {\bar u}_{k-m}(\eta-\xi)_{<N/8}$ to absorb $e^{\mathbf{1}_{k \neq 0}c_M \nu^{\frac{1}{3}}t-\mathbf{1}_{m \neq 0}c_M \nu^{\frac{1}{3}}t}$ and gain more time decay. So for a lower order estimate, we need the strong enhanced dissipation \eqref{eq-boot-lfn0}.

If $t\ge2\max(|\eta|,|\xi|,|k|,|m|)$, then we have
\begin{align*}
  \left| \frac{W_k(t,\eta)}{W_m(t,\xi)}-1\right|=0,
\end{align*}
and then ${\mathrm T}_{N,1}=0$. If $1\le t\le 2\max(|\xi|,|\eta|,|k|,|m|)$, by using Lemma \ref{lem-tran-W} and the fact $|m,\xi|\approx|k,\eta|$, we have
\begin{align*}
  \left|{\mathrm T}_{N,1}\right|\lesssim& \left(t^2e^{c_M \nu^{\frac{1}{3}}t}\left\|{\bar u}_{\neq}\right\|_{\mathcal G^{s,\lambda,5}}+t^2\left\|q\right\|_{\mathcal G^{s,\lambda,5}}\right) \\
  &\times \left(t^{-2c_1}\||\nabla|^{\frac{s}{2}}\rmA f_{\sim N}\|_{L^2}^2+ \|\sqrt{\frac{ \pa_t w}{w}}\rmA f_{\sim N}\|_{L^2}^2+\|\sqrt{\frac{ \pa_t g}{g}}\rmA f_{\sim N}\|_{L^2}^2\right) \\
  &+\left( t^2e^{c_M \nu^{\frac{1}{3}}t}\left\|{\bar u}_{\neq}\right\|_{\mathcal G^{s,\lambda,\sigma'}}+t^2\left\|q\right\|_{\mathcal G^{s,\lambda,5}}\mathbf1_{\beta<\frac{1}{6}}  \right)t^{-(1+s)}\||\nabla|^{\frac{s}{2}}\rmA f_{\sim N}\|_{L^2}^2\\
  &+t^2\left\|\pa_vq\right\|_{\mathcal G^{s,\lambda,4}}\mathbf1_{\beta\ge\frac{1}{6}}t^{-1}\|\rmA f_{\sim N}\|_{L^2}^2.
\end{align*}
From Lemma \ref{lem-elliptic-low-1} and the bootstrap hypotheses \eqref{eq-boot-hf} and \eqref{eq-boot-lq}, we have
\begin{align*}
  \left(t^2e^{c_M \nu^{\frac{1}{3}}t}\left\|{\bar u}_{\neq}\right\|_{\mathcal G^{s,\lambda,5}}+t^2\left\|q\right\|_{\mathcal G^{s,\lambda,5}}\right)\lesssim \varepsilon\nu^\beta.
\end{align*}
Similarly, by using the bootstrap hypotheses \eqref{eq-boot-lfn0}, we have for $\gamma\geq 7$,
\begin{align*}
  t^2e^{c_M \nu^{\frac{1}{3}}t}\left\|{\bar u}_{\neq}\right\|_{\mathcal G^{s,\lambda,\sigma'}}\lesssim&e^{c_M \nu^{\frac{1}{3}}t}\left\|f_{\neq}\right\|_{\mathcal G^{s,\lambda,7}}\lesssim \frac{1}{1+c_M \nu^{\frac{1}{3}}t}\|A^\gamma f_{\neq}(t)\|_{L^2}\lesssim \frac{\varepsilon \nu^{\beta}}{1+c_M \nu^{\frac{1}{3}}t}\lesssim \varepsilon t^{-3\beta}.
\end{align*}
Remark that $3\beta+s\ge \frac{1}{2}$, and $t^{-3\beta}t^{-(1+s)}\le t^{-2c_1}$.

If $\beta<\frac{1}{6}$, then $s>\frac{1}{3}$, by taking $c_1=\frac{5}{8}$ we also have
\begin{align*}
  t^2\left\|q\right\|_{\mathcal G^{s,\lambda,5}}\mathbf1_{\beta<\frac{1}{6}} t^{-(1+s)}\lesssim \varepsilon t^{-2c_1}.
\end{align*}

At last, if $\beta\ge\frac{1}{6}$, $\ln(\nu^{-1})\nu^{\beta}\lesssim \nu^{\frac{1}{2}\beta}$. By using the bootstrap hypotheses \eqref{eq-boot-lq}, we have
\begin{align*}
  &\int^{T^*}_1 \frac{1}{t}\left\|t^2\pa_vq\right\|_{\mathcal G^{s,\lambda,4}} dt\\
  =&\int^{\nu^{-3}}_1 \frac{1}{t}\left\|t^2\pa_vq\right\|_{\mathcal G^{s,\lambda,4}} dt+\int^{T^*}_{\nu^{-3}}\frac{1}{t}\left\|t^2\pa_vq\right\|_{\mathcal G^{s,\lambda,4}} dt\\
  \lesssim&\ln(\nu^{-1})\left\|t^2A^{\gamma}q\right\|_{L^\infty_{[1,T^*]}L^2}+\int^{T^*}_{\nu^{-3}}\frac{1}{t^{1-\frac{1}{6}}}\nu^{\frac{1}{2}}\left\|t^2\pa_vq\right\|_{\mathcal G^{s,\lambda,4}} dt\\
  \lesssim&\ln(\nu^{-1})\left\|t^2A^{\gamma}q\right\|_{L^\infty_{[1,T^*]}L^2}+ \left(\int^{T^*}_{\nu^{-3}}\frac{1}{t^{\frac{5}{3}}}dt\right)^{\frac{1}{2}} \left(\int^{T^*}_{\nu^{-3}}\nu\left\|t^2A^{\gamma}\pa_vq(t)\right\|_{L^2} dt\right)^{\frac{1}{2}}\\
  \lesssim&\varepsilon\nu^{\frac{1}{2}\beta}.
\end{align*}
Therefore,
\begin{align*}
  \int^{T^*}_1\sum_{N\ge8}\left|{\mathrm T}_{N,1}(t)\right|dt\lesssim&\varepsilon \int^{T^*}_1 \mathrm{CK}_{\lambda}(t)+\mathrm{CK}_{W}(t)+\mathrm{CK}_{G}(t)dt+\varepsilon \nu^{\frac{1}{2}\beta} \|\rmA f\|_{L^\infty_{[1,T^*]}L^2}^2\lesssim \varepsilon^3\nu^{2\beta}.
\end{align*}
\subsection{Term ${\mathrm T}_{N,2}$}\label{sec-tran-t2}
 As
\begin{align*}
  |e^{\mathbf{1}_{k \neq 0}c_M \nu^{\frac{1}{3}}t-\mathbf{1}_{m \neq 0}c_M \nu^{\frac{1}{3}}t}-1|\le& (\mathbf{1}_{k=0,m\neq0}+\mathbf{1}_{k\neq0,m=0})\left|e^{c_M \nu^{\frac{1}{3}}t}-1\right|\\
  \lesssim& \nu^{\frac{1}{3}}te^{c_M \nu^{\frac{1}{3}}t}(\mathbf{1}_{k=0,m\neq0}+\mathbf{1}_{k\neq0,m=0})
\end{align*}
We have
\begin{align*}
   |{\mathrm T}_{N,2}|\lesssim&\sum_{m=0,k\neq0}+\sum_{k=0,m\neq0}\int_{\eta,\xi}\rmA_k(\eta) \left|{\hat f}_k(\eta)\right|e^{\lambda(t)|k-m,\eta-\xi|^s}\left| \hat {\bar u}_{k-m}(\eta-\xi)_{<N/8}\right| \\
   &\qquad\qquad\qquad\qquad\qquad\qquad\times\nu^{\frac{1}{3}}te^{c_M \nu^{\frac{1}{3}}t}\rmA_m(\xi)|m,\xi| \left|\hat f_{m}(\xi)_{N}\right|d\eta d\xi\\
   \lesssim&\sum_{k\neq0}\int_{\eta,\xi}\rmA_k(\eta) \left|{\hat f}_k(\eta)\right|e^{\lambda(t)|k,\eta-\xi|^s} \left|\hat {\bar u}_{k}(\eta-\xi)_{<N/8} \right|\nu^{\frac{1}{3}}te^{c_M \nu^{\frac{1}{3}}t}\rmA_0(\xi)|\xi| \left|\hat f_{0}(\xi)_{N}\right|d\eta d\xi\\
   &+\sum_{m\neq0}\int_{\eta,\xi}\rmA_0(\eta) \left|{\hat f}_0(\eta)\right|e^{\lambda(t)|m,\eta-\xi|^s} \left|\hat {\bar u}_{-m}(\eta-\xi)_{<N/8} \right|\nu^{\frac{1}{3}}te^{c_M \nu^{\frac{1}{3}}t}\rmA_m(\xi)|m| \left|\hat f_{m}(\xi)_{N} \right|d\eta d\xi\\
   &+\sum_{m\neq0}\int_{\eta,\xi}\rmA_0(\eta) \left|{\hat f}_0(\eta) \right|e^{\lambda(t)|m,\eta-\xi|^s} \left|\hat {\bar u}_{-m}(\eta-\xi)_{<N/8}\right|\nu^{\frac{1}{3}}te^{c_M \nu^{\frac{1}{3}}t}\rmA_m(\xi)|\xi| \left|\hat f_{m}(\xi)_{N} \right|d\eta d\xi\\
   =&I_1+I_2+I_3.
\end{align*}
The term $I_2$ is easy to treat, indeed, by using Lemma \ref{lem-elliptic-low-1} and the bootstrap hypotheses \eqref{eq-boot-lfn0}, we have
\begin{align*}
  I_2\lesssim&\sum_{m\neq0}\int_{\eta,\xi}\frac{|\eta|^{\frac{s}{2}}}{t^{c_1}}\rmA_0(\eta) \left|{\hat f}_0(\eta)\right|\nu^{\frac{1}{6}}\rmA_m(\xi) \left|\hat f_{m}(\xi)_{N}\right|\\
  &\qquad\qquad\qquad\times \nu^{\frac{1}{6}}t^{1+c_1}e^{c_M \nu^{\frac{1}{3}}t}e^{\lambda(t)|m,\eta-\xi|^s}|m| \left|\hat {\bar u}_{-m}(\eta-\xi)_{<N/8}\right|d\eta d\xi\\
  \lesssim& \|A^\gamma f_{\neq}(t)\|_{L^2} t^{-c_1}\||\nabla|^{\frac{s}{2}}\rmA f_{\sim N}\|_{L^2}\nu^{\frac{1}{6}}\|\rmA f_{\sim N}\|_{L^2}\\
  \lesssim&\varepsilon\nu^\beta \left(t^{-2c_1}\||\nabla|^{\frac{s}{2}}\rmA f_{\sim N}\|_{L^2}^2+ \nu^{\frac{1}{3}}\|\rmA f_{\sim N}\|_{L^2}^2\right).
\end{align*}

For $I_1$ and $I_3$, it holds that
\begin{align*}
  \frac{N}{2}\le|\xi|\le \frac{3N}{2},\quad |k,\eta-\xi|\le \frac{3N}{32},\\
  \frac{N}{2}\le|m,\xi|\le \frac{3N}{2},\quad |m,\eta-\xi|\le \frac{3N}{32},
\end{align*}
thus $|m|,|k|\le|\xi|\approx |\eta|$.

Next, we give the estimate for $I_3$, and $I_1$ can be treated in the same way. For $|\xi|\ge 2|m|t$ and $|\xi|\le \frac{1}{2}|m|t$, we have
\begin{align*}
  I_3\lesssim&\sum_{m\neq0}\int_{\eta,\xi}\rmA_0(\eta)|\eta|^{\frac{s}{2}} \left|{\hat f}_0(\eta) \right|e^{\lambda(t)|m,\eta-\xi|^s} \left|\hat {\bar u}_{-m}(\eta-\xi)_{<N/8}\right|\nu^{\frac{1}{3}}te^{c_M \nu^{\frac{1}{3}}t}\rmA_m(\xi)|\xi|^{1-\frac{s}{2}} \left|\hat f_{m}(\xi)_{N} \right|d\eta d\xi.
\end{align*}
It holds that $|\xi|\lesssim |\xi-mt|$ and
\begin{align*}
  &\rmA_m(\xi)|\xi|^{1-\frac{s}{2}}|\xi|^{1-\frac{s}{2}} \left|\hat f_{m}(\xi)_{N} \right|\\
  =& \left(\rmA_m(\xi)|\xi|^{\frac{s}{2}}|\xi|^{1-\frac{s}{2}} \left|\hat f_{m}(\xi)_{N} \right|\right)^{\frac{s}{2-s}} \left(\rmA_m(\xi)|\xi|^{\frac{s}{2}}|\xi|\left|\hat f_{m}(\xi)_{N} \right| \right)^{\frac{2-2s}{2-s}}\\
  \lesssim&\nu^{-\frac{1-s}{2-s}}\left(\rmA_m(\xi)|\xi|^{\frac{s}{2}}|\xi|^{1-\frac{s}{2}} \left|\hat f_{m}(\xi)_{N} \right|\right)^{\frac{s}{2-s}} \left(\nu^{\frac{1}{2}}\rmA_m(\xi)|\xi|^{\frac{s}{2}}|\xi|\left|\hat f_{m}(\xi)_{N} \right| \right)^{\frac{2-2s}{2-s}}.
\end{align*}
Then we have
\begin{align*}
  I_3\lesssim&\left\|t^{-c_1}|\nabla|^{\frac{s}{2}}\rmA f_{\sim N}\right\|_{L^2}\left\|t^{-c_1}|\nabla|^{\frac{s}{2}}\rmA f_{\sim N}\right\|_{L^2}^{\frac{s}{2-s}}\|\nu^{\frac{1}{2}}\rmA|\nabla_L| f\|_{L^2}^{\frac{2-2s}{2-s}}\\
  &\times t^{\frac{2c_1}{2-s}-1}\nu^{\beta+\frac{1}{3}-\frac{1-s}{2-s}}e^{c_M \nu^{\frac{1}{3}}t}t^2\nu^{-\beta}\|{\bar u}_{\neq}\|_{\mathcal G^{s,\lambda,2}}\\
  \lesssim& \varepsilon \left(t^{-2c_1}\||\nabla|^{\frac{s}{2}}\rmA f_{\sim N}\|_{L^2}^2+ \nu\||\Delta_L|^{1/2}\rmA f_{\sim N}\|_{L^2}^2\right).
\end{align*}
Here we use Lemma \ref{lem-elliptic-low-1} and the fact that
\begin{align*}
  \beta+\frac{1}{3}- \frac{1-s}{2-s}=\frac{2\beta-3\beta^2}{3-3\beta}\ge0.
\end{align*}

 For $|\xi|\approx 2|m|t$, we have
\begin{align*}
  1\approx|\eta|^{\frac{1}{2}s}(|m|t)^{-\frac{1}{2}s},
\end{align*}
and
\begin{align*}
  |I_3|\lesssim&\int_{\eta,\xi}\rmA_0(\eta) \left|{\hat f}_0(\eta) \right|e^{\lambda(t)|m,\eta-\xi|^s}\nu^{\frac{1}{6}} |m|t^2e^{c_M \nu^{\frac{1}{3}}t}\left|\hat {\bar u}_{-m}(\eta-\xi)_{<N/8}\right|\nu^{\frac{1}{6}}\rmA_m(\xi) \left|\hat f_{m}(\xi)_{N} \right|d\eta d\xi\\
  \lesssim&\int_{\eta,\xi}\frac{|\eta|^{\frac{s}{2}}}{t^{c_1}}\rmA_0(\eta)\left|{\hat f}_0(\eta) \right| \nu^{\frac{1}{6}}\rmA_m(\xi) \left|\hat f_{m}(\xi)_{N} \right| \\
  &\times\nu^{-\frac{1}{3}\beta}(\nu^{\frac{1}{3}}t)^{\frac{1}{2}+\beta}t^{\frac{3}{2}+c_1-3\beta-\frac{1}{2}s}e^{c_M \nu^{\frac{1}{3}}t}e^{\lambda(t)|m,\eta-\xi|^s}|m|\hat {\bar u}_{-m}(\eta-\xi)_{<N/8}d\eta d\xi\\
  \lesssim&\nu^{-\frac{1}{3}\beta}(\nu^{\frac{1}{3}}t)^{\frac{1}{2}+\beta}\frac{1}{1+c_M \nu^{\frac{1}{3}}t}\|A^\gamma f_{\neq}(t)\|_{L^2} \left(t^{-2c_1}\||\nabla|^{\frac{s}{2}}\rmA f_{\sim N}\|_{L^2}^2+ \nu^{\frac{1}{3}}\|\rmA f_{\sim N}\|_{L^2}^2\right)\\
  \lesssim&\varepsilon\left(t^{-2c_1}\||\nabla|^{\frac{s}{2}}\rmA f_{\sim N}\|_{L^2}^2+ \nu^{\frac{1}{3}}\|\rmA f_{\sim N}\|_{L^2}^2\right).
\end{align*}
Here we use Lemma \ref{lem-elliptic-low-1} the bootstrap hypotheses \eqref{eq-boot-lfn0} and the fact that
\begin{align*}
  \frac{3}{2}+c_1-3\beta-\frac{1}{2}s\le 2.
\end{align*}

\subsection{Term ${\mathrm T}_{N,3}$}\label{sec-tran-t3}
For this term, we have
\begin{align*}
  \left(e^{\lambda\left(|k,\eta|^s-|m,\xi|^s\right)}-1\right)|m,\xi|\lesssim& \frac{\langle k-m,\xi-\eta\rangle}{\left(|k|+|m|+|\eta|+|\xi|\right)^{1-s}}e^{c\lambda|k-m,\xi-\eta|^s}|m,\xi|\\
  \lesssim&|k,\eta|^{\frac{s}{2}}|m,\xi|^{\frac{s}{2}}\langle k-m,\xi-\eta\rangle e^{c\lambda|k-m,\xi-\eta|^s}.
\end{align*}
Then it is clear that
\begin{align*}
  |{\mathrm T}_{N,3}|\lesssim&t^2\|{\bar u}\|_{\mathcal G^{s,\lambda,3}} t^{-2c_1}\left\||\nabla|^{\frac{s}{2}}\rmA f_{\sim N}\right\|_{L^2}^2\\
  \lesssim& \varepsilon t^{-2c_1}\||\nabla|^{\frac{s}{2}}\rmA f_{\sim N}\|_{L^2}^2.
\end{align*}

\subsection{Term ${\mathrm T}_{N,4}$}\label{sec-tran-t4}
For this term, we have
\begin{align*}
  \left(\frac{\langle k,\eta\rangle^\sigma}{\langle m,\xi\rangle^\sigma}-1\right)|m,\xi| \lesssim \frac{|k-m,\eta-\xi|}{\langle m,\xi\rangle}|m,\xi|  \lesssim|k,\eta|^{\frac{s}{2}}|m,\xi|^{\frac{s}{2}}\langle k-m,\xi-\eta\rangle,
\end{align*}
and
\begin{align*}
  |{\mathrm T}_{N,4}|\lesssim  \varepsilon t^{-2c_1}\||\nabla|^{\frac{s}{2}}\rmA f_{\sim N}\|_{L^2}^2.
\end{align*} 
\subsection{Term ${\mathrm T}_{N,5}$}
This term is similar to ${\mathrm T}_{N,1}$. If $t\ge 2\max(|\xi|,|\eta|,|k|,|m|)$, ${\mathrm T}_{N,5}=0$. If $1\le t\le 2\max(|\xi|,|\eta|,|k|,|m|)$, by using Lemma \ref{lem-trans-G-s} and the fact $|m,\xi|\approx|k,\eta|$, we have for $m\neq k$ or $\beta< \frac{1}{6}$ that
\begin{align*}
  \left| \frac{G_k(t,\eta)}{G_m(t,\xi)}-1\right|\frac{|m,\xi|}{t^2}\lesssim \frac{\langle k-m,\xi-\eta\rangle |k,\eta|^{\frac{s}{2}} |m,\xi|^{\frac{s}{2}}}{t^{1+s}}e^{C\kappa|\xi-\eta,k-m|^s},
\end{align*}
and for ($m\neq k$, $\beta\ge\frac{1}{6}$) that
\begin{align*}
  \left| \frac{G_k(t,\eta)}{G_m(t,\xi)}-1\right|\frac{|m,\xi|}{t^2}\lesssim \frac{|\xi-\eta |}{t}e^{C\kappa|\xi-\eta |^s}.
\end{align*}
Thus we have
\begin{align*}
  \left|{\mathrm T}_{N,5}\right|\lesssim&\left( t^2 \left\|{\bar u}_{\neq}\right\|_{\mathcal G^{s,\lambda,\sigma'}}+t^2\left\|q\right\|_{\mathcal G^{s,\lambda,5}}\mathbf1_{\beta<\frac{1}{6}}  \right)t^{-(1+s)}\||\nabla|^{\frac{s}{2}}\rmA f_{\sim N}\|_{L^2}^2\\
  &+t^2\left\|\pa_vq\right\|_{\mathcal G^{s,\lambda,4}}\mathbf1_{\beta\ge\frac{1}{6}}t^{-1}\|\rmA f_{\sim N}\|_{L^2}^2.
\end{align*}
and
\begin{align*}
  \int^{T^*}_1 \sum_{N\ge8}\left|{\mathrm T}_{N,5}(t)\right|dt \lesssim \varepsilon^3\nu^{2\beta}.
\end{align*}
\subsection{Term ${\mathrm T}_{N,6}$}\label{sec-tran-t6}
For this term, we need to estimate
\begin{align*}
  \left|\frac{ \mathfrak M_k(t,\eta)}{ \mathfrak M_m(t,\xi)}-1\right|=\left|\frac{ \mathfrak m_m(t,\xi)-\mathfrak m_k(t,\eta)}{ \mathfrak m_k(t,\eta)}\right|.
\end{align*}

Recall that $\mathfrak m_0(t,\eta)=1$ and
\begin{align*}
  \mathfrak m_k(t,\eta)=\exp \left (\frac{1}{k}\left[\arctan \left(\nu^{\frac{1}{3}} k t-\nu^{\frac{1}{3}} \eta\right)+\arctan \left(\nu^{\frac{1}{3}} \eta\right)\right]\right ), \text{ for }k \neq 0,
\end{align*}
As $1\le \mathfrak m_k(t,\eta)\le \pi$, for $m\neq k$ and $m,k\neq 0$, we have
\begin{align*}
  &\left| \mathfrak m_m(t,\xi)-\mathfrak m_k(t,\eta)\right|\\
  \lesssim& \left| \log \mathfrak m_m(t,\xi)-\log\mathfrak m_k(t,\eta)\right|\\
  =&\left|\frac{1}{m}\left[\arctan \left(\nu^{\frac{1}{3}} m t-\nu^{\frac{1}{3}} \xi\right)+\arctan \left(\nu^{\frac{1}{3}} \xi\right)\right]-\frac{1}{k}\left[\arctan \left(\nu^{\frac{1}{3}} k t-\nu^{\frac{1}{3}} \eta\right)+\arctan \left(\nu^{\frac{1}{3}} \eta\right)\right]\right|\\
  \le&\frac{|m-k|}{|mk|}\left|\arctan \left(\nu^{\frac{1}{3}} k t-\nu^{\frac{1}{3}} \eta\right)+\arctan \left(\nu^{\frac{1}{3}} \eta\right)\right|\\
  &+\frac{1}{|m|}\left|\arctan \left(\nu^{\frac{1}{3}} k t-\nu^{\frac{1}{3}} \eta\right)-\arctan \left(\nu^{\frac{1}{3}} m t-\nu^{\frac{1}{3}} \xi\right)\right|\\
  &+\frac{1}{|m|}\left|\arctan \left(\nu^{\frac{1}{3}} \eta\right)-\arctan \left(\nu^{\frac{1}{3}} \xi\right)\right|\\
  \lesssim&\nu^{\frac{1}{3}}\frac{|m-k|t+|\xi-\eta|}{|m|}.
\end{align*}
If $|\xi|\ge |m|$, then we have
\begin{align*}
  \left|\frac{ \mathfrak M_k(t,\eta)}{ \mathfrak M_m(t,\xi)}-1\right|\frac{|m,\xi|}{t^2}\lesssim \nu^{\frac{1}{3}}\frac{1}{t}|\xi|\langle m-k,\xi-\eta\rangle.
\end{align*}
we can treat it in a way similar to ${\mathrm T}_{N,2}$. And if $|\xi|\le |m|$, we have
\begin{align*}
  \left|\frac{ \mathfrak M_k(t,\eta)}{ \mathfrak M_m(t,\xi)}-1\right|\frac{|m,\xi|}{t^2}\lesssim \nu^{\frac{1}{3}}\frac{1}{t}\langle m-k,\xi-\eta\rangle,
\end{align*}
and
\begin{align*}
  \left|{\mathrm T}_{N,6}\right|\lesssim&\varepsilon \nu^{\beta}  \nu^{\frac{1}{3}}\|\rmA f_{\sim N}\|_{L^2}^2.
\end{align*}

For the case $m=k\neq 0$. We have
\begin{align*}
  &\left| \mathfrak m_m(t,\xi)-\mathfrak m_k(t,\eta)\right|\\
  \lesssim&\frac{1}{|m|}\left|\arctan \left(\nu^{\frac{1}{3}} m t-\nu^{\frac{1}{3}} \eta\right)-\arctan \left(\nu^{\frac{1}{3}} m t-\nu^{\frac{1}{3}} \xi\right)\right|\\
  &+\frac{1}{|m|}\left|\arctan \left(\nu^{\frac{1}{3}} \eta\right)-\arctan \left(\nu^{\frac{1}{3}} \xi\right)\right|\\
  \lesssim&\nu^{\frac{1}{3}}\frac{|\xi-\eta|}{|m|}.
\end{align*}
If $|\xi|< 10|m|t$, we have
\begin{align*}
  \left|\frac{ \mathfrak M_k(t,\eta)}{ \mathfrak M_m(t,\xi)}-1\right|\frac{|m,\xi|}{t^2}\lesssim \nu^{\frac{1}{3}}\frac{1}{t}\langle m-k,\xi-\eta\rangle.
\end{align*}
If $|\xi|\ge 10|m|t$, as
\begin{align*}
  |\xi-\eta|\le \frac{1}{5}|m,\xi|\le \frac{1}{4}|\xi|,
\end{align*}
we have $\xi\eta>0$ and $|\eta|\ge 5|m|t$. Then we have
\begin{align*}
  &\left| \mathfrak m_m(t,\xi)-\mathfrak m_k(t,\eta)\right|\\
  \lesssim&\nu^{\frac{1}{3}}\int^t_0 \left|\frac{1}{1+\nu^{\frac{2}{3}}(\eta-mt')^2}-\frac{1}{1+\nu^{\frac{2}{3}}(\xi-mt')^2} \right|dt'\\
  \lesssim&\nu|\xi-\eta|\int^t_0\int^1_0 \left|\frac{\xi+(\eta-\xi)s-mt'}{\left[1+\nu^{\frac{2}{3}}(\xi+(\eta-\xi)s-mt')^2\right]^2}\right|dsdt'\\
  \lesssim&\nu|\xi-\eta|t\frac{|\xi|}{\left[1+\nu^{\frac{2}{3}}|\xi|^2\right]^2}.
\end{align*}
Then we have
\begin{align*}
  \left|\frac{ \mathfrak M_k(t,\eta)}{ \mathfrak M_m(t,\xi)}-1\right|\frac{|m,\xi|}{t^2}\lesssim&\nu|\xi-\eta|\frac{1}{t}\frac{|\xi|^2}{\left[1+\nu^{\frac{2}{3}}|\xi|^2\right]^2}\lesssim \nu^{\frac{1}{3}}|\xi-\eta|\frac{1}{t}.
\end{align*}

For the case $k=0,m\neq0$ $(k\neq0,m=0)$, we have
\begin{align*}
  &\left| \mathfrak m_m(t,\xi)-\mathfrak m_k(t,\eta)\right|\\
  \lesssim& \left| \log \mathfrak m_m(t,\xi)-\log\mathfrak m_k(t,\eta)\right|\\
  =&\left|\frac{1}{m}\left[\arctan \left(\nu^{\frac{1}{3}} m t-\nu^{\frac{1}{3}} \xi\right)+\arctan \left(\nu^{\frac{1}{3}} \xi\right)\right]\right|\\
  \lesssim&\nu^{\frac{1}{3}}t.
\end{align*}
Then we can treat this case in the same method to ${\mathrm T}_{N,2}$.

Combing all these estimates, we have
\begin{align*}
  &\int^{T^*}_1\sum_{N\ge8}\left|{\mathrm T}_{N}(t)\right|dt\\
  \lesssim&\varepsilon \int^{T^*}_1 \nu^{\frac{1}{3}}\|Af(t')\|_{L^2}^2 +\nu \||\Delta_L|^{1/2}\rmA f(t')\|_{L^2}^2+\mathrm{CK}_{\lambda}(t)+\mathrm{CK}_{W}(t)+\mathrm{CK}_{G}(t)dt\\
  &+\varepsilon \nu^{\frac{1}{2}\beta} \|\rmA f\|_{L^\infty_{[1,T^*]}L^2}^2\lesssim \varepsilon^3\nu^{2\beta}.
\end{align*}
\subsection{Remainder}
In this subsection, we study the remainder term
\begin{align*}
  \mathcal R=&2\pi\sum_{N\in \mathbb D}\sum_{N/8\le N'\le 8N}\int \rmA f[\rmA({\bar u}_{N}\cdot\nabla_{z,v} f_{N'})-{\bar u}_{N}\cdot\nabla_{z,v}\rmA f_{N'}]dzdv\\
  =&2\pi\sum_{N\in \mathbb D}\sum_{N/8\le N'\le 8N}\int \rmA f\rmA({\bar u}_{N}\cdot\nabla_{z,v} f_{N'}) dzdv\\
  &-2\pi\sum_{N\in \mathbb D}\sum_{N/8\le N'\le 8N}\int \rmA f {\bar u}_{N}\cdot\nabla_{z,v}\rmA f_{N'}dzdv\\
  =&\mathcal R_a+\mathcal R_b.
\end{align*}

Consider first $\mathcal R_a$, written on the Fourier side:
\begin{align*}
  \mathcal R_a=&\sum_{N\in \mathbb D}\sum_{N/8\le N'\le 8N}\sum_{m,k\in\mathbb Z}\int_{\eta,\xi}\rmA_k(\eta){\hat f}_k(\eta)\rmA_k(\eta)\hat {\bar u}_m(\xi)_N\cdot\widehat {\nabla f}_{k-m}(\eta-\xi)_{N'}d\eta d\xi.
\end{align*}
On the support of the integrand, $|k,\eta|\lesssim|m,\xi|\approx|k-m,\eta-\xi|$. Hence by \eqref{inq-s3},
\begin{align*}
  |k,\eta|^s\le c|k-m,\eta-\xi|^s+c|m,\xi|^s.
\end{align*}
Actually, it is only true for both $N,N'\neq \frac{1}{2}$. If $|k-m,\eta-\xi|$ or $|m,\xi|$ extremely small, we have
\begin{align*}
  |k,\eta|,|k-m,\eta-\xi|,|m,\xi|\le C,
\end{align*}
and 
\begin{align*}
  e^{|k,\eta|^s}\le Ce^{c|k-m,\eta-\xi|^s+c|m,\xi|^s}.
\end{align*}
Thus we can use the $\frac{1}{t^2}$ decay of $\hat {\bar u}_m(\xi)_N$.

For $\mathcal R_b$, we can move the derivative on $f_{N'}$ to ${\bar u}_N$. It follows that
\begin{align*}
  &\int^{T^*}_1\left|\mathcal R(t)\right|dt\lesssim \varepsilon^3\nu^{2\beta}.
\end{align*}

\subsection{Dissipation error term}\label{sec-diss-error}
In this part, we give the estimate for the dissipation error term.

We have
\begin{align*}
  {\mathrm E}=&-\nu\int \rmA f\rmA \left((1-(v')^2)(\pa_v-t\pa_z)^2f\right)dz dv\\
  =&\nu\int \rmA(\pa_v-t\pa_z)f \rmA \left((1-(v')^2)(\pa_v-t\pa_z)f\right)dz dv-2\nu\int \rmA f\rmA \left(\pa_vh(\pa_v-t\pa_z)f\right)dz dv\\
  =&{\mathrm E}^1+{\mathrm E}^2.
\end{align*}
For ${\mathrm E}^1$, by using Lemma \ref{lem-product-1} and Lemma \ref{lem-product-2}, we have
\begin{align*}
  \left|{\mathrm E}^1\right|\lesssim \nu \left\|\rmA^{\mathrm R}h\right\|_{L^2} \left\| \rmA(\pa_v-t\pa_z)f\right\|_{L^2}^2\lesssim \varepsilon^{\frac{1}{2}} \nu^{\frac{1}{2}\beta}\nu \left\||\Delta_L|^{\frac{1}{2}}\rmA f\right\|_{L^2}^2.
\end{align*}

For ${\mathrm E}^2$, by using Lemma \ref{lem-product-1} we have
\begin{align*}
  \left|{\mathrm E}^2\right|\lesssim& \nu  \|\left\|\rmA f\right\|_{L^2} \left\| \rmA(\pa_v-t\pa_z)f\right\|_{L^2}\left\| \rmA^{\mathrm R}\pa_v h\right\|_{L^2} \\
  \lesssim& \varepsilon^{\frac{1}{2}} \nu^{\frac{1}{2}\beta} \left(\nu \left\||\Delta_L|^{\frac{1}{2}}\rmA f\right\|_{L^2}^2+ \varepsilon\nu^{\beta} \nu\left\| \rmA^{\mathrm R}\pa_v h\right\|_{L^2}^2\right).
\end{align*}

Then by the bootstrap hypotheses \eqref{eq-boot-hf} and \eqref{eq-boot-hh}, we have
\begin{align*}
  &\int^{T^*}_1\left|{\mathrm E}(t)\right|dt\lesssim \varepsilon^{\frac{5}{2}}\nu^{2\beta}.
\end{align*}
\section{Low order estimates for the main system}
In this section, we study the low order estimates for $f$, and give the proof of \eqref{eq-est-boot-lfn0} and \eqref{eq-est-boot-lf0}. 

\subsection{Low order enhanced dissipation estimate}\label{sec-lfn0}
It holds that
\begin{align*}
  &\frac{1}{2}\frac{d}{dt}\left\|A^{\gamma} f_{\neq}\right\|_{L^2}^2=\int A^{\gamma} f_{\neq} (\pa_t A^{\gamma})f_{\neq}- A^{\gamma} f_{\neq} A^{\gamma}P_{\neq}({\bar u}\cdot\nabla f)+\nu A^{\gamma}f_{\neq} A^{\gamma}\widetilde{\Delta}_tf_{\neq}dzdv.
\end{align*}
Recall \eqref{eq-def-Agamma}, the definition of $A^{\gamma}$. We have for $k\neq0$ that
\begin{align*}
  \pa_t A_{k}^{\gamma}(t,\eta)=&\left(c_M \nu^{\frac{1}{3}}+\frac{c_M \nu^{\frac{1}{3}}}{1+c_M \nu^{\frac{1}{3}}t}+\dot\lambda(t)|k,\eta|^s-\frac{\pa_t\mathfrak m_k(t,\eta)}{\mathfrak m_k(t,\eta)} \right)A_{k}^{\gamma}(t,\eta)\\
  \le&\left(2c_M \nu^{\frac{1}{3}}+\dot\lambda(t)|k,\eta|^s-\frac{\pa_t\mathfrak m_k(t,\eta)}{\mathfrak m_k(t,\eta)} \right)A_{k}^{\gamma}(t,\eta).
\end{align*}

Recall that $\widetilde{\Delta}_t=\pa_{zz}+(v')^2(\pa_v-t\pa_z)^2$. We have
\begin{align*}
  &\nu\int A^{\gamma}f_{\neq} A^{\gamma}\widetilde{\Delta}_tf_{\neq}dzdv\\
  =&\nu\int A^{\gamma}f_{\neq} A^{\gamma}\Delta_L  f_{\neq}dzdv-\nu\int A^{\gamma}f_{\neq}A^{\gamma} \left((1-(v')^2)(\pa_v-t\pa_z)^2f_{\neq}\right)dzdv\\
  =&-\nu \||\Delta_L|^{\frac{1}{2}}A^{\gamma}f_{\neq}\|_{L^2}^2-\nu\int A^{\gamma}f_{\neq}A^{\gamma} \left((1-(v')^2)(\pa_v-t\pa_z)^2f_{\neq}\right)dzdv\\
  =&-\nu \||\Delta_L|^{\frac{1}{2}}A^{\gamma}f_{\neq}\|_{L^2}^2+E^{\gamma,\neq}.
\end{align*}
Then by taking $c_M=\frac{1}{8}$ and using Lemma \ref{lem-nu13},  we have
\begin{align*}
  \frac{1}{2}\frac{d}{dt}\left\|A^{\gamma} f_{\neq}\right\|_{L^2}^2\lesssim&-t^{-2c_1}\left\||\nabla|^{\frac{s}{2}}A^{\gamma} f_{\neq}\right\|_{L^2}^2-\int A^{\gamma} f_{\neq} A^{\gamma}P_{\neq}({\bar u}\cdot\nabla f)dzdv\\
  &-\frac{1}{2}\nu \||\Delta_L|^{\frac{1}{2}}A^{\gamma}f_{\neq}\|_{L^2}^2-c_M\nu^{\frac{1}{3}}\|A^{\gamma}f_{\neq}\|_{L^2}^2 +E^{\gamma,\neq}.
\end{align*}

We first study the nonlinear term, and split it into two parts:
\begin{align*}
  -\int A^{\gamma} f_{\neq} A^{\gamma}P_{\neq}({\bar u}\cdot\nabla f)dzdv=&-\int A^{\gamma} f_{\neq}  A^{\gamma}(q\pa_vf_{\neq})dvdz\\
  &-\int  A^{\gamma} f_{\neq} A^{\gamma}P_{\neq}(v'\nabla^{\perp}P_{\neq0}\phi\cdot\nabla_{z,v}f)dvdz\\
  =&{\mathrm{NL}}_1+{\mathrm{NL}}_2. 
\end{align*}

For ${\mathrm{NL}}_1$, we have
\begin{align*}
  &-\int A^{\gamma} f_{\neq}  A^{\gamma}(q\pa_vf_{\neq})dvdz\\
  =&\frac{1}{2}\int \pa_vq|A^{\gamma} f_{\neq}|^2dvdz+\int  A^{\gamma} f_{\neq}[q\pa_vA^{\gamma} f_{\neq}-A^{\gamma}(q\pa_vf_{\neq})] dvdz.
\end{align*}
The first term can be controlled as $q$ has $\frac{1}{t^2}$ decay. The latter term is expanded with a paraproduct (in both $z$ and $v$):
\begin{align*}
  &\int  A^{\gamma} f_{\neq}[q\pa_vA^{\gamma} f_{\neq}-A^{\gamma}(q\pa_vf_{\neq})] dvdz\\
  =&\sum_{N\ge8}\int A^\gamma f[q_{<N/8}\pa_vA^{\gamma}f_{\neq,N}-A^{\gamma}(q_{<N/8}\pa_v f_{\neq,N})]dzdv\\
  &+\sum_{N\ge8}\int A^\gamma f[q_{N}\pa_vA^{\gamma}f_{\neq,<N/8}-A^{\gamma}(q_{N}\pa_v f_{\neq,<N/8})]dzdv\\
  &+\sum_{N\in \mathbb D}\sum_{N/8\le N'\le 8N}\int A^\gamma f[q_{N'}\pa_vA^{\gamma}f_{\neq,N}-A^{\gamma}(q_{N'}\pa_v f_{\neq,N})]dzdv\\
  =&\sum_{N\ge8}{\mathrm T}_N^{\gamma}+\sum_{N\ge8}{\mathrm R}_N^{\gamma}+\mathcal R^{\gamma}.
\end{align*}
On the Fourier side,
\begin{align*}
  {\mathrm T}_N^{\gamma}=-\frac{i}{2\pi}\sum_{k\neq0}\int A^\gamma_k\hat f_k(\eta)[A^\gamma_k(\eta)-A^\gamma_k(\xi)]\hat q(\eta-\xi)_{<N/8}\xi \hat f_k(\xi)_N d\xi d\eta,
\end{align*}
and on the support of the integrand there holds
\begin{align*}
  \left||k,\eta|-|k,\xi|\right|\le|\eta-\xi|\le \frac{3}{16}|k,\xi|,\\
  \frac{3}{16}|k,\xi|\le|k,\eta|\le \frac{19}{16}|k,\xi|.
\end{align*}
For the commutator, we write
\begin{align*}
  &A^\gamma_k(\eta)-A^\gamma_k(\xi)=A^\gamma_k(\xi)\left(\frac{A^\gamma_k(\eta)}{A^\gamma_k(\xi)}-1\right)\\
  =&A^\gamma_k(\xi)\left(e^{\lambda(t)\left(|k,\eta|^s-|k,\xi|^s\right)}-1\right)\frac{\langle k,\eta\rangle^\gamma}{\langle k,\xi\rangle^\gamma} \frac{ \mathfrak M_k(t,\eta)}{ \mathfrak M_k(t,\xi)}\\
  &+A^\gamma_k(\xi)\left(\frac{\langle k,\eta\rangle^\gamma}{\langle k,\xi\rangle^\gamma}-1\right)  \frac{ \mathfrak M_k(t,\eta)}{ \mathfrak M_k(t,\xi)}\\
  &+A^\gamma_k(\xi)\left(\frac{ \mathfrak M_k(t,\eta)}{ \mathfrak M_k(t,\xi)}-1\right).
\end{align*}
Then 
\begin{align*}
  {\mathrm T}_N^{\gamma}=&-\frac{i}{2\pi}\sum_{k\neq0}A^\gamma_k\hat f_k(\eta)A^\gamma_k(\xi)\left(e^{\lambda(t)\left(|k,\eta|^s-|k,\xi|^s\right)}-1\right)\frac{\langle k,\eta\rangle^\gamma}{\langle k,\xi\rangle^\gamma} \frac{ \mathfrak M_k(t,\eta)}{ \mathfrak M_k(t,\xi)}\hat q(\eta-\xi)_{<N/8}\xi \hat f_k(\xi)_N d\xi d\eta\\
  &-\frac{i}{2\pi}\sum_{k\neq0}A^\gamma_k\hat f_k(\eta)A^\gamma_k(\xi)\left(\frac{\langle k,\eta\rangle^\gamma}{\langle k,\xi\rangle^\gamma}-1\right) \frac{ \mathfrak M_k(t,\eta)}{ \mathfrak M_k(t,\xi)}\hat q(\eta-\xi)_{<N/8}\xi \hat f_k(\xi)_N d\xi d\eta\\
  &-\frac{i}{2\pi}\sum_{k\neq0}A^\gamma_k\hat f_k(\eta)A^\gamma_k(\xi)\left(\frac{ \mathfrak M_k(t,\eta)}{ \mathfrak M_k(t,\xi)}-1\right)\hat q(\eta-\xi)_{<N/8}\xi \hat f_k(\xi)_N d\xi d\eta\\
  =&{\mathrm T}_N^{\gamma,1}+{\mathrm T}_N^{\gamma,2}+{\mathrm T}_N^{\gamma,3}.
\end{align*}
And one can treat these terms in the same way in Section \ref{sec-tran-t3}, Section \ref{sec-tran-t4}, and Section \ref{sec-tran-t6}, respectively. It holds that
\begin{align*}
  &\int^{T^*}_1\sum_{N\ge8}\left|{\mathrm T}_N^{\gamma}\right|(t)dt\lesssim \varepsilon^3\nu^{2\beta}.
\end{align*}

Next, we study ${\mathrm R}_N^{\gamma}$ on the Fourier side,
\begin{align*}
  \left|{\mathrm R}_N^{\gamma}\right|\lesssim& \sum_{k\neq0}\int A_k^\gamma \left|\hat f_k(\eta)\right|\Big[ \left|\hat q(\xi)_{N} \right||\eta-\xi|A^\gamma_k(\eta-\xi) \left| \hat f_k(\eta-\xi)_{<N/8}\right|\\
  &\qquad\qquad+A^\gamma_k(\eta) \left|\hat q(\xi)_{N}\right||\eta-\xi|  \left|\hat f_k(\eta-\xi)_{<N/8})\right|\Big]d\xi d\eta\\
  \lesssim& \left\|A^{\gamma} f_{\neq,\sim N}\right\|_{L^2}\left\|A q_{\sim N}\right\|_{L^2}\left\|A^{\gamma} f_{\neq}\right\|_{L^2}.
\end{align*}
Here we remove the derivative from $f_{<N/8}$ to $q_{N}$. It follows from the bootstrap hypotheses \eqref{eq-boot-lfn0} that for $\sigma\ge \gamma+2$, 
\begin{align*}
   \int^{T^*}_1\sum_{N\ge8}\left|{\mathrm R}_N^{\gamma}\right|(t)dt\lesssim & \int^{T^*}_1\left\|A q(t)\right\|_{L^2}\left\|A^{\gamma} f_{\neq}(t)\right\|_{L^2}^2dt\\
   \lesssim&\left\|A^{\gamma} f_{\neq}\right\|_{L^\infty_{[1,T^*]}L^2}^2\int^{T^*}_1 \frac{\varepsilon\nu^{\frac{1}{2}\beta}}{t^2}dt\lesssim\varepsilon^3\nu^{2\beta}.
\end{align*} 
The estimate for the remainder terms $\mathcal R^{\gamma}$ is the same as ${\mathrm R}_N^{\gamma}$.

Next, we turn to 
\begin{align*}
  {\mathrm{NL}}_2=&-\int A^{\gamma} f_{\neq} A^{\gamma}P_{\neq}(v'\nabla^{\perp}P_{\neq0}\phi\cdot\nabla_{z,v}f)dvdz\\
  =&-\int A^{\gamma} f_{\neq} A^{\gamma}P_{\neq}(\nabla^{\perp}P_{\neq0}\phi\cdot\nabla_{z,v}f)dvdz-\int A^{\gamma} f_{\neq} A^{\gamma}P_{\neq}(h\nabla^{\perp}P_{\neq0}\phi\cdot\nabla_{z,v}f)dvdz\\
  =&{\mathrm{NL}}_2^1+{\mathrm{NL}}_2^{1,\varepsilon}.
\end{align*}
On the Fourier side
\begin{align*}
  {\mathrm{NL}}_2^1=\frac{1}{2\pi}\sum_{k\neq0}\int A_{k}^{\gamma} \hat f_k(\eta) A_{k}^{\gamma}(\eta)\left((\eta-\xi)m-\xi(k-m)\right)\hat\phi_m(\xi)\hat f_{k-m}(\eta-\xi)d\eta d\xi.
\end{align*}
If $k-m\neq0$, then we use $\phi_m(\xi)$ to absorb the $e^{c_M \nu^{\frac{1}{3}}t}$ in $A^{\gamma}$, and $f_{k-m}(\eta-\xi)$ to absorb the $(1+c_M \nu^{\frac{1}{3}}t)$ in $A^{\gamma}$. If $k-m=0$, than the $(1+c_M \nu^{\frac{1}{3}}t)e^{c_M \nu^{\frac{1}{3}}t}$ in $A^{\gamma}$ will totally applied on $\phi_m(\xi')$. Then for $t\le2|\xi'|$, we need to pay one order regularity, and get
\begin{equation}
  \begin{aligned}    
  &\int \mathbf 1_{t\le2|\xi'|} A_{k}^{\gamma} \hat f_k(\eta) A_{k}^{\gamma}(\eta)|\eta-\xi||k| \left|\hat\phi_k(\xi)\right| \left|\hat f_{0}(\eta-\xi)\right| d\eta d\xi\\
  \lesssim&\int  A_{k}^{\gamma} \hat f_k(\eta) e^{\lambda(t)|k,\eta|^s}\langle k,\eta\rangle^\gamma  |\eta-\xi||k| \langle\xi\rangle \left|\hat\phi_k(\xi)\right| \left|\hat f_{0}(\eta-\xi)\right| d\eta d\xi\\
  \lesssim& \left\|A^{\gamma}_k f_{k}\right\|_{L^2} \frac{1}{t^2}\left\|A_k f_{k}\right\|_{L^2}\left\|A_0 f_{0}\right\|_{L^2}.    
  \end{aligned}
\end{equation}
For $t>2|\xi'|$, we will use a similar argument to Lemma \ref{lem-elliptic-low-1} to have
  \begin{align*}
    &\sum_{k\neq0}\|\langle k\rangle e^{\lambda(t)|k,\eta|^s}\langle k,\eta\rangle^{\gamma}\mathbf 1_{t>2|\eta|} \hat \phi_k \|_{L^2}^2\\
    \le&\sum_{k\neq0}\int_\eta e^{2\lambda(t)|k,\eta|^s} \frac{\langle k\rangle^2\langle k,\eta\rangle^{2\gamma}}{\left(k^2+(\eta-kt)^2\right)^2}\mathbf 1_{t>2|\eta|} |\widehat {\Delta_L\phi}_k(\eta)|^2 d\eta\\
    \lesssim&\frac{1}{\langle t\rangle^4}\sum_{k\neq0}\int_\eta e^{2\lambda(t)|k,\eta|^s} \langle k,\eta\rangle^{2\gamma} |\widehat {\Delta_L\phi}_k(\eta)|^2 d\eta\\
    \lesssim&\frac{1}{\langle t\rangle^4} \|\Delta_L\phi_{\neq}\|_{\mathcal G^{s,\lambda,\gamma}}^2.
  \end{align*}
  Therefore, it holds that
  \begin{align*}
    (1+c_M \nu^{\frac{1}{3}}t)e^{c_M \nu^{\frac{1}{3}}t}\sum_{k,m\neq0}\|\langle k\rangle e^{\lambda(t)|k,\eta|^s}\langle k,\eta\rangle^{\gamma}\mathbf 1_{t>2|\eta|} \hat \phi_k \|_{L^2}\lesssim \frac{1}{\langle t\rangle^2}\|A^{\gamma} f_{\neq}\|_{L^2}.
  \end{align*}

In a conclusion, we have for $\sigma\ge \gamma+4$,
\begin{align*}
  \left|{\mathrm{NL}}_2^1\right|\lesssim  \frac{1}{\langle t\rangle^2}\|A^{\gamma} f_{\neq}\|_{L^2}\|\rmA f_{\neq}\|_{L^2}^2+\frac{1}{\langle t\rangle^2}\|A^{\gamma} f_{\neq}\|_{L^2}\|\rmA f_{\neq}\|_{L^2}\|\rmA f_{0}\|_{L^2}+\frac{1}{\langle t\rangle^2}\|A^{\gamma} f_{\neq}\|_{L^2}^2\|\rmA f_{0}\|_{L^2},
\end{align*}
and 
\begin{align*}
   \int^{T^*}_1\left|{\mathrm{NL}}_2^1\right|(t)dt\lesssim \varepsilon^3\nu^{2\beta}.
\end{align*} 

Similar to Section \ref{sec-react-corr} and Section \ref{sec-diss-error}, we also have
\begin{align*}
   \int^{T^*}_1\left|{\mathrm{NL}}_2^{1,\varepsilon}\right|(t)+\left|E^{\gamma,\neq}\right|(t)dt\lesssim \varepsilon^{\frac{5}{2}}\nu^{2\beta}.
\end{align*} 
And this completes the proof of \eqref{eq-est-boot-lfn0}.
\subsection{Low order estimate for zero mode}\label{sec-lf0}
Let
\begin{align*}
  \mathcal  E_{L,0}(t)=\left\|A^\gamma_0f_0\right\|_{L^2}^2+\frac{1}{2}\nu t\left\|A^\gamma_0\pa_vf_0\right\|_{L^2}^2.
\end{align*}
We have
\begin{align*}
  \frac{d}{dt}\mathcal E_{L,0}(t)=&\frac{1}{2}\nu\left\|A^\gamma_0\pa_vf_0\right\|_{L^2}^2+\frac{1}{2}t\nu\frac{d}{dt}\left\|A^\gamma_0\pa_vf_0\right\|_{L^2}^2+\frac{d}{dt}\left\|A^\gamma_0f_0\right\|_{L^2}^2\\
  =&-\frac{3}{2}\nu\left\|A^\gamma_0\pa_vf_0\right\|_{L^2}^2 -\nu^2t\left\|A^\gamma_0\pa_v^2f_0\right\|_{L^2}^2-2CK^\gamma_{\lambda}-\nu tCK^{\gamma+1}_{\lambda}\\
  &-\nu t\int A^\gamma_0\pa_vf_0A^\gamma_0\pa_v(q\pa_v f_0)dv-2\int A^\gamma_0 f_0A^\gamma_0 (q\pa_v f_0)dv\\
  &-\nu t\int A^\gamma_0\pa_vf_0A^\gamma_0\pa_v(v'\nabla^{\perp}P_{\neq0}\phi\cdot\nabla_{z,v}f)dv-2\int A^\gamma_0 f_0A^\gamma_0 (v'\nabla^{\perp}P_{\neq0}\phi\cdot\nabla_{z,v}f)dv\\
  &+\nu^2t\int A^\gamma_0\pa_vf_0A^\gamma_0\pa_v \left(\left((v')^2-1\right)\pa_v^2 f_0\right)  dv+2\nu\int A^\gamma_0 f_0A^\gamma_0 \left(\left((v')^2-1\right)\pa_v^2 f_0\right) dv\\
  =&-\frac{3}{2}\nu\left\|A^\gamma_0\pa_vf_0\right\|_{L^2}^2 -\nu^2t\left\|A^\gamma_0\pa_v^2f_0\right\|_{L^2}^2-2CK^\gamma_{\lambda}-\nu tCK^{\gamma+1}_{\lambda}\\
  &+V_{1,1}+V_{1,2}+V_{2,1}+V_{2,2}+V_{3,1}+V_{3,2},
\end{align*}
where
\begin{align*}
   \mathrm{CK}_{\gamma}(t) =& \left|\dot\lambda(t)\right|\left\||\pa_v|^{\frac{s}{2}}A^\gamma_0 f_0(t)\right\|_{L^2}^2,\quad \mathrm{CK}_{\gamma+1}(t) =\left|\dot\lambda(t)\right|\left\||\pa_v|^{\frac{s}{2}}A^\gamma_0 \pa_vf_0(t)\right\|_{L^2}^2.
\end{align*}
Here $V_{1,1}$ and $V_{1,2}$ can be controlled in the same way to non-zero mode with the help of $CK^\gamma_{\lambda}$ and $\nu tCK^{\gamma+1}_{\lambda}$.

For $V_{2,1}$, we have
\begin{align*}
  \left|V_{2,1}\right|\le&\nu t\int \left|A^\gamma_0\pa_vf_0A^\gamma_0\pa_v(v'\nabla^{\perp}P_{\neq0}\phi\cdot\nabla_{z,v}f)\right|dv\\
  \lesssim&\nu t^{\frac{1}{2}}\left\|A^\gamma_0\pa_vf_0\right\|_{L^2} \frac{1}{t^{\frac{3}{2}}}\left\|\rmA f\right\|_{L^2}^2 \left(1+\left\|\rmA^{\mathrm R}h\right\|_{L^2}\right).
\end{align*}
For $V_{2,2}$, we have
\begin{align*}
  \left|V_{2,1}\right|\le&\int \left|A^\gamma_0 f_0A^\gamma_0 (v'\nabla^{\perp}P_{\neq0}\phi\cdot\nabla_{z,v}f)\right| dv\\
  \lesssim&\frac{1}{t^2}\left\|A^\gamma_0f_0\right\|_{L^2}\left\|\rmA f\right\|_{L^2}^2 \left(1+\left\|\rmA^{\mathrm R}h\right\|_{L^2}\right).
\end{align*}
For $V_{3,1}$, we have
\begin{align*}
  V_{3,1}=&\nu^2t \int A^\gamma_0\pa_vf_0A^\gamma_0\pa_v \left(\left((v')^2-1\right)\pa_v^2 f_0\right)  dv \\
  \le&\nu^2t\int \left| A^\gamma_0\pa_v^2f_0A^\gamma_0\left(\left((v')^2-1\right)\pa_v^2 f_0\right)  \right|dv\\
  \lesssim&\nu^2t\left\|A^\gamma_0\pa_v^2f_0\right\|_{L^2}^2 \left\|\rmA^{\mathrm R} h\right\|_{L^2}.   
\end{align*}
For $V_{3,2}$, we have
\begin{align*}
  V_{3,2}=&2\nu\int A^\gamma_0 f_0A^\gamma_0 \left(\left((v')^2-1\right)\pa_v^2 f_0\right) dv\\
  =&-2\nu\int A^\gamma_0 \pa_vf_0A^\gamma_0 \left(\left((v')^2-1\right)\pa_v f_0\right) dv-4\nu\int A^\gamma_0 f_0A^\gamma_0 \left(v'\pa_vv' \pa_v f_0\right)dv\\
  \lesssim&\nu\left\|A^\gamma_0\pa_vf_0\right\|_{L^2}^2\left\|\rmA^{\mathrm R} h\right\|_{L^2}+\nu\left\|A^\gamma_0\pa_vf_0\right\|_{L^2}\left\|\rmA^{\mathrm R}\pa_vh\right\|_{L^2}\left\|A^\gamma_0f_0\right\|_{L^2} .
\end{align*}
\begin{align*}
   \int^{T^*}_1\left|V_{1,1}+V_{1,2}+V_{2,1}+V_{2,2}+V_{3,1}+V_{3,2}\right|(t)dt\lesssim\varepsilon^{\frac{5}{2}}\nu^{2\beta}.
\end{align*} 
And this completes the proof of \eqref{eq-est-boot-lf0}.

\section{Coordinate system}
In this section, we give the proof of \eqref{eq-est-boot-hq}, \eqref{eq-est-boot-hbh}, \eqref{eq-est-boot-hh}, and \eqref{eq-est-boot-lq}. 

\subsection{Low order estimates for $q$}\label{sec-lq}
Recalling the definition of $A^{\gamma}$, we have
\begin{align*}
  \|A^{\gamma} q(t)\|_{L^2}^2=\|e^{\lambda(t)|\eta|^s}\langle \eta\rangle^\gamma \hat q\|_{L^2_\eta}^2,
\end{align*}
and then  
\begin{align*}
  \frac{1}{2}\frac{d}{dt}\|A^{\gamma} q(t)\|_{L^2}^2=&\dot \lambda \|t^2 |\pa_v|^{\frac{s}{2}} A^{\gamma} q(t)\|_{L^2}^2\\
  &-t^4\int A^{\gamma} q A^{\gamma}(q\pa_v q)dv\\
  &-t^3\int A^{\gamma} q A^{\gamma}\left( v' \left(\nabla_{z,v}^{\perp}P_{\neq}\phi\cdot\nabla_{z,v}\tilde u\right)_0\right)dv\\
  &\nu t^4\int A^{\gamma} q A^{\gamma}\left((v')^2\pa_{vv} q\right)dv\\
  =&V_1+V_2+V_3+V_4.
\end{align*}
Here $V_1$ is a `$\mathrm{CK}$' term. We start from $V_2$. We write
\begin{align*}
  V_2=&-t^4\int A^{\gamma} q A^{\gamma}\left(q (1+\frac{1-v'}{v'})v'\pa_v q\right)dv\\
  =&-t^4\int A^{\gamma} q A^{\gamma}\left(q (\sum_{n=0}^{+\infty}h^n)\bar h\right)dv.
\end{align*}
It follows from Lemma \ref{lem-product}, and the bootstrap hypotheses \eqref{eq-boot-hbh}, \eqref{eq-boot-hh}, and \eqref{eq-boot-hq} that
\begin{align*}
 V_2\lesssim& t^4\|A^{\gamma} q\|_{L^2}^2 \frac{1}{t^{\frac{3}{2}}}\left\|\frac{\rmA}{\langle\pa_v \rangle^s} \bar h(t)\right\|_{L^2}\lesssim \frac{\varepsilon^3\nu^{2\beta}}{t^{\frac{3}{2}}}.
\end{align*} 

For $V_3$, noting that 
\begin{align*}
  \left(\nabla_{z,v}^{\perp}P_{\neq}\phi\cdot\nabla_{z,v}\tilde u\right)_0=\left(\nabla_{z,v}^{\perp}P_{\neq}\phi\cdot\nabla_{z,v}\tilde u_{\neq}\right)_0,
\end{align*}
we will use Lemma \ref{lem-elliptic-low-1} to gain $t^{-2}$ decay from $\phi_{\neq}$ and $t^{-1}$ decay from $\tilde u_{\neq}$. Indeed, we have
\begin{align*}
  V_3=&-t^3\int A^{\gamma} q A^{\gamma}\left( v' \left(\nabla_{z,v}^{\perp}P_{\neq}\phi\cdot\nabla_{z,v}\tilde u\right)_0\right)dv\\
  \lesssim&t^2\|A^{\gamma} q\|_{L^2}t^{-2}\|\rmA f_{\neq}\|_{L^2}^2\lesssim \frac{\varepsilon^3\nu^{2\beta}}{t^2}.
\end{align*}

For $V_4$, we have 
\begin{align*}
  V_4=&\nu t^4\int A^{\gamma} q A^{\gamma}\left((v')^2\pa_{vv} q\right)dv\\
  =&-\nu t^4\int A^{\gamma} \pa_vq A^{\gamma}\pa_{v} qdv+\nu t^4\int A^{\gamma} q A^{\gamma}\left(\left((v')^2-1\right)\pa_{vv} q\right)dv\\
  =&-\nu t^4\| A^{\gamma} \pa_vq\|_{L^2}^2-\nu t^4\int A^{\gamma} \pa_vq A^{\gamma}\left(\left((v')^2-1\right)\pa_{v} q\right)dv-2\nu t^4\int A^{\gamma} q A^{\gamma}\left(v'\pa_vh\pa_{v} q\right)dv\\
  =&-\nu t^4\| A^{\gamma} \pa_vq\|_{L^2}^2+\mathrm{E}_q^{\gamma}.
\end{align*}

Similar to Section \ref{sec-diss-error}, one can easily deduce that
\begin{align*}
  \left|\mathrm{E}_q^{\gamma}\right|\lesssim \nu t^4\| A^{\gamma} \pa_vq\|_{L^2}^2\|\rmA^{\mathrm R} h(t)\|_{L^2}+\nu t^2\| A^{\gamma} \pa_vq\|_{L^2}\|\rmA^{\mathrm R}\pa_v h\|_{L^2}t^2\| A^{\gamma} q\|_{L^2}.
\end{align*}
Therefore,
\begin{align*}
   \int^{T^*}_1\left|\mathrm{E}_q^{\gamma}\right|(t)dt\lesssim \varepsilon^{\frac{5}{2}}\nu^{2\beta}.
\end{align*} 

This completes the proof of \eqref{eq-est-boot-lq}.
\subsection{High order estimates for $\bar h$}\label{sec-hbh}
It holds that
\begin{align*}
  &\frac{1}{2}\frac{d}{dt}\left\|t^{\frac{3}{2}}\frac{\rmA}{\langle\pa_v\rangle^s}\bar h\right\|_{L^2}^2\\
  =& \int_\eta t^3\frac{\rmA}{\langle\pa_v\rangle^s}\bar h\frac{\pa_t\rmA}{\langle\pa_v\rangle^s}\bar hdv+\frac{3}{2}\|t\frac{\rmA}{\langle\pa_v\rangle^s}\bar h\|_{L^2}^2+\int_vt^3\frac{\rmA}{\langle\pa_v\rangle^s}\pa_t\bar h\frac{\rmA}{\langle\pa_v\rangle^s}\bar hdv\\
  =&-\mathrm{CK}_{\lambda}^{v,2}-\mathrm{CK}_{W}^{v,2}-\mathrm{CK}_{G}^{v,2}-\frac{1}{2}\left\|t\frac{\rmA}{\langle\pa_v\rangle^s}\bar h\right\|_{L^2}^2\\
  &-\underbrace{\int_v t^3\frac{\rmA}{\langle\pa_v\rangle^s}\bar h \left[\frac{\rmA}{\langle\pa_v\rangle^s} q\pa_v\bar h-q\frac{\rmA}{\langle\pa_v\rangle^s}\pa_v\bar h\right]dv}_{\mathrm{Com}^{\bar h}}+\frac{1}{2}\int_vt^3\pa_vq|\frac{\rmA}{\langle\pa_v\rangle^s}\bar h|^2dv\\
  &+\underbrace{t^2\int_v\frac{\rmA}{\langle\pa_v\rangle^s}\bar h \frac{\rmA}{\langle\pa_v\rangle^s}\left(\na_{z,v}^{\bot}P_{\neq}\phi\cdot\na_{z,v}f\right)_0dv}_{\mathrm{NL}^{\bar h}}-\nu t^3\left\|\frac{\rmA}{\langle\pa_v\rangle^s}\pa_v\bar h\right\|_{L^2}\\
  &+\underbrace{t^2\int_v\frac{\rmA}{\langle\pa_v\rangle^s}\bar h \frac{\rmA}{\langle\pa_v\rangle^s}\left[h\left(\na_{z,v}^{\bot}P_{\neq}\phi\cdot\na_{z,v}f\right)_0\right]dv}_{\mathrm{NL}^{\bar h,\varepsilon}}+\underbrace{\nu\int_vt^3\frac{\rmA}{\langle\pa_v\rangle^s}\bar h \frac{\rmA}{\langle\pa_v\rangle^s}[((v')^2-1)\pa_{vv}\bar h]dv}_{\mathrm E^{\bar h}}.
\end{align*}
\subsubsection{Estimate of the commutator} We expand $\mathrm{Com}^{\bar h}$ with a paraproduct in $v$:
\begin{align*}
  &\int_v t^3\frac{\rmA}{\langle\pa_v\rangle^s}\bar h \left[\frac{\rmA}{\langle\pa_v\rangle^s} q\pa_v\bar h-q\frac{\rmA}{\langle\pa_v\rangle^s}\pa_v\bar h\right]dv\\
  =&\sum_{N\ge8}\int t^3\frac{\rmA}{\langle\pa_v\rangle^s}\bar h \left[\frac{\rmA}{\langle\pa_v\rangle^s} q_{<N/8}\pa_v\bar h_N-q_{<N/8}\frac{\rmA}{\langle\pa_v\rangle^s}\pa_v\bar h_N\right]dv\\
  &+\sum_{N\ge8}\int t^3\frac{\rmA}{\langle\pa_v\rangle^s}\bar h \left[\frac{\rmA}{\langle\pa_v\rangle^s} q_N\pa_v\bar h_{<N/8}-q_N\frac{\rmA}{\langle\pa_v\rangle^s}\pa_v\bar h_{<N/8}\right]dv\\
  &+\sum_{N\in \mathbb D}\sum_{N/8\le N'\le 8N}\int t^3\frac{\rmA}{\langle\pa_v\rangle^s}\bar h \left[\frac{\rmA}{\langle\pa_v\rangle^s} q_{N'}\pa_v\bar h_{N}-q_{N'}\frac{\rmA}{\langle\pa_v\rangle^s}\pa_v\bar h_{N}\right]dv\\
  =&\sum_{N\ge8}{\mathrm T}_N^{\bar h}+\sum_{N\ge8}{\mathrm R}_N^{\bar h}+\mathcal R^{\bar h}.
\end{align*}
The treatment of the transport ${\mathrm T}_N^{\bar h}$, reaction ${\mathrm R}_N^{\bar h}$, and remainder $\mathcal R^{\bar h}$ terms are the same as the treatment of $f$. For the reaction term ${\mathrm R}_N^{\bar h}$, we have
\begin{align*}
  &\int  t^3\frac{\rmA_0(\eta)}{\langle\eta\rangle^s}\left|\hat h(\eta)\right|  \frac{\rmA_0(\eta)}{\langle\eta\rangle^s} \left|\hat q_{N}(\xi)\right||\xi-\eta| \left|\hat h(\xi-\eta)_{<N/8}\right| d\xi d\eta\\
  \lesssim&\int  t^3\frac{\rmA_0(\eta)|\eta|^{\frac{1}{2}s}}{\langle\eta\rangle^s}h(\eta) \frac{\rmA_0(\xi)|\xi|^{\frac{1}{2}s}}{\langle\xi\rangle^s} q_{N}(\xi)|\xi-\eta|^{1-s}e^{c\lambda|\xi-\eta|^s}h(\xi-\eta)_{<N/8} d\xi d\eta\\
  \lesssim&\int  t^{-2c_1}t^3\frac{\rmA_0(\eta)|\eta|^{\frac{1}{2}s}}{\langle\eta\rangle^s}h(\eta) \frac{\rmA_0(\xi)|\xi|^{\frac{1}{2}s}}{\langle\xi\rangle^s} q_{N}(\xi)|\xi-\eta|^{1-s}t^{2c_1}e^{c\lambda|\xi-\eta|^s}h(\xi-\eta)_{<N/8} d\xi d\eta\\
  \lesssim& t^{-2c_1}\left\|t^{\frac{3}{2}}|\pa_v|^{\frac{s}{2}}\frac{\rmA}{\langle\pa_v\rangle^s}\bar h_{\sim N}\right\|_{L^2}\left\|t^{\frac{3}{2}}|\pa_v|^{\frac{s}{2}}\frac{\rmA}{\langle\pa_v\rangle^s}q_{\sim N}\right\|_{L^2} t^{2c_1}\left\| \frac{\rmA}{\langle\pa_v\rangle^s}\bar h \right\|_{L^2}.
\end{align*}
The transport term and the remainder term could be treated in the same way as Section \ref{sec-transport}.
Therefore, we have that
\begin{align*}
   \int^{T^*}_1\left|\mathrm{Com}^{\bar h}\right|(t)dt\lesssim \varepsilon^3\nu^{\beta}.
\end{align*}

\subsubsection{Estimate of the nonlinear feedback}
This part is the most troublesome term, we rewrite it with a paraproduct in $v$ only:
\begin{align*}
  \mathrm{NL}^{\bar h}=&\sum_{N\ge 8}\sum_{k\neq0}t^2\int_v\frac{\rmA}{\langle\pa_v\rangle^s}\bar h \frac{\rmA}{\langle\pa_v\rangle^s} \left(\nabla_{z,v}^{\perp}P_{\neq}\phi_{<N/8}\cdot\nabla_{z,v}f_{N}\right)_0dv\\
  &+\sum_{N\ge 8}\sum_{k\neq0}t^2\int_v\frac{\rmA}{\langle\pa_v\rangle^s}\bar h \frac{\rmA}{\langle\pa_v\rangle^s} \left(\nabla_{z,v}^{\perp}P_{\neq}\phi_{N}\cdot\nabla_{z,v}f_{<N/8}\right)_0dv\\
  &+\sum_{N\in \mathbb D}\sum_{N'\sim N}\sum_{k\neq0}t^2\int_v\frac{\rmA}{\langle\pa_v\rangle^s}\bar h \frac{\rmA}{\langle\pa_v\rangle^s} \left(\nabla_{z,v}^{\perp}P_{\neq}\phi_{N'}\cdot\nabla_{z,v}f_{N}\right)_0dv\\
  =&\sum_{N\ge8}\mathrm{NL}^{\bar h}_{N,\mathrm {LH}}+\sum_{N\ge8}\mathrm{NL}^{\bar h}_{N,\mathrm {HL}}+\mathrm{NL}^{\bar h}_{\mathcal R}.
\end{align*}
We first consider $\mathrm{NL}^{\bar h}_{N,\mathrm {HL}}$. On the Fourier side,
\begin{align*}
  \mathrm{NL}^{\bar h}_{N,\mathrm{HL}}= \frac{1}{2\pi}\sum_{k\neq0}t^{\frac{1}{2}}\int_{\eta,\xi}t^{\frac{3}{2}}\frac{\rmA_0}{\langle\eta\rangle^s} \hat{\bar h}(\eta) \frac{\rmA_0(\eta)}{\langle\eta\rangle^s}\eta k\hat \phi_{-k}(\xi)_{N}\hat f_{k}(\eta-\xi)_{<N/8}d\xi d\eta.
\end{align*}

If $|k|\ge \frac{1}{10}|\xi|$, then $|-k,\xi|\approx |k,\eta-\xi|$, and this part is like a remainder term. Indeed, we have
\begin{equation}\label{eq-est-bh-NL-rem}
  \begin{aligned}    
   &\sum_{k\neq0}t^{\frac{1}{2}}\int_{\eta,\xi}t^{\frac{3}{2}}\frac{\rmA_0}{\langle\eta\rangle^s} \left|\hat{\bar h}(\eta) \right|  \frac{\rmA_0(\eta)}{\langle\eta\rangle^s}|\eta k| \left|\hat \phi_{-k}(\xi)_{N}\right| \left|\hat f_{k}(\eta-\xi)_{<N/8} \right|d\xi d\eta\\
  \lesssim &\left\|t^{\frac{3}{2}}\frac{\rmA}{\langle\pa_v\rangle^s}\bar h_{\sim N}\right\|_{L^2} t^\frac{1}{2} \| A^{\gamma} \phi_{\neq,\sim N}\|_{L^2}\| A^{\gamma} f_{\neq}\|_{L^2}\\
  \lesssim&\left\|t^{\frac{3}{2}}\frac{\rmA}{\langle\pa_v\rangle^s}\bar h_{\sim N}\right\|_{L^2} t^{-\frac{3}{2}} \| \rmA  f_{\neq,\sim N}\|_{L^2}\| A^{\gamma} f_{\neq}\|_{L^2}. 
  \end{aligned}
\end{equation}
If $|k|< \frac{1}{10}|\xi|$, the estimate is very similar to the reaction part of $f$. The difficulty is that we need to treat the extra $t^{\frac{1}{2}}$. If $t\ge 2|\xi|$, we have
\begin{align*}
  &\sum_{k\neq0}t^{\frac{1}{2}}\int_{\eta,\xi} \mathbf 1_{t\ge 2|\xi|}t^{\frac{3}{2}}\frac{\rmA_0}{\langle\eta\rangle^s} \left|\hat{\bar h}(\eta)\right| \frac{\rmA_{-k}(\xi)}{\langle\xi\rangle^s} \frac{|k\xi|}{k^2+(\xi-kt)^2} \left|\widehat {\Delta_L\phi}_{-k}(\xi)_{N}\right| \\
  &\qquad\qquad\qquad\qquad\times\langle k, \eta-\xi\rangle^{1+2C_1\kappa} e^{c\lambda|k,\eta-\xi|^s} \left|\hat f_{k}(\eta-\xi)_{<N/8}\right|d\xi d\eta\\
  \lesssim&\sum_{k\neq0}t^{\frac{1}{2}}\int_{\eta,\xi} \mathbf 1_{t\ge 2|\xi|}t^{\frac{3}{2}}|\eta|^{\frac{s}{2}}\frac{\rmA_0}{\langle\eta\rangle^s} \left|\hat{\bar h}(\eta)\right| \rmA_{-k}(\xi)\frac{|\xi|^{\frac{s}{2}}|\xi|^{1-2s}}{|k|t^2}\left|\widehat {\Delta_L\phi}_{-k}(\xi)_{N}\right| \\
  &\qquad\qquad\qquad\qquad\times\langle k, \eta-\xi\rangle^{1+2C_1\kappa} e^{c\lambda|k,\eta-\xi|^s} \left|\hat f_{k}(\eta-\xi)_{<N/8}\right|d\xi d\eta\\
\lesssim&   t^{\frac{1}{2}-1-2s+2c_1}\left\|t^{\frac{3}{2}}|\pa_v|^{\frac{s}{2}}t^{-c_1}\frac{\rmA}{\langle\pa_v\rangle^s}\bar h_{\sim N}\right\|_{L^2} \left\|\left\langle \frac{\pa_v}{t\pa_z} \right\rangle^{-1}\Delta_L |\nabla|^{\frac{s}{2}}t^{-c_1}  \rmA\phi_{\neq,\sim N}\right\|_{L^2} e^{-c_M \nu^{\frac{1}{3}}t} \|\rmA f_{\neq}\|_{L^2}\\
  \lesssim&\varepsilon t^{-\frac{1}{2}-2s+2c_1-3\beta}\left\|t^{\frac{3}{2}}|\pa_v|^{\frac{s}{2}}t^{-c_1}\frac{\rmA}{\langle\pa_v\rangle^s}\bar h_{\sim N}\right\|_{L^2} \nu^{-\frac{1}{2}\beta}\left\|\left\langle \frac{\pa_v}{t\pa_z} \right\rangle^{-1}\Delta_L |\nabla|^{\frac{s}{2}}t^{-c_1}  \rmA\phi_{\neq,\sim N}\right\|_{L^2} e^{-c_M \nu^{\frac{1}{3}}t} \\
  \lesssim&\varepsilon \left\|t^{\frac{3}{2}}|\pa_v|^{\frac{s}{2}}t^{-c_1}\frac{\rmA}{\langle\pa_v\rangle^s}\bar h_{\sim N}\right\|_{L^2} \left\|\left\langle \frac{\pa_v}{t\pa_z} \right\rangle^{-1}\Delta_L |\nabla|^{\frac{s}{2}}t^{-c_1}  \rmA\phi_{\neq,\sim N}\right\|_{L^2} .
\end{align*}
Here we use the fact that 
\begin{align*}
  -\frac{1}{2}-2s+2c_1-3\beta=-\frac{1}{2}-1+2c_1+(1-2s-3\beta)=-\frac{1}{2}-1+2c_1 -3\beta s\le 0.
\end{align*}

If $t< 2|\xi|$, we use the same argument as in Section \ref{sec-react-main}. Here we use $\frac{1}{\langle\xi\rangle^s}$ and $\nu^{\frac{1}{2}\beta}$ from $\phi_{\neq}$ to absorb the extra $t^{\frac{1}{2}}$, and this is one reason that the smallness of $\bar h$ is only $\varepsilon\nu^{\frac{1}{2}\beta}$. Indeed, for this case, we have 
\begin{align*}
  \frac{t^{\frac{1}{2}}}{\langle\xi\rangle^s}\nu^{\frac{1}{2}\beta}e^{-\nu^{\frac{1}{3}}t}\lesssim t^{\frac{1}{2}-s-\frac{3}{2}\beta}e^{-\nu^{\frac{1}{3}}t} \lesssim 1,
\end{align*}
and
\begin{align*}
  &\sum_{k\neq0}t^{\frac{1}{2}}\int_{\eta,\xi} \mathbf 1_{t< 2|\xi|}t^{\frac{3}{2}}\frac{\rmA_0}{\langle\eta\rangle^s} \left|\hat{\bar h}(\eta)\right| \frac{\rmA_{-k}(\xi)}{\langle\xi\rangle^s} \frac{|k\xi|}{k^2+(\xi-kt)^2} \left|\widehat {\Delta_L\phi}_{-k}(\xi)_{N}\right| \\
  &\qquad\qquad\qquad\qquad\times\langle k, \eta-\xi\rangle^{1+2C_1\kappa} e^{c\lambda|k,\eta-\xi|^s} \left|\hat f_{k}(\eta-\xi)_{<N/8}\right|d\xi d\eta\\
  \lesssim&\varepsilon \bigg( \mathrm{CK}_{\lambda}^{v,2}+\mathrm{CK}_{W}^{v,2}+\mathrm{CK}_{G}^{v,2} \\
  &+\nu^{-\beta}\left\|\left\langle \frac{\pa_v}{t\pa_z} \right\rangle^{-1}\Delta_L\left(|\nabla|^{\frac{s}{2}}t^{-c_1}+\sqrt{\frac{\mathfrak w\pa_t w }{w }} +\sqrt{\frac{\mathfrak w\pa_t g }{g}} \right)\rmA \phi_{\neq,\sim N}\right\|_{L^2}^2\bigg).
\end{align*}

Therefore, we get
\begin{align*}
   \int^{T^*}_1\sum_{N\ge8}\left|\mathrm{NL}^{\bar h}_{N,\mathrm{HL}}\right|(t)dt\lesssim \varepsilon^3\nu^{\beta}.
\end{align*}

Next, we focus on the term $\mathrm{NL}^{\bar h}_{N,\mathrm {LH}}$. For the same reason to $\mathrm{NL}^{\bar h}_{N,\mathrm {HL}}$, here we also only need to study the case  $|k|< \frac{1}{10}|\xi|$. Noting that
\begin{align*}
  |\eta|\le |\eta-\xi+kt|+|\xi-kt|,
\end{align*}
we have $1-s-\frac{s^2}{1-s}=\frac{1-2s}{1-s}$ and
\begin{align*}
  \left|\mathrm{NL}^{\bar h}_{N,\mathrm {LH}}\right|\lesssim&\sum_{k\neq0}t^2\int_{\eta,\xi}\frac{\rmA_0}{\langle\eta\rangle^s} \left|\hat{\bar h}(\eta) \right|  \frac{\rmA_0(\eta)}{\langle\eta\rangle^s} |\eta k| \left| \hat\phi_{-k}(\eta-\xi)_{<N/8}\right| \left|\hat f_{k}(\xi)_N\right|d\xi d\eta\\
  \lesssim&\sum_{k\neq0}t^2\int_{\eta,\xi}|\eta|^{\frac{s}{2}\frac{s}{1-s}}\frac{\rmA_0}{\langle\eta\rangle^s}\left|\hat{\bar h}(\eta) \right| \rmA_k(\xi)|\xi|^{\frac{s}{2}\frac{s}{1-s}}\left|\hat f_{k}(\xi)_N\right|\\
  &\qquad\qquad\times \langle k, \eta-\xi\rangle^{1+2C_1\kappa}e^{c\lambda|k,\eta-\xi|^s}|\eta-\xi+kt|^{\frac{1-2s}{1-s}} |k|\left| \hat\phi_{-k}(\eta-\xi)_{<N/8}\right|d\xi d\eta\\
  &+\sum_{k\neq0}t^2\int_{\eta,\xi}|\eta|^{\frac{s}{2}\frac{s}{1-s}}\frac{\rmA_0}{\langle\eta\rangle^s}\left|\hat{\bar h}(\eta) \right| \rmA_k(\xi)|\xi|^{\frac{s}{2}\frac{s}{1-s}}|\xi-kt|^{\frac{1-2s}{1-s}}\left|\hat f_{k}(\xi)_N\right|\\
  &\qquad\qquad \langle k, \eta-\xi\rangle^{1+2C_1\kappa}e^{c\lambda|k,\eta-\xi|^s} |k| \left| \hat\phi_{-k}(\eta-\xi)_{<N/8}\right| d\xi d\eta\\
  =&I+II.
\end{align*}
For $I$, we have
\begin{align*}
  I\lesssim&\sum_{k\neq0}t^{\frac{1}{2}} \left\|t^{\frac{3}{2}}|\eta|^{\frac{s}{2}\frac{s}{1-s}}\frac{\rmA_0}{\langle\eta\rangle^s} \hat{\bar h}(\eta)_{\sim N}\right\|_{L^2_\eta}\left\|\rmA_k(\xi)|\xi|^{\frac{s}{2}\frac{s}{1-s}}\hat f_{k}(\xi)_{\sim N}\right\|_{L^2_\xi}\\
  &\qquad\qquad\times\|\langle k, \eta-\xi\rangle^{2+2C_1\kappa}e^{c\lambda|k,\eta-\xi|^s} |\eta-\xi+kt|^{\frac{1-2s}{1-s}} |k| \hat \phi_{-k}(\eta-\xi)\|_{L^2_{(\eta-\xi)}}\\
  \lesssim&\varepsilon t^{\frac{1}{2}}t^{-\frac{1}{1-s}}\nu^{\beta}e^{-c_M\nu^{\frac{1}{3}}t}\left[\|t^{\frac{3}{2}}\frac{\rmA}{\langle\pa_v\rangle^s}\bar h_{\sim N}\|_{L^2}\|\rmA f_{\neq,_{\sim N}}\|_{L^2}\right]^{\frac{1-2s}{1-s}}\left[\|t^{\frac{3}{2}}|\pa_v|^{\frac{s}{2}}\frac{\rmA}{\langle\pa_v\rangle^s}\bar h_{\sim N}\|_{L^2}\|\rmA|\nabla|^{\frac{s}{2}}f_{\neq,\sim N}\|_{L^2}\right]^{\frac{s}{1-s}}.
\end{align*}
Here we use fact follows from Corollary \ref{cor-elliptic-low-2} that 
\begin{align*}
  \|A^\gamma|\nabla_L|^{\frac{1-2s}{1-s}} \phi_{\neq}\|_{L^2}\lesssim \varepsilon\nu^{\beta}t^{-2+\frac{1-2s}{1-s}}=\varepsilon\nu^{\beta}t^{-\frac{1}{1-s}}.
\end{align*}
Recall the bootstrap hypotheses \eqref{eq-boot-hf}
\begin{align*}
  \nu^{\frac{1}{6}}\|\rmA f_{\neq}\|_{L^2_tL^2}+\nu^{\frac{1}{2}}\|\nabla_L\rmA f\|_{L^2_tL^2}+\|\left|\dot\lambda\right|^{\frac{1}{2}}|\nabla|^{\frac{s}{2}}\rmA f\|_{L^2_tL^2}\lesssim \varepsilon\nu^{\beta}.
\end{align*}
We have that
\begin{align*}
  I\lesssim&\varepsilon t^{\frac{1}{2}-\frac{1}{1-s}+\frac{1}{2}\frac{1-2s}{1-s}+2c_1\frac{s}{1-s}}\nu^{\beta-\frac{1}{6}\frac{1-2s}{1-s}}e^{-c_M\nu^{\frac{1}{3}}t}\\
  &\times\left[\|t\frac{\rmA}{\langle\pa_v\rangle^s}\bar h_{\sim N}\|_{L^2} \nu^{\frac{1}{6}}\|\rmA f_{\neq,\sim N}\|_{L^2}+ \|t^{-c_1}t^{\frac{3}{2}}|\pa_v|^{\frac{s}{2}}\frac{\rmA}{\langle\pa_v\rangle^s}\bar h_{\sim N}\|_{L^2}\|t^{-c_1}\rmA|\nabla|^{\frac{s}{2}}f_{\neq,\sim N}\|_{L^2}\right].
\end{align*}
Noting that
\begin{align*}
  \frac{1-2s}{1-s}=3\beta,\quad \frac{s}{1-s}=1-3\beta,\quad \frac{1}{1-s}=2-3\beta,\quad  c_1=\frac{5}{8},
\end{align*}
we have
\begin{align*}
  &\varepsilon t^{\frac{1}{2}-\frac{1}{1-s}+\frac{1}{2}\frac{1-2s}{1-s}+2c_1\frac{s}{1-s}}\nu^{\beta-\frac{1}{6}\frac{1-2s}{1-s}}e^{-c_M\nu^{\frac{1}{3}}t}=\varepsilon t^{-\frac{1}{4}+\frac{3}{4}\beta}\nu^{\frac{\beta}{2}}e^{-c_M\nu^{\frac{1}{3}}t}\lesssim \varepsilon,
\end{align*}
and thus
\begin{align*}
  I\lesssim&\varepsilon \left(\left\|t\frac{\rmA}{\langle\pa_v\rangle^s}\bar h_{\sim N}\right\|_{L^2}^2+ \left\|t^{-c_1}t^{\frac{3}{2}}|\pa_v|^{\frac{s}{2}}\frac{\rmA}{\langle\pa_v\rangle^s}\bar h_{\sim N}\right\|_{L^2}^2+\nu^{\frac{1}{3}}\|\rmA f_{\neq,\sim N}\|_{L^2}^2+ \|t^{-c_1}\rmA|\nabla|^{\frac{s}{2}}f_{\neq,\sim N}\|_{L^2}^2\right).
\end{align*}

For $II$, we have
\begin{align*}
  II\lesssim&\sum_{k\neq0}t^{\frac{1}{2}} \|t^{\frac{3}{2}}|\eta|^{\frac{s}{2}\frac{s}{1-s}}\frac{\rmA_0}{\langle\eta\rangle^s}\hat{\bar h}(\eta)_{\sim N}\|_{L^2_\eta}\|\rmA_k(\xi)|\xi|^{\frac{s}{2}\frac{s}{1-s}}|\xi-kt|^{\frac{1-2s}{1-s}}\hat f_{k}(\xi)_{\sim N}\|_{L^2_\xi}\\
  &\qquad\qquad\times \|\langle k, \eta-\xi\rangle^{2+2C_1\kappa}e^{c\lambda|k,\eta-\xi|^s} |k|\hat\phi_{-k}(\eta-\xi)\|_{L^2_{(\eta-\xi)}}\\
  \lesssim&\varepsilon t^{\frac{1}{2}-2+\frac{1}{2}\frac{1-2s}{1-s}} \nu^{\beta}e^{-c_M\nu^{\frac{1}{3}}t}\\
  &\qquad\times\left[\|t\frac{\rmA}{\langle\pa_v\rangle^s}\bar h_{\sim N}\|_{L^2}\|\rmA|\nabla_L|f_{\neq,\sim N}\|_{L^2}\right]^{\frac{1-2s}{1-s}}\left[\|t^{\frac{3}{2}}|\pa_v|^{\frac{s}{2}}\frac{\rmA}{\langle\pa_v\rangle^s}\bar h_{\sim N}\|_{L^2}\|\rmA|\nabla|^{\frac{s}{2}}f_{\neq,\sim N}\|_{L^2}\right]^{\frac{s}{1-s}}\\
  \lesssim&\varepsilon t^{-\frac{3}{2}+\frac{1}{2}\frac{1-2s}{1-s}+2c_1\frac{s}{1-s}}\nu^{\beta-\frac{1}{2}\frac{1-2s}{1-s}}e^{-c_M\nu^{\frac{1}{3}}t}\\
  &\times\left[\|t\frac{\rmA}{\langle\pa_v\rangle^s}\bar h_{\sim N}\|_{L^2}\nu^{\frac{1}{2}}\|\rmA|\nabla_L|f_{\neq,\sim N}\|_{L^2} +\|t^{-c_1}t^{\frac{3}{2}}|\pa_v|^{\frac{s}{2}}\frac{\rmA}{\langle\pa_v\rangle^s}\bar h_{\sim N}\|_{L^2}\|t^{-c_1}\rmA|\nabla|^{\frac{s}{2}}f_{\neq,\sim N}\|_{L^2}\right].
\end{align*}
We have
\begin{align*}
  \varepsilon t^{-\frac{3}{2}+\frac{1}{2}\frac{1-2s}{1-s}+2c_1\frac{s}{1-s}}\nu^{\beta-\frac{1}{2}\frac{1-2s}{1-s}}e^{-c_M\nu^{\frac{1}{3}}t}=\varepsilon t^{-\frac{1}{6}-\frac{9}{4}\beta}\nu^{-\frac{1}{2}\beta}e^{-c_M\nu^{\frac{1}{3}}t}.
\end{align*}
Therefore, we need to use $\nu^{\frac{1}{2}\beta}$ from $f_{\neq}$, and get
\begin{align*}
II \lesssim&\varepsilon \bigg[\left\|t\frac{\rmA}{\langle\pa_v\rangle^s}\bar h_{\sim N}\right\|_{L^2}^2+\left\|t^{-c_1}t^{\frac{3}{2}}|\pa_v|^{\frac{s}{2}}\frac{\rmA}{\langle\pa_v\rangle^s}\bar h_{\sim N}\right\|_{L^2}^2\\
&\qquad+\nu^{-\beta}\nu \|\rmA|\nabla_L|f_{\neq,\sim N}\|_{L^2}^2+\nu^{-\beta}\|t^{-c_1}\rmA|\nabla|^{\frac{s}{2}}f_{\neq,\sim N}\|_{L^2}^2 \bigg].
\end{align*}
We remark that the estimate of $II$ implies that the smallness of the highest norm of $\bar h$ grows from $\varepsilon \nu^{\beta}$ to $\varepsilon\nu^{\frac{1}{2}\beta}$. 
The estimate for $\mathrm{NL}^{\bar h}_{\mathcal R}$ is the same as \eqref{eq-est-bh-NL-rem}, and we omit the details. 
As a conclusion, we have
\begin{align*}
   \int^{T^*}_1 \left|\mathrm{NL}^{\bar h} \right|(t)dt\lesssim \varepsilon^3\nu^{\beta}.
\end{align*}

For the rest terms, the  nonlinear correction term $\mathrm{NL}^{\bar h,\varepsilon}$ can be treated in the same way to Section \ref{sec-react-corr}, the dissipation error term $\mathrm E^{\bar h}$ can be treated in the same way to Section \ref{sec-diss-error}, and the term  $\frac{1}{2}\int_vt^3\pa_vq|\frac{\rmA}{\langle\pa_v\rangle^s}\bar h|^2dv$ is easy to control.

Therefore, we complete the proof of \eqref{eq-est-boot-hbh}. 

\subsection{High order estimates for $q$}\label{sec-hq} The high order estimates for $q$ are the same as $\bar h$, and we have
\begin{align*}
  &\frac{1}{2}\frac{d}{dt}\left\|t^{\frac{3}{2}}\frac{\rmA}{\langle\pa_v\rangle^s}q\right\|_{L^2}^2\\
  =&-\mathrm{CK}_{\lambda}^{v,1}-\mathrm{CK}_{W}^{v,1}-\mathrm{CK}_{G}^{v,1}-\frac{1}{2}\|t\frac{\rmA}{\langle\pa_v\rangle^s}q\|_{L^2}^2\\
  &-\underbrace{\int_v t^3\frac{\rmA}{\langle\pa_v\rangle^s}q [\frac{\rmA}{\langle\pa_v\rangle^s} q\pa_vq-q\frac{\rmA}{\langle\pa_v\rangle^s}\pa_vq]dv}_{\mathrm{Com}{q}}+\frac{1}{2}\int_vt^3\pa_vq|\frac{\rmA}{\langle\pa_v\rangle^s}q|^2dv\\
  &+\underbrace{t^2\int_v\frac{\rmA}{\langle\pa_v\rangle^s}q \frac{\rmA}{\langle\pa_v\rangle^s}\left(\na_{z,v}^{\bot}P_{\neq}\phi\cdot\na_{z,v}\tilde u\right)_0dv}_{\mathrm{NL}^{q}}-\nu t^3\|\frac{\rmA}{\langle\pa_v\rangle^s}\pa_vq\|_{L^2}\\
  &+\underbrace{t^2\int_v\frac{\rmA}{\langle\pa_v\rangle^s}q \frac{\rmA}{\langle\pa_v\rangle^s}\left[h\left(\na_{z,v}^{\bot}P_{\neq}\phi\cdot\na_{z,v}\tilde u\right)_0\right]dv}_{\mathrm{NL}^{q,\varepsilon}}+\underbrace{\nu\int_vt^3\frac{\rmA}{\langle\pa_v\rangle^s}q \frac{\rmA}{\langle\pa_v\rangle^s}[((v')^2-1)\pa_{vv}q]dv}_{\mathrm E^{q}}.
\end{align*}
The only difference is that $f$ is replaced by $\tilde u$ in the nonlinear term $\mathrm{NL}^{q}$ and the nonlinear correction term $\mathrm{NL}^{q,\varepsilon}$.

Recall that $\tilde u=-v'(\pa_v-t\pa_z)\phi$. We only need to estimate
\begin{align}
  \|\rmA \tilde u_{\neq}\|_{L^2}\lesssim& \|\rmA  f_{\neq}\|_{L^2}+\|\rmA f_{\neq}\|_{L^2}\|\rmA^{\mathrm R}h\|_{L^2},\label{eq-est-hq-1}\\
  t^{-c_1}\|\rmA|\nabla|^{\frac{s}{2}}\tilde u_{\neq}\|_{L^2}\lesssim& t^{-c_1}\|\rmA|\nabla|^{\frac{s}{2}}f_{\neq}\|_{L^2}+\|A  f_{\neq}\|_{L^2}t^{-c_1}\|\rmA^{\mathrm R}|\pa_v|^{\frac{s}{2}}h\|_{L^2},\label{eq-est-hq-2}\\
  \nu^{\frac{1}{2}}\|\rmA\nabla_L\left(v'(\pa_v-t\pa_z)\phi_{\neq}\right)\|_{L^2}\lesssim &\nu^{\frac{1}{2}}\|\rmA\nabla_Lf_{\neq}\|_{L^2}+\|\rmA f_{\neq}\|_{L^2}\nu^{\frac{1}{2}}\|\rmA^{\mathrm R}\pa_vh\|_{L^2}.\label{eq-est-hq-3}
\end{align}
With the above estimates, one can complete the proof of \eqref{eq-est-boot-hq} in the same treatment as $\bar h$. 

Next, we give the proof of \eqref{eq-est-hq-2} and \eqref{eq-est-hq-3}. The proof of \eqref{eq-est-hq-1} is similar and simpler.
\subsubsection{Proof of \eqref{eq-est-hq-2}} By using Lemma \ref{lem-product-1} and Corollary \ref{cor-elliptic-low-2}, we deduce that
\begin{align*}
  & \|\rmA|\nabla|^{\frac{s}{2}}\left(v'(\pa_v-t\pa_z)\phi_{\neq}\right)\|_{L^2}\\
  \lesssim& \|\rmA|\nabla|^{\frac{s}{2}}(\pa_v-t\pa_z) \phi_{\neq}\|_{L^2}+ \|\rmA^{\mathrm R}|\pa_v|^{\frac{s}{2}}h\|_{L^2}\|A^\gamma(\pa_v-t\pa_z)\phi_{\neq}\|_{L^2}\\
  &+ \|\rmA^{\mathrm R}h\|_{L^2}\|\rmA|\nabla|^{\frac{s}{2}} (\pa_v-t\pa_z)\phi_{\neq}\|_{L^2}\\
  \lesssim&\|\rmA|\nabla|^{\frac{s}{2}}(\pa_v-t\pa_z) \phi_{\neq}\|_{L^2}+\|\rmA^{\mathrm R}|\pa_v|^{\frac{s}{2}}h\|_{L^2}\|\rmA  f\|_{L^2}+\|\rmA^{\mathrm R}h\|_{L^2}\|\rmA|\nabla|^{\frac{s}{2}} (\pa_v-t\pa_z)\phi_{\neq}\|_{L^2}.
\end{align*}
Then it suffices to show that
\begin{align}\label{eq-est-hq-goal2}
   \|\rmA|\nabla|^{\frac{s}{2}} (\pa_v-t\pa_z)\phi_{\neq}\|_{L^2}\lesssim  \|\rmA|\nabla|^{\frac{s}{2}}f_{\neq}\|_{L^2}+ \varepsilon\nu^{\beta}\|\rmA^{\mathrm R}|\pa_v|^{\frac{s}{2}}h\|_{L^2}.
\end{align}
Recall that $f=\Delta_t\phi$. Similar to Lemma \ref{lem-elliptic-low-1}, we write $\Delta_t$ as a perturbation of $\Delta_L$ via
  \begin{align*}
    \Delta_L \phi_{\neq}=&f_{\neq}+\left(1-(v')^2\right)(\pa_v-t\pa_z)^2\phi_{\neq}-v''(\pa_v-t\pa_z)\phi_{\neq}.
  \end{align*}
It follows that
\begin{align*}
  \|\rmA|\nabla|^{\frac{s}{2}}(\pa_v-t\pa_z)\phi_{\neq}\|_{L^2}^2\le& \|\rmA|\nabla|^{\frac{s}{2}}f_{\neq}\|_{L^2}^2+\|\rmA|\nabla|^{\frac{s}{2}}(\pa_v-t\pa_z)\Delta_L^{-1}\left(1-(v')^2\right)(\pa_v-t\pa_z)^2\phi_{\neq}\|_{L^2}^2\\
  &+\|\rmA|\nabla|^{\frac{s}{2}}(\pa_v-t\pa_z)\Delta_L^{-1}v''(\pa_v-t\pa_z)\phi_{\neq}\|_{L^2}^2.
\end{align*}

Recall the definition of 
\begin{align*}
  {\mathrm T}^1={\mathrm T}^1_{\mathrm {HL}}+{\mathrm T}^1_{\mathrm {LH}}+{\mathrm T}^1_{\mathcal R},\ {\mathrm T}^2={\mathrm T}^2_{\mathrm {HL}}+{\mathrm T}^2_{\mathrm {LH}}+{\mathrm T}^2_{\mathcal R}
\end{align*}
given in \eqref{def-T1T2}. 

For $\mathrm T^1_{\mathrm {LH}}$, it holds that
\begin{align*}
  &\reallywidehat {\rmA|\nabla|^{\frac{s}{2}}(\pa_v-t\pa_z)\Delta_L^{-1}\mathrm T^1_{\mathrm {LH}}}(k,\eta)\\
  \le&\sum_{N\ge8}\int_{\xi}\rmA_k(\eta)\frac{|k,\eta|^{\frac{s}{2}}|\eta-kt|}{k^2+(\eta-kt)^2} \left| \widehat{\mathcal V}(\eta-\xi)_{<N/8}\right|(\xi-kt)^2 \left| \hat\phi_k(\xi)_{N}\right|d\xi\\
  \lesssim&\sum_{N\ge8}\int_{\xi}\langle\eta-\xi\rangle^{2+2C_1\kappa}e^{c\lambda|\eta-\xi|^s}\frac{|k|+|\xi-kt|}{|k|+|\eta-kt|} \left|\widehat{\mathcal V}(\eta-\xi)_{<N/8} \right|\rmA_k(\xi)\frac{|k,\xi|^{\frac{s}{2}}(\xi-kt)^2}{|k|+|\xi-kt|} \left|\hat\phi_k(\xi)_{N} \right|d\xi\\
  \lesssim&\sum_{N\ge8}\int_{\xi}\langle\eta-\xi\rangle^{3+2C_1\kappa}e^{c\lambda|\eta-\xi|^s} \left|\widehat{\mathcal V}(\eta-\xi)_{<N/8} \right|\rmA_k(\xi)|k,\xi|^{\frac{s}{2}}|\xi-kt| \left|\hat\phi_k(\xi)_{N} \right|d\xi.
\end{align*}
Then
\begin{align*}
  \left\| \rmA|\nabla|^{\frac{s}{2}}(\pa_v-t\pa_z)\Delta_L^{-1}\mathrm T^1_{\mathrm {LH}}\right\|_{L^2}^2\lesssim \varepsilon\nu^\beta \|\rmA|\nabla|^{\frac{s}{2}}(\pa_v-t\pa_z)\phi_{\neq}\|_{L^2}^2.
\end{align*}

For $\mathrm T^1_{\mathrm {HL}}$, if $k\ge |\eta|$, we have similar result to $\mathrm T^1_{\mathrm {LH}}(k,\eta)$. So we only study the case $k< |\eta|$. It holds that
\begin{align*}
  &\reallywidehat{\rmA|\nabla|^{\frac{s}{2}}(\pa_v-t\pa_z)\Delta_L^{-1} \mathrm T^1_{\mathrm {HL}}}(k,\eta)\\
  \le &\sum_{N\ge8}\int_{\xi}\rmA_k(\eta)\frac{|k,\eta|^{\frac{s}{2}}|\eta-kt|}{k^2+(\eta-kt)^2} \left|\widehat{\mathcal V}(\eta-\xi)_{N} \right|(\xi-kt)^2 \left|\hat\phi_k(\xi)_{<N/8}\right|d\xi\\
  \lesssim&\sum_{N\ge8}\int_{\xi}\rmA^{\mathrm R}(\eta-\xi)|\eta-\xi|^{\frac{s}{2}} \left|\widehat{\mathcal V}(\eta-\xi)_{N} \right| \langle k,\xi\rangle^{2+2C_1\kappa}e^{c\lambda|k,\xi|}(\xi-kt)^2 \left|\hat\phi_k(\xi)_{<N/8} \right|d\xi.
\end{align*}
Then by using Lemma \ref{lem-product-2} and Lemma \ref{lem-elliptic-low-1}, we deduce that
\begin{align*}
  \left\| \rmA|\nabla|^{\frac{s}{2}}(\pa_v-t\pa_z)\Delta_L^{-1}\mathrm T^1_{\mathrm {HL}}\right\|_{L^2}^2\lesssim \varepsilon^2\nu^{2\beta} \|\rmA^{\mathrm R}|\pa_v|^{\frac{s}{2}}h\|_{L^2}^2.
\end{align*}
And the estimate for ${\mathrm T}^1_{\mathcal R}$ is similar and simpler.

Next, we turn to ${\mathrm T}^2$. For ${\mathrm T}^2_{\mathrm {LH}}$, it holds that
\begin{align*}
  &\reallywidehat{\rmA|\nabla|^{\frac{s}{2}}(\pa_v-t\pa_z)\Delta_L^{-1}{\mathrm T}^2_{\mathrm {LH}}}(k,\eta)\\
  \le &\sum_{N\ge8}\int_{\xi}\rmA_k(\eta)\frac{|k,\eta|^{\frac{s}{2}}|\eta-kt|}{k^2+(\eta-kt)^2} \left|\widehat{v''}(\eta-\xi)_{<N/8} \right||\xi-kt| \left|\hat\phi_k(\xi)_{N} \right|d\xi\\
  \lesssim&\sum_{N\ge8}\int_{\xi}\langle\eta-\xi\rangle^{2+2C_1\kappa}e^{c\lambda|\eta-\xi|^s} \left|\widehat{v''}(\eta-\xi)_{<N/8} \right| \rmA_k(\xi)|k,\xi|^{\frac{s}{2}}|\xi-kt| \left|\hat\phi_k(\xi)_{N} \right|d\xi.
\end{align*}
Then
\begin{align*}
  \left\| \rmA|\nabla|^{\frac{s}{2}}(\pa_v-t\pa_z)\Delta_L^{-1}\mathrm T^2_{\mathrm {LH}}\right\|_{L^2}^2\lesssim \varepsilon\nu^\beta \|\rmA|\nabla|^{\frac{s}{2}}(\pa_v-t\pa_z)\phi_{\neq}\|_{L^2}^2.
\end{align*}

For ${\mathrm T}^2_{\mathrm {HL}}(k,\eta)$, we first assume that $|\eta|\ge \frac{1}{10}|k|$, then, if $|\eta|\ge 2|k|t$ or $|\eta|\le \frac{1}{2}|k|t$, we have 
\begin{align*}
  &\reallywidehat{\rmA|\nabla|^{\frac{s}{2}}(\pa_v-t\pa_z)\Delta_L^{-1}{\mathrm T}^2_{\mathrm {HL}}}(k,\eta)\mathbf1_{|\eta|\ge \frac{1}{10}|k|,|\eta|\ge 2|k|t \text{ or }|\eta|\le \frac{1}{2}|k|t }\\
  \le&\sum_{N\ge8}\int_{\xi}\rmA_k(\eta)\frac{|\eta|^{\frac{s}{2}}|\eta-kt|}{k^2+(\eta-kt)^2} \left|\widehat{v''}(\eta-\xi)_{N} \right||\xi-kt| \left|\hat\phi_k(\xi)_{<N/8} \right|\mathbf1_{|\eta|\ge \frac{1}{10}|k|,|\eta|\ge 2|k|t \text{ or }|\eta|\le \frac{1}{2}|k|t }d\xi\\
  \lesssim&\sum_{N\ge8}\int_{\xi}\rmA^{\mathrm R}(\eta-\xi)\frac{|\eta-\xi|^{\frac{s}{2}}}{\langle \eta-\xi\rangle} \left|\widehat{v''}(\eta-\xi)_{N} \right|\langle k,\xi\rangle^{2+2C_1\kappa} e^{c\lambda|\xi'|}|\xi-kt| \left|\hat\phi_k(\xi)_{<N/8}\right|d\xi.
\end{align*}
If $\frac{1}{2}|k|t<|\eta|< 2|k|t$, than $|\xi-\eta|\approx|\eta|\approx|k|t$, we have
\begin{align*}
  &\reallywidehat{\rmA|\nabla|^{\frac{s}{2}}(\pa_v-t\pa_z)\Delta_L^{-1}{\mathrm T}^2_{\mathrm {HL}}}(k,\eta)\mathbf1_{|\eta|\ge \frac{1}{10}|k|,\frac{1}{2}|k|t<|\eta|< 2|k|t}\\
  \le&\sum_{N\ge8}\int_{\xi}\rmA_k(\eta)\frac{|\eta|^{\frac{s}{2}}|\eta-kt|}{k^2+(\eta-kt)^2} \left|\widehat{v''}(\eta-\xi)_{N} \right||\xi-kt| \left|\hat\phi_k(\xi)_{<N/8} \right|\mathbf1_{|\eta|\ge \frac{1}{10}|k|,\frac{1}{2}|k|t<|\eta|< 2|k|t}d\xi\\
  \lesssim&\sum_{N\ge8}\int_{\xi}\rmA_k(\eta-\xi)\frac{|\eta-\xi|^{\frac{s}{2}}}{\langle \eta-\xi\rangle}\widehat{v''}(\eta-\xi)_{N}e^{c\lambda|\xi'|}|\xi-kt|t\hat\phi_k(\xi)_{<N/8}d\xi.
\end{align*}
If $|\eta|\le \frac{1}{10}|k|$, we have
\begin{align*}
  &\reallywidehat{\rmA|\nabla|^{\frac{s}{2}}(\pa_v-t\pa_z)\Delta_L^{-1}{\mathrm T}^2_{\mathrm {HL}}}(k,\eta)\mathbf1_{|\eta|\le \frac{1}{10}|k|}\\
  \le&\sum_{N\ge8}\int_{\xi}\rmA_k(\eta)\frac{|\eta|^{\frac{s}{2}}|\eta-kt|}{k^2+(\eta-kt)^2}\left|\widehat{v''}(\eta-\xi)_{N} \right||\xi-kt| \left|\hat\phi_k(\xi)_{<N/8} \right|\mathbf1_{|\eta|\le \frac{1}{10}|k|} d\xi\\
  \lesssim&\sum_{N\ge8}\int_{\xi}e^{c\lambda|\eta-\xi|^s} \langle \eta-\xi\rangle^{2+2C_1\kappa} \left|\widehat{v''}(\eta-\xi)_{N}\right|\rmA_k(\xi)|k,\xi|^{\frac{s}{2}}|\xi-kt| jdz\hat\phi_k(\xi)_{<N/8}d\xi.
\end{align*}

As a conclusion, by using Corollary \ref{cor-elliptic-low-2} we have
\begin{align*}
  \left\| \rmA|\nabla|^{\frac{s}{2}}(\pa_v-t\pa_z)\Delta_L^{-1}\mathrm T^2_{\mathrm {HL}}\right\|_{L^2}^2\lesssim \varepsilon\nu^\beta \|\rmA|\nabla|^{\frac{s}{2}}(\pa_v-t\pa_z)\phi_{\neq}\|_{L^2}^2+\varepsilon^2\nu^{2\beta}\|\rmA^{\mathrm R}|\pa_v|^{\frac{s}{2}}h\|_{L^2}^2.
\end{align*}
And the estimate for ${\mathrm T}^2_{\mathcal R}$ is similar and simpler.

The above estimates give \eqref{eq-est-hq-goal2} and then \eqref{eq-est-hq-2}.

\subsubsection{Proof of \eqref{eq-est-hq-3}} By using Lemma \ref{lem-product-1}, we deduce that
\begin{align*}
  &\nu^{\frac{1}{2}}\|\rmA\nabla_L\left(v'(\pa_v-t\pa_z)\phi_{\neq}\right)\|_{L^2}\\
  \lesssim&\nu^{\frac{1}{2}}\|\rmA\Delta_L\phi_{\neq}\|_{L^2}+\nu^{\frac{1}{2}}\|\rmA^{\mathrm R}\pa_vh\|_{L^2}\|\rmA f_{\neq}\|_{L^2}+\nu^{\frac{1}{2}}\|\rmA^{\mathrm R}h\|_{L^2}\|\rmA\Delta_L\phi_{\neq}\|_{L^2}.
\end{align*}
Then it suffices to show that
\begin{align}\label{eq-est-hq-goal}
   \nu^{\frac{1}{2}}\|\rmA\Delta_L\phi_{\neq}\|_{L^2}\lesssim \nu^{\frac{1}{2}}\|\rmA\nabla_Lf_{\neq}\|_{L^2}+\varepsilon\nu^{\beta}\nu^{\frac{1}{2}}\|\rmA^{\mathrm R}\pa_vh\|_{L^2}.
\end{align}
Recall that $f=\Delta_t\phi$. Similar to Lemma \ref{lem-elliptic-low-1}, we write $\Delta_t$ as a perturbation of $\Delta_L$ via
  \begin{align*}
    \Delta_L \phi_{\neq}=&f_{\neq}+\left(1-(v')^2\right)(\pa_v-t\pa_z)^2\phi_{\neq}-v''(\pa_v-t\pa_z)\phi_{\neq}\\
    =&f_{\neq}-\left(h^2+2h\right)(\pa_v-t\pa_z)^2\phi_{\neq}-(\pa_vh+h\pa_vh)(\pa_v-t\pa_z)\phi_{\neq}.
  \end{align*}
Therefore,
\begin{align*} 
  \|\rmA\Delta_L\phi_{\neq}\|_{L^2}^2\le& \|\rmA f_{\neq}\|_{L^2}^2+\|\rmA \left(1-(v')^2\right)(\pa_v-t\pa_z)^2\phi_{\neq}\|_{L^2}^2 +\|A v''(\pa_v-t\pa_z)\phi_{\neq}\|_{L^2}^2\\
  \lesssim&\|\rmA \nabla_Lf_{\neq}\|_{L^2}^2+\|\rmA^{\mathrm R}h\|_{L^2}^2\|\rmA\Delta_L\phi_{\neq}\|_{L^2}^2+\|\rmA^{\mathrm R}\pa_vh\|_{L^2}^2\|\rmA f_{\neq}\|_{L^2}^2,
\end{align*}
which gives \eqref{eq-est-hq-goal} and then \eqref{eq-est-hq-3}.

\subsection{High order estimates for $h$}\label{sec-hh}
Similar to $\bar h$ and $q$, it holds that
\begin{align*}
  \frac{1}{2}\frac{d}{dt}\|\rmA^{\mathrm R} h\|_{L^2}^2=&-\mathrm{CCK}_\lambda^{v,1}-\mathrm{CCK}_W^{v,1}-\mathrm{CCK}_G^{v,1}\\
  &-\underbrace{\int_v\rmA^{\mathrm R} h [\rmA^{\mathrm R} q\pa_v h-q\rmA^{\mathrm R}\pa_v h]dv}_{\mathrm{Com}^{h}}+\frac{1}{2}\int_v\pa_vq|\rmA^{\mathrm R} h|^2dv\\
  &+\underbrace{ \int_v\rmA^{\mathrm R} h \rmA^{\mathrm R} \bar hdv}_{\mathrm{NL}^h}-\nu \|\rmA^{\mathrm R}\pa_v h\|_{L^2}+\underbrace{\nu\int_v\rmA^{\mathrm R} h \rmA^{\mathrm R}[((v')^2-1)\pa_{vv} h]dv}_{\mathrm E^{h}}.
\end{align*}
\subsubsection{Estimate of the commutator}
We expand $\mathrm{Com}^{h}$ with a paraproduct in $v$:
\begin{align*}
  \mathrm{Com}^{h}=&\sum_{N\ge 8} \int_v \rmA^{\mathrm R} h [\rmA^{\mathrm R} \left(q_{N}\pa_v h_{<N/8}\right)-q_{N}\rmA^{\mathrm R}\pa_v h_{<N/8}]dv\\
  &+\sum_{N\ge 8} \int_v \rmA^{\mathrm R} h [\rmA^{\mathrm R} \left(q_{<N/8}\pa_v h_{N}\right)-q_{<N/8}\rmA^{\mathrm R}\pa_v h_N]dv\\
  &+\sum_{N\in \mathbb D}\sum_{N'\sim N} \int_v \rmA^{\mathrm R} h [\rmA^{\mathrm R} \left(q_{N'}\pa_v h_{N}\right)-q_{N'}\rmA^{\mathrm R}\pa_v h_{N}]dv\\
  =&\sum_{N\ge8}{\mathrm T}_N^{  h}+\sum_{N\ge8}{\mathrm R}_N^{ h}+\mathcal R^{ h}.
\end{align*}

We write 
\begin{align*}
  {\mathrm R}_N^{h}=&\int_v \rmA^{\mathrm R} h \rmA^{\mathrm R} \left(q_{N}\pa_v h_{<N/8}\right) dv-\int_v \rmA^{\mathrm R} h q_{N}\rmA^{\mathrm R}\pa_v h_{<N/8}dv={\mathrm R}_N^{h,2}+{\mathrm R}_N^{h,3},
\end{align*}

The estimate for ${\mathrm R}_N^{h,3}$ is the same as Section \ref{sec-reac-R3}. One can easily get
\begin{align*}
   \left|{\mathrm R}_N^{h,3}\right|\lesssim \frac{\varepsilon \nu^\beta}{t^2}\|\rmA^{\mathrm R} h\|_{L^2}^2.
\end{align*}

The estimate for ${\mathrm R}_N^{h,2}$ is similar to Section \ref{sec-reac-R2}. We write 
\begin{align*}
  {\mathrm R}_N^{h,2}=&\int_{\eta,\xi}\left[\left(\sum_{1\le|k|\le \frac{1}{2}|\eta|^s}\mathbf 1_{t\in \tilde{\mathrm{I}}_{k,\eta}\cap \tilde{\mathrm{I}}_{k,\xi}}\right)+\chi^*\right] \rmA^{\mathrm R} \hat h(\eta) \rmA^{\mathrm R}(\eta) \hat q(\xi)_{N}\widehat{\pa_v h}(\eta-\xi)_{<N/8}d\xi d\eta\\
  =&\left(\sum_{1\le|k|\le \frac{1}{2}|\eta|^s}{\mathrm R}_{N,k}^{h,2,{\mathrm R}}\right)+{\mathrm R}_{N}^{h,2,{\mathrm {NR}}},
\end{align*}
where $\chi^*=1-\sum_{1\le|k|\le \frac{1}{2}|\eta|^s}\mathbf 1_{t\in \tilde{\mathrm{I}}_{k,\eta}\cap \tilde{\mathrm{I}}_{k,\xi}}$.

The ${\mathrm R}_{N}^{h,2,{\mathrm {NR}}}$ term is easy to treat. Indeed, for $t\in {\mathrm{I}}_{k,\eta}$, on the support of the integrand, there holds
\begin{align*}
  \frac{|\eta|^{1-3\beta}}{|k|^{2-3\beta}}\lesssim \left|t-\frac{\eta}{k}\right| \text{ or }\frac{|\eta|^{1-3\beta}}{|k|^{2-3\beta}}\lesssim \langle\xi-\eta\rangle,
\end{align*}
then
\begin{align*}
  \frac{w_0(t,\xi)}{w^{\mathrm R}(t,\eta)}\lesssim \left(1+\frac{\frac{|\eta|^{1-3\beta}}{|k|^{2-3\beta}}}{\left[1+\left|t-\frac{\eta}{k}\right|\right]}\right)\frac{w_{\mathrm{NR}}(t,\xi)}{w_{\mathrm{NR}}(t,\eta)}\lesssim \langle \xi-\eta\rangle^{2+2C_1\kappa}e^{C\mu |\xi-\eta|^s}.
\end{align*}
Then we have
\begin{align*}
  \left|{\mathrm R}_{N}^{h,2,{\mathrm {NR}}}\right|\lesssim&\varepsilon^{\frac{1}{2}} t^{-2c_1}\||\pa_v|^{\frac{s}{2}}\rmA^{\mathrm R} h_{\sim N}\|_{L^2}^2+\varepsilon^{\frac{1}{2}}t^{2c_1}\left\||\pa_v|^{\frac{s}{2}}\frac{\rmA}{\langle\pa_v\rangle^s}q_{\sim N}\right\|_{L^2}^2.
\end{align*}

For $t\in \tilde{\mathrm{I}}_{k,\eta}\cap \tilde{\mathrm{I}}_{k,\xi}$ and $1\le|k|\le \frac{1}{2}|\eta|^s$, we have
\begin{align*}
  \frac{w_0(t,\xi)}{w^{\mathrm R}(t,\eta)}\lesssim& \frac{\frac{|\eta|^{1-3\beta}}{|k|^{2-3\beta}}}{\left[1+\left|t-\frac{\eta}{k}\right|\right]}\langle \xi-\eta\rangle^{1+2C_1\kappa}e^{C\mu |\xi-\eta|^s}\\
  \lesssim&  \frac{t^{1-3\beta+s}}{|k|^{1-s}}\sqrt{\frac{\pa_t w_0(t,\eta)}{w_0(t,\eta)}}\sqrt{\frac{\pa_t w_0(t,\xi)}{w_0(t,\xi)}}\frac{1}{\langle \xi\rangle^s} \langle \xi-\eta\rangle^{1+2C_1\kappa} e^{C\mu|\eta-\xi|^s}.
\end{align*}
It follows that
\begin{align*}
  \left|{\mathrm R}_{N,k}^{h,2,{\mathrm {R}}}\right|\lesssim& \varepsilon^{\frac{1}{2}}\left\|\sqrt{\frac{ \pa_t w}{w}}\rmA^{\mathrm R} h_{\sim N}\right\|_{L^2}^2+ \varepsilon^{\frac{1}{2}} t^{3}\left\|\sqrt{\frac{ \pa_t w}{w}}\frac{\rmA}{\langle\pa_v\rangle^s}q_{\sim N}\right\|_{L^2}^2.
\end{align*}

The transport term and the remainder term could be treated in the same way as Section \ref{sec-transport}.

Therefore, we have that
\begin{align*}
   \int^{T^*}_1\left|\mathrm{Com}^{h}\right|(t)dt\lesssim \varepsilon^{\frac{3}{2}}\nu^{\beta}.
\end{align*}

\subsubsection{Estimate of the nonlinear term}
We write
\begin{align*}
  \mathrm{NL}^h=&\int_v\rmA^{\mathrm R} h \rmA^{\mathrm R} \bar hdv=\frac{1}{2\pi}\int_\eta [\sum_{1\le|k|\le \frac{1}{2}|\eta|^s}\mathbf 1_{t\in \tilde{\mathrm{I}}_{k,\eta}}+\chi^*] \rmA^{\mathrm R} h(\eta) \rmA^{\mathrm R}(\eta) \bar h(\eta)d\eta\\
  =&\left(\sum_{1\le|k|\le \frac{1}{2}|\eta|^s} \mathrm{NL}^{h,{\mathrm R}}_k\right)+\mathrm{NL}^{h,{\mathrm {NR}}},
\end{align*}
where $\chi^*=1-\sum\limits_{1\le|k|\le \frac{1}{2}|\eta|^s}\mathbf 1_{t\in \tilde{\mathrm{I}}_{k,\eta}}$.

For $\mathrm{NL}^{h,{\mathrm R}}_k$, it is clear that
\begin{align*}
  \left|\mathrm{NL}^{h,{\mathrm R}}_k\right|\lesssim& \int_\eta  \mathbf 1_{t\in \tilde{\mathrm{I}}_{k,\eta},1\le|k|\le \frac{1}{2}|\eta|^s} \sqrt{\frac{\pa_t w_0(t,\eta)}{w_0(t,\eta)}}\rmA^{\mathrm R} \left|\hat h(\eta)\right| \sqrt{\frac{\pa_t w_0(t,\eta)}{w_0(t,\eta)}}\frac{|\eta|^{1-3\beta+s}}{|k|^{2-3\beta}}\frac{\rmA_0(\eta)}{\langle\eta\rangle^s} \left| \hat {\bar h}(\eta)\right| d\eta\\
  \lesssim&\int_\eta  \mathbf 1_{t\in \tilde{\mathrm{I}}_{k,\eta},1\le|k|\le \frac{1}{2}|\eta|^s} \sqrt{\frac{\pa_t w_0(t,\eta)}{w_0(t,\eta)}}\rmA^{\mathrm R} \left|\hat h(\eta)\right| \sqrt{\frac{\pa_t w_0(t,\eta)}{w_0(t,\eta)}}\frac{t^{1-3\beta+s}}{|k|^{1-s}}\frac{\rmA_0(\eta)}{\langle\eta\rangle^s} \left| \hat {\bar h}(\eta)\right|d\eta.
\end{align*}

Then we turn to $\mathrm{NL}^{h,{\mathrm {NR}}}$. In this case, $w_0=w^{\mathrm R}_0$, then we have
  \begin{align*}
  |\mathrm{NL}^{h,{\mathrm {NR}}}|\lesssim& \int_{|\eta|\ge 1}  \chi^* t^{-c_1} |\eta|^{\frac{s}{2}}\rmA^{\mathrm R} \left|\hat h(\eta)\right| t^{c_1}|\eta|^{\frac{s}{2}}\frac{A_0(\eta)}{\langle\eta\rangle^s} \left|\hat{\bar h}(\eta)\right| d\eta+\frac{1}{t^\frac{3}{2}}\int_{|\eta|\le 1}  \chi^*   \left|\hat h(\eta)\right|  t^{\frac{3}{2}} \left|\hat{\bar h}(\eta)\right|d\eta.
\end{align*}

Therefore, similar to the commutator term, we have that
\begin{align*}
   \int^{T^*}_1\left|\mathrm{NL}^h\right|(t)dt\lesssim \varepsilon^{\frac{3}{2}}\nu^{\beta}.
\end{align*}

\subsubsection{Estimates for $\mathrm{CCK}_\lambda^{v,2},\mathrm{CCK}_W^{v,2},\mathrm{CCK}_G^{v,2}$}
Recall that
\begin{align*}
  v''=\pa_vh+h\pa_vh,
\end{align*}
then by using Lemma \ref{lem-w-wgl} and Lemma \ref{lem-g-gl}, we can get the estimate in a similar way to Proposition \ref{pro-elliptic-high}. We only give the estimate for $\mathrm{CCK}_W^{v,2}$, and the proofs for the rest terms are similar.

It holds that
\begin{align*}
&\left\|\sqrt{\frac{\mathfrak w(\nu,t,\eta)\pa_t w_0(t,\eta)}{w_0(t,\eta)}} \frac{\rmA^{\mathrm R}}{\langle \eta\rangle}\widehat{v''}(t,\eta)\right\|_{L^2_\eta}^2\\
\le&\left\|\sqrt{\frac{\mathfrak w(\nu,t,\eta)\pa_t w_0(t,\eta)}{w_0(t,\eta)}} \frac{\rmA^{\mathrm R}}{\langle \eta\rangle}\widehat{\pa_vh}(t,\eta)\right\|_{L^2_\eta}^2+\left\|\sqrt{\frac{\mathfrak w(\nu,t,\eta)\pa_t w_0(t,\eta)}{w_0(t,\eta)}} \frac{\rmA^{\mathrm R}}{\langle \eta\rangle}\widehat{h\pa_vh}(t,\eta)\right\|_{L^2_\eta}^2\\
\lesssim&\mathrm{CCK}_W^{v,1}+\left\|\sqrt{\frac{\mathfrak w(\nu,t,\eta)\pa_t w_0(t,\eta)}{w_0(t,\eta)}} \frac{\rmA^{\mathrm R}}{\langle \eta\rangle}\widehat{h\pa_vh}(t,\eta)\right\|_{L^2_\eta}^2.
\end{align*}
For the second term, we write
\begin{align*}
  \sqrt{\frac{\mathfrak w(\nu,t,\eta)\pa_t w_0(t,\eta)}{w_0(t,\eta)}} \frac{\rmA^{\mathrm R}(\eta)}{\langle \eta\rangle}\widehat{h\pa_vh}(t,\eta)=\int\sqrt{\frac{\mathfrak w(\nu,t,\eta)\pa_t w_0(t,\eta)}{w_0(t,\eta)}} \frac{\rmA^{\mathrm R}(\eta)}{\langle \eta\rangle}\hat h(\eta-\xi)\widehat{\pa_vh}(\xi) d\xi.
\end{align*}

If $\frac{\mathfrak w\pa_t w_0(t,\eta)}{w_0(t,\eta)}\neq0$, we must have $|\eta|\ge \frac{1}{2}$ and $t\le 2|\eta|$. One of the following three cases must hold
\begin{itemize}
  \item[A1]:  $\frac{1}{2}|\eta|\le|\xi-\eta|\le 2|\eta|$,
  \item[A2]:  $\frac{1}{2}|\eta|\le|\xi|\le 2|\eta|$,
  \item[A3]:  $|\xi-\eta|,|\xi|\ge 2|\eta|$.
\end{itemize}
If A1 holds, by using Lemma \ref{lem-w-wgl}, we have
\begin{align*}
  &\int\sqrt{\frac{\mathfrak w(\nu,t,\eta)\pa_t w_0(t,\eta)}{w_0(t,\eta)}} \frac{\rmA^{\mathrm R}(\eta)}{\langle \eta\rangle} \left|\hat h(\eta-\xi)\right| \left|\widehat{\pa_vh}(\xi) \right| d\xi\\
  \lesssim&\int \frac{\rmA^{\mathrm R}(\eta-\xi)}{\langle\eta-\xi \rangle}\left(\sqrt{\frac{\mathfrak w\pa_t w_0(t,\eta-\xi)}{w_0(t,\eta-\xi)}}+\sqrt{\frac{\mathfrak w\pa_t g(t,\eta-\xi)}{g(t,\eta-\xi)}}+t^{-c_1}|\eta-\xi|^{\frac{s}{2}}\right)\left|\hat h(\eta-\xi)\right| \\
  &\qquad\qquad \times\langle\xi\rangle^{1+2C_1\kappa}e^{c\lambda|\xi|^s}\left|\widehat{\pa_vh}(\xi) \right|d\xi.
\end{align*}
If A2 holds, similarly
\begin{align*}
  &\int\sqrt{\frac{\mathfrak w(\nu,t,\eta)\pa_t w_0(t,\eta)}{w_0(t,\eta)}} \frac{\rmA^{\mathrm R}(\eta)}{\langle \eta\rangle} \left|\hat h(\eta-\xi)\right| \left|\widehat{\pa_vh}(\xi) \right| d\xi\\
  \lesssim&\int \frac{\rmA^{\mathrm R}(\xi)}{\langle\xi \rangle}\left(\sqrt{\frac{\mathfrak w\pa_t w_0(t, \xi)}{w_0(t, \xi)}}+\sqrt{\frac{\mathfrak w\pa_t g(t, \xi)}{g(t, \xi)}}+t^{-c_1}|\eta-\xi|^{\frac{s}{2}}\right)\left|\widehat{\pa_vh}(\xi) \right|  \\
  &\qquad\qquad \times\langle\eta-\xi\rangle^{1+2C_1\kappa}e^{c\lambda|\eta-\xi|^s}\left|\hat h(\eta-\xi)\right|d\xi\\
  \lesssim&\int \rmA^{\mathrm R}(\xi)\left(\sqrt{\frac{\mathfrak w\pa_t w_0(t, \xi)}{w_0(t, \xi)}}+\sqrt{\frac{\mathfrak w\pa_t g(t, \xi)}{g(t, \xi)}}+t^{-c_1}|\eta-\xi|^{\frac{s}{2}}\right)\left|\hat{h}(\xi) \right|  \\
  &\qquad\qquad \times\langle\eta-\xi\rangle^{1+2C_1\kappa}e^{c\lambda|\eta-\xi|^s}\left|\hat h(\eta-\xi)\right|d\xi.
\end{align*}
If A3 holds, as $t\lesssim |\xi-\eta|$, we have
\begin{align*}
  &\int\sqrt{\frac{\mathfrak w(\nu,t,\eta)\pa_t w_0(t,\eta)}{w_0(t,\eta)}} \frac{\rmA^{\mathrm R}(\eta)}{\langle \eta\rangle} \left|\hat h(\eta-\xi)\right| \left|\widehat{\pa_vh}(\xi) \right| d\xi\\
  \lesssim&\int \rmA^{\mathrm R}( \xi) t^{-c_1}|\xi|^{\frac{s}{2}}\left|\hat{h}(\xi) \right| \langle\eta-\xi\rangle^{2+2C_1\kappa}e^{c\lambda|\eta-\xi|^s}\left|\hat h(\eta-\xi)\right|d\xi.
\end{align*}

As a conclusion, we have that
\begin{align*}
  \mathrm{CCK}_W^{v,2}\lesssim \mathrm{CCK}_\lambda^{v,1}+\mathrm{CCK}_W^{v,1}+\mathrm{CCK}_G^{v,1}.
\end{align*}

For the rest terms, the dissipation error term $\mathrm E^{ h}$ can be treated in the same way to Section \ref{sec-diss-error}, and the term  $\frac{1}{2}\int_v \pa_vq|\rmA^{\mathrm R}h|^2dv$ is easy to control.

Therefore, we complete the proof of \eqref{eq-est-boot-hh}.
\begin{appendix}
\section{Littlewood-Paley decomposition and paraproducts}\label{sec-decompo}
In this section, we fix conventions and notation regarding Fourier analysis, Littlewood–Paley and paraproduct decompositions. See, for example \cite{BCD2011,Bony1981}, for more details. First, we define the Littlewood-Paley decomposition only in the $v$ variable. Let $\varphi \in C_{0}^{\infty}(\mathbb{R})$ be such that $\varphi(\xi)=1$ for $|\xi| \le \frac{1}{2}$ and $\varphi(\xi)=0$ for $|\xi| \ge \frac{3}{4}$ and define $\rho(\xi)=\varphi(\xi/2)-\varphi(\xi)$, supported in the range $\xi \in(\frac{1}{2},\frac{3}{2})$. Then we have the partition of unity
\begin{align*}
1=\varphi(\xi)+\sum_{N \in 2^{\mathbb{N}}} \rho\left(N^{-1} \xi\right),
\end{align*}
where we mean that the sum runs over the dyadic numbers $N=1,2,4,8, \ldots, 2^{j}, \ldots$ and we define the cut-off $\rho_{N}(\xi)=\rho\left(N^{-1} \xi\right)$, each supported in $\frac{1}{2}N \le|\xi| \le \frac{3}{2}N$. For $f \in L^{2}(\mathbb{R})$ we define
\begin{align*}
\begin{aligned}
f_{N} &=\rho_{N}\left(\left|\partial_{v}\right|\right) f \\
f_{\frac{1}{2}} &=\varphi\left(\left|\partial_{v}\right|\right) f \\
f_{<N} &=f_{\frac{1}{2}}+\sum_{K \in 2^{\mathbb{N}}: K<N} f_{K},
\end{aligned}
\end{align*}
which defines the decomposition
\begin{align*}
  f=f_{\frac{1}{2}}+\sum_{N\in 2^{\mathbb{N}}} f_{N}.
\end{align*}

Recall the definition of $\mathbb D$. There holds the almost orthogonality and the approximate projection property
\begin{align}
  \left\|f\right\|_{L^2}^2\approx \sum_{N\in\mathbb D}\left\|f_N\right\|_{L^2}^2,\label{eq-decomp-sum}\\
  \left\|f_N\right\|_{L^2}^2\approx \left\|(f_N)_N\right\|_{L^2}^2.
\end{align}
The following is also clear:
\begin{align*}
\left\|\left|\partial_{v}\right| f_{N}\right\|_{2} \approx N\left\|f_{N}\right\|_{2} .
\end{align*}
We make use of the notation
\begin{align*}
f_{\sim N}=\sum_{N'\sim N}=\sum_{N' \in \mathbb{D}: 1/8 N \leqq N' \leqq 8 N} f_{N'},
\end{align*}
for some constant $C$ which is independent of $N$. Generally, the exact value of $C$ which is being used is not important; what is important is that it is finite and independent of $N$.  

Another key Fourier analysis tool employed in this work is the paraproduct decomposition, introduced by Bony \cite{Bony1981}. Given suitable functions $f, g$ we may define the inhomogeneous paraproduct decomposition (in either $(z, v)$ or just $v$ ),
\begin{align*}
\begin{aligned}
f g &=T_{f} g+T_{g} f+\mathcal{R}(f, g) \\
&=\sum_{N \ge 8} f_{<N / 8} g_{N}+\sum_{N \ge 8} g_{<N / 8} f_{N}+\sum_{N \in \mathbb{D}} \sum_{N'\sim N} g_{N^{\prime}} f_{N},
\end{aligned}
\end{align*}
where all the sums are understood to run over $\mathbb{D}$. 

\section{Elementary inequalities and Gevrey spaces}
In this section, we show some basic inequalities.
\begin{lemma}\label{lem-conv-Young}
  Let $f(\xi)$, $g(\xi)\in L^2_\xi(\mathbb R^d)$, $\langle \xi\rangle^\sigma h(\xi)\in L^2_\xi(\mathbb R^d)$ and $\langle \xi\rangle^\sigma b(\xi)\in L^2_\xi(\mathbb R^d)$ for $\sigma>d/2$, then we have
\begin{align*}
  \left\|f*h\right\|_{L^2}&\lesssim \left\|f\right\|_{L^2}\left\|\langle \cdot\rangle^\sigma h\right\|_{L^2},\\
  \int \left|f(\xi)\left(g*h\right)(\xi)\right|  d \xi&\lesssim \left\|f\right\|_{L^2} \left\|g\right\|_{L^2}\left\|\langle \cdot\rangle^\sigma h\right\|_{L^2},\\
   \int \left|f(\xi)\left(g*h*b\right)(\xi)\right|  d \xi&\lesssim \left\|f\right\|_{L^2} \left\|g\right\|_{L^2}\left\|\langle \cdot\rangle^\sigma h\right\|_{L^2}\left\|\langle \cdot\rangle^\sigma b\right\|_{L^2}.
\end{align*}
\end{lemma}
\begin{proof}
  These inequalities follow from Young's convolution inequality and the fact that
  \begin{align*}
    \left\|h\right\|_{L^1}\le \left(\int \frac{1}{\langle \xi\rangle^{2\sigma}} d\xi \right)^{\frac{1}{2}}\left\|\langle \cdot\rangle^\sigma h\right\|_{L^2}\lesssim \left\|\langle \cdot\rangle^\sigma h\right\|_{L^2}.
  \end{align*}
\end{proof}

  \begin{lemma}\label{lem-dis-s}
    Let $0<s<1$ and $x\ge y\ge0$ (without loss of generality).
    \begin{itemize}
      \item If $x+y>0$,
      \begin{align}\label{inq-s1}
        |x^s-y^s|\lesssim\frac{1}{x^{1-s}+y^{1-s}}|x-y|.
      \end{align}
      \item If $|x-y|\le \frac{x}{K}$ for some $K>1$, then
      \begin{align}\label{inq-s2}
        |x^s-y^s|\le \frac{s}{(K-1)^{1-s}}|x-y|^s.
      \end{align}
      Note $\frac{s}{(K-1)^{1-s}}<1$ as soon as $s^{\frac{1}{1-s}}+1<K$.
      \item In general,
      \begin{align}
        |x+y|^s\le \left(\frac{x}{x+y}\right)^{1-s}(x^s+y^s).
      \end{align}
      In particular, if $y\le x\le Ky$ for some $K<+\infty$ then
      \begin{align}\label{inq-s3}
        |x+y|^s\le \left(\frac{K}{1+K}\right)^{1-s}(x^s+y^s).
      \end{align}
    \end{itemize}
  \end{lemma}
  \begin{proof}
    For $x>2y$, we have $x\le 2(x-y)$ 
    \begin{align*}
      x^s-y^s\le x^s \lesssim \frac{x-y}{x^{1-s}}\lesssim\frac{1}{x^{1-s}+y^{1-s}}|x-y|.
    \end{align*}  
    For $0<y\le x\le 2y$, we have 
    \begin{align*}
      x^s-y^s=s\int^x_y \frac{1}{z^{1-s}}dz\lesssim s\frac{x-y}{y^{1-s}}\lesssim\frac{s}{x^{1-s}+y^{1-s}}|x-y|.
    \end{align*} 
    This gives \eqref{inq-s1}.

    To prove \eqref{inq-s2} we deduce from $|x-y|\le \frac{x}{K}$ that $\frac{1}{y}\le \frac{K}{K-1}\frac{1}{x}$, and hence
    \begin{align*}
      x^s=y^s+s\int^x_y \frac{1}{z^{1-s}}dz\le y^s+ \frac{s}{y^{1-s}}(x-y)\le y^s+\frac{s}{(K-1)^{1-s}}|x-y|^s.
    \end{align*}
    To see \eqref{inq-s3},
    \begin{align*}
      |x+y|^s=\left(\frac{x}{x+y}\right)|x+y|^s+\left(\frac{y}{x+y}\right)|x+y|^s\le \left(\frac{x}{x+y}\right)^{1-s}(x^s+y^s).
    \end{align*}
  \end{proof}
\begin{lemma}\label{lem-product}(Product lemma). For all $0\le s<1$, $\sigma\ge0$, and $\sigma_1>1$, there exists $c\in(0,1)$ such that the following holds for all $f,g\in \mathcal{G}^{s,\lambda,\sigma}$:
\begin{align}
  \left\|fg\right\|_{\mathcal{G}^{s,\lambda,\sigma}}\lesssim \left\|f\right\|_{\mathcal{G}^{s,c\lambda,\sigma_1}}\left\|g\right\|_{\mathcal{G}^{s,\lambda,\sigma}}+\left\|g\right\|_{\mathcal{G}^{s,c\lambda,\sigma_1}}\left\|f\right\|_{\mathcal{G}^{s,\lambda,\sigma}},
\end{align}
in particular, for $\sigma>1$, $\mathcal{G}^{s,\lambda,\sigma}$ has the algebra property:
\begin{align}
  \left\|fg\right\|_{\mathcal{G}^{s,\lambda,\sigma}}\lesssim \left\|f\right\|_{\mathcal{G}^{s,\lambda,\sigma}}\left\|g\right\|_{\mathcal{G}^{s,\lambda,\sigma}}.
\end{align}
\end{lemma}
\section{Coordinate transformations in Gevrey spaces}
The proof of the main theorem requires moving from $(x,y)$ to $(z,v)$ coordinates at the outset and then back again. It is crucial to notice the regularity loss incurred in this section.

It is well-known that the $\mathcal{G}^{s,\lambda}$ norms have an equivalent `physical-side' representation which will be convenient here:
\begin{align*}
  \left\|f\right\|_{\mathcal{G}^{s,\lambda}}\approx\left\|f\right\|_{\mathcal{G}^{s,\lambda}_{ph}}=\left[\sum^\infty_{n=0}\left(\frac{\lambda^n}{(n!)^{\frac{1}{s}}}\left\| D^nf\right\|_{L^2}\right)^2\right]^{\frac{1}{2}}.
\end{align*}
In \cite{BM2015}, the authors introduced a slightly more general scale of norms:
\begin{align*}
  \left\|f\right\|_{l^qL^p,s,\lambda}=\left[\sum^\infty_{n=0}\left(\frac{\lambda^n}{(n!)^{\frac{1}{s}}}\left\| D^nf\right\|_{L^p}\right)^q\right]^{\frac{1}{q}}.
\end{align*}
By H\"older's inequality and the Sobolev embedding: for $\lambda'>\lambda$ and $p,q\in[1,\infty]$,
\begin{align*}
  \left\|f\right\|_{l^qL^p,s,\lambda}\le\left\|f\right\|_{l^1L^p,s,\lambda}\lesssim_{\lambda,\lambda'}\left\|f\right\|_{l^qL^p,s,\lambda'},\\
  \left\|f\right\|_{l^2L^\infty,s,\lambda}\lesssim \left\|f\right\|_{l^2L^2,s,\lambda}+\left\|Df\right\|_{l^2L^2,s,\lambda}.
\end{align*}

Based on the equivalent definition, we have the following result.

\begin{lemma}\label{lem-compo}(Composition inequality). For all $s\in[0,1]$ and $\lambda'>\lambda>0$, it holds that
\begin{align*}
  \left\|F\circ(\mathrm{Id}+G)\right\|_{\mathcal{G}^{s,\lambda,\sigma}}\lesssim\left\|\mathrm{det}(\mathrm{Id}+\nabla G)^{-1}\right\|_{L^\infty}^{\frac{1}{2}}\left\|F\right\|_{\mathcal{G}^{s,\lambda'+\zeta,\sigma}}
\end{align*}
  where $\zeta=\left\|G\right\|_{\mathcal{G}^{s,\lambda,\sigma}}\le \frac{1}{2}$.
\end{lemma}
\begin{proof}
When $s=0$, the result of this lemma is trivial. So we only focus on the case $s>0$. It follows from  Leibniz's rule that
  \begin{align*}
    &\left\|F\circ(\mathrm{Id}+G)\right\|_{\mathcal{G}^{s,\lambda,\sigma}}\\
    \approx& \left\|F\circ(\mathrm{Id}+G)\right\|_{\mathcal{G}^{s,\lambda}}+\left\|(D^\sigma)F\circ(\mathrm{Id}+G)\right\|_{\mathcal{G}^{s,\lambda}}\\
    =&\left\|F\circ(\mathrm{Id}+G)\right\|_{\mathcal{G}^{s,\lambda}}+\left\|\sum_{k=1}^{\sigma}(D^kF)\circ(\mathrm{Id}+G)\sum_{\sum_j=1^kjm_j=\sigma,\sum_{j=1}^km_j=k}\prod_{j=1}^k \left(D^{j}G\right)^{m_j}\right\|_{\mathcal{G}^{s,\lambda}}\\
    \lesssim&\sum_{k=0}^{\sigma}\left\|(D^kF)\circ(\mathrm{Id}+G)\right\|_{\mathcal{G}^{s,\lambda}}\\
    \lesssim&\left\|\mathrm{det}(\mathrm{Id}+\nabla G)^{-1}\right\|_{L^\infty}^{\frac{1}{2}}\sum_{k=0}^{\sigma}\left\|D^kF\right\|_{\mathcal{G}^{s,\lambda'+\zeta}}\\
    \lesssim&\left\|\mathrm{det}(\mathrm{Id}+\nabla G)^{-1}\right\|_{L^\infty}^{\frac{1}{2}}\left\|F\right\|_{\mathcal{G}^{s,\lambda'+\zeta,\sigma}}.
  \end{align*}
Here we use Lemma A.4 of \cite{BM2015}.
\end{proof}
\begin{lemma}\label{lem-inverse}
  (Inverse function theorem). Let $\alpha(x):\mathbb T\times\mathbb R\to \mathbb T\times\mathbb R$ be a given smooth function and consider the equation
\begin{align}\label{eq-beta-alpha}
  \beta(x)=\alpha \left(x+\beta(x)\right).
\end{align}
For $\sigma\ge8$ and all $\lambda'>\lambda>0$, there exists an $\varepsilon_0=\varepsilon_0(\lambda,\lambda')>0$ such that if $\left\|\alpha\right\|_{\mathcal{G}^{s,\lambda',\sigma}}\le \varepsilon_0$ then \eqref{eq-beta-alpha} has a smooth smooth solution $\beta$ which satisfies
\begin{align*}
  \left\|\beta\right\|_{\mathcal{G}^{s,\lambda,\sigma}}\lesssim \left\|\alpha\right\|_{\mathcal{G}^{s,\lambda',\sigma}}.
\end{align*}
\end{lemma}
\begin{proof}
  One can get the result by following the proof of Lemma A.5 of \cite{BM2015} with the same techniques used in Lemma \ref{lem-compo}.
\end{proof}

\section{Proof of Lemma \ref{lem-energy-01}}\label{sec-appD}
\begin{proof}[Proof of Lemma \ref{lem-energy-01}]
  The proof is classical, here we give the proof for completeness.

We first use the linear change of variables of $\grave z=x-ty$. Define $\grave \omega(t,\grave z,y)=\omega(t,\grave z+ty,y)$, $\grave u(t,\grave z,y)=u(t,\grave z+ty,y)$, and $\grave \psi(t,\grave z,y)=\psi(t,\grave z+ty,y)$ , which satisfies
\begin{equation}\label{eq-omega-lcv}
  \left\{
    \begin{array}{l}
        \pa_t\grave \omega(t,\grave z,y)+\nabla^{\perp}_{\grave z,y}\grave \psi\cdot\nabla_{\grave z,y}\grave\omega(t,\grave z,y)=\nu \Delta_{L;\grave z,y}\grave \omega(t,\grave z,y),\\
  \Delta_{L;\grave z,y}\grave \psi(t,\grave z,y)=\grave \omega(t,\grave z,y),\ \grave u(t,\grave z,y)=\nabla^{\perp}_{L;\grave z,y}\grave \psi,\\
\grave \omega(t,\grave z,y)|_{t=0}=\omega(t,x,y)|_{t=0}=\omega_{in}(x,y).
    \end{array}
  \right.
\end{equation}
  Define $\grave \lambda(t)=\lambda_0-\frac{1}{8}(\lambda_0-\lambda_1)t$ for $t\in[0,1]$. It is clear that $\inf_{t\in[0,1]}\grave \lambda(t)-\lambda(t)\ge \frac{1}{8}(\lambda_0-\lambda_1)$. We also define $\lambda_2=\frac{13}{16}\lambda_0+\frac{3}{16}\lambda_1$ which satisfies $\lambda(t)<\lambda_2<\grave \lambda(t)$ for $t\in[0,1]$. Then we have
  \begin{align*}
    \frac{1}{2}\frac{d}{dt}\|\grave \omega(t)\|_{\mathcal{G}^{s(\beta),\grave\lambda,\sigma}}^2=&\int  e^{\grave\lambda|\nabla|^s}\langle \nabla\rangle^\sigma \grave \omega  e^{\grave\lambda|\nabla|^s}\langle \nabla\rangle^\sigma \pa_t\grave \omega   dxdy+\dot{\grave\lambda}(t)\|\nabla^{\frac{s}{2}}\grave \omega(t)\|_{\mathcal{G}^{s(\beta),\grave\lambda,\sigma}}^2.
  \end{align*}
From \eqref{eq-omega-lcv}, we write
\begin{align*}
  &\int  e^{\grave\lambda|\nabla|^s}\langle \nabla\rangle^\sigma \grave\omega  e^{\grave\lambda|\nabla|^s}\langle \nabla\rangle^\sigma \pa_t\grave\omega   dxdy\\
  =&\int  e^{\grave\lambda|\nabla|^s}\langle \nabla\rangle^\sigma \grave\omega  e^{\grave\lambda|\nabla|^s}\langle \nabla\rangle^\sigma \left(\nu \Delta_{L} \grave\omega-\nabla^{\perp}\grave \psi\cdot\nabla \grave\omega \right)  d\grave zdy\\
  =&-\nu \|\nabla_L \grave\omega(t)\|_{\mathcal{G}^{s(\beta),\grave\lambda,\sigma}}^2+\int  \nabla\cdot\nabla^{\perp}\grave \psi \left| e^{\grave\lambda|\nabla|^s}\langle \nabla\rangle^\sigma \grave\omega\right|^2 d\grave zdy\\
  &-\int  e^{\grave\lambda|\nabla|^s}\langle \nabla\rangle^\sigma \grave\omega  \left[e^{\grave\lambda|\nabla|^s}\langle \nabla\rangle^\sigma(\nabla^{\perp}\grave \psi\cdot\nabla \grave\omega)-\nabla^{\perp}\grave \psi\cdot\nabla e^{\grave\lambda|\nabla|^s}\langle \nabla\rangle^\sigma\grave\omega \right]  d\grave zdy.
  \end{align*}
It is clear that $\nabla\cdot\nabla^{\perp}\grave \psi=0$, so we only need to focus on the last term. We expand this term with a paraproduct (in both $z$ and $v$):
\begin{align*}
  &\int  e^{\grave\lambda|\nabla|^s}\langle \nabla\rangle^\sigma \grave\omega  \left[e^{\grave\lambda|\nabla|^s}\langle \nabla\rangle^\sigma(\nabla^{\perp}\grave \psi\cdot\nabla \grave\omega)-\nabla^{\perp}\grave \psi\cdot\nabla e^{\grave\lambda|\nabla|^s}\langle \nabla\rangle^\sigma\grave\omega \right]  d\grave zdy\\
  =&\sum_{N\ge8}\int  e^{\grave\lambda|\nabla|^s}\langle \nabla\rangle^\sigma \grave\omega  \left[e^{\grave\lambda|\nabla|^s}\langle \nabla\rangle^\sigma(\nabla^{\perp}\grave \psi_{<N/8}\cdot\nabla \grave\omega_N)-\nabla^{\perp}\grave \psi_{<N/8}\cdot\nabla e^{\grave\lambda|\nabla|^s}\langle \nabla\rangle^\sigma\grave\omega_N \right]  d\grave zdy\\
  &+\sum_{N\ge8}\int  e^{\grave\lambda|\nabla|^s}\langle \nabla\rangle^\sigma \grave\omega  \left[e^{\grave\lambda|\nabla|^s}\langle \nabla\rangle^\sigma(\nabla^{\perp}\grave \psi_{N}\cdot\nabla \grave\omega_{<N/8})-\nabla^{\perp}\grave \psi_{N}\cdot\nabla e^{\grave\lambda|\nabla|^s}\langle \nabla\rangle^\sigma\grave\omega_{<N/8} \right]  d\grave zdy\\
  &+\sum_{N\in \mathbb D}\sum_{N/8\le N'\le 8N}\int  e^{\grave\lambda|\nabla|^s}\langle \nabla\rangle^\sigma \grave\omega  \left[e^{\grave\lambda|\nabla|^s}\langle \nabla\rangle^\sigma(\nabla^{\perp}\grave \psi_{N'}\cdot\nabla \grave\omega_{N})-\nabla^{\perp}\grave \psi_{N'}\cdot\nabla e^{\grave\lambda|\nabla|^s}\langle \nabla\rangle^\sigma\grave\omega_{N} \right]  d\grave zdy\\
  =&\frac{1}{2\pi}\sum_{N\ge8}\grave T_N+\frac{1}{2\pi}\sum_{N\ge8}\grave R_N+\frac{1}{2\pi}\grave{\mathcal R}.
\end{align*}
On the Fourier side,
\begin{align*}
  \grave T_N=&-\sum_{k,m}\int e^{\grave\lambda|k,\eta|^s}\langle k,\eta\rangle^\sigma \hat{\grave\omega}_k(\eta)[e^{\grave\lambda|k,\eta|^s}\langle k,\eta\rangle^\sigma -e^{\grave\lambda|m,\xi|^s}\langle m,\xi\rangle^\sigma]\\
  &\qquad\qquad\qquad\qquad\qquad\qquad\times(k\xi-m\eta)\hat{\grave\psi}_{k-m}(\eta-\xi)_{<N/8}\hat{\grave\omega}_m(\xi)_N d\xi d\eta,
\end{align*}
and on the support of the integrand there holds
\begin{align*}
  \left||k,\eta|-|m,\xi|\right|\le|k-m,\eta-\xi|\le \frac{3}{16}|k,\xi|,\\
  \frac{3}{16}|m,\xi|\le|k,\eta|\le \frac{19}{16}|m,\xi|.
\end{align*}
Then we deduce from \eqref{inq-s1} and \eqref{inq-s2} that
\begin{align*}
   \left||k,\eta|^s-|m,\xi|^s\right|\lesssim \frac{|k-m,\eta-\xi|}{|m,\xi|^{1-s}},\quad |k,\eta|^s\le  |m,\xi|^s+c |k-m,\eta-\xi|^s,
\end{align*}
and then
\begin{align*}
  &\left|e^{\grave\lambda|k,\eta|^s}\langle k,\eta\rangle^\sigma -e^{\grave\lambda|m,\xi|^s}\langle m,\xi\rangle^\sigma\right|\\
  =&e^{\grave\lambda|m,\xi|^s}\langle m,\xi\rangle^\sigma\left|e^{\grave\lambda|k,\eta|^s-\grave\lambda|m,\xi|^s}\frac{\langle k,\eta\rangle^\sigma}{\langle m,\xi\rangle^\sigma} -1\right|\\
  \le&e^{\grave\lambda|m,\xi|^s}\langle m,\xi\rangle^\sigma\frac{\langle k,\eta\rangle^\sigma}{\langle m,\xi\rangle^\sigma}\left|e^{\grave\lambda|k,\eta|^s-\grave\lambda|m,\xi|^s} -1\right|+e^{\grave\lambda|m,\xi|^s}\langle m,\xi\rangle^\sigma\left|\frac{\langle k,\eta\rangle^\sigma}{\langle m,\xi\rangle^\sigma} -1\right|\\
  \lesssim&e^{\grave\lambda|m,\xi|^s}\langle m,\xi\rangle^\sigma e^{c\grave\lambda|k-m,\eta-\xi|^s}\frac{|k-m,\eta-\xi|}{|m,\xi|^{1-s}}+e^{\grave\lambda|m,\xi|^s}\langle m,\xi\rangle^\sigma \frac{|k-m,\eta-\xi|}{\langle m,\xi\rangle}.
\end{align*}
Then by Lemma \ref{lem-conv-Young}, we have
\begin{align*}
  \left|\grave T_N\right|\lesssim \|\grave \psi\|_{\mathcal{G}^{s(\beta),\grave\lambda,4}}\|\nabla^{\frac{s}{2}}\grave \omega_{\sim N}\|_{\mathcal{G}^{s(\beta),\grave\lambda,\sigma}}^2\lesssim (\|\grave u\|_{L^2}+\|\grave \omega\|_{\mathcal{G}^{s(\beta),\grave\lambda,\sigma}})\|\nabla^{\frac{s}{2}}\grave \omega_{\sim N}\|_{\mathcal{G}^{s(\beta),\grave\lambda,\sigma}}^2.
\end{align*}

Next, we turn to
\begin{align*}
  \grave R_N=&2\pi\int  e^{\grave\lambda|\nabla|^s}\langle \nabla\rangle^\sigma \grave\omega  \left[e^{\grave\lambda|\nabla|^s}\langle \nabla\rangle^\sigma(\nabla^{\perp}\grave \psi_{N}\cdot\nabla \grave\omega_{<N/8})-\nabla^{\perp}\grave \psi_{N}\cdot\nabla e^{\grave\lambda|\nabla|^s}\langle \nabla\rangle^\sigma\grave\omega_{<N/8} \right]  d\grave zdy\\
  =&\grave R_{N,1}+\grave R_{N,2}.
\end{align*}
The second term $\grave R_{N,2}$ is easy to estimate, as the derivative applied on $\grave\omega_{<N/8}$ could remove to $\grave \psi_{N}$. For $\grave R_{N,1}$, we have
\begin{align*}
  \grave R_{N,1}=&-\sum_{k,m}\int e^{\grave\lambda|k,\eta|^s}\langle k,\eta\rangle^\sigma \hat{\grave\omega}_k(\eta)e^{\grave\lambda|k,\eta|^s}\langle k,\eta\rangle^\sigma(m\eta-k\xi)\hat{\grave\psi}_m(\xi)_N\hat{\grave\omega}_{k-m}(\eta-\xi)_{<N/8} d\xi d\eta\\
  \lesssim&\sum_{k,m}\int e^{\grave\lambda|k,\eta|^s}\langle k,\eta\rangle^\sigma |\hat{\grave\omega}_k(\eta)|e^{\grave\lambda|m,\xi|^s}\langle m,\xi\rangle^\sigma|\widehat{\nabla\grave\psi}_m(\xi)_N||\widehat{\nabla\grave\omega}_{k-m}(\eta-\xi)_{<N/8}| d\xi d\eta\\
  \lesssim&\sum_{k,m}\int e^{\grave\lambda|k,\eta|^s}\langle k,\eta\rangle^\sigma |\hat{\grave\omega}_k(\eta)|e^{\grave\lambda|m,\xi|^s}\langle m,\xi\rangle^\sigma|\hat{\grave\omega}_m(\xi)_N||\widehat{\nabla\grave\omega}_{k-m}(\eta-\xi)_{<N/8}| d\xi d\eta\\
  \lesssim& \|\grave \omega\|_{H^4}\| \grave \omega_{\sim N}\|_{\mathcal{G}^{s(\beta),\grave\lambda,\sigma}}^2\lesssim  \|\grave \omega\|_{\mathcal{G}^{s(\beta),\grave\lambda,\sigma}}\| \grave \omega_{\sim N}\|_{\mathcal{G}^{s(\beta),\grave\lambda,\sigma}}^2.
\end{align*}
Here we use the fact that $\frac{|m,\xi|}{m^2+|\xi-mt|^2}\le 1$ for $t\in[0,1]$ and $\frac{N}{2}\le|m,\xi|\le \frac{3N}{2}$ with $N\ge8$, and
\begin{align*}
  \left||m,\xi|\hat{\grave\psi}_m(\xi)_N\right|=\left|\frac{|m,\xi|}{m^2+|\xi-mt|^2}\hat{\grave\omega}_m(\xi)_N\right|\le \left|\hat{\grave\omega}_m(\xi)_N\right|.
\end{align*}

The remainder term $\grave{\mathcal R}$ can be treated in the same way as $\grave T_N$. Then we could get
\begin{align*}
  &\frac{1}{2}\frac{d}{dt}\|\grave \omega(t)\|_{\mathcal{G}^{s(\beta),\grave\lambda,\sigma}}^2+\left|\dot{\grave\lambda}(t)\right|\|\nabla^{\frac{s}{2}}\grave \omega(t)\|_{\mathcal{G}^{s(\beta),\grave\lambda,\sigma}}^2+\nu \|\nabla_L \grave\omega(t)\|_{\mathcal{G}^{s(\beta),\grave\lambda,\sigma}}^2\\
  \lesssim& (\|\grave u\|_{L^2}+\|\grave \omega\|_{\mathcal{G}^{s(\beta),\grave\lambda,\sigma}})(\|\nabla^{\frac{s}{2}}\grave \omega_{\sim N}\|_{\mathcal{G}^{s(\beta),\grave\lambda,\sigma}}^2+\|\grave \omega(t)\|_{\mathcal{G}^{s(\beta),\grave\lambda,\sigma}}^2).
\end{align*}
As $\left\|\grave u(t)\right\|_{L^2}\le\left\|\grave u_{in}\right\|_{L^2}$, we deduce by using continuous argument that
\begin{align*}
  \sup_{t\in[0,1]} \left(\left\|\grave u(t)\right\|_{L^2}+\|\grave \omega(t)\|_{\mathcal{G}^{s(\beta),\grave \lambda,\sigma}}^2\right)\lesssim {\grave \varepsilon}^2\nu^{2\beta}.
\end{align*}

Next, we need to convert estimates on $\grave \omega$ into estimates on $f$ and the associated nonlinear coordinate system. 

Recall that $v(t,y)$ solves \eqref{eq-v}. As $\grave u^x_0=u^x_0$ and $\grave \omega_0=\omega_0$, we have
\begin{align*}
  v(t,y)-y=\frac{1}{t}\int^t_0 e^{(t-s)\nu\pa_y^2}\grave u^x_0(s,y)dy,\\
  v'(t,y)-1=-\frac{1}{t}\int^t_0 e^{(t-s)\nu\pa_y^2} \grave\omega_0(s,y)dy.
\end{align*}
It follows that
\begin{align*}
  \left(\left\|v(t,y)-y\right\|_{\mathcal{G}^{s(\beta),\grave \lambda,\sigma}}^2+\left\|v'(t,y)-1\right\|_{\mathcal{G}^{s(\beta),\grave \lambda,\sigma}}^2\right)\lesssim{\grave \varepsilon}^2\nu^{2\beta}.
\end{align*}

From the definition of $z=z(t,x,y)$, we also have Gevrey control on $z=z(t,x,y)$. We can write $z$ in terms of $\grave z$ and $y$ via
\begin{align}
  z(t,\grave z,y)=\grave z-t(v(t,y)-y),
\end{align}
and $f(t,z,v)=\grave \omega(t,\grave z(t,z,v),y(t,v))$. Therefore, in order to control the Gevrey norm of $f$ we need to solve for $\grave z,\ y$ in terms of $ z,\ v$. Writing $\alpha(y)=y-v(t,y)$, $\beta(v)=y(t,v)-v$ and $\beta(v)=\alpha(v+\beta(v))$ we may apply Lemma \ref{lem-inverse} to solve for $y(t,v)-v$ with 
\begin{align*}
  \left\|y(t,v)-v\right\|_{\mathcal{G}^{s,\lambda_2,\sigma}}\lesssim \grave \varepsilon\nu^\beta.
\end{align*}
In turn, this allows us to write 
\begin{align*}
  \left\|\grave z(t,z,v)-z\right\|_{\mathcal{G}^{s,\lambda_2,\sigma}}\lesssim \grave \varepsilon\nu^\beta.
\end{align*}
Then by Lemma \ref{lem-compo}, we can deduce
\begin{align*}
  \left\|f\right\|_{\mathcal{G}^{s,\lambda_2,\sigma}}^2\lesssim {\grave \varepsilon}^2\nu^{2\beta}.
\end{align*}
Recall the definition of the multiplier $\rmA$. With the help of Lemma \ref{lem-grow-w} and Lemma \ref{lem-g-growth}, we have
\begin{align*}
  \|\rmA(t,\nabla)f(t)\|_{L^2}^2\bigg|_{t=1}\lesssim {\grave \varepsilon}^2\nu^{2\beta}.
\end{align*}
Then by the definition of $\bar h$, $q$, and $h$, one can easily check that
\begin{align*}
  \mathcal E^{s}(1)\le \varepsilon^2\nu^{2\beta}.
\end{align*}
\end{proof}

\end{appendix}

\section*{Acknowledgements}
The work of N. M. is supported by NSF grant DMS-1716466 and by Tamkeen under the NYU Abu Dhabi Research Institute grant of the center SITE. 

\bibliographystyle{siam.bst} 
\bibliography{references.bib}

\begin{thebibliography}{10}

\bibitem{DRM2021}
{\sc D.~Albritton, R.~Beekie, and M.~Novack}, {\em Enhanced dissipation and
  h\"ormander's hypoellipticity}, arXiv preprint arXiv:2105.12308,  (2021).

\bibitem{Alinhac2001}
{\sc S.~Alinhac}, {\em The null condition for quasilinear wave equations in two
  space dimensions {I}}, Invent. Math., 145 (2001), ~597--618.

\bibitem{AntonelliDolceMarcati2020}
{\sc P.~Antonelli, M.~Dolce, and P.~Marcati}, {\em Linear stability analysis
  for 2d shear flows near couette in the isentropic compressible euler
  equations}, arXiv preprint arXiv:2003.01694,  (2020).

\bibitem{AntonelliDolceMarcati2021}
{\sc P.~Antonelli, M.~Dolce, and P.~Marcati}, {\em Linear stability analysis of
  the homogeneous {C}ouette flow in a 2{D} isentropic compressible fluid}, Ann.
  PDE, 7 (2021), ~Paper No. 24, 53.

\bibitem{BCD2011}
{\sc H.~Bahouri, J.-Y. Chemin, and R.~Danchin}, {\em Fourier analysis and
  nonlinear partial differential equations}, vol.~343 of Grundlehren der
  mathematischen Wissenschaften [Fundamental Principles of Mathematical
  Sciences], Springer, Heidelberg, 2011.

\bibitem{bedrossian2021nonlinear}
{\sc J.~Bedrossian, R.~Bianchini, M.~C. Zelati, and M.~Dolce}, {\em Nonlinear
  inviscid damping and shear-buoyancy instability in the two-dimensional
  boussinesq equations}, arXiv preprint arXiv:2103.13713,  (2021).

\bibitem{BC2017}
{\sc J.~Bedrossian and M.~Coti~Zelati}, {\em Enhanced dissipation,
  hypoellipticity, and anomalous small noise inviscid limits in shear flows},
  Arch. Ration. Mech. Anal., 224 (2017), ~1161--1204.

\bibitem{BCV2017}
{\sc J.~Bedrossian, M.~Coti~Zelati, and V.~Vicol}, {\em Vortex
  axisymmetrization, inviscid damping, and vorticity depletion in the
  linearized 2{D} {E}uler equations}, Ann. PDE, 5 (2019), ~Paper No. 4, 192.

\bibitem{BGM2015}
{\sc J.~Bedrossian, P.~Germain, and N.~Masmoudi}, {\em Dynamics near the
  subcritical transition of the 3d couette flow ii: Above threshold case.},
  arXiv preprint arXiv:1506.03721,  (2015).

\bibitem{BGM2017}
{\sc J.~Bedrossian, P.~Germain, and N.~Masmoudi}, {\em On the stability
  threshold for the 3{D} {C}ouette flow in {S}obolev regularity}, Ann. of Math.
  (2), 185 (2017), ~541--608.

\bibitem{BGM2020}
{\sc J.~Bedrossian, P.~Germain, and N.~Masmoudi}, {\em Dynamics near the
  subcritical transition of the 3{D} {C}ouette flow {I}: {B}elow threshold
  case}, Mem. Amer. Math. Soc., 266 (2020), ~v+158.

\bibitem{BM2015}
{\sc J.~Bedrossian and N.~Masmoudi}, {\em Inviscid damping and the asymptotic
  stability of planar shear flows in the 2{D} {E}uler equations}, Publ. Math.
  Inst. Hautes \'{E}tudes Sci., 122 (2015), ~195--300.

\bibitem{BMM2016}
{\sc J.~Bedrossian, N.~Masmoudi, and C.~Mouhot}, {\em Landau damping:
  paraproducts and {G}evrey regularity}, Ann. PDE, 2 (2016), ~Art. 4, 71.

\bibitem{BMV2016}
{\sc J.~Bedrossian, N.~Masmoudi, and V.~Vicol}, {\em Enhanced dissipation and
  inviscid damping in the inviscid limit of the {N}avier-{S}tokes equations
  near the two dimensional {C}ouette flow}, Arch. Ration. Mech. Anal., 219
  (2016), ~1087--1159.

\bibitem{BVW2018}
{\sc J.~Bedrossian, V.~Vicol, and F.~Wang}, {\em The {S}obolev stability
  threshold for 2{D} shear flows near {C}ouette}, J. Nonlinear Sci., 28 (2018),
  ~2051--2075.

\bibitem{Bony1981}
{\sc J.-M. Bony}, {\em Calcul symbolique et propagation des singularit\'{e}s
  pour les \'{e}quations aux d\'{e}riv\'{e}es partielles non lin\'{e}aires},
  Ann. Sci. \'{E}cole Norm. Sup. (4), 14 (1981), ~209--246.

\bibitem{BM2010}
{\sc F.~Bouchet and H.~Morita}, {\em Large time behavior and asymptotic
  stability of the 2{D} {E}uler and linearized {E}uler equations}, Phys. D, 239
  (2010), ~948--966.

\bibitem{Case1960}
{\sc K.~M. Case}, {\em Stability of inviscid plane {C}ouette flow}, Phys.
  Fluids, 3 (1960), ~143--148.

\bibitem{CLWZ2020}
{\sc Q.~Chen, T.~Li, D.~Wei, and Z.~Zhang}, {\em Transition threshold for the
  2-{D} {C}ouette flow in a finite channel}, Arch. Ration. Mech. Anal., 238
  (2020), ~125--183.

\bibitem{chen2019linear}
{\sc Q.~Chen, D.~Wei, and Z.~Zhang}, {\em Linear stability of pipe poiseuille
  flow at high reynolds number regime}, arXiv preprint arXiv:1910.14245,
  (2019).

\bibitem{ChenWeiZhang2020}
{\sc Q.~Chen, D.~Wei, and Z.~Zhang}, {\em Transition threshold for the 3{D}
  couette flow in a finite channel}, arXiv preprint arXiv:2006.00721,  (2020).

\bibitem{Coti2020}
{\sc M.~Coti~Zelati}, {\em Stable mixing estimates in the infinite {P}\'{e}clet
  number limit}, J. Funct. Anal., 279 (2020), ~108562, 25.

\bibitem{CotiElgindiWidmayer2020}
{\sc M.~Coti~Zelati, T.~M. Elgindi, and K.~Widmayer}, {\em Enhanced dissipation
  in the {N}avier-{S}tokes equations near the {P}oiseuille flow}, Comm. Math.
  Phys., 378 (2020), ~987--1010.

\bibitem{DengWuZhang2021}
{\sc W.~Deng, J.~Wu, and P.~Zhang}, {\em Stability of {C}ouette flow for 2{D}
  {B}oussinesq system with vertical dissipation}, J. Funct. Anal., 281 (2021),
  ~Paper No. 109255, 40.

\bibitem{DM2018}
{\sc Y.~Deng and N.~Masmoudi}, {\em Long time instability of the couette flow
  in low gevrey spaces}, arXiv preprint arXiv:1803.01246,  (2018).

\bibitem{ding2020enhanced}
{\sc S.~Ding and Z.~Lin}, {\em Enhanced dissipation and transition threshold
  for the 2-d plane poiseuille flow via resolvent estimate}, arXiv preprint
  arXiv:2008.10057,  (2020).

\bibitem{GNRS2020}
{\sc E.~Grenier, T.~T. Nguyen, F.~Rousset, and A.~Soffer}, {\em Linear inviscid
  damping and enhanced viscous dissipation of shear flows by using the
  conjugate operator method}, J. Funct. Anal., 278 (2020), ~108339, 27.

\bibitem{He2021}
{\sc S.~He}, {\em Enhanced dissipation, hypoellipticity for passive scalar
  equations with fractional dissipation}, arXiv:2103.07906,  (2021).

\bibitem{IMM2019}
{\sc S.~Ibrahim, Y.~Maekawa, and N.~Masmoudi}, {\em On pseudospectral bound for
  non-selfadjoint operators and its application to stability of {K}olmogorov
  flows}, Ann. PDE, 5 (2019), ~Paper No. 14, 84.

\bibitem{IJ2020}
{\sc A.~Ionescu and H.~Jia}, {\em Nonlinear inviscid damping near monotonic
  shear flows}, arXiv preprint arXiv:2001.03087,  (2020).

\bibitem{IonescuJia2020cmp}
{\sc A.~D. Ionescu and H.~Jia}, {\em Inviscid damping near the {C}ouette flow
  in a channel}, Comm. Math. Phys., 374 (2020), ~2015--2096.

\bibitem{IonescuJia2021}
{\sc A.~D. Ionescu and H.~Jia}, {\em Axi-symmetrization near point vortex
  solutions for the 2{D} {E}uler equation}, Comm. Pure Appl. Math., 75 (2022),
  ~818--891.

\bibitem{Jia2020arma}
{\sc H.~Jia}, {\em Linear inviscid damping in {G}evrey spaces}, Arch. Ration.
  Mech. Anal., 235 (2020), ~1327--1355.

\bibitem{Jia2020siam}
{\sc H.~Jia}, {\em Linear inviscid damping near monotone shear flows}, SIAM J.
  Math. Anal., 52 (2020), ~623--652.

\bibitem{Jiahao2022}
{\sc H.~Jia}, {\em Uniform linear inviscid damping and enhanced dissipation
  near monotonic shear flows in high reynolds number regime (i): the whole
  space case}, 2022.

\bibitem{Kelvin1887}
{\sc L.~Kelvin}, {\em Stability of fluid motion: rectilinear motion of viscous
  fluid between two parallel plates}, Phil. Mag, 24 (1887), ~188--196.

\bibitem{Landau1946}
{\sc L.~Landau}, {\em On the vibrations of the electronic plasma}, Acad. Sci.
  USSR. J. Phys., 10 (1946), ~25--34.

\bibitem{LMZ2022}
{\sc H.~Li, N.~Masmoudi, and W.~Zhao}, {\em New energy method in the study of
  the instability near couette flow}, arXiv preprint arXiv:2203.10894,  (2022).

\bibitem{li2021metastability}
{\sc H.~Li and W.~Zhao}, {\em Metastability for the dissipative
  quasi-geostrophic equation and the non-local enhancement}, arXiv preprint
  arXiv:2107.10594,  (2021).

\bibitem{LiWeiZhang2020}
{\sc T.~Li, D.~Wei, and Z.~Zhang}, {\em Pseudospectral bound and transition
  threshold for the 3{D} {K}olmogorov flow}, Comm. Pure Appl. Math., 73 (2020),
  ~465--557.

\bibitem{LWZZ2020}
{\sc Z.~Lin, D.~Wei, Z.~Zhang, and H.~Zhu}, {\em The number of traveling wave
  families in a running water with coriolis force}, arXiv preprint
  arXiv:2009.05733,  (2020).

\bibitem{LinXu2019}
{\sc Z.~Lin and M.~Xu}, {\em Metastability of {K}olmogorov flows and inviscid
  damping of shear flows}, Arch. Ration. Mech. Anal., 231 (2019), ~1811--1852.

\bibitem{LinZeng2011}
{\sc Z.~Lin and C.~Zeng}, {\em Inviscid dynamical structures near {C}ouette
  flow}, Arch. Ration. Mech. Anal., 200 (2011), ~1075--1097.

\bibitem{Liss2020cmp}
{\sc K.~Liss}, {\em On the {S}obolev stability threshold of 3{D} {C}ouette flow
  in a uniform magnetic field}, Comm. Math. Phys., 377 (2020), ~859--908.

\bibitem{LMZZ2021}
{\sc H.~Liu, N.~Masmoudi, C.~Zhai, and W.~Zhao}, {\em Linear damping and
  depletion in flowing plasma with strong sheared magnetic fields}, Journal de
  Math\'ematiques Pures et Appliqu\'ees, 158 (2022), ~1--41.

\bibitem{MSZ2020}
{\sc N.~Masmoudi, B.~Said-Houari, and W.~Zhao}, {\em Stability of couette flow
  for 2d boussinesq system without thermal diffusivity}, Archive for Rational
  Mechanics and Analysis, 245 (2022), ~645--752.

\bibitem{MasmoudiZhaiZhao2022}
{\sc N.~Masmoudi, C.~Zhai, and W.~Zhao}, {\em Asymptotic stability for
  two-dimensional boussinesq systems around the couette flow in a finite
  channel}, arXiv preprint arXiv:2201.06832,  (2022).

\bibitem{MasmoudiZhao2020cpde}
{\sc N.~Masmoudi and W.~Zhao}, {\em Enhanced dissipation for the 2{D} {C}ouette
  flow in critical space}, Comm. Partial Differential Equations, 45 (2020),
  ~1682--1701.

\bibitem{MasmoudiZhao2020}
{\sc N.~Masmoudi and W.~Zhao}, {\em Nonlinear inviscid damping for a class of
  monotone shear flows in finite channel}, arXiv preprint arXiv:2001.08564,
  (2020).

\bibitem{MasmoudiZhao2019}
{\sc N.~Masmoudi and W.~Zhao}, {\em Stability threshold of two-dimensional
  couette flow in sobolev spaces}, Annales de l'Institut Henri Poincar\'e C,
  Analyse Non lin\'eaire, 39 (2022), ~245--325.

\bibitem{MouhotVillani2011}
{\sc C.~Mouhot and C.~Villani}, {\em On {L}andau damping}, Acta Math., 207
  (2011), ~29--201.

\bibitem{Orr1907}
{\sc W.~Orr}, {\em Mcf. stability and instability of steady motions of a
  perfect liquid}, Proc. Ir. Acad. Sect. A, Math Astron. Phys. Sci, 27 (1907),
  ~66.

\bibitem{RenWeiZhang-siam-2022}
{\sc S.~Ren, D.~Wei, and Z.~Zhang}, {\em Long time behavior of {A}lfv\'{e}n
  waves in flowing plasma: the destruction of the magnetic island}, SIAM J.
  Math. Anal., 53 (2021), ~5548--5579.

\bibitem{RenZhao2017}
{\sc S.~Ren and W.~Zhao}, {\em Linear damping of {A}lfv\'{e}n waves by phase
  mixing}, SIAM J. Math. Anal., 49 (2017), ~2101--2137.

\bibitem{RosSat1966}
{\sc S.~Rosencrans and D.~Sattinger}, {\em On the spectrum of an operator
  occurring in the theory of hydrodynamic stability}, Journal of Mathematics
  and Physics, 45 (1966), ~289--300.

\bibitem{Stepin1995}
{\sc S.~A. Stepin}, {\em Nonself-adjoint friedrichs model in hydrodynamic
  stability}, Functional analysis and its applications, 29 (1995), ~91--101.

\bibitem{Vanneste2001}
{\sc J.~Vanneste}, {\em Nonlinear dynamics of anisotropic disturbances in plane
  {C}ouette flow}, SIAM J. Appl. Math., 62 (2001/02), ~924--944.

\bibitem{VMW1998}
{\sc J.~Vanneste, P.~Morrison, and T.~Warn}, {\em Strong echo effect and
  nonlinear transient growth in shear flows}, Physics of Fluids, 10 (1998),
  ~1398--1404.

\bibitem{WangZhangZhu2020}
{\sc L.~Wang, Z.~Zhang, and H.~Zhu}, {\em Dynamics near couette flow for the
  $\beta$-plane equation}, arXiv preprint arXiv:2202.05708,  (2022).

\bibitem{Wei2021}
{\sc D.~Wei}, {\em Diffusion and mixing in fluid flow via the resolvent
  estimate}, Sci. China Math., 64 (2021), ~507--518.

\bibitem{WeiZhang2020}
{\sc D.~Wei and Z.~Zhang}, {\em Transition threshold for the 3d couette flow in
  sobolev space}, Communications on Pure and Applied Mathematics,  (2020).

\bibitem{WeiZhangZhao2018}
{\sc D.~Wei, Z.~Zhang, and W.~Zhao}, {\em Linear inviscid damping for a class
  of monotone shear flow in {S}obolev spaces}, Comm. Pure Appl. Math., 71
  (2018), ~617--687.

\bibitem{WeiZhangZhao2019}
{\sc D.~Wei, Z.~Zhang, and W.~Zhao}, {\em Linear inviscid damping and vorticity
  depletion for shear flows}, Ann. PDE, 5 (2019), ~Paper No. 3, 101.

\bibitem{WeiZhangZhao2020}
{\sc D.~Wei, Z.~Zhang, and W.~Zhao}, {\em Linear inviscid damping and enhanced
  dissipation for the {K}olmogorov flow}, Adv. Math., 362 (2020), ~106963, 103.

\bibitem{WeiZhangZhu2020cmp}
{\sc D.~Wei, Z.~Zhang, and H.~Zhu}, {\em Linear inviscid damping for the
  $\beta$-plane equation}, Communications in Mathematical Physics, 375 (2020),
  ~127--174.

\bibitem{YangLin2018}
{\sc J.~Yang and Z.~Lin}, {\em Linear inviscid damping for {C}ouette flow in
  stratified fluid}, J. Math. Fluid Mech., 20 (2018), ~445--472.

\bibitem{YD2002}
{\sc J.~Yu and C.~Driscoll}, {\em Diocotron wave echoes in a pure electron
  plasma}, IEEE transactions on plasma science, 30 (2002), ~24--25.

\bibitem{YDO2005}
{\sc J.~Yu, C.~Driscoll, and T.~O'Neil}, {\em Phase mixing and echoes in a pure
  electron plasma}, Physics of plasmas, 12 (2005), ~055701.

\bibitem{ZengZhangZi2021}
{\sc L.~Zeng, Z.~Zhang, and R.~Zi}, {\em Linear stability of the couette flow
  in the 3d isentropic compressible navier-stokes equations}, arXiv preprint
  arXiv:2105.10200,  (2021).

\bibitem{ZhaiZhangZhao2021}
{\sc C.~Zhai, Z.~Zhang, and W.~Zhao}, {\em Long-time behavior of alfv{\'{e}}n
  waves in a flowing plasma: Generation of the magnetic island}, Archive for
  Rational Mechanics and Analysis, 242 (2021), ~1317--1394.

\bibitem{ZhaiZhao2022}
{\sc C.~Zhai and W.~Zhao}, {\em Stability threshold of the couette flow for
  navier-stokes boussinesq system with large richardson number $\gamma^2>1/4$},
  arXiv preprint arXiv:2204.09662,  (2022).

\bibitem{Zillinger2017}
{\sc C.~Zillinger}, {\em Linear inviscid damping for monotone shear flows},
  Trans. Amer. Math. Soc., 369 (2017), ~8799--8855.

\end{thebibliography}

\end{document}